\documentclass[12pt, lettersize]{amsart}
\usepackage{setspace}

\usepackage{amssymb, amsmath, amsthm}   
\usepackage{mathrsfs}
\usepackage{stmaryrd}
\usepackage{newlfont}
\usepackage{amscd}
\usepackage{mathtools}
\usepackage{bm}
\usepackage[all]{xy}
\usepackage{tikz-cd}
\usepackage{graphicx}
\usepackage{hyperref}
\usepackage{framed}
\usepackage{comment}

\usepackage[normalem]{ulem}

\usepackage{enumitem}

\usepackage{textcomp}

\usepackage{amsbsy}

\usepackage[OT2, T1]{fontenc}

\newtheorem{thm}{Theorem}[section]

\newtheorem{thm-defn}[thm]{Theorem/Definition}
\newtheorem{lem}[thm]{Lemma}
\newtheorem{prop}[thm]{Proposition}
\newtheorem{cor}[thm]{Corollary}

\theoremstyle{definition}
\newtheorem{defn}[thm]{Definition}

\newtheorem{eg}[thm]{Example}

\newtheorem{construction}[thm]{Construction}
\newtheorem{set-up}[thm]{Set-up}

\theoremstyle{remark}
\newtheorem{rem}[thm]{Remark}

\numberwithin{equation}{section}

\usepackage{relsize}
\usepackage[bbgreekl]{mathbbol}
\usepackage{amsfonts}

\DeclareSymbolFontAlphabet{\mathbb}{AMSb}
\DeclareSymbolFontAlphabet{\mathbbl}{bbold}
\newcommand{\Prism}{{\mathlarger{\mathbbl{\Delta}}}}

\DeclareMathOperator{\Gal}{Gal}

\DeclareMathOperator{\Spec}{Spec}
\DeclareMathOperator{\Spa}{Spa}
\DeclareMathOperator{\Spf}{Spf}

\DeclareMathOperator{\Ker}{Ker}

\DeclareMathOperator{\Fil}{Fil}
\DeclareMathOperator{\Hom}{Hom}

\DeclareMathOperator{\Mor}{Mor}
\DeclareMathOperator{\Sh}{Sh}

\DeclareMathOperator{\Ab}{Ab}

\DeclareMathOperator{\Mod}{Mod}
\DeclareMathOperator{\Tot}{Tot}

\newcommand{\pt}{\mathrm{pt}}
\newcommand{\gp}{\mathrm{gp}}
\newcommand{\op}{\mathrm{op}}
\newcommand{\id}{\mathrm{id}}
\newcommand{\CR}{\mathrm{CR}}

\newcommand{\Vect}{\mathrm{Vect}}
\newcommand{\Loc}{\mathrm{Loc}}
\newcommand{\perf}{\mathrm{perf}}
\newcommand{\str}{\mathrm{str}}

\newcommand{\an}{\mathrm{an}}
\newcommand{\aff}{\mathrm{aff}}

\newcommand{\cris}{\mathrm{cris}}
\newcommand{\CRIS}{\mathrm{CRIS}}
\newcommand{\dCRIS}{{\delta\text{-}\mathrm{CRIS}}}

\renewcommand{\inf}{\mathrm{inf}}
\newcommand{\st}{\mathrm{st}}
\newcommand{\HK}{\mathrm{HK}}

\newcommand{\dR}{\mathrm{dR}}
\newcommand{\et}{\mathrm{\acute{e}t}}
\newcommand{\proet}{\mathrm{pro\acute{e}t}}
\newcommand{\proetdR}{\proet\mhyphen\dR}

\newcommand{\Perfd}{\mathrm{Perfd}}
\newcommand{\conv}{\mathrm{conv}}
\newcommand{\pd}{\mathrm{pd}}

\newcommand{\BK}{\mathrm{BK}}
\newcommand{\Br}{\mathrm{Br}}
\newcommand{\rel}{\mathrm{rel}}
\newcommand{\can}{\mathrm{can}}
\newcommand{\logcris}{\mathrm{log\text{-}cris}}

\newcommand{\N}{\mathbf{N}}
\newcommand{\F}{\mathbf{F}}
\newcommand{\Z}{\mathbf{Z}}
\newcommand{\Q}{\mathbf{Q}}

\newcommand{\A}{\mathbf{A}}
\newcommand{\B}{\mathbf{B}}

\newcommand{\Ainf}{\mathbf{A}_{\mathrm{inf}}}

\renewcommand{\log}{\mathrm{log}}

\newcommand{\fkm}{\mathfrak{m}}

\newcommand{\fkF}{{\mathfrak F}}

\newcommand{\fkS}{{\mathfrak S}}

\newcommand{\calD}{\mathcal{D}}
\newcommand{\calE}{\mathcal{E}}
\newcommand{\calF}{\mathcal{F}}
\newcommand{\calG}{\mathcal{G}}

\newcommand{\calI}{\mathcal{I}}
\newcommand{\calJ}{\mathcal{J}}
\newcommand{\calK}{\mathcal{K}}

\newcommand{\calM}{\mathcal{M}}
\newcommand{\calN}{\mathcal{N}}
\newcommand{\calO}{\mathcal{O}}

\newcommand{\calS}{\mathcal{S}}
\newcommand{\calT}{\mathcal{T}}

\newcommand{\calX}{\mathcal{X}}
\newcommand{\calY}{\mathcal{Y}}

\newcommand{\bA}{\mathbb{A}}
\newcommand{\bB}{\mathbb{B}}

\newcommand{\bL}{\mathbb{L}}

\newcommand{\PD}{\mathrm{PD}}

\mathchardef\mhyphen="2D

\begin{document}

\pagenumbering{arabic}

\title{A log prismatic-crystalline comparison theorem}

\author{Heng Du} 
\address{Yau Mathematical Sciences Center, Tsinghua University, Beijing 100084, China}
\email{hengdu@mail.tsinghua.edu.cn}
\author{Yong Suk Moon} 
\address{Beijing Institute of Mathematical Sciences and Applications, Beijing 101408, China}
\email{ysmoon@bimsa.cn}
\author{Koji Shimizu}
\address{Yau Mathematical Sciences Center, Tsinghua University, Beijing 100084, China;
Beijing Institute of Mathematical Sciences and Applications, Beijing 101408, China}
\email{shimizu@tsinghua.edu.cn}

\begin{abstract} 
We show a comparison theorem between log prismatic cohomology and log crystalline cohomology for a $p$-adic formal scheme with semistable reduction. Combined with the prismatic-\'etale comparison theorem recently proved by Tian \cite{Tian-prism-etale-comparison}, this implies the $C_{\st}$-conjecture in the semistable case with coefficients given by semistable local systems. 
\end{abstract}

\maketitle

\tableofcontents

\section{Introduction} \label{sec: intro}

Fix a prime $p$, and let $K$ be a complete discretely valued field over $\Q_p$ whose residue field $k$ is perfect. Let $K_0 \coloneqq W(k)[p^{-1}]$. Understanding the relation among various $p$-adic cohomology theories for varieties over $K$ has been of great interest. The $C_{\st}$-conjecture of Fontaine--Jannsen compares the $p$-adic \'etale cohomology and Hyodo--Kato cohomology for a proper semistable scheme $X$ over $\calO_K$:
\begin{equation} \label{eq: Cst-const-coeff}
H^i_{\et}(X_{\overline{K}}, \Z_p)\otimes_{\Z_p} B_{\st} \cong H^i_{\HK}(X_0)\otimes_{K_0} B_{\st}.    
\end{equation}
Here, $X_0\coloneqq X\otimes_{\calO_K}  k$, and $B_{\st}$ is the semistable period ring constructed by Fontaine, and the above isomorphism is compatible with Galois actions, Frobenii, monodromy operators, as well as filtrations after base change to the de Rham period ring $B_{\dR}$ (see below for the detail). This conjecture and its variants have been proved in various works, including \cite{Tsuji-Cst, faltings-almostetale, niziol-semist-conj-K-theory, andreatta-iovita-semistable-relative, bhatt-derived-deRham-cohom, Beilinson-crystalline-period-map, colmez-niziol, Cesnavicius-Koshikawa}. 

In studying geometric families of $p$-adic Galois representations, one is naturally led to generalize the $C_{\st}$-conjecture~\eqref{eq: Cst-const-coeff} to allow certain coefficients. The goal of this paper is to give a $C_\st$-type comparison theorem for coefficients formulated as follows.

\begin{thm}[{Theorem~\ref{thm: Cst-conj-main}, based on \cite[Thm.~0.4]{Tian-prism-etale-comparison} and Theorems~\ref{thm: prism-cris-comparison-BK-case-intro}, \ref{thm:Hyodo-Kato-complex-intro}}] \label{thm: Cst-intro}
Let $(X, M_X)$ be a proper semistable $p$-adic formal scheme over $\mathcal{O}_K$. Let $\mathbb{L}$ be a semistable $\mathbf{Z}_p$-local system on the generic fiber $X_{K}$, and let $\mathcal{E}_{\cris, \Q}$ be the corresponding filtered $F$-isocrystal on $X_0$. Then there exists an isomorphism
\[
\alpha_\st\colon R\Gamma(X_{C, \et}, \mathbb{L})\otimes_{\Z_p}^L B_{\st} \xrightarrow{\cong} R\Gamma_{\HK}((X_0, M_{X_0}), \mathcal{E}_{\cris, \Q})\otimes_{K_0}^L B_{\st}
\]
compatible with Galois actions, Frobenii, and monodromy actions, as well as filtrations after base change to $B_\dR$.
\end{thm}

In the above statement, $M_X$ denotes the log structure $\calO_X\cap (\calO_X[p^{-1}])^\times$, and $C$ is the completion of an algebraic closure $\overline{K}$ of $K$, and $X_C\coloneqq X_K\times_KC$.

The notion of \emph{semistable $\mathbf{Z}_p$-local systems on $X_{K}$} articulates the families of semistable $p$-adic Galois representations parametrized by $X_{K}$ and is defined via the notion of association to a \emph{(filtered) $F$-isocrystal} (see \cite[Def.~3.39]{du-liu-moon-shimizu-purity-F-crystal} or \S~\ref{subsec: semist-local-syst}). 

Crucial to Theorem~\ref{thm: Cst-intro} is the \emph{Hyodo--Kato complex} $R\Gamma_{\HK}((X_0, M_{X_0}), \mathcal{E}_{\cris, \Q})$. Consider the crystalline cohomology $D=R\Gamma(((X_0, M_{X_0})/(W(k),M_0))_\CRIS, \mathcal{E}_{\cris, \Q})$ where the log structure $M_0$ on $W(k)$ is given by $\N\rightarrow W(k)$, $1\mapsto 0$. We show 
\begin{itemize}
 \item[(HK1)] $D$ is a perfect complex of $K_0$-vector spaces with a Frobenius $\phi_D$;
 \item[(HK2)] $\phi_D$ induces an isomorphism $D\otimes_{K_0,\phi}^L K_0\xrightarrow{\cong} D$;
 \item[(HK3)] $D$ admits a monodromy operator $N_D$ satisfying $p\phi_D N_D=N_D\phi_D$;
 \item[(HK4)] given a uniformizer $\pi$, there is an isomorphism $\rho_\pi\colon D\otimes_{K_0}^L K\cong R\Gamma_{\dR}(X_K, \mathbf{E})$ where $\mathbf{E}$ is a filtered vector bundle with integral connection on $X_K$ associated to $\calE_{\cris, \Q}$;
 for another $\pi'$, we have $\rho_{\pi'}=\rho_{\pi}\circ \exp(\operatorname{log}(\pi'/\pi) [K:K_0]N)$.
\end{itemize}
The Hyodo--Kato complex consists of the quadruple $(D,\phi_D,N_D,\rho_\pi)$. Such a formalism already appears in the classical $C_\st$-conjecture, which we refer to as $H^i_{\HK}(X_0)$ above (cf.~\cite[\S~5]{jannsen}, \cite[Conj.~6.2.7]{fontaine-exposeIII}, \cite[Intro.]{hyodo-kato}, \cite[Conj.~1.1]{Kato-exposeVI}).

\begin{rem}
We have been informed that in an upcoming work of Inoue--Koshikawa, the $C_{\st}$-conjecture in a more general set-up than Theorem~\ref{thm: Cst-intro} is proved. For example, the comparison between log prismatic cohomology and log crystalline cohomology as in Theorem~\ref{thm: pris-crys-comparison-associated-sheaves} were already known to them.   
\end{rem}

In this paper, we prove Theorem~\ref{thm: Cst-intro} using log prismatic cohomology theory. The $A_{\inf}$-cohomology by Bhatt--Morrow--Scholze \cite{bhatt-morrow-scholze-integralpadic} and prismatic cohomology by Bhatt--Scholze \cite{bhatt-scholze-prismaticcohom} have been major breakthroughs in integral $p$-adic Hodge theory. When $X$ is a proper smooth $p$-adic formal scheme over $\calO_K$, these cohomology theories bridge \'etale, crystalline, and de Rham cohomology theories. In particular, each of these works proved the $C_{\cris}$-conjecture for the constant coefficient case. When $X$ is a proper semistable $p$-adic formal scheme over $\calO_K$, \v{C}esnavi\v{c}ius--Koshikawa \cite{Cesnavicius-Koshikawa} obtained the $C_{\st}$-conjecture in the constant coefficient case using $A_{\inf}$-cohomology. See \cite{diao-yao} and \cite{koshikawa-yao} for further developments in these directions.

To go beyond the constant coefficient, Faltings \cite{faltings-almostetale} and Andreatta--Iovita \cite{Andreatta-Iovita-smooth, andreatta-iovita-semistable-relative} worked on the $C_\st$-conjecture with coefficients using the Faltings site. Guo--Reinecke \cite{GuoReinecke-Ccris} generalized the prismatic-\'etale comparison and prismatic-crystalline comparison in \cite{bhatt-scholze-prismaticcohom} to allow coefficients given by crystalline local systems on $X_{K}$ in the good reduction case, and thereby proved a $C_{\cris}$-conjecture with coefficients. 

This paper studies the interplay between log prismatic theory by Koshikawa \cite{koshikawa} and log crystalline theory in the semistable case; our main goal is to prove a log prismatic-crystalline comparison theorem with coefficients and study the Hyodo--Kato complex. Tian \cite{Tian-prism-etale-comparison} recently established a prismatic-\'etale comparison with coefficients given by semistable local systems. Combining these results will give Theorem~\ref{thm: Cst-intro}. Let us explain these results in detail now.
 
\medskip
\noindent \textbf{Logarithmic prismatic theory.}  
Let $X$ be a bounded $p$-adic formal scheme. The absolute prismatic site $X_{\Prism}$ consists of bounded prisms $(A, I)$ equipped with a map $\Spf(A/I) \rightarrow X$. Koshikawa \cite{koshikawa} developed logarithmic prismatic theory to generalize this to log formal schemes. A \emph{log prism} is given by a bounded prism $(A, I)$ and a log structure $M_{\Spf A}$ on $\Spf A$ equipped with a $\delta_\log$-structure. For a bounded integral $p$-adic log formal scheme $(X,M_X)$, one can define the \emph{absolute (strict) prismatic site} $(X,M_X)_\Prism$ consisting of log prisms over $(X,M_X)$ with strict log structures as well as its relative analogue (see Definitions~\ref{def: absolute-prismatic-site} and \ref{def: relative-prismatic-site}).

\begin{eg}[Breuil--Kisin log prism; Example~\ref{eg: Breuil-Kisin prism}] \label{eg:intro-BK-log-prisms}
Assume $X = \Spf R^0$ with 
\[
R^0 = \mathcal{O}_K \langle T_1, \ldots, T_m, T_{m+1}^{\pm 1}, \ldots, T_d^{\pm 1}\rangle / (T_1\cdots T_m - \pi),
\]
where $\pi \in \calO_K$ is a uniformizer with monic irreducible polynomial $E(u) \in W(k)[u]$. Let $M_X$ be the log structure on $X$ given by $\mathbf{N}^d \rightarrow R^0$, $e_i \mapsto T_i$ ($1\leq i\leq d$). Consider 
\[
\mathfrak{S} \coloneqq W(k) \langle T_1, \ldots, T_m, T_{m+1}^{\pm 1}, \ldots, T_d^{\pm 1}\rangle[\![u]\!] / (T_1\cdots T_m - u)
\]
equipped with Frobenius given by $\phi(T_i) = T_i^p$. Then we have $(\mathfrak{S}, (E(u))) \in X_{\Prism}$ with the structure map $\mathfrak{S}/(E(u)) \cong R^0$. The prelog structure $\mathbf{N}^d \rightarrow \mathfrak{S}$, $e_i \mapsto T_i$, further gives rise to a log prism in $(X, M_X)_{\Prism}$, called the \emph{Breuil--Kisin log prism}. The same construction works when $X$ is \emph{small affine}, i.e., when $X$ is affine and admits an \'etale map to $\Spf R_0$. The Breuil--Kisin log prism will play a key role in our local study. 
\end{eg}

\medskip
\noindent \textbf{Prismatic $F$-crystals and semistable local systems.}
The log prismatic site $(X,M_X)_\Prism$ admits a structure sheaf $\calO_\Prism$ of rings together with a Frobenius endomorphism $\phi$ and an ideal sheaf $\calI_\Prism$, given by assigning $A$, $\phi_A$, and $I$ to each log prism $(A,I, M_{\Spf A})$. Based on these structures, we can define various $F$-crystals on $X_\Prism$ and relate them to \'etale $\Z_p$-local systems on the generic fiber $X_K$.

From now on, assume that $(X, M_X)$ is a \emph{semistable} $p$-adic formal scheme over $\calO_K$. To study semistable local systems on $X_{K}$, the current authors and Tong Liu \cite{du-liu-moon-shimizu-purity-F-crystal} consider the category $\Vect^{\an, \phi}((X,M_X)_\Prism)$ of \emph{analytic prismatic $F$-crystals} (following \cite{GuoReinecke-Ccris} in the smooth case). It consists of pairs $(\calE, \phi_{\calE})$ where $\calE$ assigns to each $(A, I, M_{\Spf A})\in (X, M_X)_\Prism$ a vector bundle $\calE_A$ over $\Spec A \smallsetminus V(p,I)$, and $\phi_{\calE}$ is an isomorphism $(\phi_A^\ast\calE_A)[I^{-1}]\xrightarrow{\cong}\calE_A[I^{-1}]$. Let $(X_1,M_{X_1})_\CRIS$ denote the absolute logarithmic crystalline site of $(X_1, M_{X_1})\coloneqq (X,M_X)\otimes_{\Z_p}\Z_p/p$ (see Appendix~\ref{subsec:crystalline-sites}). To $(\calE_{\Prism}, \phi_{\calE_{\Prism}}) \in \Vect^{\an, \phi}((X,M_X)_\Prism)$, one can associate a $\Z_p$-local system on $X_{K}$ (\emph{\'etale realization}) and an $F$-isocrystal $(\calE_{\cris, \Q}, \phi_{\calE_{\cris, \Q}})$ on $(X_1,M_{X_1})_\CRIS$ (\emph{crystalline realization}); see \cite[\S~3.3, 3.5]{du-liu-moon-shimizu-purity-F-crystal} or \S~\ref{subsec: semist-local-syst}.
Write $\Loc_{\Z_p}^{\st}(X_{K})$ for the category of semistable $\Z_p$-local systems on $X_{K}$. The following is a main result of \cite{du-liu-moon-shimizu-purity-F-crystal}.

\begin{thm}[{\cite{du-liu-moon-shimizu-purity-F-crystal}}; see Theorem~\ref{thm:DLMS-association-via-prismatic-F-crystal}]  \label{thm: analy-prism-F-cryst-semist-loc-syst}
The \'etale realization induces an equivalence between $\Vect^{\an, \phi}((X,M_X)_\Prism)$ and $\Loc_{\Z_p}^{\st}(X_{K})$.
Furthermore, if $\bL \in \Loc_{\Z_p}(X_{K})^{\st}$ is a semistable $\Z_p$-local system and $(\calE_{\Prism}, \phi_{\calE_{\Prism}}) \in \Vect^{\an, \phi}((X,M_X)_\Prism)$ is the corresponding analytic prismatic $F$-crystal, then $\bL$ and the crystalline realization $(\calE_{\cris, \Q}, \phi_{\calE_{\cris, \Q}})$ of $(\calE_{\Prism}, \phi_{\calE_{\Prism}})$ are associated in the sense of \cite[Def.~3.39]{du-liu-moon-shimizu-purity-F-crystal}.
\end{thm}

Hence one needs to compare the prismatic cohomology of $\calE_\Prism$ with the \'etale cohomology of $\mathbb{L}$ and the crystalline cohomology of $(\calE_{\cris, \Q}, \phi_{\calE_{\cris, \Q}})$. At this point, it is convenient to consider two variants of analytic prismatic $F$-crystals.

A \emph{completed prismatic $F$-crystal} is a sheaf $\calG$ of $\calO_\Prism$-modules together with a Frobenius structure on $\calG[\calI_\Prism^{-1}]$ such that for every $(A,I,M_{\Spf A})$, $\calG_A\coloneqq \calG(A,I,M_{\Spf A})$ is a finitely generated classically $(p,I)$-complete $A$-module and restricts to a vector bundle on $\Spec(A)\smallsetminus V(p,I)$ and such that for every map $(A,I,M_{\Spf A})\rightarrow (B,IB,M_{\Spf B})$, the induced map $\calG_A\widehat{\otimes}_AB\rightarrow \calG_B$ from the $(p,I)$-completed tensor product is an isomorphism. This notion is first introduced in \cite{du-liu-moon-shimizu-completed-prismatic-F-crystal-loc-system}, and used in \cite{Tian-prism-etale-comparison} (cf. \cite[Def.~1.13]{Tian-prism-etale-comparison}; see also \cite[Def.~11.1]{Tsuji-prismatic-q-Higgs} for relevant notions). Using the Breuil--Kisin prism and its self-coproducts, one can associate to $\calE_\Prism$ a completed prismatic $F$-crystal $j_\ast \calE_{\Prism}$ such that $(\calE_{\Prism})_A=(j_\ast \calE_\Prism)_A|_{\Spec A\smallsetminus V(p,I)}$. Furthermore, using the fact that each Breuil--Kisin prism underlies a regular Noetherian ring, one associates to $j_\ast \calE_\Prism$ a \emph{prismatic $F$-crystal in perfect complexes} $i(\calE_\Prism)$ together with a morphism $i(\calE_\Prism)\rightarrow j_\ast \calE_\Prism$ (see \S~\ref{sec: analy-pris-cryst-perf-cx}, \ref{subsec: analy-prism-crys-perf-cx}). We study the prismatic cohomology of $j_\ast \calE_\Prism$ and $i(\calE_\Prism)$. 

Let $(\fkS_K, (E), M_{\Spf\fkS_K})$ be the Breuil--Kisin log prism in $(\Spf\calO_K, M_{\Spf\calO_K})_{\Prism}$. The $p$-completed PD-envelope $S_K$ of $\mathfrak{S}_K$ with respect to $(E)$ defines the \emph{Breuil log prism} $(S_K, (p), M_{\Spf S_K})\coloneqq (S_K, (p), \phi_* \mathbf{N})^a$, and $\phi\colon \frak{S}_K \rightarrow S_K$ induces a map of log prisms $(\fkS_K, (E), M_{\Spf \fkS_K}) \rightarrow (S_K, (p), M_{\Spf S_K})$ in $(\Spf \mathcal{O}_K, M_{\Spf \mathcal{O}_K})_{\Prism}$ (cf.~Example~\ref{eg: examples-log-prisms}).

For any $p$-adically complete $\calO_K$-algebra $B$, set $X_B\coloneqq X\times_{\Spf \calO_K}\Spf B$ with pullback log structure $M_{X_B}$.
For $\calK=j_\ast \calE_\Prism, i(\calE_\Prism)$, consider its \emph{Breuil--Kisin cohomology} and \emph{Breuil cohomology}:
\begin{itemize}
\item $R\Gamma_{\BK}(X, \calK)\coloneqq R\Gamma(((X, M_X) / (\fkS_{K}, M_{\Spf\fkS_{K}}))_{\Prism}, \calK)$;
\item $R\Gamma_{\Br}(X, \calK)\coloneqq R\Gamma(((X_{S_K/p}, M_{X_{S_K/p}}) / (S_K, M_{\Spf S_k}))_{\Prism}, \calK)$.
\end{itemize}

By the \v{C}ech--Alexander method, we see that they satisfy the following properties:
\begin{enumerate}
 \item (\cite[Thm.~6.18]{Tian-prism-etale-comparison}) If $X$ is proper over $\calO_K$, then $R\Gamma_{\BK}(X, j_\ast \calE_\Prism)$ is a perfect complex in $\fkS_K$-modules.
 \item (Proposition~\ref{prop: BK-cohom}) $R\Gamma_{\BK}(X, i( \calE_\Prism))\rightarrow R\Gamma_{\BK}(X, j_\ast \calE_\Prism)$ is an isomorphism.
 \item (Theorem~\ref{thm: general-base-change-BK-cohom}) If $X$ is proper over $\calO_K$, the maps
 \[
 R\Gamma_{\BK}(X, i( \calE_\Prism))\otimes_{\fkS_K}^L S_K \rightarrow R\Gamma_{\Br}(X, i( \calE_\Prism))
 \]
 and
  \[
 R\Gamma_{\BK}(X, j_\ast \calE_\Prism)\otimes_{\fkS_K,\phi}^LS_K[p^{-1}] \rightarrow R\Gamma_{\Br}(X, j_\ast \calE_\Prism)\otimes_{S_K}^LS_K[p^{-1}]
 \]
 are isomorphisms.
\end{enumerate}

Here the perfectness uses the fact that $j_\ast \calE_\Prism$ underlies an $\calO_\Prism$-module; we have a stronger base change property for $i(\calE_\Prism)$ as it underlies a perfect complex in $\calO_\Prism$-modules, and the base change property for $j_\ast \calE_\Prism$ is deduced from the one for $i(\calE_\Prism)$.

\smallskip
\noindent \textbf{Prismatic-crystalline comparison and the Frobenius isogeny property.}
The Breuil cohomology for $j_{\ast}\mathcal{E}_{\Prism}$ is identified with the crystalline cohomology of the crystalline realization $\calE_{\cris, \Q}$ of $\calE_{\Prism}$, which is an $F$-isocrystal on $(X_1, M_{X_1})_{\CRIS}$.

\begin{thm}[{Theorems~\ref{thm: pris-cris-comparison-over-Breuil S} and \ref{thm:Frobenius-isogeny-property}}] \label{thm: prism-cris-comparison-BK-case-intro}
Keep the notation as above and assume that $X$ is proper over $\calO_K$.
We have an isomorphism   
\[
R\Gamma_\Br(X, j_{\ast}\mathcal{E}_{\Prism})\otimes_{S_K}^LS_K[p^{-1}]\cong R\Gamma(((X_1, M_{X_1})/(S_K,M_{S_K}))_{\CRIS}, \mathcal{E}_{\cris, \Q})
\]
of perfect complexes in $D(S_K[p^{-1}])$ that is compatible with Frobenii. Moreover, they satisfy the Frobenius isogeny property: if we write $\calM$ for the above complex, the induced map $1\otimes\phi\colon L\phi^* \calM \rightarrow \calM$ is an isomorphism.
\end{thm}

This is indeed a core result of our work, and we need to combine several different ideas.
The first part of Theorem~\ref{thm: prism-cris-comparison-BK-case-intro} is deduced from a more general form of \emph{prismatic-crystalline comparison} (Theorem~\ref{thm: pris-crys-comparison-associated-sheaves}) where we study completed prismatic crystals on a certain relative set-up. Following \cite[\S~6]{koshikawa} in the constant coefficient case, Theorem~\ref{thm: pris-crys-comparison-associated-sheaves} is proved via explicit local computations based on \v{C}ech--Alexander method. For the Frobenius isogeny property, we use a slightly ad hoc method; such a property is proved in a more general set-up for a finite locally free $F$-crystal on a crystalline site in \cite[Thm.~11.3]{du-moon-shimizu-cris-pushforward}. Since the $F$-crystal underlying $\mathcal{E}_{\cris, \Q}$ is not necessarily finite locally free, we also introduce the crystalline realization of $i(\calE_\Prism)$ in Construction~\ref{const: cryst-realization-perfect-cx}. We use the base change property of the Breuil cohomology and the comparison results of $i(\calE_\Prism)$ and $j_\ast\calE_\Prism$ as additional inputs to modify the proof of \cite{du-moon-shimizu-cris-pushforward}.

Similarly, we show a prismatic-crystalline comparison over $W(k)$ and $A_\cris$: the triple $(W(k), (p), M_0)$ gives a log prism of $(\Spf\calO_K, M_{\Spf\calO_K})_{\Prism}$  via $\calO_K \rightarrow k \cong W(k)/(p)$.
Consider the map $\theta\colon A_{\inf}\coloneqq W(\calO_C^{\flat})\rightarrow \calO_C$ lifting $\calO_C^\flat\rightarrow \calO_C/p$. Choose a compatible system $(\pi_n)_{n \geq 0}$ of $p$-power roots of $\pi$ in $C$ with $\pi_0 = \pi$ and $\pi_{n+1}^p = \pi_n$, and let $\pi^{\flat}\coloneqq (\pi_n)_{n \geq 0}$ be the corresponding element in $\calO_{C}^{\flat}$. Let $A_{\cris}$ be the $p$-completed PD-envelope of $A_{\inf}$ with respect to $(\Ker(\theta),p)$, which yields a log prism
 \[
 (A_{\cris}, (p), M_{\Spf A_{\cris}})\coloneqq (A_{\cris}, (p), (\N \rightarrow A_{\cris}, 1\mapsto [\pi^\flat]^p)^a) \in (\Spf\calO_K, M_{\Spf\calO_K})_{\Prism}
 \]
with $\calO_K\rightarrow\calO_K/p\xrightarrow{\phi}\calO_C/p\rightarrow A_\cris/p$ (cf.~Example~\ref{eg: examples-log-prisms}). Set $B_\cris^+\coloneqq A_\cris[p^{-1}]$.

\begin{thm}[cf.~Theorems~\ref{thm: pris-cris-comparison-over-Breuil S} and \ref{thm:Frobenius-isogeny-property}] 
Keep assuming that $X$ is proper over $\calO_K$. Then there exist isomorphisms of perfect complexes 
\[
R\Gamma(((X_0, M_{X_0}) / (W(k), M_0))_{\Prism}, j_{\ast}\calE_\Prism)\otimes_{W(k)}^L K_0
\cong R\Gamma(((X_0, M_{X_0})/(W(k), M_0))_{\CRIS}, \calE_{\cris,\Q})
\]
in $D(K_0)$ and 
\begin{align*}
R\Gamma(((X_0, M_{X_0}) / (A_\cris, M_{\Spf A_\cris}))_{\Prism}, &j_{\ast}\calE_\Prism)\otimes_{A_{\cris}}^L B_\cris^+\\
&\cong R\Gamma(((X_{\calO_C/p}, M_{X_{\calO_C/p}})/(A_{\cris}, M_{A_\cris}))_{\CRIS}, \calE_{\cris,\Q})
\end{align*}
in $D(B_\cris^+)$ that are compatible with Frobenii. Moreover, these perfect complexes satisfy the Frobenius isogeny property.
\end{thm}

\medskip
\noindent \textbf{The Hyodo--Kato complex.}
We are now ready to state the properties of the Hyodo--Kato complex.

\begin{thm}[Theorem~\ref{thm:Hyodo-Kato-cohomology}]\label{thm:Hyodo-Kato-complex-intro}
The complex $R\Gamma(((X_0, M_{X_0})/(W(k), M_0))_{\CRIS}, \calE_{\cris, \Q})$
underlies a complex $R\Gamma_{\HK}((X_0, M_{X_0}), \mathcal{E}_{\cris, \Q})$ of $(\phi,N)$-modules over $K_0$  satisfying properties (HK1)-(HK4) (explained after Theorem~\ref{thm: Cst-intro}). Moreover, 
\begin{itemize}
 \item[\emph{(HK5)}]
 there exists an isomorphism of perfect complexes
\begin{align*}
\rho_{\st}\colon R\Gamma(((X_0, M_{X_0})&/(W(k), M_0))_{\CRIS}, \calE_{\cris, \Q})\otimes_{K_0}^L B_\st^+ \\
&\cong  R\Gamma(((X_{\calO_C/p}, M_{X_{\calO_C/p}})/(A_{\cris}, M_{A_\cris}))_{\CRIS}, \calE_{\cris, \Q})\otimes_{B_\cris^+}^L B_\st^+
\end{align*}
  in $D(B_\st^+)$ that is compatible with $\phi$, $N$, and $\Gal(\overline{K}/K)$-actions.
\end{itemize}
\end{thm}

We use Beilinson's framework of the derived category $D_{\phi, N}(K_0)$ of $(\phi,N)$-modules over $K_0$ and Hyodo--Kato theory in \cite{Beilinson-crystalline-period-map-arXiv} to formulate and prove Theorem~\ref{thm:Hyodo-Kato-complex-intro}.
In fact, Beilinson's Hyodo--Kato theory associates to an $F$-isocrystal in perfect complexes on $(\Spec\calO_K/p,M_\can)_\CRIS$\footnote{Here, $M_\can$ denotes the pullback of the canonical log structure on $\calO_K$} an object of $D_{\phi, N}(K_0)$ (see \S~\ref{subsec: hyodo-kato-theory}). We use $R\Gamma_{\Br}(X, i(\calE_{\Prism}))$ to define such an $F$-isocrystal in perfect complexes and apply the various comparison and base change results for $i(\calE_\Prism)$ and $j_\ast\calE_\Prism$ mentioned above to deduce Theorem~\ref{thm:Hyodo-Kato-complex-intro}.

\medskip
\noindent \textbf{Prismatic-\'etale comparison in \cite{Tian-prism-etale-comparison}.}
The comparison between prismatic and etale cohomology has been established recently by Tian \cite{Tian-prism-etale-comparison}: let $\mu\coloneqq [\epsilon]-1 \in A_{\inf}$ where $\epsilon = (\epsilon_n)_{n \geq 0} \in \calO_C^{\flat}$ is a non-trivial compatible system of $p$-power roots of unity such that $\epsilon_0 = 1$ and $\epsilon_{n+1}^p = \epsilon_n$. 

\begin{thm}[{\cite[Thm.~0.4(2), 5.6(2)]{Tian-prism-etale-comparison}}] \label{thm: prism-etale-comparison-intro}
Let $X$ be a proper semistable $p$-adic formal scheme over $\calO_K$. Let $\mathcal{E}_{\Prism}$ be an analytic prismatic $F$-crystal on $(X, M_X)_{\Prism}$, and let $\mathbb{L}$ be its \'etale realization. Then there is an isomorphism in $D(A_{\inf}[\mu^{-1}])$
\[
R\Gamma(((X_{\calO_C}, M_{X_{\calO_C}})/(A_{\inf}, M_{\Spf A_{\inf}}))_{\Prism}, j_\ast\mathcal{E}_{\Prism})\otimes_{A_{\inf}}^L A_{\inf}[\mu^{-1}] \cong R\Gamma(X_{C, \et}, \mathbb{L})\otimes_{\mathbf{Z}_p}^L A_{\inf}[\mu^{-1}]
\]
compatible with $\Gal(\overline{K}/K)$-actions and Frobenii.
\end{thm}

\medskip
\noindent \textbf{The $C_{\st}$-conjecture.}
Let us return to Theorem~\ref{thm: Cst-intro}. With the notation and assumption therein, we define the isomorphism $\alpha_\st$ to be the composite
\begin{align*}
R\Gamma(X_{C, \et}, \mathbb{L})\otimes_{\Z_p}^L B_{\st}
&\underset{\cong}{\longrightarrow} R\Gamma(((X_{\calO_C/p}, M_{X_{\calO_C/p}})/(A_{\cris}, M_{A_{\cris}}))_{\CRIS}, \calE_{\cris, \Q})\otimes_{A_{\cris}}^L B_{\st}\\
&\overset{\rho_\st\otimes B_\st}{\underset{\cong}{\longleftarrow}} R\Gamma_{\HK}((X_0, M_{X_0}), \mathcal{E}_{\cris, \Q})\otimes_{K_0}^L B_{\st},
\end{align*}
where the first isomorphism is deduced from Theorem~\ref{thm: prism-etale-comparison-intro} and a base change property along $A_\inf\rightarrow A_\cris$ (Theorem~\ref{thm: general-base-change-BK-cohom}), and the second is the one in (HK5) of Theorem~\ref{thm:Hyodo-Kato-complex-intro}. The results mentioned so far show that $\alpha_\st$ is an isomorphism that is compatible with Galois actions, Frobenii, and monodromy operators.

Finally, we need to prove the filtration compatibility after base change to $B_{\dR}$. The precise statement is given in Theorem~\ref{thm: Cst-conj-main}, using the \'etale-de Rham comparison by Scholze \cite[Thm. 8.4]{scholze-p-adic-hodge}. 
Following \cite[\S~10]{GuoReinecke-Ccris}, we study the infinitesimal cohomology over $B_\dR^+$ and show the filtration compatibility. It is worth noting that an embedding $B_\st\hookrightarrow B_\dR$ (or the definition of $B_\st$) depends on the choice of a uniformizer $\pi$. We explain in Remark~\ref{rem:independence-of-Cst-comparison-map} that $\alpha_\st$ is independent of the choice $\pi$.

\medskip
\noindent \textbf{Outline.} In Section~\ref{sec: log-prismatic-sites}, we recall the definition of (strict) logarithmic prismatic sites. We establish some necessary results on functoriality for log prismatic sites in Section~\ref{sec: functorial on log prismatic site}. Various notions of ($F$-)crystals on the log prismatic site are discussed in Section~\ref{sec: analy-pris-cryst-perf-cx}.

In Section~\ref{sec: prismatic-crystalline comparison}, we study the comparison between log prismatic cohomology and crystalline cohomology with coefficients given by $p$-adically completed crystals satisfying the association condition (given in Definition~\ref{defn: associated-crystals}).

Section~\ref{sec: semistable-case} is devoted to the semistable case: we set up notations in \S~\ref{sec: semistable-case-notation} and explain a \v{C}ech--Alexander method for computing cohomology. In \S~\ref{subsec: BK-prism}, we recall Breuil--Kisin log prism and construct \v{C}ech nerve given by a family of Breuil--Kisin log prisms. In \S~\ref{subsec: analy-prism-crys-perf-cx}, we associate to analytic prismatic crystals $\calE_\Prism$ completed prismatic crystals $j_\ast\calE_\Prism$ and prismatic crystals in perfect complexes $i(\calE_\Prism)$, which are the coefficients of prismatic cohomology. We introduce the Breuil--Kisin cohomology and Breuil cohomology and establish their comparison results and base change properties in \S~\ref{subsec: BK-cohom}. In \S~\ref{sec: pris-cris-comparison-crystals-semist}, a comparison between the prismatic and crystalline cohomology is given, and the Frobenius isogeny property for crystalline and Breuil cohomology is proved in \S~\ref{subsec: Frobenius-isogeny-property}. We discuss the Hyodo--Kato theory in \S~\ref{subsec: hyodo-kato-theory} based on \cite{Beilinson-crystalline-period-map-arXiv}.  We apply these results in \S~\ref{subsec: hyodo-kato-cohom} to define and study the Hyodo--Kato complex $R\Gamma_{\HK}((X_0, M_{X_0}), \mathcal{E}_{\cris, \Q})$.

We formulate and prove the $C_{\st}$-conjecture in Section~\ref{sec: Cst}. We recall relevant facts from \cite{du-liu-moon-shimizu-purity-F-crystal} on semistable local systems and their relation to analytic prismatic $F$-crystals in \S~\ref{subsec: semist-local-syst}. Then we prove Theorem~\ref{thm: Cst-intro} in the following subsections; the isomorphism $\alpha_{\st}$ is given in \S~\ref{subsec: Cst-conj}, and the filtration compatibility is proved in \S~\ref{subsec: etale-deRham comparison}. 

Appendix~\ref{sec:log-formal-scheme} discusses some general facts on log formal schemes. In Appendix~\ref{sec: cryst-sites-variants}, we explain certain variants of crystalline sites and recall $\delta_{\log}$-crystalline site, following \cite[\S~6]{koshikawa}.

\medskip
\noindent
\textbf{Notation and conventions}.
Throughout the paper, we fix a prime $p$.

When we work on the monoid $\N^d$, we write $e_i$ ($1\leq i\leq d$) for the element $(0,\ldots,0,1,0,\ldots,0)$ where $1$ is in the $i$-th entry.

Let $\Delta$ denote the category whose objects are $[n]\coloneqq\{0,1,2,\ldots,n\}$ ($n\geq 0)$, and a morphism between $[n]$ to $[m]$ is a nondecreasing map between the corresponding sets (see \cite[0164]{stacks-project} for example). For general notations and terminologies on $\infty$-categories, we follow \cite{Lurie-higher-topos-theory, Lurie-HA}.

We refer the reader to \cite{koshikawa, koshikawa-yao} for the terminology related to log prisms and log prismatic sites. Recall that a \emph{(bounded) prelog prism} is a $\delta_{\mathrm{log}}$-triple $(A,I,M)$ where $(A,I)$ is a (bounded) prism. A \emph{log prism} $(A,I,M_{\Spf A})$ consists of a bounded prism $(A, I)$ and a log structure $M_{\Spf A}$ on $(\Spf A)_\et$ together with $\delta_{\mathrm{log}}$-structure that comes from some prelog prism of the form $(A, I, M)$. Write $M_{\Spf (A/I)}$ for the log structure on $\Spf (A/I)$ induced from $M_{\Spf A}$. For a bounded prelog prism $(A,I,M)$, let $(A,I,M)^a\coloneqq (A,I,M_{\Spf A}^a)$ denote the associated log prism.

We use two crystalline sites appearing in \cite{koshikawa} and \cite{du-liu-moon-shimizu-purity-F-crystal, du-moon-shimizu-cris-pushforward}; the former is denoted by $((Y,M_Y)/(A,M_A))_\CRIS$, and an object is usually written as $(B,J)$ or $(B,B/J)$; the latter is $((Y,M_Y)/\calS^\sharp)_\CRIS$ with objects $(U,T)$. When $(\Spf A,M_A)$ underlies the $p$-adic log PD-formal scheme $\calS^\sharp$, $(B,J)\in((Y,M_Y)/(A,M_A))_\CRIS$ defines an ind-object $(\Spec B/J,\Spf B)=\varinjlim_n (\Spec B/J,\Spec B/p^n)$ of $((Y,M_Y)/\calS^\sharp)_\CRIS$. See \S~\ref{subsec:crystalline-sites} for the definition of these sites and the comparison of cohomology.

For a general bounded $p$-adic formal scheme $X$, let $X_n$ denote the scheme $X\otimes_{\Z_p}\Z_p/p^n$ equipped with the pullback log structure $M_{X_n}$ of $M_X$. We follow \cite[Notation~3.1]{bhatt-scholze-prismaticFcrystal} for the convention on the generic fiber and $\Z_p$-local systems: write $X_{\eta}$ for the generic fiber considered as a locally spatial diamond. For a locally spatial diamond $\calX$, write $\Loc_{\Z_p}(\calX)$ for the category of $\Z_p$-local systems on the quasi-pro-\'etale site of $\calX$. When $X$ is a semistable $p$-adic formal scheme over $\calO_K$, we write $X_{K}$ for $X_\eta$ and also regard it as the adic generic fiber; write $X_{K,\proet}$ for the pro-\'etale site in the sense of \cite{scholze-p-adic-hodge, Scholze-p-adicHodgeerrata}. Then $\Loc_{\Z_p}(X_{K})$ can be naturally identified with the category of $p$-torsion free lisse $\widehat{\Z}_p$-sheaves on $X_{K,\proet}$ in the sense of \cite[Def.~8.1]{scholze-p-adic-hodge} by \cite[Prop.~8.2]{scholze-p-adic-hodge}, \cite[Prop.~3.6, 3.7]{MannWerner-Localsystemsondiamonds}, and \cite[Lem.~15.6]{scholze-etalecohomologyofdiamonds}, and we simply call the objects \emph{$\Z_p$-local systems} on $X_K$ or $X_{K,\proet}$.

\medskip
\noindent
\textbf{Acknowledgments.}

Our special thanks go to Tong Liu for his collaboration during the early stages of the project and for numerous helpful discussions throughout. We are grateful for Teruhisa Koshikawa and Kentaro Inoue for many insightful conversations and sharing about their ongoing joint work as well as pointing out several errors in earlier drafts and providing suggestions . We also thank Shizhang Li for helpful discussions and correspondences. Finally, it is clear to the reader that the $C_{\st}$-conjecture for semistable formal schemes with coefficients consists of two pillars: prismatic-\'etale comparison by Tian \cite{Tian} and the current paper. We thank Yichao Tian for several inspiring conversations.

The first author was partially supported by the National Key R\&D Program of China (No. 2023YFA1009703) and the Beijing Natural Science Foundation (Youth Program, No. 1254044). The third author was partially supported by the NSFC Excellent Young Scientists Fund Program (Overseas).

\section{Logarithmic prismatic sites} \label{sec: log-prismatic-sites}

Let $(X,M_X)$ be a bounded $p$-adic log formal scheme such that $M_X$ is integral.

\begin{defn}[{cf.~\cite[Rem.~4.6]{koshikawa}, \cite[Def.~2.3]{du-liu-moon-shimizu-purity-F-crystal}}] \label{def: absolute-prismatic-site} 
Let $(X,M_X)_\Prism\coloneqq (X,M_X)_\Prism^\str$ denote the \emph{absolute strict prismatic site} of $(X,M_X)$. Recall that an object is a diagram
\[
(\Spf B,M_{\Spf B})\hookleftarrow (\Spf (B/J),M_{\Spf (B/J)})\xrightarrow{f_B} (X,M_X),
\]
where $(B,J,M_{\Spf B})$ is a log prism and $(\Spf (B/J),M_{\Spf (B/J)})\rightarrow (X,M_X)$ is a \emph{strict} morphism of log formal schemes. To simplify the notation, we often write an object as $(\Spf B,J, M_{\Spf B})$, or even as $(B,J,M_{\Spf B})$ with the opposite arrows. 
We equip $(X,M_X)_\Prism$ with strict flat topology. 
It comes with a sheaf of rings $\calO_\Prism$ defined by $\calO_\Prism(B,J,M_{\Spf B})=B$, which is equipped with the Frobenius endomorphism $\phi_{\calO_{\Prism}}$ induced by the Frobenius on each $B$. We often write $\phi$ for $\phi_{\calO_\Prism}$ for simplicity.

Let $\calI_{\Prism} \subset \calO_{\Prism}$ be the ideal sheaf given by $\calI_{\Prism} (B, J, M_{\Spf B}) = J$. For each integer $n \geq 1$, write $\calO_{\Prism, n} \coloneqq \calO_{\Prism} / (p, \calI_{\Prism})^n$. By fpqc descent, we have $\calO_{\Prism, n} (B, J, M_{\Spf B}) = B / (p, J)^n$, and $H^i((\Spf B, J, M_{\Spf B}), \calO_{\Prism, n}) = 0$ for any $i \geq 1$. Thus, we have $R^i \varprojlim_n \calO_{\Prism, n} = 0$ for each $i \geq 1$ and the natural isomorphism $\calO_{\Prism} \cong \varprojlim_n \calO_{\Prism, n}$ by \cite[Lem.~3.18]{scholze-p-adic-hodge} (see Corollary~\ref{cor:repleteness} below for a relevant result).
\end{defn}

\begin{defn} \label{def: relative-prismatic-site}
Let $(A,I,M_{\Spf A})$ be a log prism such that $M_{\Spf A}$ is integral and assume that we are given a morphism  $(X,M_X)\rightarrow (\Spf A/I,M_{\Spf A/I})$ of $p$-adic log formal schemes. We let $((X,M_X)/(A,M_{\Spf A}))_\Prism$ denote the \emph{relative strict prismatic site} of $(X,M_X)$ with strict flat topology as in \cite[Def.~4.1, Rem.~4.3]{koshikawa}. Moreover, it comes with a sheaf of rings $\calO_\Prism$. Note that the smoothness assumption involving $M_A$ and $X$ therein is not necessary to show that $((X,M_X)/(A,M_{\Spf A}))_\Prism$ is a site or $\calO_\Prism$ is a sheaf. When emphasizing the ideal $I$, we also write $((X,M_X)/(A,I, M_{\Spf A}))_\Prism$ for the site.

The obvious forgetful functor
\[
\iota\colon ((X,M_X)/(A,M_{\Spf A}))_\Prism \rightarrow (X,M_X)_\Prism
\]
is continuous and cocontinuous, which induces a morphism of topoi
\[
g\colon \Sh(((X,M_X)/(A,M_{\Spf A}))_\Prism) \rightarrow \Sh((X,M_X)_\Prism)
\]
such that $(g^{-1}\calF)(B,J,M_{\Spf B})=\calF(\iota(B,J,M_{\Spf B}))$ for $\calF\in \Sh((X,M_X)_\Prism)$ and $(B,J,M_{\Spf B})\in ((X,M_X)/(A,M_{\Spf A}))_\Prism$ (see \cite[00XO, 00XR]{stacks-project}). We usually write $\calF|_{((X,M_X)/(A,M_{\Spf A}))_\Prism}$ for $g^{-1}\calF$, or even $\calF$ when there is no ambiguity.
\end{defn}

\begin{defn}[{cf.~\cite[Rem.~4.5]{koshikawa}}]
\label{def:projection to etale site}
Define a morphism of topoi
\[
\nu_X=(\nu_X^{-1},\nu_{X,\ast})\colon \Sh((X,M_X)_\Prism)\rightarrow \Sh(X_\et)
\]
as follows. 
Let $\mathrm{ForSch}_\et/X$ denote the localized site of the site $\mathrm{ForSch}$ of $p$-adic formal schemes with \'etale topology by $X$ (equivalently, the category of log $p$-adic formal schemes that are strict over $(X,M_X)$ with \'etale topology). The association $(\Spf B,J,M_{\Spf B})\mapsto (f_{B}\colon\Spf  B/J\rightarrow X)$ defines a cocontinuous functor $(X,M_X)_\Prism\rightarrow \mathrm{ForSch}_\et/X$ by \cite[Rem.~4.2]{koshikawa} and thus a morphism $\Sh((X,M_X)_\Prism)\rightarrow \Sh(\mathrm{ForSch}_\et/X)$ of topoi. We define $\nu_X$ to be the composite of this with the restriction $\Sh(\mathrm{ForSch}_\et/X)\rightarrow\Sh(X_\et)$.
Concretely, $(\nu_{X,\ast}\calF)(U\rightarrow X)=\Gamma((U,M_U)_\Prism,\calF|_{(U,M_U)_\Prism})$ where $\calF|_{(U,M_U)_\Prism}$ denotes the obvious restriction (see \S~\ref{sec: functorial on log prismatic site} for a general discussion of functoriality of prismatic sites); $\nu_X^{-1}\calG$ is the flat sheafification of the presheaf sending $(\Spf B,J,M_{\Spf B})$ to $\Gamma((\Spf B/J)_\et,f_{B,\et}^{-1}\calG)$.
In the relative set-up, we also use $\nu_X$ to denote the composite $g\circ \nu_X\colon \Sh(((X,M_X)/(A,M_{\Spf A}))_\Prism)\rightarrow \Sh(X_\et)$ by abuse of notation.
\end{defn}

Let us discuss the repleteness of the log prismatic topoi.

\begin{lem}\label{lem:limit of flat covers}
Let 
\[
(\Spf A_1,I_1, M_{\Spf A_1})\xleftarrow{h_1}(\Spf A_2,I_2, M_{\Spf A_2})\xleftarrow{h_2}(\Spf A_3,I_3, M_{\Spf A_3})\xleftarrow{h_3} \cdots
\]
be a sequence of morphisms in the prismatic site $(X,M_X)_\Prism$ 
such that $h_i$ is a flat cover for every $i$. Then the inverse limit $\varprojlim_i (\Spf A_i,I_i, M_{\Spf A_i})$ exists in $(X,M_X)_\Prism$ and each projection $(\Spf A_i,I_i, M_{\Spf A_i})\leftarrow\varprojlim_i (\Spf A_i,I_i, M_{\Spf A_i})$ is a flat cover.
\end{lem}

The proof is eventually reduced to the discussion on the (sequential) colimit of $\delta_\log$-triples, but one also needs to take care of the difference between morphisms of prelog prisms and those of log prisms.

\begin{proof}
By definition, $(\Spf A_1,I_1, M_{\Spf A_1})$ comes from a prelog prism $(A_1,I_1,M_1)$ for some monoid $M_1$, and thus $M_1\rightarrow A_1\rightarrow A_1/I_1$ gives a chart for $M_{\Spf A_1}$. By \cite[Cor.~2.2(1)]{du-liu-moon-shimizu-purity-F-crystal}, we obtain a prelog prism $(A_2,I_2,M_2)$ defining $(\Spf A_2,I_2, M_{\Spf A_2})$ together with a monoid map $M_1\rightarrow M_2$ compatible with $\delta_\log$-structures. By repeating this argument, we obtain a sequence of prelog prisms
\[
(A_1,I_1, M_1)\rightarrow(A_2,I_2, M_2)\rightarrow( A_3,I_3, M_3)\rightarrow \cdots
\]
giving the inverse system of log prisms in question. 
We know from \cite[Rem.~2.6]{koshikawa} that the colimit $\varinjlim_i (A_i,M_i)$ as a $\delta_\log$-ring exists, and the underlying ring and monoid are given by $A_\infty\coloneqq \varinjlim_i A_i$ and $M_\infty\coloneqq \varinjlim_i M_i$, respectively. Note $I_i=I_1A_i$ and let $B$ be the derived $(p,I_1)$-completion of $A_\infty$. Then each $(A_i,I_1A_i)\rightarrow (B,I_1B)$ becomes a $(p,I_1)$-completely faithfully flat map of bounded prisms with $B$ being discrete and classically $(p,I_1)$-complete by \cite[Lem.~3.7(1)(2)(3)]{bhatt-scholze-prismaticcohom}.
Moreover, $(B,M_\infty\rightarrow B)$ admits a natural $\delta_\log$-structure by \cite[Lem.~2.9]{koshikawa}. Hence we obtain a prelog prism $(B,I_1B,M_\infty)$ and morphisms $(\Spf A_i,I_i, M_{\Spf A_i})\leftarrow (\Spf B,I_1B, \underline{M_\infty}^a)$ of log prisms compatible with $h_i$. It is now straightforward to check that the latter represents the inverse limit and satisfies the required property by using \cite[Cor.~2.2(1)]{du-liu-moon-shimizu-purity-F-crystal}.
\end{proof}

\begin{cor}\label{cor:repleteness}
The absolute strict prismatic topos $\Sh((X,M_X)_\Prism)$ is replete. In particular, for any inverse system $\calF\colon \N^{\op}\rightarrow \Ab((X,M_X)_\Prism))$ of abelian sheaves with each $F_{n+1}\rightarrow F_n$ surjective, we have 
\[
\varprojlim_n \calF_n\cong R\varprojlim_n \calF_n \quad\text{and}\quad R\Gamma((X,M_X)_\Prism,\varprojlim_n \calF_n)\cong R\varprojlim_n R\Gamma((X,M_X)_\Prism,\calF_n).
\]
The same holds for any localized topoi.
\end{cor}

\begin{proof}
The first assertion is proved as in \cite[Ex.~3.1.7]{bhatt-scholze-proetale} or \cite[Rem.~2.4]{bhatt-scholze-prismaticFcrystal} using Lemma~\ref{lem:limit of flat covers}. The assertion $\varprojlim_n \calF_n\cong R\varprojlim_n \calF_n$ is  \cite[Prop.~3.1.10]{bhatt-scholze-proetale}, and the assertion on cohomology is \cite[0D6K]{stacks-project}. The last assertion is \cite[Lem.~3.1.2]{bhatt-scholze-proetale}.
\end{proof}

\section{Functoriality of log prismatic sites} \label{sec: functorial on log prismatic site}

We discuss the functoriality of the absolute strict log prismatic sites. 
Throughout this section, fix a morphism of 
bounded \emph{integral} $p$-adic log formal schemes
\[
f\colon (X,M_X)\rightarrow (Y,M_Y).
\]
We will discuss a morphism of topoi 
\[
f_\Prism=(f^{-1}_\Prism, f_{\Prism,\ast})\colon \operatorname{Sh}((X,M_X)_\Prism)\rightarrow\operatorname{Sh}((Y,M_Y)_\Prism).
\]
If $f$ is strict, we have a fully faithful cocontinuous functor $(X,M_X)_\Prism \rightarrow (Y,M_Y)_\Prism$, which yields $f_\Prism$. We will construct $f_\Prism$ when $M_Y$ is quasi-coherent (Proposition~\ref{prop:functoriality of log prismatic site}).

\begin{defn}\label{def:f-prismatic morphism}
An \emph{$f$-prismatic morphism} (or an $f$-$\Prism$ morphism for short) from $(\Spf B, J, M_{\Spf B})\in (X,M_X)_\Prism$ to $(\Spf A, I, M_{\Spf A})\in (Y,M_Y)_\Prism$ is a morphism of log prisms $g\colon (\Spf B, J, M_{\Spf B}) \rightarrow (\Spf A, I, M_{\Spf A})$ that is compatible with $f$, namely, makes the following induced diagram commutative:
\begin{equation}\label{eq:f-prismatic morphism}
\xymatrix{
(X,M_X)\ar[d]^-f & (\Spf (B/J), M_{\Spf (B/J)})\ar[l]\ar[d]^-{\overline{g}}\ar@{^{(}->}[r] & (\Spf B, M_{\Spf B})\ar[d]^-{g}\\
(Y,M_Y) & (\Spf (A/I), M_{\Spf (A/I)})\ar[l]\ar@{^{(}->}[r] & (\Spf A, M_{\Spf A}).
}
\end{equation}
\end{defn}

\begin{lem}\label{lem:pullback of f-prismatic morphism}
Let $g\colon (\Spf B, J, M_{\Spf B}) \rightarrow (\Spf A, I, M_{\Spf A})$ be an $f$-prismatic morphism and let $h\colon (\Spf A', IA', M_{\Spf A'})\rightarrow(\Spf A, I, M_{\Spf A})$ be a $(p,I)$-completely flat morphism in $(Y,M_Y)_\Prism$. Consider the category of commutative diagrams of the form
\[
\xymatrix{
(\Spf B', JB', M_{\Spf B'})\ar[d]_-{g'}\ar[r]^-{h'} & (\Spf B,  J, M_{\Spf B})\ar[d]^-{g}\\
(\Spf A', IA', M_{\Spf A'})\ar[r]^-h & (\Spf A, I, M_{\Spf A}),
}
\]
where $(\Spf B', JB', M_{\Spf B'})\in (X, M_X)_\Prism$, $h'$ is a morphism in $(X,M_X)_\Prism$, and $g'$ is an $f$-prismatic morphism, together with the obvious notion of morphisms.
Then the category has a final object.
\end{lem}

\begin{proof}
This follows from an argument similar to \cite[Rem.~2.4]{du-liu-moon-shimizu-purity-F-crystal}: let $B'$ denote the classical $(p,I)$-completion of $B\otimes_AA'$. Then $(B',IB')$ represents the pushout of the diagram $(B, J)\leftarrow (A,I)\rightarrow (A', IA')$ of bounded prisms. Since $h\colon (\Spf A',M_{\Spf A'})\rightarrow (\Spf A,M_{\Spf A})$ is strict, one can check that $\Spf (B'/JB')\hookrightarrow \Spf B'$ is upgraded to a log prism in $(X,M_X)_{\Prism}$ and gives a final object of the category. 
\end{proof}

\begin{lem}\label{lem:pullback of representable prismatic sheaf}
For $(\Spf A, I, M_{\Spf A})\in (Y,M_Y)_\Prism$, the presheaf $f^{-1}(\Spf A, I, M_{\Spf A})$ on $(X,M_X)_\Prism$ defined by
\[
(X,M_X)_\Prism\ni(\Spf B, J, M_{\Spf B})\mapsto \Mor_{f\text{-}\Prism}\bigl((\Spf B, J, M_{\Spf B}),(\Spf A, I, M_{\Spf A})\bigr)
\]
is a sheaf of sets. 
\end{lem} 

We often write $f^{-1} h_A$ for $f^{-1}(\Spf A, I, M_{\Spf A})$ if there is no confusion. When $f=\mathrm{id}$, the lemma claims that the site $(X,M_X)_\Prism$ is subcanonical.

\begin{proof}
Since affine formal schemes are quasi-compact, every cover in $(X,M_X)_\Prism$ is refined by a finite cover. Moreover, every finite disjoint union exists in $(X,M_X)_\Prism$. Hence it suffices to check the sheaf condition for a single cover: take a $(p,J)$-completely faithfully flat morphism $(\Spf B',JB', M_{\Spf B'})\rightarrow (\Spf B, J,M_{\Spf B})$ in $(X,M_X)_\Prism$.
By \cite[Rem.~2.4]{du-liu-moon-shimizu-purity-F-crystal}, the self-fiber product of $(\Spf B',JB', M_{\Spf B'})$ over $(\Spf B, J,M_{\Spf B})$ is representable by a log prism $(\Spf B'', JB'', M_{\Spf B''})$ where $B''$ is the classical $(p, J)$-completion of $B'\otimes_B B'$. 
Let $p_1,p_2\colon (\Spf B'', JB'', M_{\Spf B''})\rightarrow (\Spf B', JB', M_{\Spf B'})$ be the projections. 

Take an $f$-prismatic morphism $g'\colon (\Spf B', JB', M_{\Spf B'})\rightarrow (\Spf A, I, M_{\Spf A})$ with $g'\circ p_1=g'\circ p_2$. By \cite[Cor.~3.12]{bhatt-scholze-prismaticcohom} and its proof, the self-fiber product of the prism $(\Spf B',JB')$ over $(\Spf B, J)$ is represented by $(\Spf B'', JB'')$ and $g'$ descends uniquely to a morphism of prisms $g\colon (\Spf B, J)\rightarrow (\Spf A, I)$.
By Lemma~\ref{lem:ff descent of log str}, the monoid sheaf morphism $g'^{-1}M_{\Spf A}\rightarrow M_{\Spf B'}$ descends uniquely to a morphism $g^{-1}M_{\Spf A}\rightarrow M_{\Spf B}$ and thus defines a morphism of log schemes $g\colon (\Spf B,M_{\Spf B})\rightarrow (\Spf A, M_{\Spf A})$. Now it is easy to check that $g$ is compatible with $f$ and $\delta_\log$-structure and yields a morphism of log prisms from $(\Spf B, J,M_{\Spf B})$ to $(\Spf A, I, M_{\Spf A})$. 
\end{proof}

By \cite[Prop.~5.7]{berthelot-ogus-book}, there exists a unique pair of functors
\[
f_{\Prism,\ast}\colon \operatorname{PSh}((X,M_X)_\Prism)\rightarrow \operatorname{PSh}((Y,M_Y)_\Prism)
\]
and
\[
f_\Prism^\bullet\colon \operatorname{PSh}((Y,M_Y)_\Prism)\rightarrow \operatorname{PSh}((X,M_X)_\Prism)
\]
such that $f_\Prism^\bullet h_A$ agrees with $f^{-1}h_A$ in Lemma~\ref{lem:pullback of representable prismatic sheaf} and such that $f_\Prism^\bullet$ is left adjoint to $f_{\Prism,\ast}$. In fact, the proof of \textit{loc. cit.} gives explicit definitions of $f_{\Prism,\ast}$ and $f_\Prism^\bullet$:
for $\calF\in \operatorname{PSh}((X,M_X)_\Prism)$ and $(\Spf A, I, M_{\Spf A})\in (Y,M_Y)_\Prism$, we have
\[
(f_{\Prism,\ast}\calF)(\Spf A, I, M_{\Spf A})=\Hom_{\operatorname{PSh}((X,M_X)_\Prism)}(f^{-1}h_A,\calF).
\]
For 
$\calG\in \operatorname{PSh}((Y,M_Y)_\Prism)$ and $(\Spf B, J, M_{\Spf B})\in (X, M_X)_\Prism$, we have
\[
(f_\Prism^\bullet\calG)(\Spf B, J, M_{\Spf B})=\varinjlim_{(\Spf B, J, M_{\Spf B})\rightarrow (\Spf A, I, M_{\Spf A})}\calG(\Spf A, I, M_{\Spf A}),
\]
where the colimit is taken over all the $f$-prismatic morphisms from $(\Spf B, J, M_{\Spf B})$.

\begin{lem}
If $\calF$ is a sheaf on $(X,M_X)_\Prism$, then $f_{\Prism,\ast} \calF$ is a sheaf on $(Y,M_Y)_\Prism$.
\end{lem}

\begin{proof}
    This is almost formal: we can argue as in the proof of \cite[Prop.~5.8.2]{berthelot-ogus-book} by using Lemma~\ref{lem:pullback of f-prismatic morphism}.
\end{proof}

\begin{defn}
    For a sheaf $\calG$ on $(Y,M_Y)_\Prism$, let $f_\Prism^{-1}\calG$ denote the sheafification of $f_\Prism^\bullet\calG$.
    Then $f_\Prism^{-1}\colon \operatorname{Sh}((Y,M_Y)_\Prism)\rightarrow \operatorname{Sh}((X,M_X)_\Prism)$ is left adjoint to $f_{\Prism,\ast}$
\end{defn}

\begin{prop}\label{prop:functoriality of log prismatic site}
Assume that $M_Y$ is quasi-coherent and integral. Then the functor $f^{-1}\colon \operatorname{Sh}((Y,M_Y)_\Prism)\rightarrow \operatorname{Sh}((X,M_X)_\Prism)$ commutes with finite limits. Hence the pair 
\[
f_\Prism\coloneqq (f^{-1}_\Prism,f_{\Prism,\ast})\colon \operatorname{Sh}((X,M_X)_\Prism)\rightarrow\operatorname{Sh}((Y,M_Y)_\Prism)
\]
is a morphism of topoi. 
\end{prop}

\begin{construction}\label{construction:initial object for f-prismatic morphism}
Assume that $M_Y$ admits an integral chart $P\rightarrow \Gamma(Y,M_Y)$. For every $(\Spf B, J, M_{\Spf B})\in (X,M_X)_\Prism$, let $M_{\Spf (B/J)}^P$ and $M_{\Spf B}^P$ denote the log structures on $\Spf(B/J)$ and $\Spf B$ associated to
the monoid homomorphisms
\begin{align*}
&P\rightarrow \Gamma(Y,M_Y) \rightarrow \Gamma(X,M_X)\rightarrow \Gamma(\Spf (B/J), M_{\Spf (B/J)}) \rightarrow B/J \quad\text{and}    \\
&P_B\coloneqq P\times_{\Gamma(\Spf (B/J), M_{\Spf (B/J)})}\Gamma(\Spf B, M_{\Spf B})\rightarrow \Gamma(\Spf B, M_{\Spf B}) \rightarrow B,
\end{align*}
respectively. Note that these log structures are integral, and $M_{\Spf (B/J)}^P=(f\circ f_B)^\ast M_Y$ where $f_B\colon \Spf (B/J)\rightarrow X$ is the structure map.
\end{construction}

\begin{lem}\label{lem:f-prismatic morphism induced by chart}
The induced map $(\Spf (B/J), M_{\Spf (B/J)}^P)\rightarrow (\Spf B, M_{\Spf B}^P)$ is an exact closed immersion, 
and the induced map $M_{\Spf B}^P\rightarrow M_{\Spf (B/J)}^P\times_{M_{\Spf (B/J)}}M_{\Spf B}$ is an isomorphism where the fiber product is taken in $\Sh((\Spf (B/J))_\et)=\Sh((\Spf B)_\et)$. 
Moreover, the $\delta_\log$-structure on $\Gamma(\Spf B, M_{\Spf B})$ defines a $\delta$-log structure on $P_B$ and makes $(\Spf B, J, M_{\Spf B}^P)$ a log prism in $(Y,M_Y)_\Prism$. 
Finally, the natural map
\[
(\Spf B, J, M_{\Spf B})\rightarrow (\Spf B, J, M_{\Spf B}^P)
\]
is an $f$-prismatic morphism, and is an initial object in the category of $f$-prismatic morphisms from $(\Spf B, J, M_{\Spf B})$.
\end{lem}

\begin{proof}
    The argument in the first paragraph of \cite[Prop.~3.7]{koshikawa} works: as in the proof of \cite[Lem.~3.8]{koshikawa}, we see that $\Gamma(\Spf B, M_{\Spf B})\rightarrow \Gamma(\Spf (B/J), M_{\Spf (B/J)})$ is exact surjective and a $(1+J)$-torsor on monoids. Hence so is $P_B=P\times_{\Gamma(\Spf (B/J), M_{\Spf (B/J)})}\Gamma(\Spf B, M_{\Spf B})\rightarrow P$. 
    One can also check that the map $M_{\Spf B}^P\rightarrow M_{\Spf (B/J)}^P\times_{M_{\Spf (B/J)}}M_{\Spf B}$ is an isomorphism by computing the stalks at every geometric point of $\Spf B$.
    The remaining assertions are straightforward.
\end{proof}

\begin{proof}[Proof of Proposition~\ref{prop:functoriality of log prismatic site}]
We need to show that $f_\Prism^{-1}$ commutes with finite limits.
Since $M_Y$ is integral and quasi-coherent, it admits an integral chart \'etale locally. In particular, every $(\Spf B, J, M_{\Spf B})\in (X,M_X)_\Prism$ admits an \'etale cover $(\Spf B', JB', M_{\Spf B'})$ for which Construction~\ref{construction:initial object for f-prismatic morphism} works. Let $(\Spf B', JB', M_{\Spf B'})\rightarrow (\Spf B', JB', M_{\Spf B'}^P)$ denote the resulting $f$-prismatic morphism in Lemma~\ref{lem:f-prismatic morphism induced by chart}. Then for every $\calG\in \Sh((Y,M_Y)_\Prism)$, we see
\[
(f_\Prism^{\bullet}\calG)(\Spf B', JB', M_{\Spf B'})=\calG(\Spf B', JB', M_{\Spf B'}^P).
\]
It follows that $(f_\Prism^{-1}\calG)(\Spf B', JB', M_{\Spf B'})=\calG(\Spf B', JB', M_{\Spf B'}^P)$. Since evaluation at $(\Spf B', JB', M_{\Spf B'}^P)$ commutes with taking limits, we conclude that $f_\Prism^{-1}$ commutes with finite limits.
\end{proof}

We have an obvious variant for the relative prismatic sites.

\begin{prop}\label{prop:functoriality of relative log prismatic site}
Let $(A_0,I_0,M_{\Spf A_0})$ be a log prism such that $M_{\Spf A_0}$ is integral and assume that we are given a morphism $(Y,M_Y)\rightarrow (\Spf A_0/I_0,M_{\Spf A_0/I_0})$ of $p$-adic log formal schemes.
If $M_Y$ is quasi-coherent, then there exists a morphism of topoi
\[
f_\Prism\coloneqq (f^{-1}_\Prism,f_{\Prism,\ast})\colon \operatorname{Sh}(((X,M_X)/(A_0,M_{\Spf A_0}))_\Prism)\rightarrow\operatorname{Sh}(((Y,M_Y)/(A_0,M_{\Spf A_0})_\Prism)
\]
such that 
\[
(f_{\Prism,\ast}\calF)(\Spf A, I, M_{\Spf A})=\Hom_{\operatorname{PSh}((X,M_X)_\Prism)}(f^{-1}h_A,\calF)
\]
where $f^{-1}h_A$ is defined by sending $(\Spf B, J, M_{\Spf B})\in ((X,M_X)/(A_0,M_{\Spf A_0}))_\Prism$ to
the set $\Mor_{f\text{-}\Prism/A_0}\bigl((\Spf B, J, M_{\Spf B}),(\Spf A, I, M_{\Spf A})\bigr)$ of $f$-prismatic morphisms $g$ as in \eqref{def:f-prismatic morphism} such that $g$ is a map of log prisms over $(A_0,I_0,M_{\Spf A_0})$.
\end{prop}

\begin{proof}
The proof of Proposition~\ref{prop:functoriality of log prismatic site} works.
\end{proof}

We now turn to a description of the higher direct image.
Let $(\Spf A, I, M_{\Spf A})\in (Y,M_Y)_\Prism$ and set 
\[
f_{A/I}\colon (X_{A/I},M_{X_{A/I}})\coloneqq (X,M_X)\times_{(Y,M_Y)}(\Spf (A/I), M_{\Spf (A/I)})\rightarrow(\Spf (A/I), M_{\Spf (A/I)}).
\]
Note that $(X_{A/I},M_{X_{A/I}})\rightarrow (X,M_X)$ is strict.
By construction, we have an obvious diagram of morphisms of topoi
\[
\xymatrix{
\Sh(((X_{A/I},M_{X_{A/I}})/(A,M_{\Spf A}))_\Prism) \ar[d]\ar[r]^-{(f_{A/I})_{\Prism}}& \Sh(((\Spf A/I,M_{\Spf A/I})/(A,M_{\Spf A}))_\Prism)\ar[d]\\
\Sh((X,M_X)_\Prism)\ar[r]^-{f_{\Prism}}& \Sh((Y,M_Y)_\Prism)
}
\]
that is commutative up to canonical isomorphism, and the top horizontal morphism is canonically identified with the induced morphism between the localized topoi
\[
\Sh((X,M_X)_\Prism)/f_\Prism^{-1}h_A\xrightarrow{f_{\Prism}} \Sh((Y,M_Y)_\Prism)/h_A
\]
given in \cite[04H1]{stacks-project}.

\begin{prop} \label{prop:prismatic-higher-direct-image}
Keep the notation as above.
Let $\calK\in D((X,M_X)_\Prism, \calO_\Prism)$ and consider $(Rf_{\Prism,\ast}\calK)\in D((Y,M_Y)_\Prism,\calO_\Prism)$. Then there is a natural quasi-isomorphism
\[
R\Gamma((\Spf A, I, M_{\Spf A}), Rf_{\Prism,\ast}\calK)\cong R\Gamma(((X_{A/I},M_{X_{A/I}})/(A,M_{\Spf A}))_\Prism,\calK),
\]
where we continue to write $\calK$ for the restriction of $\calK$ to $\Sh(((X_{A/I},M_{X_{A/I}})/(A,M_{\Spf A}))_\Prism)$.
\end{prop}

\begin{proof}
This follows from \cite[0D6H]{stacks-project}. Note that the statement therein concerns a morphism of ringed sites, but the same proof works for a morphism of ringed topoi thanks to \cite[04IF]{stacks-project}.
\end{proof}

\section{Prismatic \texorpdfstring{$F$}{F}-crystals in perfect complexes} \label{sec: analy-pris-cryst-perf-cx}

Let $(X, M_X)$ be a bounded $p$-adic log formal scheme such that $M_X$ is integral. We review prismatic ($F$-)crystals in perfect complexes as well on $(X, M_X)$ as completed prismatic $F$-crystals and analytic prismatic $F$-crystals. These notions and their variants have appeared in \cite[\S~4]{bhatt-scholze-prismaticFcrystal}, \cite[\S~3.2]{du-liu-moon-shimizu-completed-prismatic-F-crystal-loc-system}, \cite[\S~3.1, 5.1]{GuoReinecke-Ccris}, \cite[\S~3.1]{du-liu-moon-shimizu-purity-F-crystal}, \cite[\S~11]{Tsuji-prismatic-q-Higgs}, \cite[\S~1]{Tian-prism-etale-comparison}.

\smallskip
\noindent
\textbf{Prismatic crystals in perfect complexes.} Given a (discrete) ring $A$, write $\mathcal{D}(A)$ for the derived $\infty$-category of $A$-modules, and $\mathcal{D}_{\perf}(A) \subset \mathcal{D}(A)$ for the full subcategory of perfect complexes over $A$ (cf.~\cite[Def.~7.2.4.1, Prop.~7.2.4.2]{Lurie-HA}).
By a perfect complex on $((X, M_X)_{\Prism}, \mathcal{O}_{\Prism})$, we mean an object $\mathcal{E}$ in the derived $\infty$-category $\mathcal{D}((X, M_X)_{\Prism}, \mathcal{O}_{\Prism})$ such that there exists a cover $\{U_i\}$ of $(X, M_X)_{\Prism}$ (jointly covering the final object in the topos) and perfect complexes $P_i \in \mathcal{D}(\mathcal{O}_{\Prism}(U_i))$ satisfying $\mathcal{E}|_{U_i} \cong P_i\otimes_{\mathcal{O}_{\Prism}(U_i)} \mathcal{O}_{\Prism, U_i}$. 
Write $\mathcal{D}_{\perf}((X, M_X)_{\Prism}) \subset \mathcal{D}((X, M_X)_{\Prism}, \mathcal{O}_{\Prism})$ for the full subcategory of perfect complexes, which we call \emph{prismatic crystals in perfect complexes} on $(X,M_X)$. By $(p, I)$-completely faithfully flat descent (see \cite[Prop.~2.7 Pf.]{bhatt-scholze-prismaticFcrystal}), we have a natural equivalence
\[
\mathcal{D}_{\perf}((X, M_X)_{\Prism}) \xrightarrow{\cong} \lim_{(A, I, M_{\Spf A}) \in (X, M_X)_{\Prism}^\op} \mathcal{D}_{\perf}(A).
\]

For any $(A, I, M_{\Spf A}) \in (X, M_X)_{\Prism}$, note that the Frobenius endomorphism $\phi_A$ on $A$ induces an action on $\Spec(A)\smallsetminus V(p,I)$, denoted also by $\phi_A$.
Write $\mathcal{D}_{\perf}^{\phi}(A)$ for the $\infty$-category of pairs $(\mathcal{E}_A, \phi_{\mathcal{E}_A})$ where $\mathcal{E}_A \in \mathcal{D}_{\perf}(A)$ and 
\[
\phi_{\mathcal{E}_A}\colon L\phi_A^*\mathcal{E}_A\otimes_A^{L}A[I^{-1}] \stackrel{\cong} {\rightarrow} \mathcal{E}_A\otimes_A^{L}A[I^{-1}].
\]
is an isomorphism. Following \cite[Rem.~4.2]{bhatt-scholze-prismaticFcrystal}, let $\mathcal{D}_{\perf}^{\phi}((X, M_X)_{\Prism})$ be the $\infty$-category of pairs $(\mathcal{E}, \phi_{\mathcal{E}})$ where $\mathcal{E} \in \mathcal{D}_{\perf}((X, M_X)_{\Prism})$ and $\phi_{\mathcal{E}}\colon L\phi_{\calO_\Prism}^*\mathcal{E}[\mathcal{I}_{\Prism}^{-1}] \stackrel{\cong}{\rightarrow} \mathcal{E}[\mathcal{I}_{\Prism}^{-1}]$ is an isomorphism. The objects are called \emph{prismatic $F$-crystals in perfect complexes} on $(X,M_X)$.

\smallskip
\noindent
\textbf{Completed prismatic crystals.} Let us first recall the definitions.

\begin{defn}[{\cite[Def.~3.11]{du-liu-moon-shimizu-completed-prismatic-F-crystal-loc-system}}]\hfill
\begin{enumerate}
  \item A \emph{completed crystal of $\calO_\Prism$-modules} on $(X,M_X)_\Prism$ is a sheaf $\calF$ of $\calO_\Prism$-modules such that for every $(A,I,M_{\Spf A})\in (X,M_X)_\Prism^\op$, the $A$-module $\calF_A\coloneqq \calF(A,I,M_{\Spf A})$ is classically $(p,I)$-complete and for every map $(A,I,M_{\Spf A})\rightarrow (B,IB,M_{\Spf B})$ in $(X, M_X)_{\Prism}^{\op}$, the induced map
\[
\calF_A\widehat{\otimes}_AB\rightarrow \calF_B
\]
is an isomorphism, where $\calF_A\widehat{\otimes}_AB$ denotes the completed tensor product $\varprojlim_n (\calF_A\otimes_AB)/(p,I)^n(\calF_A\otimes_AB)$.
 We also use the terminology \emph{completed prismatic crystal} interchangeably.  Let $\CR^{\wedge}((X, M_X)_{\Prism}, \calO_{\Prism})$ denote the category of \emph{finitely generated} completed crystal of $\calO_\Prism$-modules on $(X,M_X)_\Prism$. 
  \item A \emph{crystal of $\calO_{\Prism,n}$-modules} on $(X,M_X)_\Prism$ is a sheaf $\calF_n$ of $\calO_{\Prism,n}$-modules such that for every $(A,I,M_{\Spf A})\rightarrow (B,IB,M_{\Spf B})$, the induced map $\calF_{n,A}\otimes_AB\rightarrow \calF_{n,B}$ is an isomorphism, where $\calF_{n}\coloneqq \calF_{n,A}(A,I,M_{\Spf A})$.
  \end{enumerate}
\end{defn}

\begin{rem}\label{rem:completed crystal as limit}
Let us recall the discussions in \cite[Rem.~3.12, Lem.~3.13]{du-liu-moon-shimizu-completed-prismatic-F-crystal-loc-system} (see also \cite[Rem.~11.2, 11.4]{Tsuji-prismatic-q-Higgs}).
\begin{enumerate}
 \item If $\calF$ is a completed crystal of $\calO_\Prism$-modules on $(X,M_X)_\Prism$, then for every $f\colon (A,I,M_{\Spf A})\rightarrow (B,IB,M_{\Spf B})$, the induced map
\[
\calF_A\otimes_AB/(p,IB)^n\rightarrow \calF_B/(p,IB)^n\calF_B
\]
is an isomorphism for every $n$ (use \cite[Thm.~1.2(2)]{yekutieli-flatnesscompletion}).
Moreover, the association $(A,I,M_{\Spf A})\mapsto \calF_A/(p,I)^n\calF_A$ represents the quotient sheaf $\calF_{n}\coloneqq \calF/(p,\calI_\Prism)^n\calF$ and defines a crystal of $\calO_{\Prism,n}$-modules. To see this, assume that $f$ is a flat cover (hence, strict) and set $B'\coloneqq (B\otimes_AB)^\wedge_{(p,I)}$ (the classical completion). Then $(B',IB',\underline{M_A}^a)\in (X,M_X)_\Prism^\op$ represents the self-cofiber product of $f$, and one can use the argument in \cite[Lem.~3.13, Pf., 1st para.]{du-liu-moon-shimizu-completed-prismatic-F-crystal-loc-system} verbatim to conclude the above association is indeed a sheaf. Since $\calF_A$ is classically $(p,I)$-complete, we also have an isomorphism $\calF\xrightarrow{\cong}\varprojlim_n\calF_n$. If $\calF$ is finitely generated, then for any $f$, the induced map
\[
\calF\otimes_AB\rightarrow \calF_A\widehat{\otimes}_AB\xrightarrow{\cong}\calF_B
\]
is surjective; it becomes an isomorphism if, moreover, $A$ and $B$ are Noetherian, or if $A$ is Noetherian and $A\rightarrow B$ is classically flat.
 \item Let $(\calF_n)_n$ be an inverse system of $\calO_\Prism$-modules such that $\calF_n$ is a crystal of $\calO_{\Prism,n}$-modules and the induced map $\calF_{n+1}\otimes_{\calO_{\Prism,n+1}}\calO_{\Prism,n}\rightarrow \calF_n$ is an isomorphism for each $n$. Then by arguing as in (1), we see that the association $A\mapsto \calF_{n+1,A}\otimes_{A}A/(p,I)^n$ represents the sheaf $\calO_{\Prism,n}\otimes_{\calO_{\Prism,n+1}}\calF_{n+1}$. Hence $\calF\coloneqq \varprojlim_n \calF_n$ is a completed crystal of $\calO_\Prism$-modules. Finally, we see from \cite[Thm.~1.2(2)]{yekutieli-flatnesscompletion} that the induced map $\calF\otimes_{\calO_\Prism}\calO_{\Prism,n}=\calF/(p,\calI_\Prism)^n\calF\rightarrow \calF_n$ is an isomorphism. Note that $\calF_n$ is finitely generated for each (or some) $n$ if and only if so is $\calF$.
  \item For a completed crystal $\calF$ of $\calO_\Prism$-modules on $(X,M_X)_\Prism$, we define
\[
\widehat{\phi}^\ast\calF\coloneqq \calF\widehat{\otimes}_{\calO_\Prism,\phi_{\calO_\Prism}}\calO_\Prism\coloneqq \varprojlim_n (\calF\otimes_{\calO_\Prism,\phi_{\calO_\Prism}}\calO_\Prism)/(p,\calI_\Prism)^n(\calF\otimes_{\calO_\Prism,\phi_{\calO_\Prism}}\calO_\Prism).
\]
Then $\widehat{\phi}^\ast\calF$ is also a completed crystal of $\calO_\Prism$-modules on $(X,M_X)_\Prism$. In fact, we have $\calF=\varprojlim_n\calF_n$ as in (1). Then $\widehat{\phi}^\ast\calF=\varprojlim_n \phi_{\calO_\Prism}^\ast\calF_n$ because $\phi_{\calO_\Prism}^\ast\calF_n=(\calF\otimes_{\calO_\Prism,\phi_{\calO_\Prism}}\calO_\Prism)/(p,\phi_\Prism(\calI_\Prism))^n(\calF\otimes_{\calO_\Prism,\phi_{\calO_\Prism}}\calO_\Prism)$ and $(p,\calI_\Prism)\supset (p,\phi_{\calO_\Prism}(\calI_\Prism))\supset (p,\calI_\Prism)^p$. Since $(\phi_{\calO_\Prism}^\ast\calF_n)_n$ satisfies the assumption in (2), we conclude that $\widehat{\phi}^\ast\calF$ is a completed prismatic crystal. This argument also shows that $(\widehat{\phi}^\ast\calF)_A=\calF_A\widehat{\otimes}_{A,\phi_A}A$ for every $(A,I,M_{\Spf A})\in (X,M_X)_\Prism$.
\end{enumerate}
\end{rem}

\begin{defn}[{cf.~\cite[Def.~1.15]{Tian-prism-etale-comparison}}] \hfill
\begin{enumerate}
    \item Define $\CR^{\wedge, \an}((X, M_X)_{\Prism}, \calO_{\Prism})$ to be the full subcategory of $\CR^{\wedge}((X, M_X)_{\Prism}, \calO_{\Prism})$ consisting of $\calF$ such that $\calF[\calI_\Prism^{-1}]$ (resp.~$\calF[p^{-1}]$) is a locally free sheaf of $\calO_\Prism[\calI_\Prism^{-1}]$-modules (resp.~$\calO_\Prism[p^{-1}]$-modules). Equivalently, one can require that for every $(A, I, M_{\Spf A}) \in (X, M_X)_{\Prism}^\op$, the restriction $\calF_A^\sim|_{\Spec(A)\smallsetminus V(p, I)}$ of the quasi-coherent sheaf $\calF_A^\sim$ on $\Spec(A)$ associated to $\calF_A$ be a vector bundle.

    \item Define the category $\CR^{\wedge, \an, \phi}((X, M_X)_{\Prism}, \calO_{\Prism})$ of \emph{completed prismatic $F$-crystals} on $(X,M_X)_\Prism$ to be the category of pairs $(\calF, \phi_{\calF})$ where $\calF \in \CR^{\wedge, \an}((X, M_X)_{\Prism}, \calO_{\Prism})$ and 
    \[
    \phi_{\calF}\colon (\widehat{\phi}_{\calO_{\Prism}}^\ast\calF)[\calI_{\Prism}^{-1}] \stackrel{\cong}{\rightarrow} \calF[\calI_{\Prism}^{-1}] 
    \]
    is an isomorphism.
\end{enumerate}
\end{defn}

\begin{rem}
In the small affine case with trivial log structure, the category of completed prismatic $F$-crystals 
defined in \cite[Def.~3.16]{du-liu-moon-shimizu-completed-prismatic-F-crystal-loc-system} is a full subcategory of the one defined as above by \cite[Lem.~3.24(iv)]{du-liu-moon-shimizu-completed-prismatic-F-crystal-loc-system}. However, the latter is strictly larger: first, we do not impose the saturation condition here (e.g. $\calO_\Prism/(p,\calI_\Prism)$ together with induced $\phi$ is a completed prismatic $F$-crystal in the current sense); second, we also include non-effective completed prismatic $F$-crystals. Namely, the isomorphism $\phi_\calF\colon (\widehat{\phi}_{\calO_{\Prism}}^\ast\calF)[\calI_{\Prism}^{-1}] \stackrel{\cong}{\rightarrow} \calF[\calI_{\Prism}^{-1}]$ does not necessarily come from $\phi_\calF\colon \calF\rightarrow \calF$ whose cokernel is killed by a power of $\calI_\Prism$.
In Proposition~\ref{prop: relations-prism-crystals} below, we introduce a fully faithful functor $j_\ast$ to $\CR^{\wedge, \an, \phi}((X, M_X)_{\Prism}, \calO_{\Prism})$. The objects in the essential image satisfy the saturation condition.
\end{rem}

\smallskip
\noindent
\textbf{Analytic prismatic $F$-crystals.} 

\begin{defn}[{\cite[Def.~3.1]{GuoReinecke-Ccris}, \cite[Def.~3.3]{du-liu-moon-shimizu-purity-F-crystal}}]\hfill
\begin{enumerate}
 \item Let $(A, I, M_{\Spf A}) \in (X, M_X)_{\Prism}$. Write $\Vect^{\mathrm{an}}(A,I)$ for the category of vector bundles over $\Spec(A)\smallsetminus V(p,I)$ and $\Vect^{\mathrm{an},\phi}(A,I)$ for the category of pairs $(\mathcal{E}_A, \phi_{\mathcal{E}_A})$ where $\mathcal{E}_A\in \Vect^{\mathrm{an}}(A,I)$ and $\phi_{\mathcal{E}_A}$ is an isomorphism of vector bundles
\[
    \phi_{\mathcal{E}_A}\colon (\phi_A^\ast\mathcal{E}_A)[I^{-1}] \simeq \mathcal{E}_A[I^{-1}]
\]
over $\Spec(A)\smallsetminus V(I)$.
 \item Define the category $\Vect^{\mathrm{an}(,\phi)}((X,M_X)_\Prism)$ of \emph{analytic prismatic ($F$-)crystals} over $(X,M_X)$ by
\[
\Vect^{\mathrm{an}(,\phi)}((X,M_X)_\Prism)\coloneqq \lim_{(A,I,M_{\Spf A})\in (X,M_X)_\Prism^\op} \Vect^{\mathrm{an}(,\phi)}(A,I).
\]
\end{enumerate}
\end{defn}

\smallskip
\noindent
\textbf{Relations among them.} 
Let us discuss particular situations in which we can relate these notions. 
To simplify the notation, for $(A,I,M_{\Spf A})\in (X,M_X)_\Prism$, we write $h_A$ for the sheaf it represents; by abuse of notation, we also use phrases such as a map $h_A\rightarrow h_{A'}$ or the evaluation $\calF_A\coloneqq \calF(h_A)$ of a sheaf $\calF$.

\begin{prop} \label{prop: relations-prism-crystals}
Assume that there exists $(A,I,M_{\Spf A})\in (X,M_X)_\Prism$ such that 
\begin{itemize}
  \item the sheaf $h_A$ it represents covers the final object,
  \item $A$ is a regular Noetherian ring, and
  \item the self-product $h_A^{n+1}$ is representable, say by $(A_n,I_n,M_{\Spf A_n})$ for every $n\geq 0$ (with $A=A_0$), and each of the projections $p_0^n,\ldots, p_n^n\colon \Spec A_n\rightarrow \Spec A$ is (classically) faithfully flat.
\end{itemize}
\begin{enumerate}
 \item (cf.~\cite[Prop.~6.8]{Tian-prism-etale-comparison}, \cite[Thm.~5.10]{GuoReinecke-Ccris}, \cite[Lem.~3.8]{du-liu-moon-shimizu-purity-F-crystal})
 The natural restriction functor
\[
j^\ast\colon \CR^{\wedge, \an(, \phi)}((X, M_X)_{\Prism}, \calO_{\Prism})\rightarrow \Vect^{\mathrm{an}(,\phi)}((X,M_X)_\Prism)
\]
sending $\calF$ to $(\calF_A^\sim|_{\Spec (A)\smallsetminus V(p,I)})_A$
admits a right adjoint
\[
j_\ast\colon \Vect^{\mathrm{an}(,\phi)}((X,M_X)_\Prism)\rightarrow \CR^{\wedge, \an(, \phi)}((X, M_X)_{\Prism}, \calO_{\Prism}).
\]
Moreover, the counit map $j^\ast j_\ast\rightarrow \id$ is an isomorphism.

\item One can functorially and uniquely associate to $\calF \in \CR^{\wedge, \an}((X, M_X)_{\Prism}, \calO_{\Prism})$ a perfect complex $\iota_A(\calF)\in \calD_\perf((X,M_X)_\Prism)$ together with a morphism $f\colon \iota_A(\calF)\rightarrow \calF$ in $\calD((X,M_X)_\Prism, \calO_{\Prism})$ satisfying the following properties.
\begin{enumerate}
 \item The restriction $\iota_A(\calF)|_{h_A}\rightarrow \calF|_{h_A}$ to the localized topos $\Sh((X,M_X)_\Prism)/h_A$ is given by $P^\bullet\otimes_A\calO_\Prism\rightarrow \calF_A\otimes_A\calO_\Prism\rightarrow \calF|_{h_A}$ where $P^\bullet\rightarrow \calF_A$ is a resolution by a bounded complex of finite projective $A$-modules. 
 \item For every $(B,J,M_{\Spf B})\in (X,M_X)_\Prism/h_A$ with $A\rightarrow B$ being classically flat, $f$ gives an isomorphism $\iota_A(\calF)(h_B)\xrightarrow{\cong}\calF_B$ in $\calD(B)$. 
 \item The induced map $R\Gamma((X,M_X)_\Prism,\iota_A(\calF))\rightarrow R\Gamma((X,M_X)_\Prism,\calF)$ is a quasi-isomorphism.
\end{enumerate}
If $\calF$ comes with an isomorphism $\phi_\calF\colon (\widehat{\phi}_{\calO_{\Prism}}^\ast\calF)[\calI_{\Prism}^{-1}] \stackrel{\cong}{\rightarrow} \calF[\calI_{\Prism}^{-1}]$, then $\iota_A(\calF)$ comes with an isomorphism $\phi_{\iota_A(\calF)}\colon (L\phi_{\calO_{\Prism}}^\ast \iota_A(\calF))[\calI_{\Prism}^{-1}] \stackrel{\cong}{\rightarrow} \iota_A(\calF)[\calI_{\Prism}^{-1}]$ compatible with $\phi_\calF$ via $\iota_A(\calF)\rightarrow \calF$.
\end{enumerate}
Hence we obtain fully faithful functors
\[
\Vect^{\mathrm{an} (, \phi)}((X,M_X)_\Prism) \rightarrow 
\CR^{\wedge, \an (, \phi)}((X, M_X)_{\Prism}, \calO_{\Prism}) \rightarrow \calD_\perf^{(\phi)}((X,M_X)_\Prism).
\]
\end{prop}

The construction in (2) a priori depends on the choice of $h_A$. In the semistable case, we show that it is independent of the choice and is globalized (see \S~\ref{subsec: analy-prism-crys-perf-cx}).

\begin{proof}
(1) The uniqueness follows from the adjoint property. 
 For each $n \geq 0$, let $j^n\colon \Spec(A_n)\smallsetminus V(p, I_n) \hookrightarrow \Spec(A_n)$ denote the open immersion (note $A_0=A$). We write $p_0^n,\ldots, p_n^n\colon \Spec A_n\rightarrow \Spec A$ for the projections.
 
Let $\calF \in \Vect^{\mathrm{an}}((X,M_X)_\Prism)$, and write $\calF_{/B}\in \Vect^{\an}(B,J)$ for the value at $(B,J,M_{\Spf B})$. By the crystal property of $\calF$, we have an isomorphism $\varepsilon_{/A_1}\colon p_0^{1,*}\calF_{/A} \xrightarrow{\cong}\calF_{/A_1}\xleftarrow{\cong} p_1^{1,*}\calF_{/A}$ in $\Vect^{\an}(A_1, I_1)$ satisfying the cocycle condition in $\Vect^{\an}(A_2, I_2)$. 

Since $A$ is a Noetherian regular ring, $j_*^0 \calF_{/A}$ is a coherent sheaf on $\Spec(A)$ by \cite[0BK3]{stacks-project}. Thus $\fkF_A\coloneqq H^0(\Spec(A), j_*^0 \calF_{/A}) = H^0(\Spec(A)\smallsetminus V(p, I), \calF_{/A})$ is a finitely generated $A$-module, and so it is classically $(p, I)$-complete. Furthermore, we have $j^{0, *}\fkF_A^\sim = \calF_{\fkS_{R, \square}}$ by \cite[0BK0]{stacks-project}. Since $A\rightarrow A_n$ is assumed to be flat, the flat base change for quasi-coherent sheaves \cite[02KH]{stacks-project} gives  a canonical isomorphism
\[
\fkF_A\otimes_{A, p^n_i} A_n \cong H^0(\Spec(A_n), j^n_*(p^{n, *}_i \calF_{/A})).
\]
We also note that $\fkF_A\otimes_{A, p^n_i} A_n$ is classically $(p,I_n)$-complete by \cite[0912]{stacks-project} and $p^{n, *}_i \calF_{/A}\cong j^{n, *}(\fkF_A\otimes_{A, p^n_i} A_n)^\sim$. In particular, $\varepsilon_{/A_1}$ induces an isomorphism 
\[
\varepsilon\colon \fkF_A\otimes_{A, p^1_0} A_1 \cong \fkF_A\otimes_{A, p^1_1} A_2
\]
of $A_1$-modules satisfying the cocycle condition over $A_2$. We also remark $\Delta^\ast\varepsilon=\id_{\fkF_A}$ where $\Delta$ denotes the diagonal $\Spec A\rightarrow \Spec A_1$; to see this, consider the pullback of the cocycle condition along the triple diagonal $\Spec A\rightarrow \Spec A_2$ to obtain $\Delta^\ast\varepsilon\circ \Delta^\ast\varepsilon=\Delta^\ast\varepsilon$.

For every morphism $(A,I,M_{\Spf A})\rightarrow (B,IB,M_{\Spf B})$ in $(X,M_X)_\Prism^\op$, define $\fkF_B$ to be the classical $(p,IB)$-completion $\fkF_A\widehat{\otimes}_AB$. Thanks to $\varepsilon$, the $B$-module $\fkF_B$ depends only on the log prism $(B,IB,M_{\Spf B})$ and is independent of the map $h_B\rightarrow h_A$. Moreover, the induced maps $\fkF_A[p^{-1}]\otimes_AB\rightarrow\fkF_B[p^{-1}]$ and $\fkF_A[I^{-1}]\otimes_AB\rightarrow\fkF_B[I^{-1}]$ are isomorphisms; this follows from \cite[Lem.~3.24(iv)]{du-liu-moon-shimizu-completed-prismatic-F-crystal-loc-system} as the proof therein only uses the Noetherian property of the Breuil--Kisin prism. In particular, we have a functorial isomorphism $\calF_{/B}\cong \fkF_B^\sim|_{\Spec (B)\smallsetminus V(p,J)}$.

Using the faithfully flat descent and Remark~\ref{rem:completed crystal as limit}, it is standard to see that the association $(h_B\rightarrow h_A)\mapsto \fkF_B$ is a sheaf of $\calO_\Prism$-modules on $(X,M_X)_\Prism/h_A$ and it further descends to a completed crystal $\fkF$ of $\calO_\Prism$-modules on $(X,M_X)_\Prism$; see \cite[Prop.~3.26, Pf.]{du-liu-moon-shimizu-completed-prismatic-F-crystal-loc-system} for example. The construction also shows that $\fkF$ is indeed an object of $\CR^{\wedge, \an}((X, M_X)_{\Prism}, \calO_{\Prism})$ by \cite[Thm.~7.8]{mathew-descent}, and provides a canonical isomorphism $\calF\cong j^\ast\fkF$. Since the association $\calF\mapsto \fkF$ is functorial, setting $j_\ast\calF\coloneqq \fkF$ defines a functor $j_\ast\colon \Vect^{\mathrm{an}}((X,M_X)_\Prism)\rightarrow \CR^{\wedge, \an}((X, M_X)_{\Prism}, \calO_{\Prism})$. One can easily check that $j_\ast$ is right adjoint to $j^\ast$ and the counit map $j^\ast j_\ast\rightarrow \id$ agrees with the above identification, and thus an isomorphism. It is also straightforward to see from the construction that $j_\ast$ can be upgraded to a functor $\Vect^{\mathrm{an},\phi}((X,M_X)_\Prism)\rightarrow 
\CR^{\wedge, \an, \phi}((X, M_X)_{\Prism}, \calO_{\Prism})$ with the same properties. Note that $j_{\ast}$ is fully faithful since $j^\ast j_\ast\xrightarrow{\cong}\id$.

(2) We first note that the evaluation induces an equivalence
\begin{equation} \label{eq: compl-faith-flat-descent-wrt-A}
\calD_\perf^{(\phi)}((X,M_X)_\Prism) \xrightarrow{\cong} \lim_{[n] \in \Delta} \calD_\perf^{(\phi)}(A_n)    
\end{equation}
by $(p, I)$-completely faithfully flat descent since $h_A$ covers the final object. Let $\calF \in \CR^{\wedge, \an}((X, M_X)_{\Prism}, \calO_{\Prism})$. Since $A$ is regular Noetherian, we have a resolution $P^{\bullet} \rightarrow \calF_A$ by a bounded complex of finite projective $A$-modules. 
For any map $h_B\rightarrow h_A$, we have morphisms of complexes of $B$-modules $P^\bullet\otimes_AB\rightarrow \calF_A\otimes_AB\rightarrow \calF_B$.
Any coface map $p^n_i\colon A \rightarrow A_n$ is classically flat by assumption, so the resulting maps
\[
P^{\bullet}\otimes_{A, p^n_i} A_n 
\xrightarrow{\cong} \calF_A\otimes_{A, p^n_i} A_n \xrightarrow{\cong} \calF_{A_n}
\]
are quasi-isomorphisms by Remark~\ref{rem:completed crystal as limit}(1). By the equivalence \eqref{eq: compl-faith-flat-descent-wrt-A}, this defines $\iota_A(\calF) \in \calD_\perf((X,M_X)_\Prism)$ together with a morphism $f\colon \iota_A(\calF) \rightarrow \calF$ in $\calD((X,M_X)_\Prism, \calO_{\Prism})$ such that $\iota_A(\calF)|_{h_A}$ is given by $P^{\bullet}\otimes_A \calO_{\Prism}$ and $f|_{h_A}$ is the composite $P^\bullet\otimes_A\calO_\Prism\rightarrow \calF_A\otimes_A\calO_\Prism\rightarrow \calF|_{h_A}$.
It is straightforward to see that this is independent of the choice of the resolution $P^\bullet\rightarrow \calF_A$ and the construction is functorial in $\calF$.
Let $(B,J,M_{\Spf B})\in (X,M_X)_\Prism/h_A$ with $A\rightarrow B$ being classically flat. Then we have the natural isomorphism $\calF_A\otimes_A B \cong \calF_B$ again by Remark~\ref{rem:completed crystal as limit}(1), so $f$ gives an isomorphism $\iota_A(\calF)(h_B)\xrightarrow{\cong}\calF_B$ in $\calD(B)$. 
We have thus verified (a) and (b), and it is easy to see that these two properties  uniquely characterize $\iota_A(\calF)$ and $f$. 

For (c), observe that $f$ yields a commutative diagram
\[
\xymatrix{
R\Gamma((X,M_X)_\Prism,\iota_A(\calF)) \ar[r] \ar[d]^{\cong} & R\Gamma((X,M_X)_\Prism,\calF) \ar[d]^{\cong}\\
\Tot(\iota_A(\calF)(h_{A_\bullet})) \ar[r]^-{\cong} & \Tot(\calF_{A_\bullet})
}
\]
where $\Tot(C^\bullet)$ denotes the cochain complex associated to a cosimplicial complex (or a module) $C^\bullet$, and the vertical quasi-isomorphisms are given by \v{C}ech cohomological descent and 
the higher vanishing $H^i((X,M_X)_\Prism/h_{A_n},\calG)=0$ for every $n\geq 0$, $i\geq 1$ with $\calG$ a finite projective $\calO_\Prism$-module or a completed crystal of $\calO_\Prism$-module (cf.~Lemmas~\ref{lem:CA method} and \ref{lem: CA-method-infinity-category}). Since the bottom horizontal morphism is a quasi-isomorphism by (b), so is the top horizontal morphism. 

Suppose that $\calF$ comes with $\phi_\calF\colon (\widehat{\phi}_{\calO_{\Prism}}^\ast\calF)[\calI_{\Prism}^{-1}] \stackrel{\cong}{\rightarrow} \calF[\calI_{\Prism}^{-1}]$. Then by the equivalence \eqref{eq: compl-faith-flat-descent-wrt-A}, the above construction also gives an isomorphism $\phi_{\iota_A(\calF)}\colon (L\phi_{\calO_{\Prism}}^\ast \iota_A(\calF))[\calI_{\Prism}^{-1}] \stackrel{\cong}{\rightarrow} \iota_A(\calF)[\calI_{\Prism}^{-1}]$ compatible with $\phi_\calF$ via $f$.

It remains to show that the functor $\iota_A\colon \CR^{\wedge, \an (, \phi)}((X, M_X)_{\Prism}, \calO_{\Prism}) \rightarrow \calD_\perf^{(\phi)}((X,M_X)_\Prism)$ is fully faithful. Let $\calF, \calG \in \CR^{\wedge, \an}((X, M_X)_{\Prism}, \calO_{\Prism})$. 
On one hand, as in the proof of (1), we deduce from the faithfully flat descent for morphisms that 
\[
\Hom_{\mathrm{Mod}((X,M_X)_\Prism,\calO_\Prism)}(\calF,\calG)\rightarrow \Hom_{\mathrm{Mod}_{A}}(\calF_A,\calG_A)\rightrightarrows\Hom_{\mathrm{Mod}_{A_1}}(\calF_{A_1},\calG_{A_1})
\]
is exact. On the other hand, for any $n \geq 0$, $\calF_{A_n}$ and $\calG_{A_n}$ in $\calD_{\perf}(A_n)$ lie in cohomological degree zero, so we have
\[
\Hom_{\mathrm{Mod}_{A_n}}(\calF_{A_n}, \calG_{A_n}) = \Hom_{\calD(A_n)}(\calF_{A_n}, \calG_{A_n}) = \Hom_{\calD(A_n)}(\iota_A(\calF)(h_{A_n}), \iota_A(\calG)(h_{A_n})).
\]
Thus, $\iota_A$ is fully faithful by the equivalence \eqref{eq: compl-faith-flat-descent-wrt-A}.
\end{proof}

\section{Prismatic-crystalline comparison for associated sheaves} \label{sec: prismatic-crystalline comparison}

We study the comparison between log prismatic cohomology and crystalline cohomology with coefficients given by $p$-adically completed crystals satisfying the association condition (given in Definition~\ref{defn: associated-crystals}). Our argument is based on \cite[\S~6]{koshikawa} where the constant coefficient case is proved (cf.~\cite[Thm.~6.3]{koshikawa}). The main result in this section (Theorem~\ref{thm: pris-crys-comparison-associated-sheaves}) will be used in \S~\ref{sec: pris-cris-comparison-crystals-semist} for the semistable case. We refer the reader to Appendix~\ref{sec: cryst-sites-variants} for the definitions and properties of crystalline and $\delta_{\log}$-crystalline sites; in this section, we use the crystalline site used in \cite{koshikawa} and $p$-adically completed crystals (Definitions~\ref{def: Koshikawa-crystalline-site} and \ref{defn:p-complete crystals over OCRIS}), and their relation to the ones in \cite{du-moon-shimizu-cris-pushforward} is explained in Proposition~\ref{prop:comparison of crystals in two crystalline sites}.

\begin{set-up} \label{set-up: base-cryst-prism}
Fix a crystalline log prism $(A, (p), M_{\Spf A}) = (A, (p), M_A)^a$ where $(A, (p), M_A)$ is a prelog prism in the sense of \cite[Def.~3.3]{koshikawa}. Assume that $M_A$ is integral, and that either the $\delta_{\log}$-ring $(A, M_A)$ is of rank $1$ or is a log ring. Let $I \subset A$ be a PD-ideal containing $p$. Then the Frobenius $\phi$ on the prelog ring $(A/(p), M_A)$ factors as
\[
(A/p, M_A)\rightarrow (A/I, M_A)\xrightarrow{\psi} (\phi_\ast A/p,\phi_\ast M_A),
\]
where the first map is the natural quotient map, and $(\phi_\ast A/p, \phi_\ast M_A)$ denotes $(A/p, M_A)$ regarded as a prelog ring over itself by $\phi$.

Let $(Y, M_Y)$ be a smooth log scheme over $(A/I, M_A)$ of Cartier type as in \cite[Def.~4.8]{Kato-log}. 
Set $(Y^{(1)}, M_{Y^{(1)}})\coloneqq (Y, M_Y)\times_{(\Spec A/I,M_A^a),\psi}(\Spec \phi_\ast A/p,(\phi_\ast M_A)^a)$.
\end{set-up}

Recall from Construction~\ref{const: delta-cris to affine cris} that we have morphisms of topoi
\[
u_{Y}^\delta\colon \Sh(((Y,M_Y)/(A,M_A))_{\dCRIS})\xrightarrow{\nu_\CRIS} \Sh(((Y,M_Y)/(A,M_A))_\CRIS)\xrightarrow{u_Y} \Sh(Y_\et).
\]

\begin{prop}[{cf.~\cite[Prop.~6.8]{koshikawa}}] \label{prop:comparison-delta-log-cris-log-cris}
Let $\calE_\CRIS$ be a $p$-adically completed crystal of $\calO_\CRIS$-modules on $((Y,M_Y)/(A,M_A))_{\CRIS}$, and write $\calE_{\dCRIS}\coloneqq \nu_{\CRIS}^\ast\calE_{\CRIS}$ for the associated sheaf of $\calO_\dCRIS$-modules on $((Y,M_Y)/(A,M_A))_{\dCRIS}$. Then the induced morphism
\[
Ru_{Y,\ast}\calE_{\CRIS} \rightarrow Ru_{Y,\ast}^\delta\calE_{\dCRIS}.
\]
is an isomorphism in $\mathcal{D}(Y_{\et}, A)$.
\end{prop}

\begin{proof}
We follow a similar argument as in the proof of \cite[Prop.~6.8]{koshikawa}. Working \'etale locally on $Y$, we may assume that we have a smooth chart $P \rightarrow (Y, M_Y)$ as in \cite[Def.~A.11]{koshikawa} and $Y$ is affine, say $Y = \Spf R$. Then we have an exact surjection $(\widetilde{R}, \widetilde{P}) \rightarrow (R, P)$ from a smooth lift $(\widetilde{R}, \widetilde{P})$ over $(A/p, M_A)$. Choose a surjection
\[
(B_0, M_{B_0}) \coloneqq (A\langle \mathbf{N}^T\rangle, M_A\oplus \mathbf{N}^T) \twoheadrightarrow (\widetilde{R}, \widetilde{P})
\]
for some finite set $T$ whose kernel is Zariski locally generated by $p$ and a $p$-completely regular sequence relative to $A$. Write $(B_0^{\bullet}, M_{B_0}^{\bullet})$ for the \v{C}ech nerve of $(A, M_A) \rightarrow (B_0, M_{B_0})$, and $(C_0^{\bullet}, M_{C_0}^{\bullet})$ for the $p$-completed log PD-envelopes of surjections $(B_0^{\bullet}, M_{B_0}^{\bullet}) \rightarrow (R, P)$ as $p$-adic log-affine prelog rings. Note that $C_0^{\bullet}$ are $p$-completely flat over $A$ by \cite[Lem.~2.43]{bhatt-scholze-prismaticcohom}. Let $J$ be the kernel of $C_0 \rightarrow R$. Since $(C_0, J, M_{C_0})^a$ is weakly final in $((Y,M_Y)/(A,M_A))_{\CRIS}$, $R\Gamma(((Y,M_Y)/(A,M_A))_{\CRIS},\calE_{\CRIS})$ is computed by the totalization of the cosimplicial $A$-complex $\mathcal{E}_{\CRIS}(C_0^{\bullet})$ (cf.~Lemmas~\ref{lem: affine vanishing on crystalline site} and \ref{lem:CA method}). 

On the other hand, to compute $Ru_{Y,\ast}^\delta\calE_{\dCRIS}$, consider $B\coloneqq (A\{\mathbf{N}^T\}^{\delta}_{\log})^{\wedge}_p$ where $\{\mathbf{N}^T\}_{\log}^{\delta}$ denotes the free $\delta_{\log}$-ring with variables in $T$ as in \cite[Notation~2.7, Rem.~2.8]{koshikawa}, and write $B^{\bullet}$ for the $p$-completed \v{C}ech nerve of $(A, M_A) \rightarrow (B, M_B)$. Let $C = (B\otimes_{B_0}C_0)^{\wedge}_p$. Since $B_0 \rightarrow B$ is $p$-completely flat, $(JC)^{\wedge}_p$ is a PD-ideal of $C$, and $C$ agrees with the $p$-completed log PD-envelope of $(B, M_B)$ with respect to an ideal Zariski locally generated by $p$ and a $p$-completely regular sequence relative to $A$. In particular, $C$ is a $\delta$-ring over $C_0$ by \cite[Cor.~2.39]{bhatt-scholze-prismaticcohom} and Zariski descent. Note that $(C, (JC)^{\wedge}_p, M_C)^a$ is a weakly final object in $((Y,M_Y)/(A,M_A))_{\dCRIS}$, and denote its \v{C}ech nerve by $C^{\bullet}$. Again by \cite[Cor.~2.39]{bhatt-scholze-prismaticcohom} and Zariski descent, each $(B^{n}\otimes_{B_0^{n}}C_0^{n})^{\wedge}_p$ is a $\delta$-ring over $C_0^n$.  we have
\[
C^{\bullet} \cong (B^{\bullet}\otimes_{B_0^{\bullet}}C_0^{\bullet})^{\wedge}_p.
\]

Now, $Ru_{Y,\ast}^\delta\calE_{\dCRIS}$ is computed by the totalization of the cosimplicial $A$-complex $\mathcal{E}_{\CRIS}(C^{\bullet})$. Note that $B$ is $p$-completely free over $B_0$ by construction, so the map $B_0^{\bullet} \rightarrow B^{\bullet}$ is a cosimplicial homotopy equivalence by a similar construction as in \cite[Ex.~2.16]{bhatt-deJong-crys-cohom-deRham-cohom}. Since $\mathcal{E}_{\CRIS}$ is a $p$-adically completed crystal, we have
\[
\mathcal{E}_{\CRIS}(C^{\bullet}) \cong \mathcal{E}_{\CRIS}(C_0^{\bullet})\widehat{\otimes}_{C_0^{\bullet}} C^{\bullet} \cong \mathcal{E}_{\CRIS}(C_0^{\bullet})\widehat{\otimes}_{B_0^{\bullet}} B^{\bullet}.
\]
Thus, the map $\mathcal{E}_{\CRIS}(C_0^{\bullet}) \rightarrow \mathcal{E}_{\CRIS}(C^{\bullet})$ is a cosimplicial homotopy equivalence.
\end{proof}

We now compare the log prismatic cohomology and $\delta$-log crystalline cohomology. Recall from Construction~\ref{const: f^J_Prism} that we have a cocontinuous functor
\[
((Y,M_Y)/(A,M_A))_{\dCRIS}\rightarrow ((Y^{(1)},M_Y^{(1)})/(\phi_\ast A,(p), \phi_\ast M_A)^a)_{\Prism,\et},
\]
which induces a morphism of ringed topoi
\[
\Sh(((Y,M_Y)/(A,M_A))_{\dCRIS}, \phi_{\ast}\calO_{\dCRIS}) \xrightarrow{\nu_\Prism} \Sh(((Y^{(1)},M_Y^{(1)})/(\phi_\ast A, (p), \phi_\ast M_A)^a)_{\Prism,\et}, \calO_{\Prism}).
\]
By construction, we have $\nu_\Prism^{-1}\calO_{\Prism} = \phi_{\ast} \calO_{\dCRIS}$.

\begin{defn} \label{defn: associated-crystals}
Let $\calE_{\dCRIS}$ be a $p$-adically completed crystal of $\calO_{\dCRIS}$-modules on $(Y,M_Y)/(A,M_A))_{\dCRIS}$, and let $\calE^{(1)}_{\Prism}$ be a completed prismatic crystal of $\calO_{\Prism}$-modules on $((Y^{(1)},M_Y^{(1)})/(\phi_{\ast} A, (p), \phi_{\ast} M_A)^a)_{\Prism, \et}$.
We say that $\calE_{\dCRIS}$ and $\calE^{(1)}_\Prism$ are \emph{associated} if there exists an isomorphism
\[
c\colon \nu_\Prism^{-1}\calE^{(1)}_\Prism \xrightarrow{\cong} \phi_\ast \calE_{\dCRIS}
\]    
of $\phi_{\ast} \calO_{\dCRIS}$-modules.
\end{defn}

Let $\nu^{(1)}$ be the natural morphism of topoi
\[
\nu^{(1)} \colon \Sh(((Y^{(1)},M_{Y}^{(1)})/(\phi_{\ast} A, (p), \phi_{\ast} M_A)^a)_{\Prism}) \rightarrow \Sh((Y^{(1)})_{\et})
\]
as given in \cite[Rem.~4.5]{koshikawa}.

\begin{prop}[{cf.~\cite[\S~6.2]{koshikawa}}] \label{prop:comparison-prismatic-delta-log-cris}
Keep the notation as above, and assume that $\calE_{\dCRIS}$ and $\calE^{(1)}_\Prism$ are associated by an isomorphism $c$. Then the map
\[
R\nu^{(1)}_{*}\calE^{(1)}_{\Prism} \rightarrow \phi_* R(\nu^{(1)}\circ \nu_\Prism)_{*}\calE_{\dCRIS}
\]
induced by $c$ is an isomorphism in $\mathcal{D}((Y^{(1)})_{\et}, \phi_* A)$.
\end{prop}

\begin{proof}
We follow a similar argument as in \cite[\S~6.2]{koshikawa}. Working \'etale locally, we can assume that $Y = \Spf R$ with a smooth chart $P \rightarrow (Y, M_Y)$ as in \cite[Def.~A.11]{koshikawa}. We first compute the $\delta_{\log}$-crystalline cohomology as in the proof of Proposition~\ref{prop:comparison-delta-log-cris-log-cris}: let $(\widetilde{R}, \widetilde{P}) \rightarrow (R, P)$ be an exact surjection from a smooth lift $(\widetilde{R}, \widetilde{P})$ over $(A/p, M_A)$, and choose a surjection
\[
(B_0, M_{B_0}) \coloneqq (A\langle \mathbf{N}^T\rangle, M_A\oplus \mathbf{N}^T) \twoheadrightarrow (\widetilde{R}, \widetilde{P})
\]
for some finite set $T$ whose kernel is Zariski locally generated by $p$ and a $p$-completely regular sequence relative to $A$. Let $B\coloneqq (A\{\mathbf{N}^T\}^{\delta}_{\log})^{\wedge}_p$ and $B^{\bullet}$ the $p$-completed \v{C}ech nerve of $(A, M_A) \rightarrow (B, M_B)$. Write $(B_0^{\bullet}, M_{B_0}^{\bullet})$ for the \v{C}ech nerve of $(A, M_A) \rightarrow (B_0, M_{B_0})$, and $(C_0^{\bullet}, M_{C_0}^{\bullet})$ for the $p$-completed log PD envelopes of surjections $(B_0^{\bullet}, M_{B_0}^{\bullet}) \rightarrow (R, P)$. Let $((B_0^{\bullet})', (M_{B_0}^{\bullet})')$ be the $p$-completed exactification of the surjection $(B_0^{\bullet}, M_B^{\bullet}) \rightarrow (\widetilde{R}, \widetilde{P})$. Then $C_0^{\bullet}$ agrees with the $p$-completed PD envelope of $\widetilde{J}^{\bullet} \subset (B_0^{\bullet})'$, where $\widetilde{J}^{\bullet} \coloneqq \Ker((B_0^{\bullet})' \rightarrow \widetilde{R})$. 

Write $B = (A\{\mathbf{N}^T\}^{\delta}_{\log})^{\wedge}_p$ and $B^{\bullet}$ for the $p$-completed \v{C}ech nerve of $(A, M_A) \rightarrow (B, M_B)$. Denoting by $C^{\bullet}$ the $p$-completed PD envelope of $(B^{\bullet})' \coloneqq (B^{\bullet}\otimes_{B_0^{\bullet}}(B_0^{\bullet})')^{\wedge}_p$ with respect to $(\widetilde{J}^{\bullet}(B^{\bullet})')^{\wedge}_p$, the $\delta_{\log}$-crystalline cohomology is computed by $(C^{\bullet}, M_C^{\bullet})$. Since $B_0^{\bullet} \rightarrow B^{\bullet}$ is $p$-completely flat, we have $C^{\bullet} = (B^{\bullet}\otimes_{B_0^{\bullet}}C_0^{\bullet})^{\wedge}_p$. 

For any $A$-algebra $E$, we write $\phi_A^* E\coloneqq A\otimes_{\phi, A} E$ regarded as a $\phi_* A$-algebra. To compute the log prismatic cohomology, consider the induced surjection
\[
((\phi_A^*B_0)^{\wedge}_p, M_B^{(1)}) \rightarrow (\widetilde{R}^{(1)}, \widetilde{P}^{(1)})
\]
over $(\phi_* A, \phi_* M_A)$. Then the prismatic envelope $D^{\bullet}$ over $(\phi_*A, (p))$ of $(\phi_A^*\widetilde{J}^{\bullet}(B^{\bullet})')^{\wedge}_p \subset (\phi_A^*(B^{\bullet})')^{\wedge}_p$ computes the log prismatic cohomology. Consider the relative Frobenius
\[
1\otimes \phi_{(B^{\bullet})'}\colon \phi_A^*(B^{\bullet})' \rightarrow \phi_*(B^{\bullet})'.
\]
By \cite[Cor.~2.39]{bhatt-scholze-prismaticcohom}, this extends to $D^{\bullet} \rightarrow \phi_* C^{\bullet}$. By construction, this is compatible with the morphism of ringed topoi. 

For each $m\geq 1$, let $\mathcal{E}^{(1)}_{\Prism, m}$ denote the mod $p^m$ reduction of the completed prismatic crystal $\mathcal{E}^{(1)}_{\Prism}$ as in Remark~\ref{rem:completed crystal as limit}(1) (recall that our base log prism is crystalline); we have $\mathcal{E}^{(1)}_{\Prism} \cong \varprojlim_m \mathcal{E}^{(1)}_{\Prism, m} \cong R\varprojlim_m \mathcal{E}^{(1)}_{\Prism, m}$ by \textit{loc.~cit.} and Corollary~\ref{cor:repleteness}. Similarly, we deduce from the fact that $\calE_{\dCRIS}$ is a $p$-adically completed crystal that the presheaf $\calE_{\dCRIS, m}$ given by $\calE_{\dCRIS, m}(B_1, J_1, M_{B_1})\coloneqq \calE_{\dCRIS}(B_1, J_1, M_{B_1})/p^m \calE_{\dCRIS}(B_1, J_1, M_{B_1})$ is a sheaf and that $\calE_{\dCRIS} \cong \varprojlim_m \calE_{\dCRIS, m} \cong R\varprojlim_m \calE_{\dCRIS, m}$ (imitate the proof of Lemma~\ref{lem: affine vanishing on crystalline site}). The map $c$ yields an isomorphism $c_m\colon \nu_\Prism^{-1}\calE^{(1)}_{\Prism,m} \xrightarrow{\cong} \phi_\ast \calE_{\dCRIS,m}$.

It suffices to show that for each $m \geq 1$, the map 
\[
\mathcal{E}^{(1)}_{\Prism, m}(D^{\bullet}) \rightarrow \mathcal{E}^{(1)}_{\Prism, m}(\phi_* C^{\bullet}) \cong \mathcal{E}^{(1)}_{\Prism, m}(D^{\bullet})\otimes_{D^{\bullet}} \phi_* C^{\bullet} 
\]
induced by $c_m$ gives a quasi-isomorphism on the associated totalizations; in fact, the proposition is deduced from this by taking $R\varprojlim_m$ and using \cite[0D6K]{stacks-project}.
Finally, the claimed quasi-isomorphism follows from \cite[Lem.~5.4]{bhatt-scholze-prismaticcohom} and \cite[Prop.~B.3]{koshikawa} together with the description of $D^{\bullet} \rightarrow \phi_* C^{\bullet}$ in \cite[\S~6.2, last para.]{koshikawa}: recall from the construction that $B_0'$ is $p$-completely free over the $p$-completion of $A\otimes_{\mathbf{Z}_p[M_A]} \mathbf{Z}_p[M_B']$. Note that $M_A \rightarrow M_B'$ is of Cartier type by assumption. Write $G \coloneqq (M_B')^{\gp} / M_A^{\gp}$, which is a free abelian group. Then $(M_B^n)'$ is isomorphic to $M_B'\oplus G^n$ for each $n$, and we have the cosimplicial $A$-algebra
\begin{equation} \label{eq: log-Cartier-cosimplicial-algebra}
A\otimes_{\mathbf{Z}_p[M_A]} \mathbf{Z}_p[M_B'] \rightarrow A\otimes_{\mathbf{Z}_p[M_A]} \mathbf{Z}_p[M_B'\oplus G] \rightarrow A\otimes_{\mathbf{Z}_p[M_A]} \mathbf{Z}_p[M_B'\oplus G^2] \rightarrow \cdots    
\end{equation}
as in \cite[Appx.~B]{koshikawa}. 
So the mod $p^m$ reduction of $D^{\bullet} \rightarrow \phi_* C^{\bullet}$ decomposes into two maps: the relative Frobenius on the free polynomial algebra part induced by the Frobenii on $A$ and $B$, and the relative Frobenius on \eqref{eq: log-Cartier-cosimplicial-algebra} as in \cite[Appx.~B.1]{koshikawa}. We obtain the quasi-isomorphism by applying \cite[Lem.~5.4]{bhatt-scholze-prismaticcohom} to the free part\footnote{In the statement of \cite[Lem.~5.4]{bhatt-scholze-prismaticcohom}, $k$ is assumed to be of characteristic $p$. However, given the relative Frobenius, the same argument works in our set-up.} and \cite[Prop.~B.3]{koshikawa} to the part on \eqref{eq: log-Cartier-cosimplicial-algebra}.
\end{proof}

\begin{thm}[cf. {\cite[Thm.~6.3]{koshikawa}}] \label{thm: pris-crys-comparison-associated-sheaves}
Assume Set-up~\ref{set-up: base-cryst-prism}, and let $(Y, M_Y)$ be a smooth log scheme over $(A/I, M_A)$ of Cartier type. Let $\calE^{(1)}_{\Prism}$ be a completed prismatic crystal of $\calO_{\Prism}$-modules on $((Y^{(1)},M_Y^{(1)})/(\phi_\ast A, (p), \phi_\ast M_A)^a)_{\Prism}$, and let $\calE_{\CRIS}$ be a $p$-adically completed crystal of $\calO_{\CRIS}$-modules on $((Y,M_Y)/(A,M_A))_{\CRIS}$.  Suppose $\calE^{(1)}_{\Prism}$ and $\nu_{\CRIS}^\ast\calE_{\CRIS}$ are associated by $c$. Then $c$ induces an isomorphism
\[
R\nu^{(1)}_{*}\calE^{(1)}_{\Prism} \cong  \phi_* Ru_{Y,\ast}\calE_{\CRIS}
\]
in $\mathcal{D}(Y_{\et}, \phi_* A)$ (under the natural identification $\mathcal{D}(Y_{\et}, \phi_* A) \cong \mathcal{D}(Y^{(1)}_{\et}, \phi_* A))$.
\end{thm}

This generalizes \cite[Thm.~6.3]{koshikawa}, which works on the constant coefficient.

\begin{proof}
This follows from Propositions~\ref{prop:comparison-delta-log-cris-log-cris} and \ref{prop:comparison-prismatic-delta-log-cris}.
\end{proof}

\section{Prismatic-crystalline comparison on semistable formal schemes} \label{sec: semistable-case}

This section discusses the prismatic-crystalline comparison theorem on semistable formal schemes and defines the Hyodo--Kato complex. In \S~\ref{sec: semistable-case-notation}, we first set up notations related to semistable formal schemes and explain a general \v{C}ech--Alexander method for computing cohomology (Lemmas~\ref{lem:CA method} and \ref{lem: CA-method-infinity-category}). In \S~\ref{subsec: BK-prism}, we recall Breuil--Kisin log prism and construct \v{C}ech nerve given by a family of Breuil--Kisin log prisms (Construction~\ref{const: Cech-nerve-global-separated}). In \S~\ref{subsec: analy-prism-crys-perf-cx}, we associate to analytic prismatic crystals $\calE_\Prism$ completed prismatic crystals $j_\ast\calE_\Prism$ and prismatic crystals in perfect complexes $i(\calE_\Prism)$ (Propositions~\ref{prop: right-adjoint-semistable-case} and \ref{prop: functor-completed-prism-cryst-perfect-cx}). Based on the results and constructions in \S~\ref{sec: semistable-case-notation}--\S~\ref{subsec: analy-prism-crys-perf-cx}, we introduce and study Breuil--Kisin cohomology and Breuil cohomology (Definition~\ref{def: BK-cohom}), and establish their comparison results and base change properties (Proposition~\ref{prop: BK-cohom}, Theorem~\ref{thm: general-base-change-BK-cohom}) in \S~\ref{subsec: BK-cohom}. In \S~\ref{sec: pris-cris-comparison-crystals-semist}, a comparison between the prismatic and crystalline cohomology (Theorems~\ref{thm: prism-cryst-comparison-semist} and \ref{thm: pris-cris-comparison-over-Breuil S}) is studied, and the Frobenius isogeny property for crystalline and Breuil cohomology (Theorem~\ref{thm:Frobenius-isogeny-property}) is given in \S~\ref{subsec: Frobenius-isogeny-property}. We discuss the Hyodo--Kato theory in \S~\ref{subsec: hyodo-kato-theory} based on \cite{Beilinson-crystalline-period-map-arXiv}. We apply these results in \S~\ref{subsec: hyodo-kato-cohom} to define and study the Hyodo--Kato complex $R\Gamma_{\HK}((X_0, M_{X_0}), \mathcal{E}_{\cris, \Q})$ (Theorem~\ref{thm:Hyodo-Kato-cohomology}).

\subsection{Notation and overview}\label{sec: semistable-case-notation}
Let $k$ be a perfect field of characteristic $p$. Write $W = W(k)$. Let $K$ be a finite totally ramified extension of $K_0\coloneqq W[p^{-1}]$ and $e \coloneqq [K : K_0]$ the ramification index. Fix a uniformizer $\pi$ of $K$ and let $E = E(u)\in W[u]$ denote the monic minimal polynomial of $\pi$. Fix an algebraic closure $\overline{K}$ of $K$, and let $C \coloneqq \overline{K}^{\wedge}$ be its $p$-adic completion. 
Choose a nontrivial compatible system $(\epsilon_n)_{n \geq 0}$ of $p$-power roots of unity in $C$ such that $\epsilon_0 = 1$ and $\epsilon_{n+1}^p = \epsilon_n$, and write $\epsilon \coloneqq (\epsilon_n)_{n \geq 0}$ for the corresponding element in $\calO_C^{\flat}$. Also, pick a compatible system $(\pi_n)_{n \geq 0}$ of $p$-power roots of $\pi$ in $C$ with $\pi_0 = \pi$ and $\pi_{n+1}^p = \pi_n$, and let $\pi^{\flat}\coloneqq (\pi_n)_{n \geq 0}$ be the corresponding element in $\calO_{C}^{\flat}$.
We equip $\Spf \calO_K$ with the canonical log structure $M_{\can}$ (i.e., the one associated to $\calO_K\cap K^\times$). 

A \emph{semistable $p$-adic formal scheme} over $\calO_K$ refers to a $p$-adic formal scheme over $\calO_K$ that is, \'etale locally, \'etale over an affine formal scheme of the form $\Spf R^0$ with 
\[
R^0 = \mathcal{O}_K \langle T_1, \ldots, T_m, T_{m+1}^{\pm 1}, \ldots, T_d^{\pm 1}\rangle / (T_1\cdots T_m - \pi).
\]

In this section, we fix a semistable $p$-adic formal scheme $f\colon X\rightarrow \Spf \calO_K$ and always equip $X$ with the log structure $M_X$ given by the subsheaf associated to the subpresheaf $\calO_{X,\et}\cap (\calO_{X,\et}[p^{-1}])^\times \hookrightarrow \calO_{X,\et}$ (cf.~\cite[1.6(2)]{Cesnavicius-Koshikawa}).

We mainly concern the morphism of topoi
\[
f_\Prism\colon \Sh((X,M_X)_\Prism)\rightarrow \Sh((\Spf\calO_K,M_\can)_\Prism).
\]
Recall from Proposition~\ref{prop:prismatic-higher-direct-image} that for every $\calK\in D((X,M_X)_\Prism,\calO_\Prism)$ and $(A,I,M_{\Spf A})\in (\Spf \calO_K,M_\can)_\Prism^\op$, we have
\[
R\Gamma((\Spf A, I, M_{\Spf A}), Rf_{\Prism,\ast}\calK)\cong R\Gamma(((X_{A/I},M_{X_{A/I}})/(A,M_{\Spf A}))_\Prism,\calK),
\]
where $(X_{A/I},M_{X_{A/I}})\coloneqq (X,M_X)\times_{(\Spf\calO_K,M_\can)}(\Spf (A/I), M_{\Spf (A/I)})$.
Below are the main examples of $(A,I,M_{\Spf A})$.

\begin{eg} \label{eg: examples-log-prisms}
The following are log prisms in $(\Spf\calO_K, M_{\can})_{\Prism}^\op$.
\begin{enumerate}
 \item (Breuil--Kisin log prism). 
 A bounded prelog prism 
\[
(\fkS_{K}\coloneqq W(k)[\![u]\!], (E(u)), \N\rightarrow \fkS_{K}, 1\mapsto u)
\] 
 with $\delta_{\log}(1)=0$ and $\delta(u)=0$ defines a log prism
$(\fkS_{K}, (E), M_{\Spf\fkS_{K}})$ via $\calO_K \cong \fkS_{K}/(E)$. 
 \item (Breuil log prism and its Frobenius twist). Let $S_{K}$ denote the $p$-adically completed (log) PD-envelope of $\fkS_{K}$ with respect to $\Ker(\fkS_{K}\rightarrow \calO_K/p)$ equipped with the Frobenius extending the one on $\fkS_{K}$. Since $\phi(E)/p$ is a unit in $S_{K}$, we have a map $\phi\colon (\fkS_{K},(E))\rightarrow (S_{K},(p))$ of prisms, which defines a log prism
\[
(S_{K}, (p), M_{\Spf S_{K}})\coloneqq (S_{K}, (p), \phi_{\ast}\N)^a.
\]

Write $(\phi_{\ast}S_K, (p), M_{\Spf \phi_{\ast}S_K})$ for the log prism in $(\Spf\calO_K, M_{\can})_{\Prism}^{\op}$ whose underlying prism is $(S_K, (p))$ with log structure given by $\N \rightarrow S_K$, $1 \mapsto u^{p^2}$ and structure map $\calO_K \cong \fkS_K/(E) \stackrel{\phi^2}{\rightarrow} S/(p)$. Note that $\phi\colon (S_{K}, (p), M_{\Spf S_{K}}) \rightarrow (\phi_{\ast}S_K, (p), M_{\Spf \phi_{\ast}S_K})$ is a morphism in $(\Spf\calO_K, M_{\can})_{\Prism}^{\op}$.

 \item A bounded prelog prism 
 \[
 (W(k), (p), \N \rightarrow W(k), 1 \mapsto 0)
 \]
 with $\delta_{\log}(1) = 0$ defines a log prism  $(W(k), (p), M_0)$ via $\calO_K \rightarrow k \cong W(k)/(p)$.
 \item 
 Consider $A_\inf\coloneqq W(\calO_C^\flat)$ together with the surjection $\theta\colon A_\inf\rightarrow\calO_C$.
 A bounded prelog prism
\[
(A_\inf,\Ker(\theta\circ \phi^{-1}), \N\rightarrow A_\inf, 1 \mapsto [\pi^\flat]^p)
\]
defines a log prism $(A_\inf,\Ker(\theta\circ \phi^{-1}), M_{\Spf A_\inf})$\footnote{We work on the strict prismatic site, so we use the log structure associated to $\N\rightarrow A_\inf$ instead of $\calO_C^\flat\smallsetminus \{0\}\rightarrow A_\inf$. Note that $M_{\Spf A_\inf}$ is independent of the choice of $\pi$ or $\pi^\flat$.} via $\calO_K\rightarrow\calO_C\cong A_\inf/\Ker(\theta\circ \phi^{-1})$.

 \item Let $A_{\cris}$ be the $p$-completed (log) PD-envelope of $A_{\inf}$ with respect to $(\Ker(\theta),p)$, which comes with the Frobenius extending that on $A_{\inf}$. The natural inclusion map $A_{\inf} \rightarrow A_{\cris}$ gives a map of prisms $(A_{\inf}, \Ker(\theta\circ \phi^{-1}) \rightarrow (A_{\cris}, (p))$. This defines a log prism
 \[
 (A_{\cris}, (p), M_{\Spf A_{\cris}})\coloneqq (A_{\cris}, (p), \N \rightarrow A_{\inf} \rightarrow A_{\cris})^a.
 \]
\end{enumerate}
We have a following commutative diagram in $(\Spf\calO_K, M_{\can})_{\Prism}^{\op}$:
\begin{equation} \label{eq: diag-log-prism-site-O_K}
\xymatrix{
(\fkS_{K}, (E), M_{\Spf\fkS_{K}}) \ar[r]^{\phi} \ar[d]^{u\mapsto [\pi^{\flat}]^p} & (S_{K}, (p), M_{\Spf S_{K}}) \ar[d]^{u \mapsto [\pi^{\flat}]} \\
(A_\inf,\Ker(\theta\circ \phi^{-1}), M_{\Spf A_\inf}) \ar[r] &(A_{\cris}, (p), M_{\Spf A_{\cris}}).
}    
\end{equation}
\end{eg}

\begin{lem} \label{lem: finite-p-compl-Tor-amp} 
Consider the map of underlying rings in the diagram \eqref{eq: diag-log-prism-site-O_K}.
\begin{enumerate}
    \item The map $\fkS_{K} \rightarrow A_{\inf}$ is classically faithfully flat.

    \item The map $S_{K} \rightarrow A_{\cris}$ is $p$-completely faithfully flat.
    
    \item The maps $\phi\colon \fkS_{K} \rightarrow S_{K}$ and $A_{\inf} \rightarrow A_{\cris}$ have finite $p$-complete Tor amplitudes.    
\end{enumerate}
\end{lem}

\begin{proof}
Note that $\phi\colon \fkS_{K} \rightarrow \fkS_{K}$ is classically faithfully flat. So it suffices to consider the maps $\fkS_{K} \rightarrow A_{\inf}$, $u \mapsto [\pi^{\flat}]$ and $\fkS_{K} \rightarrow S_{K}$, $u \mapsto u$. The statement (2) is \cite[Lem.~5.18(4)]{Cais-Liu}. In fact, in the proof of \textit{loc. cit.}, it is shown that the map $\fkS_{K}/(p) \rightarrow A_{\inf}/(p)$ is classically faithfully flat and  $A_{\inf}\otimes_{\fkS_{K}} S_{K}/(p^n) \cong A_{\cris}/(p^n)$ for each $n \geq 1$. Since $\fkS_{K}$ is Noetherian, this implies (1). The statement (3) follows from these and \cite[pp.~1225, fn.~3]{Cais-Liu}.
\end{proof}

We end this subsection with the \v{C}ech--Alexander method.
Let us fix $(A,I,M_{\Spf A})\in (\Spf \calO_K,M_\can)_\Prism^\op$. To simplify the notation, set $\overline{A}\coloneqq A/I$ and write  $((X_{\overline{A}},M_{X_{\overline{A}}})/A)_\Prism$ for the log prismatic site $(X_{A/I},M_{X_{A/I}})/(A,M_{\Spf A}))_\Prism$. We further assume that there exists an object $(B,IB,M_{\Spf B})\in (X_{\overline{A}}/A)_\Prism^\op$ that covers the final object of the associated topos and that for each $n\geq 0$, the $(n+1)$-st self-coproduct exists, which we denote by $(B^{[n]},IB^{[n]},M_{\Spf B^{[n]}})$ or even $B^{[n]}$. This gives rise to a simplicial ringed topos with the augmentation 
\[
a\colon (\Sh(((X_{\overline{A}},M_{X_{\overline{A}}})/A)_\Prism/B^{[\bullet]}),\calO_\Prism^\bullet)\rightarrow (\Sh((X_{\overline{A}},M_{X_{\overline{A}}})/A)_\Prism,\calO_\Prism),
\]
where we write $\calO_\Prism^\bullet$ for the sheaf of rings $a^{-1}\calO_\Prism$ on the simplicial topos; see \cite[03CI, 0GMA, 0D9Y]{stacks-project}. We let 
\[
g_n\colon (\Sh(((X_{\overline{A}},M_{X_{\overline{A}}})/A)_\Prism/B^{[n]}),\calO_\Prism)\rightarrow (\Sh(((X_{\overline{A}},M_{X_{\overline{A}}})/A)_\Prism/B^{[\bullet]}),\calO_\Prism^\bullet)
\]
denote the associated morphism of ringed topoi; note that $g_n^\ast$ admits a left adjoint, so it is exact. 

Together with the punctual topos $\Sh(\pt)$ and the associated constant simplicial topos $\Sh(\pt\times \Delta)$, we have a commutative diagram of morphisms of ringed topoi
\[
\xymatrix{
(\Sh(((X_{\overline{A}},M_{X_{\overline{A}}})/A)_\Prism/B^{[\bullet]}),\calO_\Prism^\bullet)\ar[r]^-{b^\bullet}\ar[d]_-a
& (\Sh(\pt\times \Delta),A^\bullet)\ar[d]_-{a_A}\\
(\Sh((X_{\overline{A}},M_{X_{\overline{A}}})/A)_\Prism,\calO_\Prism)\ar[r]^-b & (\Sh(\pt),A),
}
\]
where $b_\ast(\calF)=\Gamma(((X_{\overline{A}},M_{X_{\overline{A}}})/A)_\Prism,\calF)$, and $b^{\ast}$ sends an $A$-module $M$ to the associated sheaf $M\otimes_{A}\calO_\Prism$ of $\calO_\Prism$-modules.

\begin{lem}\label{lem:CA method}
Keep the above notation. \hfill
\begin{enumerate}
 \item For $\calK\in D(((X_{\overline{A}},M_{X_{\overline{A}}})/A)_\Prism,\calO_\Prism)$, the adjunction map
 \[
 \calK\rightarrow Ra_\ast a^{\ast}\calK
 \]
 is an isomorphism in $ D(((X_{\overline{A}},M_{X_{\overline{A}}})/A)_\Prism,\calO_\Prism)$.
 \item Let $\calF$ be a sheaf of $\calO_\Prism$-modules on $((X_{\overline{A}},M_{X_{\overline{A}}})/A)_\Prism$ such that $H^i(B^{[n]},\calF)=0$ for every $n\geq 0$ and $i>0$. Then $R\Gamma(((X_{\overline{A}},M_{X_{\overline{A}}})/A)_\Prism,\calF)$ is given by the cochain complex $s(\calF(B^{[\bullet]}))$ associated to the cosimplicial $A$-module $\calF(B^{[\bullet]})$ (as in \cite[019H]{stacks-project}). This is the case when $\calF$ is a finite projective $\calO_\Prism$-module or (the restriction of) a completed prismatic crystal on $(X,M_X)_\Prism$.
\end{enumerate}
\end{lem}

\begin{proof}
Part (1) is \cite[0DA0]{stacks-project}. The first part of (2) is \cite[fn.~10]{bhatt-scholze-prismaticcohom}. For the convenience of the reader and for later use, we give a slightly different argument here. By (1), we have an isomorphism
$R\Gamma(((X_{\overline{A}},M_{X_{\overline{A}}})/A)_\Prism,\calF)\xrightarrow{\cong} Ra_{A,\ast}Rb^\bullet_\ast a^\ast\calF$. It follows from the assumption and \cite[0DH1]{stacks-project} that the map $b^\bullet_\ast a^\ast\calF\xrightarrow{\cong} Rb^\bullet_\ast a^\ast\calF$ is an isomorphism. Note that $b^\bullet_\ast a^\ast\calF\in \Mod(A^\bullet)$ is given by the cosimplicial $A$-module $\calF(B^{[\bullet]})$ under (a ringed topos variant of) \cite[09WF]{stacks-project}.
Hence it remains to show  that for every $M^\bullet\in\Mod(A^\bullet)$, $Ra_{A,\ast}M^\bullet$ is given by the associated cochain complex $s(M^\bullet)$. This claim is a special case of \cite[0D7D]{stacks-project} if $M^\bullet$ is an injective $A^\bullet$-module on $\pt\times \Delta$; in fact, we have a functorial map $a_{A,\ast}M^\bullet\rightarrow s(M^\bullet)$ for any $M^\bullet$, and it is a quasi-isomorphism if $M^\bullet$ is injective. To show the general case, take an injective resolution $M^\bullet\rightarrow I^{0,\bullet}\rightarrow I^{1,\bullet}\rightarrow I^{2,\bullet}\rightarrow\cdots$. Then $Ra_{A,\ast}M^\bullet$ is given by $a_{A,\ast}I^{0,\bullet}\rightarrow a_{A,\ast}I^{1,\bullet}\rightarrow a_{A,\ast}I^{2,\bullet}\rightarrow\cdots$. Consider the double complex $I^{\star,\bullet}$ where we consider the cochain complex associated to the cosimplicial $A$-module in the $\bullet$-direction, and let $\Tot (I^{\star,\bullet})$ denote the associated total complex.
By construction, we have maps of complexes of $A$-modules
\[
s(M^\bullet)\rightarrow \Tot (I^{\star,\bullet})\leftarrow [a_{A,\ast}I^{0,\bullet}\rightarrow a_{A,\ast}I^{1,\bullet}\rightarrow a_{A,\ast}I^{2,\bullet}\rightarrow\cdots].
\]
By the $E_1$-spectral sequence associated to a double complex, we see that these two maps are quasi-isomorphisms; use the fact that each $g_{A,n}^\ast$ ($g_n^\ast$ for $(\Sh(\pt\times\Delta),A^\bullet)$) is exact for the first map and the preceding discussion for the second map. The second part of (2) follows from \cite[Cor.~3.12]{bhatt-scholze-prismaticcohom} in the former case; in the latter case, a similar argument works based on the faithfully flat descent, Remark~\ref{rem:completed crystal as limit}(1), Corollary~\ref{cor:repleteness}, and  \cite[0DD8]{stacks-project}.
\end{proof}

We also need an $\infty$-categorical analogue of this lemma. Since the category of modules on a ringed site is a Grothendieck abelian category, we have a commutative diagram of the associated derived $\infty$-categories
\[
\xymatrix{
\calD((((X_{\overline{A}},M_{X_{\overline{A}}})/A)_\Prism/B^{[\bullet]}),\calO_\Prism^\bullet)\ar[r]^-{Rb^\bullet_\ast}\ar[d]_-{Ra_\ast}
& \calD((\pt\times \Delta),A^\bullet)\ar[d]_-{Ra_{A,\ast}}\\
\calD(((X_{\overline{A}},M_{X_{\overline{A}}})/A)_\Prism,\calO_\Prism)\ar[r]^-{Rb_\ast} & \calD(\pt,A),
}
\]
where $Rb_\ast=R\Gamma(((X_{\overline{A}},M_{X_{\overline{A}}})/A)_\Prism,-)$ under the identification $\calD(\pt,A)=\calD(A)$.
For $\calK\in \calD(((X_{\overline{A}},M_{X_{\overline{A}}})/A)_\Prism,\calO_\Prism)$, the adjunction map $\calK\rightarrow Ra_\ast a^\ast\calK$ is an isomorphism by Lemma~\ref{lem:CA method}(1). We also need a variant of Lemma~\ref{lem:CA method}(2). Observe that the collection $(g_{A,n}^\ast)_{n\geq 0}$ yields a functor $g_A^\ast\colon\calD((\pt\times \Delta),A^\bullet)\rightarrow \mathrm{Fun}(N(\Delta),\calD(A))$. We remark that $g_{A,n}^\ast\circ Rb_{\ast}^\bullet=R\Gamma(B^{[n]},-)$ by \cite[0DH0]{stacks-project}.
Since $\calD(A)$ is presentable, we have the totalization functor $\Tot\colon \mathrm{Fun}(N(\Delta),\calD(A))\rightarrow \calD(A)$ by taking the limit. By abuse of notation, we continue to write $\Tot$ for the composite $\Tot\circ g_A^\ast$.

\begin{lem} \label{lem: CA-method-infinity-category}
For $M\in \calD^+((\pt\times \Delta),A^\bullet)$, $Ra_{A,\ast}M$ is given by $\Tot M$.
\end{lem}

\begin{proof}
Take a bounded below complex $I^\star$ of injectives in $\Mod(A^\bullet)$ that represents $M$. As in the proof of Lemma~\ref{lem:CA method}(2), $Ra_{A,\ast}M$ is given by the total complex associated to the double complex $(g_{A,\bullet}^\ast I^\star)$, where we take the cochain complex associated to the cosimplicial $A$-module in the $\bullet$-direction. We know from \cite[019I]{stacks-project} that the latter is quasi-isomorphic to the total complex associated to the double complex $(g_{A,\bullet}^\ast I^\star)'$ where we take the \emph{normalized} cochain complex associated to the cosimplicial $A$-module in the $\bullet$-direction. Hence it remains to prove that $\Tot g_A^\ast M$ is given by the total complex $\Tot (g_{A,\bullet}^\ast I^\star)'$; this is explained, for example, in \cite[Prob.~4.23]{Bunke} when $A=\Z$, and the proof therein works for a general $A$.
\end{proof}

\begin{rem} \label{rem: CA-for-jointly-surjective-case}
Let $\Lambda$ be a finite set and assume that there exists a collection of objects $\{(B_\lambda,IB_\lambda,M_{\Spf B_\lambda})\}_{\lambda\in \Lambda}$ of $(X_{\overline{A}}/A)_\Prism^\op$ that jointly covers the final object of the associated topos. We also assume that for every $n\geq 0$ and $\lambda_0,\ldots,\lambda_n\in \Lambda$, the coproduct $(B_{\lambda_0,\ldots\lambda_n},IB_{\lambda_0,\ldots\lambda_n}, M_{\Spf B_{\lambda_0,\ldots\lambda_n}})$ of $(B_{\lambda_0},IB_{\lambda_0},M_{\Spf B_{\lambda_0}}),\ldots,(B_{\lambda_0},IB_{\lambda_0},M_{\Spf B_{\lambda_0}})$  exists in $(X_{\overline{A}}/A)_\Prism^\op$.
Set $B\coloneqq \prod_{\lambda\in\Lambda}B_\lambda$. Then the resulting log prism $(B,IB,M_{\Spf B})$ covers the final object. In this case, we have  $B^{[n]}=\prod_{(\lambda_0,\ldots,\lambda_n)\in\Lambda^{n+1}}B_{\lambda_0,\ldots\lambda_n}$ and $\calF(B^{[n]})=\prod_{(\lambda_0,\ldots,\lambda_n)\in\Lambda^{n+1}}\calF(B_{\lambda_0,\ldots\lambda_n})$ for any sheaf $\calF$.
\end{rem}

Let us end with two purely algebraic lemmas from \cite{bhatt-scholze-prismaticcohom, du-liu-moon-shimizu-purity-F-crystal}, which will be key algebraic inputs when we discuss the base change properties.
\begin{lem}\label{lem:algebraic-lemmas-for-base-change}
\hfill
\begin{enumerate}
 \item Let $A\rightarrow B$ be a map of commutative rings that has finite $I$-complete Tor amplitude for a finitely generated ideal $I\subset A$. The $I$-completed base change functor $-\widehat{\otimes}_A^LB$ commutes with the totalization functor $\Tot$ on $\calD^{\geq c}$, for any $c\in \Z$.
 \item Let $A\rightarrow B$ be a map of $p$-adically complete rings with $A$ $p$-torsion free and let $M$ be a finitely generated $p$-adically complete $A$-module. If $M[p^{-1}]$ is a finite projective $A[p^{-1}]$-module, then the map
\[
M[p^{-1}]\otimes_AB\rightarrow (M\widehat{\otimes}_AB)\otimes_BB[p^{-1}]
\]
is an isomorphism. In particular, the latter is finite projective over $B[p^{-1}]$.
\end{enumerate}
\end{lem}

\begin{proof}
Part (1) for $c=0$ is \cite[Lem.~4.22]{bhatt-scholze-prismaticcohom}, which also implies the general case since the functors involved commute with shifts. Part (2) is \cite[Lem.~B.17]{du-liu-moon-shimizu-purity-F-crystal}.
\end{proof}

\subsection{Breuil--Kisin log prism} \label{subsec: BK-prism}

We first recall the notion of Breuil--Kisin log prism, which will be useful when working locally.

\smallskip
\noindent \textbf{Small affine case.}
\begin{defn} \label{defn: small affine}
For $d\geq m\geq 1$, let
\[
R^0 = \mathcal{O}_K \langle T_1, \ldots, T_m, T_{m+1}^{\pm 1}, \ldots, T_d^{\pm 1}\rangle / (T_1\cdots T_m - \pi),
\]
and consider the prelog structure $\mathbf{N}^d \rightarrow R^0$ sending $e_i \mapsto T_i$ for $1 \leq i \leq d$.
Let $R$ be a \emph{connected} $\mathcal{O}_K$-algebra equipped with a $p$-adically completed \'etale map 
\[
\square\colon R^0 \rightarrow R.
\]
We call a $p$-adic log formal scheme $(X,M_X)$ \emph{small affine with framing $\square$} if it is an affine log formal scheme of the form $(X,M_X)=(\Spf R, \N^d\rightarrow R^0\xrightarrow{\square} R)^a$.    
\end{defn}

\begin{eg}[Breuil--Kisin log prism and Breuil log prism in the small affine case] \label{eg: Breuil-Kisin prism}
Assume $(X,M_X)=(\Spf R, \N^d\rightarrow R^0\xrightarrow{\square} R)^a$ is small affine as in Definition~\ref{defn: small affine}. We will often consider the \emph{Breuil--Kisin log prism} $(\mathfrak{S}, (E(u)), M_{\Spf \mathfrak{S}}) \in (X,M_X)_{\Prism}$ (with respect to $\pi$ and $\square$) given as follows.

Let $\fkS_{R^0}\coloneqq W(k)\langle T_1, \ldots, T_m, T_{m+1}^{\pm 1}, \ldots, T_d^{\pm 1}\rangle[\![u]\!] / (T_1\cdots T_m-u)$. Note that $\fkS_{R^0}$ is $(p, E)$-adically complete and $\fkS_{R^0}/(E) \cong R^0$. We equip $\fkS_{R^0}$ with the prelog structure $\mathbf{N}^d \rightarrow \mathfrak{S}_{R^0}$, $e_i \mapsto T_i$ ($1 \leq i \leq d$) lifting the one on $R^0$. Since $\square$ is $p$-adically completed \'etale, it lifts uniquely to a $(p, E)$-adically completed \'etale map 
\[
\square_{\mathfrak{S}}\colon \mathfrak{S}_{R^0} \rightarrow \mathfrak{S}_{\square}.
\]
We often write $\fkS$ for $\fkS_{\square}$. Setting $\delta_{\mathrm{log}}(e_i) = 0 $ and $\delta(T_j) = 0$ makes $(\fkS_{R^0}, \delta, (E), \N^d, \delta_{\mathrm{log}})$ a bounded prelog prism; we will also write it as $(\mathfrak{S}_{R^0}, (E), \N^d)$ for short. By \cite[Lem.~2.13]{koshikawa}, $(\mathfrak{S}, (E), \N^d\rightarrow \mathfrak{S}_{R^0}\rightarrow \mathfrak{S})$ admits a bounded prelog prism structure. The \emph{Breuil--Kisin log prism} $(\mathfrak{S}, (E), M_{\Spf \mathfrak{S}})$ is the log prism associated to $(\mathfrak{S}, (E), \N^d)$, which is an object of $((\Spf R,\N^d)^a)_\Prism$ via $R \cong \fkS/(E)$.

To relate the log prismatic site to log crystalline site in the small affine case, we will also consider the \emph{Breuil log prism} $(S, (p), M_{\Spf S})$: let $S_{\square}$ be the $p$-adic completion of the log PD-envelope of $\fkS_{\square}$ with respect to the kernel of the surjection $\fkS_{\square} \to R/p$. By \cite[Cor.~2.39]{bhatt-scholze-prismaticcohom}, we have $S_{\square} \simeq \fkS_{\square}\{\phi(E)/p\}_\delta^\wedge$, where $\{-\}_\delta$ denotes adjoining elements as $\delta$-ring. If there is no confusion, we write $S$ for $S_{\square}$. The pullback of the prelog structure on $\fkS$ defines a prelog structure $\phi^{\ast}\N^d$ on $S$, and we have $(S, (p), M_{\Spf S})\coloneqq (S, (p), \phi^\ast \N^d)^a \in (X,M_X)_{\Prism}$. By construction, we have a well-defined map $\phi\colon(\fkS,(E), M_{\Spf \fkS}) \to (S,(p), M_{\Spf S})$ of log prisms in $(X,M_X)_{\Prism}^{\op}$. 
\end{eg}

For more details on the Breuil--Kisin log prism and Breuil log prism, we refer to \cite[\S~2.2, 3.5]{du-liu-moon-shimizu-purity-F-crystal}. For example, the following properties hold.

\begin{lem}[{\cite[Lem.~2.9]{du-liu-moon-shimizu-purity-F-crystal}}] \label{lem: coprod with fkS is cover}
Keep the small affine assumption, and let $(A,I,M_{\Spf A})$ be any log prism in the opposite category $(X,M_X)_\Prism^{\mathrm{op}}$. Then the coproduct $(A,I,M_{\Spf A})\amalg (\mathfrak{S}, (E(u)), M_{\Spf \mathfrak{S}})$ of $(A,I,M_{\Spf A})$ and $(\mathfrak{S}, (E(u)), M_{\Spf \mathfrak{S}})$ exists in  $(X,M_X)_\Prism^{\mathrm{op}}$, and the structure map 
\[
(A,I,M_{\Spf A}) \to (A,I,M_{\Spf A})\amalg (\mathfrak{S}, (E(u)), M_{\Spf \mathfrak{S}}) 
\]
is a cover. In particular, $(\mathfrak{S}, (E(u)), M_{\Spf \mathfrak{S}})$ covers the final object of $\Sh((X, M_X)_{\Prism})$.   
\end{lem}

For two framings $\square$, $\square'$ of small affine $X = \Spf R$, write $(\fkS_{\square,\square'}, (E), M_{\Spf \fkS_{\square,\square'}})$ be the coproduct of $(\fkS_{\square},(E),M_{\Spf\fkS_{\square}})$ and $(\fkS_{\square'},(E),M_{\Spf\fkS_{\square'}})$ in $(X,M_X)_\Prism^{\mathrm{op}}$. When the framing $\square$ is fixed, we write $(\fkS_{\square}^{[n]}, (E), M_{\Spf \fkS_{\square}^{[n]}})$ or simply $(\fkS^{[n]}, (E), M_{\Spf \fkS^{[n]}})$ for the $(n+1)$-st self-coproduct of $(\fkS_{\square},(E),M_{\Spf\fkS_{\square}})$. 
\footnote{In \cite{du-liu-moon-shimizu-purity-F-crystal}, the object $\mathfrak{S}^{[n]}$ is denoted by $\mathfrak{S}^{(n)}$, and the notation $\mathfrak{S}^{[n]}$ therein denotes a different object (cf.~\cite[Ex.~2.15]{du-liu-moon-shimizu-purity-F-crystal}). }
For $n \geq 1$ and $0 \leq i \leq n$, write $p_i^n\colon \mathfrak{S} \rightarrow \mathfrak{S}^{[n]}$ for the $i$-th projection (coface) map.

\begin{lem}[{\cite[Cor.~2.12]{du-liu-moon-shimizu-purity-F-crystal}}] \label{lem: BK-prism-coface-map-faithfully flat}
Each coface map $p_i^n\colon \mathfrak{S} \rightarrow \mathfrak{S}^{[n]}$ is classically faithfully flat.    
\end{lem}

\begin{lem} \label{lem:Breuil-prism-weakly-initial-for-crystalline-prisms}
Suppose that $(X, M_X) = (\Spf R, \mathbf{N}^d)^a$ is small affine.
Then the coproduct $(A, (p), M_{\Spf A}) \amalg (S, (p), M_{\Spf S})$ of any crystalline log prism $(A, (p), M_{\Spf A})$ and a Breuil log prism $(S, (p), M_{\Spf S})$ exists in $(X, M_X)_{\Prism}^{\op}$, and the structure map
\[
(A, (p), M_{\Spf A}) \rightarrow (A, (p), M_{\Spf A}) \amalg  (S, (p), M_{\Spf S})
\]
is a cover.
\end{lem}

\begin{proof}
By Lemma~\ref{lem: coprod with fkS is cover}, we have the coproduct $(A, (p), M_{\Spf A})) \amalg (\mathfrak{S}, (E), M_{\Spf \mathfrak{S}})$ of $(A, (p), M_{\Spf A}))$ and $(\mathfrak{S}, (E), M_{\Spf \mathfrak{S}})$ in $(X, M_X)_{\Prism}^{\op}$, and the map
\[
(A, (p), M_{\Spf A})) \rightarrow (A, (p), M_{\Spf A})) \amalg (\mathfrak{S}, (E), M_{\Spf \mathfrak{S}})
\]
is a cover. Consider the Frobenius twist $(\mathfrak{S}, (\phi(E)), \phi^*\mathbf{N}^d)^a \in (X, M_X)_{\Prism}^\op$ of the Breuil--Kisin log prism as in \cite[Ex.~2.15]{du-liu-moon-shimizu-purity-F-crystal}. The map $\phi\colon (\mathfrak{S}, (E), \mathbf{N}^d)^a \rightarrow (\mathfrak{S}, (\phi(E)), \phi^*\mathbf{N}^d)^a$ is a cover by \cite[Lem.~2.7]{du-liu-moon-shimizu-purity-F-crystal}. Let $(B, (p), M_{\Spf B}) \in (X, M_X)_{\Prism}^\op$ be the pushout of $\phi\colon (\mathfrak{S}, (E), \mathbf{N}^d)^a \rightarrow (\mathfrak{S}, (\phi(E)), \phi^*\mathbf{N}^d)^a$ along the structure map $(\mathfrak{S}, (E), M_{\Spf \mathfrak{S}}) \rightarrow (A, (p), M_{\Spf A}) \amalg (\mathfrak{S}, (E), M_{\Spf \mathfrak{S}})$ given by \cite[Rem.~2.4]{du-liu-moon-shimizu-purity-F-crystal}. By construction, $(B, (p), M_{\Spf B})$ is the coproduct of $(A, (p), M_{\Spf A})$ and $(\mathfrak{S}, (\phi(E)), \phi^*\mathbf{N}^d)^a$ in $(X, M_X)_{\Prism}^{\op}$, and the map $(A, (p), M_{\Spf A}) \rightarrow (B, (p), M_{\Spf B})$ is a cover.

Since $(B, (p), M_{\Spf B})$ admits a map from $(\mathfrak{S}, (\phi(E)), \phi^*\mathbf{N}^d)^a$, we have $\phi(E)/p \in B$. By \cite[Cor.~2.39]{bhatt-scholze-prismaticcohom}, the map $(\mathfrak{S}, (\phi(E)), \phi^*\mathbf{N}^d)^a \rightarrow (B, (p), M_{\Spf B})$ factors through
\[
(\mathfrak{S}, (\phi(E)), \phi^*\mathbf{N}^d)^a \rightarrow (S, (p), \phi^*\mathbf{N}^d)^a \rightarrow (B, (p), M_{\Spf B}).
\]
Thus, the log prism $(B, (p), M_{\Spf B}) = (A, (p), M_{\Spf A})) \amalg (\mathfrak{S}, (\phi(E)), \phi^*\mathbf{N}^d)^a$ gives the coproduct $(A, (p), M_{\Spf A})) \amalg (S, (p), M_{\Spf S})$. 
\end{proof}

Now, denote by $(\fkS_{K}, (E), M_{\Spf S_{K}})$ (resp. $(S_{K}, (p), M_{\Spf S_{K}})$) the Breuil--Kisin (resp. Breuil) log prism in $(\Spf\calO_K, M_{\Spf\calO_K})_{\Prism}$. Consider the map $\fkS_{K}/(E) = \calO_K \rightarrow S_{K}/(p)$ induced by $\phi\colon \fkS_{K} \rightarrow S_{K}$, and let $X^{(1)}\coloneqq X\times_{\Spf \calO_K} \Spf S_{K}/(p)$ the corresponding base change with the induced log structure $M_{X^{(1)}}$. When $X = \Spf R$ is small affine, the following relative version is proved in \cite{min-wang-HT-crys-log-prism}.

\begin{lem}[cf. {\cite[\S~3.2.1]{min-wang-HT-crys-log-prism}}] \label{lem: BK-prism-self-products-relative}
Let $X = \Spf R$ be small affine with a framing $\square$. The Breuil--Kisin log prism $(\fkS_{\square}, (E), M_{\Spf\fkS_{\square}})$ in $(X, M_X)_{\Prism}^\op$ covers the final object of $\Sh(((X, M_X)/(\fkS_{K}, M_{\Spf \fkS_{K}}))_{\Prism})$. Furthermore, for each $n \geq 0$, the $(n+1)$-st self-coproduct of $(\fkS_{\square}, (E), M_{\Spf\fkS_{\square}})$ in $((X, M_X)/(\fkS_{K}, M_{\Spf \fkS_{K}}))_{\Prism}^\op$ exists (which we denote by $(\fkS_{\square}^{\rel, [n]}, (E), M_{\Spf \fkS_{\square}^{\rel, [n]}}$). Any coface map $\fkS_{\square} \rightarrow \fkS_{\square}^{\rel, [n]}$ is classically faithfully flat.
\end{lem}

Note that the map $\fkS_{K} \rightarrow \fkS_{\square}$ is classically faithfully flat, and so the structure map $\fkS_{K} \rightarrow \fkS_{\square}^{\rel, [n]}$ is classically faithfully flat for any $n$. Thus, the derived $p$-completion of $\fkS_{\square}^{\rel, [n]}\otimes_{\fkS_{K}, \phi}^L S_{K}$ is discrete and agrees with the classical $p$-completion $\fkS_{\square}^{\rel, [n]}\widehat{\otimes}_{\fkS_{K}, \phi} S_{K}$ of $\fkS_{\square}^{\rel, [n]}\otimes_{\fkS_{K}, \phi} S_{K}$. Furthermore, since the map $\phi\colon (\fkS_{K}, (E), M_{\Spf\fkS_{K}}) \rightarrow (S_{K}, (p), M_{\Spf S_{K}})$ in $(\Spf\calO_K, M_{\Spf\calO_K})_{\Prism}^\op$ is strict, $S_{\square}^{\rel, [n]}\coloneqq \fkS_{\square}^{\rel, [n]} \widehat{\otimes}_{\fkS_{K}, \phi} S_{K}$ with the induced log and $\delta_\log$ structures gives rise to an object $(S_{\square}^{\rel, [n]}, (p), M_{\Spf S_{\square}^{\rel, [n]}})$ of $((X^{(1)}, M_{X^{(1)}})/(S_{K}, M_{\Spf S_{K}}))_{\Prism}^\op$. 

\begin{lem} \label{lem: Breuil-prism-self-products-relative}
Let $X = \Spf R$ be small affine with framing $\square$. Then $(S_{\square}^{\rel, [0]}, (p), M_{\Spf S_{\square}^{\rel, [0]}})$ covers the final object of $\Sh(((X^{(1)}, M_{X^{(1)}})/(S_{K}, M_{\Spf S_{K}}))_{\Prism})$, and its $(n+1)$-st self-coproduct is given by $(S_{\square}^{\rel, [n]}, (p), M_{\Spf S_{\square}^{\rel, [n]}})$.
\end{lem}

\begin{proof}
This follows from Lemma~\ref{lem: BK-prism-self-products-relative} and the construction of $S_{\square}^{\rel, [n]}$.
\end{proof}

\smallskip
\noindent \textbf{Separated case.}
In the global case when $X$ is separated, the following \v{C}ech nerve construction via a family of Breuil--Kisin log prisms will be useful.

\begin{construction}[\v{C}ech nerve for separated semistable case] \label{const: Cech-nerve-global-separated}
Suppose $(X, M_X)$ is a \emph{separated} semistable $p$-adic formal scheme over $\Spf \calO_K$. Take a $p$-completely \'etale covering $\{g_\lambda\colon X_\lambda\rightarrow X\}_{\lambda\in \Lambda}$ such that each $X_\lambda$ is small affine with framing $\square_\lambda$. For each $\lambda$, let $(\fkS_{\square_{\lambda}}, (E), M) \in (X_{\lambda}, M_{X_{\lambda}})_{\Prism}$ be the Breuil--Kisin log prism.\footnote{Here we often write $M$ for the log structure of a log prism (e.g.~$M_{\Spf \fkS_{\square_{\lambda}}}$) to simplify the notation.}
Since $X$ is separated, $X_{\lambda \lambda'}\coloneqq X_{\lambda}\times_X X_{\lambda'}$ is small affine for each $\lambda, \lambda' \in \Lambda$. Since $X_{\lambda \lambda'} \rightarrow X_{\lambda}$ is $p$-completely \'etale, it lifts uniquely to a $(p, E)$-completely \'etale map $(\fkS_{\square_{\lambda}}, (E), M) \rightarrow (\fkS_{\lambda \lambda', \square_{\lambda}}, (E), M)$ in $(X_{\lambda}, M_{X_{\lambda}})_{\Prism}^{\op}$ by \cite[Lem.~2.13, Cor.~2.15]{koshikawa}, where $(\fkS_{\lambda \lambda', \square_{\lambda}}, (E), M)$ is the Breuil--Kisin log prism in $(X_{\lambda \lambda'}, M_{X_{\lambda \lambda'}})_{\Prism}$ with respect to the framing $\square_{\lambda}$. Similarly, the $p$-completely \'etale map $X_{\lambda \lambda'} \rightarrow X_{\lambda'}$ lifts uniquely to a $(p, E)$-completely \'etale map $(\fkS_{\square_{\lambda}}, (E), M) \rightarrow (\fkS_{\lambda \lambda', \square_{\lambda'}}, (E), M)$ where $(\fkS_{\lambda \lambda', \square_{\lambda'}}, (E), M)$ is the Breuil--Kisin log prism in $(X_{\lambda \lambda'}, M_{X_{\lambda \lambda'}})_{\Prism}$ for the framing $\square_{\lambda'}$. By Lemma~\ref{lem: coprod with fkS is cover}, we have the coproduct $(\fkS_{\lambda \lambda', \square_{\lambda}, \square_{\lambda'}}, (E), M)$ of $(\fkS_{\lambda \lambda', \square_{\lambda}}, (E), M)$ and $(\fkS_{\lambda \lambda', \square_{\lambda'}}, (E), M)$ in $(X_{\lambda \lambda'}, M_{X_{\lambda \lambda'}})_{\Prism}^{\op}$. Note that $(\fkS_{\lambda \lambda', \square_{\lambda}, \square_{\lambda'}}, (E), M)$ is also the coproduct of $(\fkS_{\square_{\lambda}}, (E), M)$ and $(\fkS_{\square_{\lambda'}}, (E), M)$ in $(X, M_X)_{\Prism}^{\op}$, and write 
\[
(\fkS_{\square_{\lambda}, \square_{\lambda'}}, (E), M_{\Spf \fkS_{\square_{\lambda}, \square_{\lambda'}}})\coloneqq (\fkS_{\lambda \lambda', \square_{\lambda}, \square_{\lambda'}}, (E), M_{\Spf \fkS_{\lambda \lambda', \square_{\lambda}, \square_{\lambda'}}}).
\]
Repeating this argument, we deduce that for any $\lambda_0, \ldots, \lambda_n \in \Lambda$, the coproduct $(\fkS_{\square_{\lambda_0}, \ldots, \square_{\lambda_n}}, (E), M)$ of $(\fkS_{\square_{\lambda_0}}, (E), M), \ldots, (\fkS_{\square_{\lambda_n}}, (E), M)$ exists in $(X, M_X)_{\Prism}^{\op}$. 

The family $(\fkS_{\square_{\lambda}}, (E), M_{\Spf\fkS_{\square_{\lambda}}})_{\lambda\in \Lambda}$ jointly covers the final object of $\Sh((X, M_X)_{\Prism})$ by Lemma~\ref{lem: coprod with fkS is cover} and \cite[Lem.~2.13, Cor.~2.15]{koshikawa}, and we get the \v{C}ech nerve of the cover $\coprod_{\lambda\in \Lambda}(\Spf \fkS_{\square_{\lambda}}, (E), M) \twoheadrightarrow \ast$; the $n$-th spot is $\coprod_{(\lambda_0, \ldots, \lambda_n) \in \Lambda^{n+1}}(\Spf\fkS_{\square_{\lambda_0}, \ldots, \square_{\lambda_n}}, (E), M)$. We simply denote the \v{C}ech nerve by $\Spf( \fkS^{[\bullet]}_{\Lambda})$. By $(p, E)$-completely faithfully flat descent, we have a natural equivalence
\[
\cal{D}_{\perf}((X, M_X)_{\Prism}) \cong \lim_{[n] \in \Delta} \calD_{\perf}(\fkS^{[n]}_{\Lambda}).
\]
Furthermore, we also have a natural equivalence
\[
\mathcal{D}_{\perf}^{\phi}((X, M_X)_{\Prism}) \cong \lim_{[n] \in \Delta} \calD_{\perf}^{\phi}(\fkS^{[n]}_{\Lambda}).
\]
For this, by Zariski descent, we may reduce to the case that $X$ is quasi-compact so that the set $\Lambda$ above can be taken to be finite. Then the equivalence follows from $(p, E)$-completely faithfully flat descent for Frobenius structure \cite[Thm.~7.8]{mathew-descent}, by considering an effective morphism which induces the isogeny.

For the relative site $((X, M_X)/(\fkS_{K}, M_{\Spf \fkS_{K}}))_{\Prism}$, we also have an analogous \v{C}ech nerve construction generalizing Lemma~\ref{lem: BK-prism-self-products-relative}. Note that the family $(\fkS_{\square_{\lambda}}, (E), M_{\Spf\fkS_{\square_{\lambda}}})_{\lambda\in \Lambda}$ jointly covers the final object of $\Sh(((X, M_X)/(\fkS_{K}, M_{\Spf \fkS_{K}}))_{\Prism})$. For each $n \geq 0$, let $(\fkS_{K}^{[n]}, (E), M_{\Spf\fkS_{K}^{[n]}})$ be the $(n+1)$-th self-coproduct of $(\fkS_{K}, (E), M_{\Spf\fkS_{K}})$ in $(\Spf\calO_K, M_{\can})_{\Prism}$. Let $\lambda_0, \ldots, \lambda_n \in \Lambda$. Since the structure map $\fkS_{K}^{[n]} \rightarrow \fkS_{\square_{\lambda_0}, \ldots, \square_{\lambda_n}}$ is $(p, E)$-completely faithfully flat, the pushout of this map along the diagonal map $\fkS_{K}^{[n]} \rightarrow \fkS_{K}$ in the category of bounded prisms is given by the classical $(p, E)$-completion of $\fkS_{\square_{\lambda_0}, \ldots, \square_{\lambda_n}}\otimes_{\fkS_{K}^{[n]}} \fkS_{K}$ by \cite[Lem.~3.3]{du-liu-moon-shimizu-completed-prismatic-F-crystal-loc-system} equipped with the induced log and $\delta_{\log}$-structures. This gives the coproduct of $(\fkS_{\square_{\lambda_i}}, (E), M_{\Spf\fkS_{\square_{\lambda_i}}})$'s with $0 \leq i \leq n$ in $((X, M_X)/(\fkS_{K}, M_{\Spf \fkS_{K}}))_{\Prism}^{\op}$. Denote the corresponding \v{C}ech nerve by $\Spf(\fkS_{\Lambda}^{\rel, [\bullet]})$.

Lastly, for each $\lambda$, let $(S_{\square_{\lambda}}, (p), M_{\Spf S_{\square_{\lambda}}}) \in (X_{\lambda}, M_{X_{\lambda}})_{\Prism}$ be the Breuil log prism. By Lemma~\ref{lem:Breuil-prism-weakly-initial-for-crystalline-prisms} and \cite[Lem.~2.13, Cor.~2.15]{koshikawa}, the family $(S_{\square_{\lambda}}, (p), M)_{\lambda \in \Lambda}$ jointly covers the final object of $\Sh(((X, M_X)/(S_{K}, M_{\Spf S_{K}}))_{\Prism})$, and a similar construction as above gives the corresponding \v{C}ech nerve which we denote by $\Spf(S_{\Lambda}^{\rel, [\bullet]})$.
\end{construction}

\subsection{Analytic prismatic $F$-crystals and perfect complexes} \label{subsec: analy-prism-crys-perf-cx} 

Throughout this subsection, fix a semistable $p$-adic formal scheme over $\Spf \calO_K$.
Note that the restriction to the analytic locus gives a functor
\[
j^*\colon \CR^{\wedge, \an (, \phi)}((X, M_X)_{\Prism}, \calO_{\Prism}) \rightarrow \Vect^{\mathrm{an} (, \phi)}((X,M_X)_\Prism).
\]

\begin{prop}[{cf. \cite[Prop.~6.8]{Tian-prism-etale-comparison}}] \label{prop: right-adjoint-semistable-case}
 The functor $j^*$ admits a right adjoint
\[
j_*\colon \Vect^{\mathrm{an} (, \phi)}((X,M_X)_\Prism) \rightarrow \CR^{\wedge, \an (,\phi)}((X, M_X)_{\Prism}, \calO_{\Prism})
\]
such that the counit map $j^*j_* \rightarrow \id$ is an equivalence. In particular, $j_*$ is fully faithful. Furthermore, such $j_*$ is unique up to canonical isomorphism.    
\end{prop}

\begin{proof}
The uniqueness of $j_*$ follows from the adjunction property. Suppose first that $X$ is separated. By Construction~\ref{const: Cech-nerve-global-separated}, we have a family of Breuil--Kisin log prisms which jointly covers the final object of the topos associated to $(X, M_X)_{\Prism}$, and have corresponding \v{C}ech nerve. Note that the underlying ring of any Breuil--Kisin log prism is regular Noetherian. Furthermore, each coface map from a Breuil--Kisin log prism of the \v{C}ech nerve is classically flat by Lemma~\ref{lem: coprod with fkS is cover} and \cite[Tag 0912]{stacks-project}. Thus, we can apply Proposition~\ref{prop: relations-prism-crystals}(1) to obtain $j_*$ in the separated case. It also follows from the uniqueness that $j_*$ is compatible with \'etale pullbacks.  

In the general case, we have a Zariski covering of $X$ by separated semistable $p$-adic formal schemes over $\calO_K$. So we prove the general case by observing that both source and target categories satisfy Zariski descent.
\end{proof}

\begin{prop} \label{prop: functor-completed-prism-cryst-perfect-cx}
 We have a canonical fully faithful functor 
\[
\iota\colon \CR^{\wedge, \an (, \phi)}((X, M_X)_{\Prism}, \calO_{\Prism}) \rightarrow \calD_{\perf}^{(\phi)}((X, M_X)_{\Prism})
\]
together with a morphism  $\iota(\calF)\rightarrow \calF$ in $\calD((X, M_X)_{\Prism}, \calO_{\Prism})$ functorial in $\calF$ (and compatible with $\phi$).
\end{prop}

\begin{proof}
Suppose first that $X$ is small affine. For each framing $\square$, Proposition~\ref{prop: relations-prism-crystals}(2) gives a fully faithful functor $\iota_{\fkS_{\square}}\colon \CR^{\wedge, \an (, \phi)}((X, M_X)_{\Prism}, \calO_{\Prism}) \rightarrow \calD_{\perf}^{(\phi)}((X, M_X)_{\Prism})$ together with a natural transformation assigning $\iota_{\fkS_{\square}}(\calF)\rightarrow \calF$ to each $\calF$. We claim that this is independent of the choice of $\square$. To see this, take another framing $\square'$.
Since $(\Spf \fkS_{\square,\square'},(E),M_{\Spf }\fkS_{\square,\square'})$ covers the final object of $\Sh((X, M_X)_{\Prism})$, it suffices to show that $\iota_{\fkS_{\square}}$ and $\iota_{\fkS_{\square'}}$ are canonically identified over each of the self-products $\Spf \fkS_{\square,\square'}^{[\bullet]}$ compatibly with the projections between them. Since all the maps $\fkS_{\square}\rightarrow \fkS_{\square,\square'}^{[\bullet]}$ and $\fkS_{\square'}\rightarrow \fkS_{\square,\square'}^{[\bullet]}$ are classically flat by \cite[0912]{stacks-project}, the canonical identifications follow from Properties (a) and (b) of Proposition~\ref{prop: relations-prism-crystals}(2).

Next suppose that $X$ is separated. Take a $p$-completely \'etale covering $\{X_\lambda\rightarrow X\}_{\lambda\in \Lambda}$ such that each $X_\lambda$ is small affine with framing $\square_\lambda$. By Construction~\ref{const: Cech-nerve-global-separated}, we have a family of Breuil--Kisin log prisms $(\fkS_{\square_{\lambda}}, (E), M_{\Spf\fkS_{\square_{\lambda}}})_{\lambda\in \Lambda}$, which jointly covers the final object of $\Sh((X, M_X)_{\Prism})$ with induced \v{C}ech nerve $\Spf \fkS_{\Lambda}^{[\bullet]}$. Using $\fkS_{\Lambda}^{[\bullet]}$ and arguing as above, we obtain a fully faithful functor $\iota_{\fkS_{\Lambda}}\colon \CR^{\wedge, \an (, \phi)}((X, M_X)_{\Prism}, \calO_{\Prism}) \rightarrow \calD_{\perf}^{(\phi)}((X, M_X)_{\Prism})$ together with a natural transformation $\iota_{\fkS_{\Lambda}}(\calF)\rightarrow \calF$ such that the formation is compatible with refinements $\Lambda\subset \Lambda'$ and \'etale localizations. In the general case, we have a Zariski covering of $X$ by separated semistable $p$-adic formal schemes over $\calO_K$. Since the construction in the separated case is independent of the choice of coverings and framings and is compatible with \'etale pullbacks, we obtain by Zariski descent a fully faithful functor 
\[
\iota\colon \CR^{\wedge, \an (, \phi)}((X, M_X)_{\Prism}, \calO_{\Prism}) \rightarrow \calD_{\perf}^{(\phi)}((X, M_X)_{\Prism}).
\]
with a morphism  $\iota(\calF)\rightarrow \calF$ in $\calD((X, M_X)_{\Prism}, \calO_\Prism)$ functorial in $\calF$ (and compatible with $\phi$).
\end{proof}

\begin{defn}\label{def: analy-prism-cryst-perf-cx}
Define a functor $i\colon \Vect^{\mathrm{an} (, \phi)}((X,M_X)_\Prism) \rightarrow \calD_{\perf}^{(\phi)}((X, M_X)_{\Prism})$
by $i\coloneqq\iota\circ j_*$.
Note that $i$ is fully faithful by Propositions~\ref{prop: right-adjoint-semistable-case} and \ref{prop: functor-completed-prism-cryst-perfect-cx}.
\end{defn}

\subsection{The Breuil--Kisin and Breuil cohomologies} \label{subsec: BK-cohom}

For application to the $C_{\st}$-conjecture, we consider the following special cases with coefficients arising from analytic prismatic $F$-crystals. Let $\calE \in \Vect^{\an, \phi}((X,M_X)_\Prism)$ be an analytic prismatic $F$-crystal on $(X, M_X)_{\Prism}$. Write $j_{\ast}\calE \in \CR^{\wedge, \an, \phi}((X, M_X)_{\Prism}, \calO_{\Prism})$ and $i(\calE) \in \calD_{\perf}^{\phi}((X, M_X)_{\Prism})$ for the corresponding completed prismatic $F$-crystal and prismatic $F$-crystal in perfect complexes given by Proposition~\ref{prop: right-adjoint-semistable-case} and Definition~\ref{def: analy-prism-cryst-perf-cx} respectively. By construction, we have a morphism $i(\calE) \rightarrow j_{\ast}\calE$.

\begin{defn} \label{def: BK-cohom}
For log prisms in Example~\ref{eg: examples-log-prisms}, we define
\begin{enumerate}
 \item $R\Gamma_{\BK}(X, j_{\ast}\calE)\coloneqq R\Gamma(((X, M_X) / (\fkS_{K}, M_{\Spf\fkS_{K}}))_{\Prism}, j_{\ast}\calE) \in \calD(\fkS_{K})$;
 \item $R\Gamma_{\Br}(X, j_{\ast}\calE)\coloneqq R\Gamma(((X_{S_K/p}, M_{X_{S_K/p}}) / (S_K, M_{\Spf S_k}))_{\Prism}, j_{\ast}\calE) \in \calD(S_{K})$;
 \item $R\Gamma_{W}(X, j_{\ast}\calE)\coloneqq R\Gamma(((X_k, M_{X_k}) / (W(k), M_0))_{\Prism}, j_{\ast}\calE) \in \calD(W(k))$;
 \item $R\Gamma_{A_\inf}(X, j_{\ast}\calE)\coloneqq R\Gamma(((X_{\calO_C}, M_{X_{\calO_C}}) / (A_\inf, M_{\Spf A_\inf}))_{\Prism}, j_{\ast}\calE) \in \calD(A_\inf)$;
 \item $R\Gamma_{A_\cris}(X, j_{\ast}\calE)\coloneqq R\Gamma(((X_{A_\cris/p}, M_{X_{A_\cris/p}}) / (A_\cris, M_{\Spf A_\cris}))_{\Prism}, j_{\ast}\calE) \in \calD(A_\cris)$.
\end{enumerate}
We use similar notation for the coefficient $i(\calE)$.
Note that $R\Gamma_{\BK}$ and $R\Gamma_{\Br}$ depend on the choice of the uniformizer $\pi$, whereas $R\Gamma_W$, $R\Gamma_{A_\inf}$, and $R\Gamma_{A_\cris}$ are not. Moreover, $\Gal(\overline{K}/K)$ acts on the prisms in (4) and (5), which induces an action of $\Gal(\overline{K}/K)$ on $R\Gamma_{A_\inf}$ and $R\Gamma_{A_\cris}$.
\end{defn}

The Breuil--Kisin cohomology yields a perfect complex over $\fkS_K$ and provides the base change theorems (Theorems~\ref{thm: BK-cohom-perfect-cx} and \ref{thm: general-base-change-BK-cohom}), and the Breuil cohomology is shown to satisfy the Frobenius isogeny property via a crystalline method (Theorem~\ref{thm: pris-cris-comparison-over-Breuil S}).

\begin{thm}[{\cite[Thm.~6.18, Rem.~6.19]{Tian-prism-etale-comparison}}] \label{thm: BK-cohom-perfect-cx}
Assume that $X$ be proper over $\calO_K$, and let $\calE \in \Vect^{\an, \phi}((X,M_X)_\Prism)$. Then $R\Gamma_{\BK}(X, j_{\ast}\calE)$ is a perfect complex in $\calD(\fkS_{K})$. 
\end{thm}

\begin{prop} \label{prop: BK-cohom}
Assume that $X$ is quasi-compact and separated, and let $\calE \in \Vect^{\an, \phi}((X,M_X)_\Prism)$. Then the natural map 
\[
R\Gamma_{\BK}(X, i(\calE)) \rightarrow R\Gamma_{\BK}(X, j_{\ast}\calE)
\]
given by Proposition~\ref{prop: functor-completed-prism-cryst-perfect-cx} is an isomorphism in $\calD(\fkS_{K})$.
\end{prop}

\begin{proof}
We have a $p$-completely \'etale covering $\{g_\lambda\colon X_\lambda\rightarrow X\}_{\lambda\in \Lambda}$ with $\Lambda$ finite such that each $X_\lambda$ is small affine with framing $\square_\lambda$. By Construction~\ref{const: Cech-nerve-global-separated}, the finite family $(\fkS_{\square_{\lambda}}, (E), M_{\Spf_{\fkS_{\square_{\lambda}}}})_{\lambda\in \Lambda}$ jointly covers the final object of $\Sh(((X, M_X)/(\fkS_{K}, M_{\Spf\fkS_{K}}))_{\Prism})$, and we have the corresponding \v{C}ech nerve $\Spf(\fkS_{\Lambda}^{\rel, [\bullet]})$.

We have a natural commutative diagram
\[
\xymatrix{
R\Gamma_{\BK}(X, i(\calE)) \ar[r] \ar[d]^{\cong} & R\Gamma_{\BK}(X, j_{\ast}\calE) \ar[d]^{\cong} \\
\Tot(i(\calE)(\fkS_{\Lambda}^{\rel, [\bullet]})) \ar[r] & \Tot(j_{\ast}\calE(\fkS_{\Lambda}^{\rel, [\bullet]}))
}
\]
where the vertical isomorphisms are given by \v{C}ech cohomological descent in Lemmas~\ref{lem:CA method} and \ref{lem: CA-method-infinity-category} and Remark~\ref{rem: CA-for-jointly-surjective-case} (applied to $(A, I, M_{\Spf A}) = (\fkS_{K}, (E), M_{\Spf\fkS_{K}})$ and $(B, IB, M_{\Spf B})$ given by $B = \prod_{\lambda \in \Lambda} \fkS_{\square_{\lambda}}$). Since any coface map $\fkS_{\square_{\lambda}} \rightarrow \fkS_{\Lambda}^{\rel, [n]}$ is classically flat, the bottom horizontal map is an isomorphism by Property~(b) therein (cf. proof of Proposition~\ref{prop: functor-completed-prism-cryst-perfect-cx}). Thus, the top horizontal map is an isomorphism.
\end{proof}

For later use, we consider the following base change property.

\begin{thm} \label{thm: general-base-change-BK-cohom}
Assume that $X$ is quasi-compact and separated over $\calO_K$, and let $\calE \in \Vect^{\an, \phi}((X,M_X)_\Prism)$. Let $(A, (E), M_{\Spf A}) \in ((\Spf\calO_K, M_{\can})/(\fkS_{K}, M_{\Spf \fkS_{K}}))_{\Prism}^\op$ such that the map $\fkS_{K} \rightarrow A$ has finite $(p, E)$-complete Tor amplitude. Write $X_{A/(E)}\coloneqq X\times_{\Spf\calO_K} \Spf A/(E)$ equipped with the log structure induced from $M_X$. 
\begin{enumerate}
  \item The map
\[
(R\Gamma_{\BK}(X, i(\calE))\widehat{\otimes}_{\fkS_{K}} A \rightarrow R\Gamma((X_{A/(E)}/(A, M_{\Spf A}))_{\Prism}, i(\calE))
\]
is an isomorphism in $\calD(A)$, where the tensor product is derived $(p, E)$-completed.

    \item The map
\[
(R\Gamma_{\BK}(X, j_{\ast}\calE)\widehat{\otimes}_{\fkS_{K}} A)\otimes_A^L A[p^{-1}] \rightarrow R\Gamma((X_{A/(E)}/(A, M_{\Spf A}))_{\Prism}, j_{\ast}\calE)\otimes_A^L A[p^{-1}]
\]
is an isomorphism in $\calD(A[p^{-1}])$.
\end{enumerate}
This is the case when $A$ is either $A_{\inf}$, $A_{\cris}$, $W(k)$, or $S^{[n]}_K$ for any $n \geq 0$. Here $(S^{[n]}_K, (p), \phi^{\ast}\N)^a$ denotes the $(n+1)$-st self-coproduct of the Breuil log prism $(S_{K}, (p), \phi^{\ast}\N)^a$ in $(\Spf\calO_K, M_{\Spf\calO_K})_{\Prism}^{\op}$, given by Lemma~\ref{lem:Breuil-prism-weakly-initial-for-crystalline-prisms}.
\end{thm}

\begin{proof}
We will use \v{C}ech cohomological descent in Lemmas~\ref{lem:CA method} and \ref{lem: CA-method-infinity-category} and Remark~\ref{rem: CA-for-jointly-surjective-case}. As in the proof of Proposition~\ref{prop: BK-cohom}, we have a $p$-completely \'etale covering $\{g_\lambda\colon X_\lambda\rightarrow X\}_{\lambda\in \Lambda}$ with $\Lambda$ a finite set such that each $X_\lambda$ is small affine with framing $\square_\lambda$. The family $(\fkS_{\square_{\lambda}}, (E), M_{\Spf_{\fkS_{\square_{\lambda}}}})_{\lambda\in \Lambda}$ jointly covers the final object of $\Sh(((X, M_X)/(\fkS_{K}, M_{\Spf\fkS_{K}}))_{\Prism})$, with the corresponding \v{C}ech nerve $\Spf(\fkS_{\Lambda}^{\rel, [\bullet]})$. 

On the other hand, since the structure map $\fkS_{K} \rightarrow \fkS_{\Lambda}^{\rel, [n]}$ for any $n$ is classically faithfully flat, the derived $(p, E)$-completion of $\fkS_{\Lambda}^{\rel, [n]}\otimes_{\fkS_{K}}^L A$ is discrete and agrees with the classical $(p, E)$-completion $A_{\Lambda}^{\rel, [n]}\coloneqq \fkS_{\Lambda}^{\rel, [n]}\widehat{\otimes}_{\fkS_{K}} A$ of $\fkS_{\Lambda}^{\rel, [n]}\otimes_{\fkS_{K}} A$. Note that $(A_{\Lambda}^{\rel, [n]}, (E), M_{\Spf A_{\Lambda}^{\rel, [n]}})$ equipped with the induced log and $\delta_\log$-structures from those on $\fkS_{\Lambda}^{\rel, [n]}$ is an object in $(X_{A/(E)}/(A, M_{\Spf A}))_{\Prism}$. Furthermore, the family $(A_{\square_{\lambda}}^{\rel}, (E), M_{\Spf A_{\square_{\lambda}}^{\rel}})_{\lambda\in \Lambda}$ where $A_{\square_{\lambda}}^{\rel}\coloneqq \fkS_{\square_{\lambda}}\widehat{\otimes}_{\fkS_{K}} A$ jointly covers the final object of $\Sh((X_{A/(E)}/(A, M_{\Spf A}))_{\Prism})$, and its \v{C}ech nerve is given by $\Spf A_{\Lambda}^{\rel, [\bullet]}$.

(1) By above, we have
\[
R\Gamma_{\fkS_{K}}(X, i(\calE)) \cong \Tot(i(\calE)(\fkS_{\Lambda}^{\rel, [\bullet]}))
\]
and
\[
R\Gamma((X_{A/(E)}/(A, M_{\Spf A}))_{\Prism}, i(\calE)) \cong \Tot(i(\calE)(A_{\Lambda}^{\rel, [\bullet]})).
\]
Consider the commutative diagram
\[
\xymatrix{
R\Gamma_{\BK}(X, i(\calE)) \ar[r] \ar[d]^{\cong} & R\Gamma((X_{A/(E)}/(A, M_{\Spf A}))_{\Prism}, i(\calE)) \ar[d]^{\cong} \\
\Tot(i(\calE)(\fkS_{\Lambda}^{\rel, [\bullet]})) \ar[r] & \Tot(i(\calE)(A_{\Lambda}^{\rel, [\bullet]})).
}
\]
We need to show the map 
\begin{equation} \label{eq: base-change-totalization-i(E)}
\Tot(i(\calE)(\fkS_{\Lambda}^{\rel, [\bullet]}))\widehat{\otimes}_{\fkS_{K}}^L A \rightarrow \Tot(i(\calE)(A_{\Lambda}^{\rel, [\bullet]}))    
\end{equation}
is an isomorphism. For each $\lambda \in \Lambda$, $i(\calE)(\fkS_{\square_{\lambda}}) \in \calD(\fkS_{\square_{\lambda}})$ is represented by a bounded complex $P_{\square_{\lambda}}^{\bullet}$ of finite projective $\fkS_{\square_{\lambda}}$-modules. Then for any coface map $\fkS_{\square_{\lambda}} \rightarrow \fkS_{\Lambda}^{\rel, [n]}$, $i(\calE)(\fkS_{\Lambda}^{\rel, [n]})$ is represented by $P_{\square_{\lambda}}^{\bullet}\otimes_{\fkS_{\square_{\lambda}}} \fkS_{\Lambda}^{\rel, [n]}$, and $i(\calE)(A_{\Lambda}^{\rel, [n]})$ is represented by
\[
P_{\square_{\lambda}}^{\bullet}\otimes_{\fkS_{\square_{\lambda}}} A_{\Lambda}^{\rel, [n]} \cong (P_{\square_{\lambda}}^{\bullet}\otimes_{\fkS_{\square_{\lambda}}} \fkS_{\Lambda}^{\rel, [n]})\otimes_{\fkS_{\Lambda}^{\rel, [n]}} A_{\Lambda}^{\rel, [n]}.
\]
Thus, we have $i(\calE)(A_{\Lambda}^{\rel, [n]}) \cong i(\calE)(\fkS_{\Lambda}^{\rel, [n]})\widehat{\otimes}_{\fkS_{K}}^L A$. 
Note there exist $a, b \in \Z$ such that $i(\calE)(\fkS_{\Lambda}^{\rel, [n]}) \in \calD^{[a, b]}(\fkS_{\Lambda}^{\rel, [n]})$ for any $n$, since $\Lambda$ is finite.\footnote{In fact, $i(\calE)(B) \in \calD^{[a, b]}(B)$ for any $(B, I, M_{\Spf B}) \in ((X, M_X)/(\fkS_{K}, M_{\Spf\fkS_{K}}))_{\Prism}^{\op}$ by $(p, I)$-completely faithfully flat descent; namely, the argument in \cite[Prop.~3.11, Pf., Lem.~3.12]{GuoLi-FrobHeight} works in our situation by Lemma~\ref{lem: coprod with fkS is cover}}
Since $\fkS_{K} \rightarrow A$ has finite $(p, E)$-complete Tor amplitude, the map \eqref{eq: base-change-totalization-i(E)} is an isomorphism by Lemma~\ref{lem:algebraic-lemmas-for-base-change}(1). 

(2) Consider the commutative diagram
\[
\xymatrix{
R\Gamma_{\BK}(X, j_{\ast}\calE) \ar[r] \ar[d]^{\cong} & R\Gamma((X_{A/(E)}/(A, M_{\Spf A}))_{\Prism}, j_{\ast}\calE) \ar[d]^{\cong} \\
\Tot(j_{\ast}\calE(\fkS_{\Lambda}^{\rel, [\bullet]})) \ar[r] & \Tot(j_{\ast}\calE(A_{\Lambda}^{\rel, [\bullet]}))
}
\]
where the vertical isomorphisms are explained above. We need to show that the map
\begin{equation} \label{eq: base-change-totalization-jE}
(\Tot(j_{\ast}\calE(\fkS_{\Lambda}^{\rel, [\bullet]}))\widehat{\otimes}^L_{\fkS_{K}} A)\otimes_A^L A[p^{-1}] \rightarrow \Tot(j_{\ast}\calE(A_{\Lambda}^{\rel, [\bullet]}))\otimes_A^L A[p^{-1}]    
\end{equation}
is an isomorphism. 

For any coface map $\fkS_{\square_{\lambda}} \rightarrow \fkS_{\Lambda}^{\rel, [n]}$, note that 
\[
(j_{\ast}\calE(\fkS_{\Lambda}^{\rel, [n]})\widehat{\otimes}^L_{\fkS_{K}} A) \otimes_A^L A[p^{-1}]
\]
is represented by 
\[
(P_{\square_{\lambda}}^{\bullet}\otimes_{\fkS_{\square_{\lambda}}} \fkS_{\Lambda}^{\rel, [n]})\otimes_{\fkS_{\Lambda}^{\rel, [n]}} A_{\Lambda}^{\rel, [n]}[p^{-1}].
\]
Since $j_{\ast}\calE$ restricted to the analytic locus is a vector bundle, we deduce from Lemma~\ref{lem:algebraic-lemmas-for-base-change}(2) that the above is isomorphic to $j_{\ast}\calE(A_{\Lambda}^{\rel, [n]})\otimes_A A[p^{-1}]$ in $\calD(A_{\Lambda}^{\rel, [n]}[p^{-1}])$. Hence the map \eqref{eq: base-change-totalization-jE} is an isomorphism by Lemma~\ref{lem:algebraic-lemmas-for-base-change}(1), since $A \rightarrow A[p^{-1}]$ is classically flat.

For the last statement, note that any projection map $S_{K} \rightarrow S_{K}^{[n]}$ is $p$-completely faithfully flat by Lemma~\ref{lem:Breuil-prism-weakly-initial-for-crystalline-prisms}. So the map $\fkS_{K} \rightarrow A$ for each $A = A_{\inf}$, $A_{\cris}$, or $S^{[n]}_K$ has finite $(p, E)$-complete Tor amplitude by Lemma~\ref{lem: finite-p-compl-Tor-amp}. The map $\fkS_{K} \rightarrow W(k)$ also has finite $(p, E)$-complete Tor amplitude, 
due to $0 \rightarrow \fkS_{K} \stackrel{\times u}{\rightarrow} \fkS_{K} \rightarrow W(k) \rightarrow 0$.
\end{proof}

\begin{cor} \label{cor: Frob-twist-Breuil-cohom}
Consider the morphism $(\fkS_K, M_{\Spf \fkS_K}) \rightarrow (\phi_{\ast}S_K, M_{\Spf \phi_{\ast}S_K})$ in $(\Spf \calO_K, M_{\can})_{\Prism}^{\op}$ as in Example~\ref{eg: examples-log-prisms}, and write $X_{\phi_{\ast}S_K/p}\coloneqq X\times_{\calO_K} \phi_{\ast}S_K/p$. If $X$ is proper over $\calO_K$, then the map
\[
R\Gamma_{\Br}(X, j_{\ast}\calE)\otimes_{S_K, \phi}^L S_K[p^{-1}] \rightarrow R\Gamma((X_{\phi_{\ast}S_K/p}/(\phi_{\ast}S_K, M_{\Spf \phi_{\ast}S_K}))_{\Prism}, j_{\ast}\calE)\otimes_{S_K}^L S_{K}[p^{-1}]  
\]
is an isomorphism in $\calD(S_K[p^{-1}])$.
\end{cor}

\begin{proof}
Under the properness assumption, Theorems~\ref{thm: BK-cohom-perfect-cx} and \ref{thm: general-base-change-BK-cohom} tell that the map
\[
R\Gamma_{\BK}(X, j_{\ast}\calE)\otimes_{\fkS_K, \phi}^L S_K[p^{-1}] \rightarrow R\Gamma_{\Br}(X, j_{\ast}\calE)\otimes_{S_K}^L S_K[p^{-1}]
\]
is an isomorphism. Furthermore, since $\phi\colon \fkS_K \rightarrow \fkS_K$ is classically flat, the map $\phi^2\colon \fkS_K \rightarrow S_K$ has finite $(p, E)$-complete Tor amplitude by Lemma~\ref{lem: finite-p-compl-Tor-amp}. So the map
\[
R\Gamma_{\BK}(X, j_{\ast}\calE)\otimes_{\fkS_K, \phi^2}^L S_K[p^{-1}] \rightarrow R\Gamma((X_{\phi_{\ast}S_K/p}/(\phi_{\ast}S_K, M_{\Spf \phi_{\ast}S_K}))_{\Prism}, j_{\ast}\calE)\otimes_{S_K}^L S_{K}[p^{-1}]
\]
is an isomorphism again by Theorems~\ref{thm: BK-cohom-perfect-cx} and \ref{thm: general-base-change-BK-cohom}. The statement follows from the above two isomorphisms. 
\end{proof}

\subsection{Prismatic-crystalline comparison for crystals} \label{sec: pris-cris-comparison-crystals-semist}

We study a comparison between the prismatic and crystalline cohomology on semistable formal schemes, when the coefficients are given by completed prismatic ($F$-)crystals and their crystalline realizations. 

Let $f_1\colon (X_1,M_{X_1})\rightarrow (\Spec \calO_K/p, M_{\Spec \calO_K/p})$ denote the mod $p$ reduction of $f$. Let $(A, (p), M_A)$ be a crystalline prelog prism with $M_A$ integral such that either $(A, (p), M_A)$ is of rank $1$ or $(A, M_A)$ is a log ring. Let $J \subset A$ be a $p$-complete PD-ideal containing $p$. Suppose $(A, J, M_{\Spf A}) = (A, J, M_A)^a \in (\Spf\calO_K, M_\can)_{\dCRIS}$. Write $X_{A/J}\coloneqq X\times_{\Spf\calO_K} \Spf(A/J)$, and consider
$(X_{A/J}^{(1)}, M_{X_{A/J}}^{(1)})$ over $(\phi_{*}A, (p), \phi_* M_A)^a$ in Set-up~\ref{set-up: base-cryst-prism}.

Let $\calE_{\Prism} \in \Vect^{\an, \phi}((X,M_X)_\Prism)$ be an analytic prismatic ($F$)-crystal on $(X, M_X)_{\Prism}$, and write $j_{\ast}\calE_{\Prism} \in \CR^{\wedge, \an, \phi}((X, M_X)_{\Prism}, \calO_{\Prism})$ for the corresponding completed prismatic ($F$)-crystal given by Proposition~\ref{prop: right-adjoint-semistable-case}. 

\begin{defn}[{\cite[\S~3.5, Cor.~3.31]{du-liu-moon-shimizu-purity-F-crystal}}]\label{def:crystalline-realization}
Let $(X_1, M_{X_1})_{\CRIS}$ denote the absolute log crystalline site of $X_1$ (see also Remark~\ref{rem: two absolute crystalline sites}). Define $\calE_{\cris}$ to be the \emph{crystalline realization} of $j_{\ast}\calE_{\Prism}$ and let $(\calE_{\cris, \Q},\phi_{\calE_{\cris, \Q}})$ denote the associated finite locally free $F$-isocrystal on $(X_1, M_{X_1})_{\CRIS}$.
Recall that $\calE_{\cris}$ is
a quasi-coherent crystal of $\calO_{\CRIS}$-modules of finite type on $(X_1, M_{X_1})_{\CRIS}$ in the sense of \cite{du-liu-moon-shimizu-purity-F-crystal, du-moon-shimizu-cris-pushforward}, and in the small affine case with framing $\square$, it is given by 
\[
\calE_{\cris}(X_1, D_X) = j_{\ast}\calE_{\Prism}(S_{\square}, M_{\Spf S_{\square}})
\]
where $D_X \coloneqq \Spf S_{\square}$ equipped with the log structure $M_{D_X}$ associated to $\N^d$ with $e_i \mapsto T_i$; the exact closed immersion $X_1\hookrightarrow D_X$ is a log PD-thickening given by a PD ideal $J_\square\subset S_\square$, and $(X_1, D_X)$\footnote{This is also denoted by $(X_1, D_X, M_{D_X})$ to indicate the log structure in \cite[App.~B]{du-liu-moon-shimizu-purity-F-crystal}.} is a weakly final ind-object $\varinjlim_n (X_1,\Spec (S_{\square}/p^n))$ of $(X_1,M_{X_1})_{\CRIS}^{\mathrm{aff}}$ by \cite[Lem.~3.28]{du-liu-moon-shimizu-purity-F-crystal}. By Proposition~\ref{prop:comparison of crystals in two crystalline sites}, we also regarded $\calE_\cris$ as a $p$-adically completed crystal on $(X_1, M_{X_1})_{\CRIS}$ in the sense of Definitions~\ref{def: Koshikawa-crystalline-site} and \ref{defn:p-complete crystals over OCRIS}, in which case the evaluation is written as $\calE_{\cris}(S_{\square},J_\square)$ or $\calE_{\cris}(S_{\square},S_{\square}/J_\square)$.

Write $(E\otimes_{\calO_X}\omega^{\bullet}_{X/\calO_K}, \nabla_E)$ for the log de Rham complex associated to $\calE_{\cris}$ as in \cite[\S~6]{du-moon-shimizu-cris-pushforward}. Note from the construction that $E$ is a finitely generated $\calO_X$-module, since $\calE_{\cris}$ is a finitely generated $\calO_{\CRIS}$-module.  
\end{defn}

The log crystalline cohomology of $\calE_\cris$ is computed by log de Rham cohomology.  

\begin{thm} \label{thm: log-dR-cris-comparison}
Keep the notation as above. We have an isomorphism
\[
R\Gamma(((X_1, M_{X_1})/(\calO_K, M_\can))_{\CRIS}, \calE_{\cris}) \cong R\Gamma(X, (E\otimes_{\calO_X}\omega^{\bullet}_{X/\calO_K}, \nabla_E)). 
\]
In particular, $R\Gamma(((X_1, M_{X_1})/(\calO_K, M_\can))_{\CRIS}, \calE_{\cris})$ is a perfect complex over $\calO_K$ if $X$ is proper over $\calO_K$.
\end{thm}

Recall that we use two crystalline sites in this paper, and Proposition~\ref{prop:comparison of crystals in two crystalline sites} gives
$R\Gamma(((X_1, M_{X_1})/(\calO_K, M_\can))_{\CRIS}, \calE_{\cris})\cong R\Gamma(((X_1, M_{X_1})/(\Spf\calO_K)^\sharp)_{\CRIS}, \calE_{\cris})$, where $(\Spf\calO_K)^\sharp$ denotes $(\Spf \calO_K,M_\can)$ with the PD-ideal $(p)$.

\begin{proof}
 The comparison isomorphism is a special case of \cite[Thm.~6.21]{du-moon-shimizu-cris-pushforward}. Perfectness follows from the fact that $\calO_K$ is a Noetherian regular ring.  
\end{proof}

We next discuss the consequence of Theorem~\ref{thm: pris-crys-comparison-associated-sheaves}.

\begin{thm} \label{thm: prism-cryst-comparison-semist}
There exists an isomorphism of $\phi_* A$-complexes
\begin{align*}
R\Gamma(((X_{A/J})^{(1)}, M_{X_{A/J}}^{(1)})/&(\phi_{*}A, (p), \phi_* M_A)^a)_{\Prism}, (j_{\ast}\mathcal{E}_{\Prism})^{(1)})\\ 
&\cong \phi_\ast R\Gamma(((X_{A/J},M_{X_{A/J}})/(A,M_A))_\CRIS,\calE_\cris)
\end{align*}
which is functorial in $(A,(p),M_A)^a$. When $\calE_{\Prism}$ is an analytic prismatic $F$-crystal, then the above isomorphism is compatible with Frobenii after inverting $p$. Here, $(j_{\ast}\mathcal{E}_{\Prism})^{(1)}$ denotes the completed prismatic ($F$)-crystal on $((X_{A/J})^{(1)}, M_{X_{A/J}}^{(1)}/(\phi_{*}A, (p), \phi_* M_A)^a)_{\Prism}$ given by the pullback of $j_{\ast}\mathcal{E}_{\Prism}$. 
\end{thm}

\begin{proof}
We first claim that there is a natural isomorphism
\[
\nu_{\Prism}^{-1}(j_{\ast}\mathcal{E}_{\Prism})^{(1)} \stackrel{\cong}{\rightarrow} \phi_* \nu_{\CRIS}^*\mathcal{E}_{\cris},\quad\text{where}
\]
$\nu_{\Prism}\colon \Sh(((X_{A/J}, M_{X_{A/J}})/(A, M_A))_{\dCRIS}) \rightarrow \Sh(((X_{A/J})^{(1)}, M_{X_{A/J}}^{(1)}/(\phi_{*}A, (p), \phi_* M_A)^a)_{\Prism, \et})$ and $\nu_{\CRIS}\colon \Sh(((X_{A/J}, M_{X_{A/J}})/(A, M_A))_{\dCRIS}) \rightarrow \Sh(((X_{A/J}, M_{X_{A/J}})/(A, M_A))_{\CRIS})$ are morphisms of topoi as in \S~\ref{sec: prismatic-crystalline comparison}. Consider an object in $((X_{A/J}, M_{X_{A/J}})/(A, M_A))_{\dCRIS}$ given by a log prism $(B, (p), M_{\Spf B}) = (B, (p), M_B)^a$ over $(A, (p), M_{\Spf A})$ together with a $p$-completed PD-ideal $J_1 \subset B$ and map $\Spf(B/J_1) \rightarrow X_{A/J}$, as in \cite[Def.~6.4]{koshikawa}. Without loss of generality, we may assume that $(B, M_B)$ is a log ring. Note 
\[
\nu_{\Prism}^{-1}(j_{\ast}\mathcal{E}_{\Prism})^{(1)}(B, J_1, M_B) = (j_{\ast}\mathcal{E}_{\Prism})^{(1)}(\phi_* B, (p), M_B^{(1)})^a,
\]
and $(\phi_* B, (p), \phi_* M_B)^a \in (X, M_X)_{\Prism}$.

On the other hand, we claim that there is a natural identification 
\[
\phi_*\nu_{\CRIS}^*\mathcal{E}_{\cris}(B, J_1, M_B) = j_{\ast}\mathcal{E}_{\Prism}(\phi_* B, (p), \phi_* M_B)^a.
\]
By \cite[Rem.~4.2]{koshikawa} and \'etale descent, we may assume that $(X, M_X)$ is small affine. Furthermore, by Lemma~\ref{lem:Breuil-prism-weakly-initial-for-crystalline-prisms} and $p$-completely faithfully flat descent, we may assume to have a map $(S, (p), M_{\Spf S}) \rightarrow (\phi_* B, (p), \phi_* M_B)^a$ in $(X, M_X)_{\Prism}$ where $(S, (p), M_{\Spf S}) \in (X, M_X)_{\Prism}$ is the Breuil log prism. Then the claim follows directly from the construction of crystalline realization in \cite[Const.~3.29]{du-liu-moon-shimizu-purity-F-crystal}. Thus, we obtain a natural isomorphism $\nu_{\Prism}^{-1}(j_{\ast}\mathcal{E}_{\Prism})^{(1)} \stackrel{\cong}{\rightarrow} \phi_*\nu_{\CRIS}^*\mathcal{E}_{\cris}$.

Now, the statement follows from Theorem~\ref{thm: pris-crys-comparison-associated-sheaves}, Propositions~\ref{prop:prismatic-higher-direct-image} and \ref{prop: proj to etale for cris topos}.
\end{proof}

We apply the above comparison theorem in the following situation. For the Breuil ring $S_K$ as in Example~\ref{eg: examples-log-prisms}(2), write $J_{S_K} \subset S_K$ for the $p$-complete PD-ideal $\Ker(S_K \twoheadrightarrow \calO_K/p)$. We have $(S_K, S_K/J_{S_K} = \calO_K/p) \in (\Spec \calO_K/p, M_{\Spec \calO_K/p})_{\CRIS}^\op$ where the log structure $M_{S_K}$ on $\Spf S_K$ is given by $\N \rightarrow S_K$, $1 \mapsto u$.\footnote{Note that $M_{S_K}$ is different from $M_{\Spf S_K}$ in \textit{loc.~cit.}; $M_{S_K}=M_A^a$ and $M_{\Spf S_K}=(\phi_{\ast} M_A)^a$ in the notation of Theorem~\ref{thm: prism-cryst-comparison-semist}.} Note that $(S_K, \calO_K/p)$ is a weakly initial object of $(\Spec \calO_K/p, M_{\Spec \calO_K/p})_{\CRIS}^{\op}$. Let $(S_K^{(n)}, \calO_K/p)$ be the $(n+1)$-st self-coproduct of $(S_K, \calO_K/p)$ in $(\Spec \calO_K/p, M_{\Spec \calO_K/p})_{\CRIS}^{\op}$. We have a natural identification $S_K^{(n)} = S_K^{[n]}$ of rings by the construction and \cite[Cor.~2.39]{bhatt-scholze-prismaticcohom} (cf. \cite[Lem.~3.32]{du-liu-moon-shimizu-purity-F-crystal}). 

\begin{thm} \label{thm: pris-cris-comparison-over-Breuil S}
Suppose $(X, M_X)$ is a proper semistable $p$-adic formal scheme over $\Spf \calO_K$. We have a canonical isomorphism
\[
R\Gamma_{\Br}(X, j_{\ast}\calE_{\Prism}) \cong R\Gamma(((X_1, M_{X_1})/(S_K, M_{S_K}))_{\CRIS}, \calE_{\cris}).
\]
in $\calD(S_K)$. In particular, we have an isomorphism
\[
R\Gamma_{\BK}(X, j_{\ast}\calE_{\Prism})\otimes_{\fkS_K, \phi}^L S_K[p^{-1}] \cong R\Gamma(((X_1, M_{X_1})/(S_K,M_{S_K}))_{\CRIS}, \calE_{\cris, \Q})
\]
in $\calD_{\perf}(S_K[p^{-1}])$ compatible with the Frobenii.
Moreover, the following isomorphisms hold, and they are compatible with the Frobenii.
\begin{enumerate}
 \item For $(X_0, M_{X_0})\coloneqq (X_1, M_{X_1})\otimes_{\calO_K/p}k$, 
we have an isomorphism in $\calD_{\perf}(K_0)$:
\[
R\Gamma_W(X,j_\ast\calE_\Prism)\otimes_{W(k)}^L K_0\cong R\Gamma(((X_0, M_{X_0})/(W(k), M_0))_{\CRIS}, \calE_{\cris, \Q}).
\]

 \item The base change along any projection $S_K \rightarrow S_K^{[n]}$ induces an isomorphism
\[
R\Gamma_{\Br}(X, j_{\ast}\calE_{\Prism})\otimes_{S_K}^L S_K^{[n]}[p^{-1}] \cong R\Gamma(((X_1, M_{X_1})/(S_K^{[n]}, M_{S_K^{[n]}}))_{\CRIS}, \calE_{\cris, \Q})
\]
in $D_{\perf}(S_K^{[n]}[p^{-1}])$.

 \item The base change along $S_K \rightarrow A_{\cris}$, $u \mapsto [\pi^{\flat}]$, induces an isomorphism
\[
R\Gamma_{A_{\cris}}(X, j_{\ast}\calE_{\Prism})\otimes_{A_{\cris}}^L A_{\cris}[p^{-1}] \cong R\Gamma(((X_{\calO_C/p}, M_{X_{\calO_C/p}})/(A_{\cris}, M_{A_{\cris}}))_{\CRIS}, \calE_{\cris, \Q})
\]
in $D_{\perf}(A_{\cris}[p^{-1}])$.
 \end{enumerate}
Here the log structures $M_{S_K^{[n]}}$ and $M_{A_\cris}$ are the ones induced from $M_{S_K}$ via the described maps from $S_K$.
\end{thm}

Again, Proposition~\ref{prop:comparison of crystals in two crystalline sites} gives the identification of two crystalline sites
\[
R\Gamma(((X_{A/J}, M_{X_{A/J}})/(A,M_A))_{\CRIS}, \calE_{\cris, \Q})\cong R\Gamma(((X_{A/J}, M_{X_{A/J}})/(\Spf A)^\sharp)_{\CRIS}, \calE_{\cris, \Q})
\]
where $(\Spf A)^\sharp$ is the $p$-adic log PD-formal scheme $(\Spf A,M_A)$ with the PD-ideal $J$, and $R\Gamma(-,\calE_{\cris,\Q})$ denotes $R\Gamma(-,\calE_{\cris})\otimes_\Z^L\Q$.

\begin{proof}
Note that $R\Gamma_{\BK}(X, j_{\ast}\calE_{\Prism})$ is a perfect complex in $D(\fkS_K)$ by Theorem~\ref{thm: BK-cohom-perfect-cx}. 
Let us apply Theorem~\ref{thm: prism-cryst-comparison-semist} to $(A,J) = (S_K,J_{S_K})$; note $((X_{A/J})^{(1)}, M_{X_{A/J}}^{(1)})/(\phi_{*}A, (p), \phi_* M_A)^a)_{\Prism}$ is identified with $((X_{S_K/p}, M_{X_{S_K/p}}) / (S_K, M_{\Spf S_k}))_{\Prism}$ according to Example~\ref{eg: examples-log-prisms}(2). 
So as complexes of modules over $S_K=\phi_{*}A$, we have a canonical isomorphism 
\[
R\Gamma_{\Br}(X, j_{\ast}\calE_{\Prism}) \cong R\Gamma(((X_1, M_{X_1})/(S_K, M_{S_K}))_{\CRIS}, \calE_{\cris}).
\]
Thus, the first statement follows from Theorem~\ref{thm: general-base-change-BK-cohom} (for $A = S_K$). 

Parts (1), (2), and (3) follow from the first two, Theorems~\ref{thm: general-base-change-BK-cohom}(2) and \ref{thm: prism-cryst-comparison-semist} for $A = W(k)$, $A = S_K^{[n]} = S_K^{(n)}$, and $A = A_{\cris}$, respectively.
\end{proof}

\subsection{The Frobenius isogeny property for the crystalline and Breuil cohomology} \label{subsec: Frobenius-isogeny-property}

We use the method of \cite[Thm.~11.3, Pf.]{du-moon-shimizu-cris-pushforward} to show the following Frobenius isogeny property.

\begin{thm}\label{thm:Frobenius-isogeny-property}
Suppose that $(X, M_X)$ is a proper semistable $p$-adic formal scheme over $\calO_K$ and keep the notation as in \S~\ref{sec: pris-cris-comparison-crystals-semist}.
Then $R\Gamma(((X_1, M_{X_1})/(S_K, M_{S_K}))_{\CRIS}, \calE_{\cris, \Q})$ satisfies the Frobenius isogeny property: if we write $\calM \in \calD_{\perf}(S_K[p^{-1}])$ for the above complex, the induced map $1\otimes\phi\colon L\phi^* \calM \rightarrow \calM$ is an isomorphism. Consequently, the objects in Theorem~\ref{thm: pris-cris-comparison-over-Breuil S}(1)(2)(3) also satisfy the Frobenius isogeny property.
\end{thm}

To prove this, we need to introduce a perfect complex $i(\calE_\Prism)_\cris$ on $(X_1,M_{X_1})_\CRIS$, which is the crystalline realization of $i(\calE_\Prism)$.\footnote{The following argument works for $(X_1,M_{X_1})_\CRIS$ in the sense of Definition~\ref{def: Koshikawa-crystalline-site} as well as in the sense of \cite[Def.~10.1]{du-moon-shimizu-cris-pushforward}/Definition~\ref{def:DMS-crystalline-site}. Since we use the techniques in \cite{du-moon-shimizu-cris-pushforward} in the proof of Theorem~\ref{thm:Frobenius-isogeny-property}, we implicitly adopt the latter here.}

Let $\mathcal{D}_{\perf}((X_1, M_{X_1})_{\CRIS})$ be the derived $\infty$-category of perfect complexes, i.e., the full subcategory of $\mathcal{D}((X_1, M_{X_1})_{\CRIS}, \mathcal{O}_{X_1/\mathbf{Z}_p})$ consisting of perfect complexes. Set $\mathcal{D}_{\perf}((X_1, M_{X_1})_{\CRIS})_{\mathbf{Q}} \coloneqq \mathcal{D}_{\perf}((X_1, M_{X_1})_{\CRIS})\otimes_{\mathcal{D}(\mathbf{Z})}\mathcal{D}(\mathbf{Q})$. Let $\mathcal{D}_{\perf}^{\phi}((X_1, M_{X_1})_{\CRIS})$ be the $\infty$-category of pairs $(\mathcal{E}, \phi_{\mathcal{E}})$ where $\mathcal{E} \in \mathcal{D}_{\perf}((X_1, M_{X_1})_{\CRIS})$ and $\phi_{\mathcal{E}}\colon LF_{X_1, \CRIS}^*\mathcal{E} \stackrel{\cong}{\rightarrow} \mathcal{E}$ is an isomorphism in $\mathcal{D}_{\perf}((X_1, M_{X_1})_{\CRIS})_{\mathbf{Q}}$.\footnote{As in \cite{du-moon-shimizu-cris-pushforward}, we also write $F_{X_1, \CRIS}^*$ for $LF_{X_1, \CRIS}^*$ since $F_{X_1, \CRIS}^*$ is exact.}
Objects of $\mathcal{D}_{\perf}^{\phi}((X_1, M_{X_1})_{\CRIS})$ are called \emph{$F$-isocrystals in perfect complexes on $(X_1, M_{X_1})_{\CRIS}$}.

\begin{construction}[Crystalline realization of $i(\calE_{\Prism})$] \label{const: cryst-realization-perfect-cx}
Keep the set-up and notation as in \S~\ref{sec: pris-cris-comparison-crystals-semist}. Let $\{g_\lambda\colon X_\lambda\rightarrow X\}_{\lambda\in \Lambda}$ be a $p$-completely \'etale covering with $\Lambda$ a finite set such that each $X_\lambda$ is small affine with framing $\square_\lambda$. For each $\lambda$, consider the Breuil log prism $(S_{\square_{\lambda}}, (p), M_{\Spf S_{\square_{\lambda}}}) \in (X, M_X)_{\Prism}$ and the ind-object $(X_{\lambda, 1}, \Spf S_{\square_{\lambda}})=\varinjlim_m (X_{\lambda, 1}, \Spec S_{\square_{\lambda}}/p^m) \in (X_1, M_{X_1})_{\CRIS}^{\aff}$.\footnote{Here, the log structure of $(X_{\lambda, 1}, \Spf S_{\square_{\lambda}})$ is given as in Definition~\ref{def:crystalline-realization}.} Then the family $(X_{\lambda, 1}, \Spf S_{\square_{\lambda}})_{\lambda \in \Lambda}$ jointly covers the final object of $\Sh((X_1, M_{X_1})_{\CRIS})$. 

Similarly as in Construction~\ref{const: Cech-nerve-global-separated}, for any $\lambda_0, \ldots, \lambda_n \in \Lambda$, we have the coproduct $(S_{\square_{\lambda_0}, \ldots, \square_{\lambda_n}}, (p), M_{\Spf S_{\square_{\lambda_0}, \ldots, \square_{\lambda_n}}})$ of $(S_{\square_{\lambda_0}}, (p), M_{\Spf S_{\square_{\lambda_0}}}), \ldots, (S_{\square_{\lambda_n}}, (p), M_{\Spf S_{\square_{\lambda_n}}})$ in $(X, M_X)_{\Prism}^{\op}$. By \cite[Cor.~2.39]{bhatt-scholze-prismaticcohom}, $(X_{\lambda_0 \ldots \lambda_n, 1}, \Spf S_{\square_{\lambda_0}, \ldots, \square_{\lambda_n}})$ represents the product of $(X_{\lambda_0, 1}, \Spf S_{\square_{\lambda_0}}), \ldots, (X_{\lambda_n, 1}, \Spf S_{\square_{\lambda_n}})$ as ind-objects of $(X_1, M_{X_1})_{\CRIS}$, where $X_{\lambda_0 \ldots \lambda_n, 1}\coloneqq X_{\lambda_0, 1}\times_{X_1}\cdots \times_{X_1} X_{\lambda_n, 1}$ (cf.~\cite[Lem.~3.32]{du-liu-moon-shimizu-purity-F-crystal} for small affine case). Note that $H^q((X_{\lambda_0 \ldots \lambda_n, 1}, \Spf S_{\square_{\lambda_0}, \ldots, \square_{\lambda_n}}),\calO_{X_1/\Z_p})$ is $S_{\square_{\lambda_0}, \ldots, \square_{\lambda_n}}$ when $q=0$ and it vanishes when $q\geq 1$.

Write $\Spf(S^{[\bullet]}_{\Lambda})$ for the \v{C}ech nerve of the cover $\coprod_{\lambda\in \Lambda}(X_{\lambda, 1}, \Spf S_{\square_{\lambda}}) \rightarrow \ast$; the $n$-th spot is given by $\coprod_{(\lambda_0, \ldots, \lambda_n) \in \Lambda^{n+1}}(X_{\lambda_0 \ldots \lambda_n, 1}, \Spf(S_{\square_{\lambda_0}, \ldots, \square_{\lambda_n}}))$. By $p$-completely faithfully flat descent, we have a natural equivalence 
\[
\cal{D}_{\perf}^{\phi}((X_1, M_{X_1})_{\CRIS})\cong \lim_{[n] \in \Delta} \calD_{\perf}^{\phi}(S^{[n]}_{\Lambda}).
\]
We define \emph{crystalline realization of $i(\calE_{\Prism})$} to be the object $i(\calE_{\Prism})_{\cris} \in \cal{D}_{\perf}^{(\phi)}((X_1, M_{X_1})_{\CRIS})$ given by $\lim_{[n]\in \Delta} i(\calE_{\Prism})(S_{\Lambda}^{[n]})$ under the above equivalence.
\end{construction}

Let $(\Spf S_K)^\sharp$ be the $p$-adic log PD-formal scheme $(\Spf S_K, M_{S_K})$ with the PD-ideal $J_{S_K}=\Ker(S_K\rightarrow \calO_K/p)$.
So we have the identification
\[
R\Gamma(((X_1, M_{X_1})/(S_K,M_{S_K}))_{\CRIS}, \calE_{\cris, \Q})\cong R\Gamma(((X_1, M_{X_1})/(\Spf S_K)^\sharp)_{\CRIS}, \calE_{\cris, \Q}).
\]

\begin{construction}[\v{C}ech nerve of $((X_1, M_{X_1})/(\Spf S_K)^\sharp)_{\CRIS}$] \label{const: Cech-nerve-cristalline-site-over-SK}
Let $\{g_\lambda\colon X_\lambda\rightarrow X\}_{\lambda\in \Lambda}$ be a $p$-completely \'etale covering with $\Lambda$ a finite set such that each $X_\lambda$ is small affine with framing $\square_\lambda$. Construction~\ref{const: Cech-nerve-global-separated} gives a family of Breuil log prisms $(S_{\square_{\lambda}}, (p), M_{\Spf S_{\square_{\lambda}}}) \in (X_{\lambda}, M_{X_{\lambda}})_{\Prism}^\op$, which jointly covers the final object of the topos $\Sh(((X, M_X)/(S_{K}, M_{\Spf S_{K}}))_{\Prism})$, and its \v{C}ech nerve is denoted by $\Spf(S_{\Lambda}^{\rel, [\bullet]})$. Similarly as in small affine case, this gives rise to a family of ind-objects $(X_1, \Spf S_{\square_{\lambda}})_{\lambda \in \Lambda}$, which jointly covers the final object of $\Sh(((X_1, M_{X_1})/(\Spf S_K)^\sharp)_{\CRIS})$; its associated \v{C}ech nerve is given by $(X_1, \Spf S_{\Lambda}^{\rel, [\bullet]})$ by the construction and \cite[Cor.~2.39]{bhatt-scholze-prismaticcohom}. 
\end{construction}

\begin{prop} \label{prop:comparison-of-cohomology-of-crystalline-realization}
The induced map
\[
R\Gamma(((X_1, M_{X_1})/(\Spf S_K)^\sharp)_{\CRIS}, i(\calE_\Prism)_\cris)_\Q
\rightarrow 
R\Gamma(((X_1, M_{X_1})/(\Spf S_K)^\sharp)_{\CRIS}, \calE_{\cris, \Q})
\]
is an isomorphism in $\calD_\perf(S_K[p^{-1}])$ and compatible with Frobenii, where the subscript $\Q$ in the source refers to $-\otimes_\Z^L\Q=-\otimes_{S_K}^LS_K[p^{-1}]$.
\end{prop}

\begin{proof}
Consider the \v{C}ech nerve $(X_1, \Spf S_{\Lambda}^{\rel, [\bullet]})$ in Construction~\ref{const: Cech-nerve-cristalline-site-over-SK}. We have natural isomorphisms
\[
i(\calE_{\Prism})(S_{\Lambda}^{\rel, [\bullet]})\otimes_{S_{\Lambda}^{\rel, [\bullet]}}^L S_{\Lambda}^{\rel, [\bullet]}[p^{-1}] \cong j_{\ast}\calE_{\Prism}(S_{\Lambda}^{\rel, [\bullet]})\otimes_{S_{\Lambda}^{\rel, [\bullet]}}^L S_{\Lambda}^{\rel, [\bullet]}[p^{-1}] \cong  \calE_{\cris, \Q}(S_{\Lambda}^{\rel, [\bullet]}),
\]
where the first follows from the constructions of $i(\calE_{\Prism})$ and $j_{\ast}\calE_{\Prism}$ with Lemma~\ref{lem:algebraic-lemmas-for-base-change}(2) (applied to $\fkS_{\square_{\lambda}} \rightarrow S_{\Lambda}^{\rel, [\bullet]}$), and the second uses the construction of $\calE_{\cris}$ and Lemma~\ref{lem:algebraic-lemmas-for-base-change}(2) (applied to $S_{\square_{\lambda}} \rightarrow S_{\Lambda}^{\rel, [\bullet]}$). Furthermore,  Construction~\ref{const: cryst-realization-perfect-cx} gives
\[
i(\calE_{\Prism})(S_{\Lambda}^{\rel, [\bullet]}) \cong i(\calE_{\Prism})_{\cris}(S_{\Lambda}^{\rel, [\bullet]}).
\]
Thus, we obtain $(i(\calE_{\Prism})_{\cris}(S_{\Lambda}^{\rel, [\bullet]}))_{\Q} \cong \calE_{\cris, \Q}(S_{\Lambda}^{\rel, [\bullet]})$.

We now apply the \v{C}ech cohomological descent; the proofs for Lemmas~\ref{lem:CA method} and \ref{lem: CA-method-infinity-category} with Remark~\ref{rem: CA-for-jointly-surjective-case} work verbatim for the current set-up for crystalline cohomology. So we have a commutative diagram
\[
\xymatrix{
R\Gamma(((X_1, M_{X_1})/(\Spf S_K)^\sharp)_{\CRIS}, i(\calE_\Prism)_\cris)_\Q \ar[r] \ar[d]^{\cong} & R\Gamma(((X_1, M_{X_1})/(\Spf S_K)^\sharp)_{\CRIS}, \calE_{\cris, \Q}) \ar[d]^{\cong} \\
\Tot(i(\calE_{\Prism})_{\cris}(S_{\Lambda}^{\rel, [\bullet]}))_{\Q} \ar[r] & \Tot(\calE_{\cris, \Q}(S_{\Lambda}^{\rel, [\bullet]}))
}
\]
where the vertical maps are isomorphisms. Since $\Z \rightarrow \Q$ is classically flat, it follows from Lemma~\ref{lem:algebraic-lemmas-for-base-change}(1) that
\[
\Tot(i(\calE_{\Prism})_{\cris}(S_{\Lambda}^{\rel, [\bullet]}))_{\Q} \cong \Tot((i(\calE_{\Prism})_{\cris}(S_{\Lambda}^{\rel, [\bullet]}))_{\Q}).
\]
Thus, the bottom map of the above diagram is an isomorphism, and the statement follows. The compatibility with Frobenii is straightforward.
\end{proof}

\begin{proof}[Proof of Theorem~\ref{thm:Frobenius-isogeny-property}]
By Proposition~\ref{prop:comparison-of-cohomology-of-crystalline-realization}, it suffices to prove the Frobenius isogeny property for $R\Gamma(((X_1, M_{X_1})/(\Spf S_K)^\sharp)_{\CRIS}, i(\calE_\Prism)_\cris)_\Q$; the second statement follows from Theorem~\ref{thm: general-base-change-BK-cohom} as the involved base changes are compatible with Frobenii. 
Set $(Y,M_Y)\coloneqq (\Spec \calO_K/p,M_{\Spec\calO_K/p})$ and $(X_1',M_{X_1'})\coloneqq (Y,M_Y)\times_{F_Y,(Y,M_Y)}(X_1,M_{X_1})$; they sit in the following commutative diagram:
\begin{equation}\label{eq:relative Frobenius}
\xymatrix{
(X_1,M_{X_1}) \ar[r]^-{F_{X_1/Y}}\ar[dr]_-{f_1} \ar@/^2pc/[rr]^-{F_{X_1}} & (X_1',M_{X_1'})\ar[d]_-{f_1'}\ar[r]^g \ar@{}[dr] | {\square} &(X_1,M_{X_1})\ar[d]^-{f_1}\\
& (Y,M_Y) \ar[r]^-{F_Y} & (Y,M_Y).
}
\end{equation}
Set $\calK\coloneqq i(\calE_\Prism)_\cris$ and $\calK'\coloneqq g_{\CRIS}^\ast \calK$.
Write $\phi_{\calK_{\Q}}=p^{r}\phi_{\calK}$ for some $r\in\Z$ and a morphism $\phi_{\calK}\colon F_{X_1,\CRIS}^\ast\calK\rightarrow \calK$ in $D_\perf((X_1,M_{X_1})_\CRIS,\calO_{X_1/\Z_p})$. 
To simplify the notation, we write $R\Gamma((X_1/S_K)_\CRIS,\calK)$ for $R\Gamma(((X_1, M_{X_1})/(\Spf S_K)^\sharp)_{\CRIS}, \calK)$ and so on.
Consider the composite $\alpha$ of the maps
\begin{equation}\label{eq:Frobenius pullback and higher direct image}
\xymatrix{
R\Gamma((X_1/S_K)_\CRIS,\calK)\otimes_{S_K,\phi}S_K \ar[r]^-{\delta}\ar@{-->}[dd]_-{\alpha}
&R\Gamma((X_1'/S_K)_\CRIS,\calK')\ar[d]_-{R\Gamma((X_1'/S_K)_\CRIS,\mathrm{adj})}^-{\beta}\\
&R\Gamma((X_1'/S_K)_\CRIS,RF_{X_1/Y,\CRIS,\ast}F_{X_1/Y,\CRIS}^\ast\calK')
\ar[d]^-\cong\\
R\Gamma((X_1/S_K)_\CRIS,\calK)
&
R\Gamma((X_1/S_K)_\CRIS,F_{X_1,\CRIS}^\ast\calK)\ar[l]_-{R\Gamma((X_1/S_K)_\CRIS,\phi_{\calK})},
}
\end{equation}
where $\delta$ is defined by the base change with respect to the Cartesian diagram in \eqref{eq:relative Frobenius} and the second right vertical map is an isomorphism given by the composition of right derived functors.
The morphism in question
\[
R\Gamma((X_1/S_K)_\CRIS,\calK)_\Q\otimes_{S_K[p^{-1}],\phi}S_K[p^{-1}] \rightarrow R\Gamma((X_1/S_K)_\CRIS,\calK)_\Q
\]
is given by $p^r\alpha_\Q$ where $\alpha_\Q\coloneqq \alpha\otimes \id_{S_K[p^{-1}]}$. 
Since $\phi_{\calK_\Q}$ is an isomorphism, so is $R\Gamma((X_1/S_K)_\CRIS,\phi_{\calK})_\Q$.
Hence it remains to show that $\beta_\Q$ and $\delta_\Q$ are isomorphisms. 

For $\beta_\Q$, consider the diagram
\[
\xymatrix{
\calK' \ar[r]^-{\mathrm{adj}}\ar@{=}[d]
&RF_{X_1/Y,\CRIS,\ast}(F_{X_1/Y,\CRIS}^\ast\calK'\otimes F_{X_1/Y,\CRIS}^\ast\calO_{X_1'/\Z_p})\\
\calK'\otimes \calO_{X_1'/\Z_p}\ar[r]^-{\id\otimes\mathrm{adj}} &\calK'\otimes RF_{X_1/Y,\CRIS,\ast} F_{X_1/Y,\CRIS}^\ast\calO_{X_1'/\Z_p}\ar[u]_-\gamma^-\cong.
}
\]
Here $\gamma$ is the canonical quasi-isomorphism given by the projection formula, which is applicable since $\calK'$ is a perfect complex, and the above diagram is commutative by construction (see \cite[0943]{stacks-project}). Since $\beta=R\Gamma((X_1/S_K)_\CRIS,\gamma\circ \id\otimes\mathrm{adj})$, it remains to prove that the canonical morphism given by adjunction
\[
\calO_{X_1'/\Z_p}\rightarrow RF_{X_1/Y,\CRIS,\ast}F_{X_1/Y,\CRIS}^\ast\calO_{X_1'/\Z_p}=RF_{X_1/Y,\CRIS,\ast}\calO_{X_1/\Z_p}
\]
is an isogeny (i.e., an isomorphism in $D_\perf((X,M_X)_\CRIS)_\Q$). This follows from \cite[Prop.~2.24]{hyodo-kato}; see \cite[Thm.~11.3, Pf., Rem.~11.6(1)]{du-moon-shimizu-cris-pushforward}.

For $\delta_\Q$, recall from Theorem~\ref{thm: pris-cris-comparison-over-Breuil S} that $R\Gamma_{\Br}(X, j_{\ast}\calE_{\Prism}) \cong R\Gamma((X_1/S_K)_{\CRIS}, \calE_{\cris})$.
Furthermore, we have a canonical isomorphism
\[
R\Gamma((X_{\phi_{\ast}S_K/p}/(\phi_{\ast}S_K, M_{\Spf \phi_{\ast}S_K}))_{\Prism}, j_{\ast}\calE_{\Prism}) \cong R\Gamma((X_1'/S_K)_{\CRIS}, \calE_{\cris}') 
\]
for $\calE_{\cris}'\coloneqq g_\CRIS^{\ast}\calE_{\cris}$
by Theorem~\ref{thm: prism-cryst-comparison-semist} applied to $A = \phi_{\ast}S_K$ with $M_A$ given by $\N \rightarrow \phi_{\ast}S_K$, $1\mapsto u^p$. On the other hand,  we have
\[
R\Gamma((X_1^{(')}/S_K)_{\CRIS}, \calK^{(')})_{\Q} \cong R\Gamma((X_1^{(')}/S_K)_{\CRIS}, \calE_{\cris, \Q}^{(')}).
\]
by Proposition~\ref{prop:comparison-of-cohomology-of-crystalline-realization} and its proof (for $\calK'$).
Thus, we deduce from the above isomorphisms and Corollary~\ref{cor: Frob-twist-Breuil-cohom} that $\delta_{\Q}$ is an isomorphism.
\end{proof}

\subsection{The Hyodo--Kato theory} \label{subsec: hyodo-kato-theory}

In this subsection, we discuss the Hyodo--Kato theory. This will be used to equip $R\Gamma(((X_0, M_{X_0})/((W(k),M_0))_{\CRIS}, \calE_{\cris, \Q})$ in Theorem~\ref{thm: pris-cris-comparison-over-Breuil S} with additional structures in \S~\ref{subsec: hyodo-kato-cohom}. We follow the Hyodo--Kato theory for log crystalline cohomology in \cite{hyodo-kato,Kato-exposeVI, Tsuji-Cst, Beilinson-crystalline-period-map, Beilinson-crystalline-period-map-arXiv}. For the main construction, we use Beilinson's work \cite{Beilinson-crystalline-period-map-arXiv} (the arXiv version) and the absolute crystalline site as in Definition~\ref{def:DMS-crystalline-site}, which is most suitable for our construction.

A \emph{$(\phi, N)$-module over $K_0$} is a triple $(V, \phi_V, N_V)$ where $V$ is a finite-dimensional $K_0$-vector space, $\phi_V\colon V \rightarrow V$ is a $\phi$-semilinear automorphism, and $N_V\colon V \rightarrow V$ is a $K_0$-linear endomorphism such that $N_V\phi_V = p\phi_V N_V$. It is said to be \emph{effective} if $V$ contains a $\phi_V$-stable $W(k)$-lattice. Write $D_{\phi, N}(K_0)$ for the bounded derived category of $(\phi, N)$-modules over $K_0$, and $D_{\phi, N}(K_0)^{\mathrm{eff}}$ for its subcategory constructed out of effective ones. Note that for every $V\in D_{\phi, N}(K_0)^{\mathrm{eff}}$, there exists $m\geq 0$ with $N_V^m=0$.

Denote $\calY = \Spf \calO_K$ with the canonical log structure $M_\calY$ and $(Y_n, M_{Y_n}) \coloneqq (\calY, M_{\calY})\times_{\Z_p} \Z_p/p^n$. Set $(Y_0, M_{Y_0}) \coloneqq (\calY, M_{\calY})\times_{\calO_K} k$. Note that $(Y_0, M_{Y_0})$ can be regarded as the mod $p$ reduction $(Y_1^0, M_{Y_1^0})$ of $\Spf W(k)$ with the canonical log structure. The map $k = W(k)/(p) \rightarrow \calO_K/(p) \rightarrow k$ gives 
\[
(Y_0, M_{Y_0}) \stackrel{i}{\rightarrow} (Y_1, M_{Y_1}) \stackrel{i^0}{\rightarrow} (Y_1^0, M_{Y_1^0}).
\]

We refer the reader to \cite[\S~1.14]{Beilinson-crystalline-period-map-arXiv}\footnote{More precisely, \cite[Def.~11.1, Rem.~11.2]{du-moon-shimizu-cris-pushforward} as we work on the big site. However, this does not cause any essential difference.} for the notion of non-degenerate perfect $F$-crystals on an integral and quasi-coherent log scheme $(Z, M_Z)$ over $\F_p$. As in \textit{loc. cit.}, let $D_{\phi}^{\mathrm{pcr}}(Z, M_Z)^{(\mathrm{nd})}$ denote the category of (non-degenerate) perfect $F$-crystals on $(Z, M_Z)_{\CRIS}$, and write $D_{\phi}^{\mathrm{pcr}}(Z, M_Z)^{(\mathrm{nd})}_{\Q}$ for its isogeny category.

We first recall from \cite{Beilinson-crystalline-period-map-arXiv} a description of the category of $F$-isocrystals in perfect complexes on $(Y_1, M_{Y_1})$ or $(Y_0, M_{Y_0})$. 

\begin{thm}[{\cite[\S~1.15, Thm.-Const.]{Beilinson-crystalline-period-map-arXiv}}] \label{thm: Beilison-HK-equiv}
Let $f\colon (Z, M_Z) \rightarrow (Y_1^0, M_{Y_1^0})$ be a morphism of integral quasi-coherent log schemes. There is a natural functor
\[
\epsilon_f\colon D_{\phi, N}(K_0)^{\mathrm{eff}} \rightarrow D_{\phi}^{\mathrm{pcr}}(Z, M_Z)^{\mathrm{nd}}_{\Q},
\]
which is compatible with base change: for $g\colon (Z', M_{Z'}) \rightarrow (Z, M_Z)$, we have a canonical identification $\epsilon_{f\circ g} \stackrel{\cong}{\rightarrow} g^{\ast}_{\CRIS}\epsilon_f$.

Moreover, the following functors give equivalences of categories
\[
D_{\phi, N}(K_0)^{\mathrm{eff}} \underset{\cong}{\xrightarrow{\epsilon\coloneqq\epsilon_\id}} D_{\phi}^{\mathrm{pcr}}(Y_1^0, M_{Y_1^0})^{\mathrm{nd}}_{\Q} \underset{\cong}{\xrightarrow{i^{0, \ast}_{\CRIS}}} D_{\phi}^{\mathrm{pcr}}(Y_1, M_{Y_1})^{\mathrm{nd}}_{\Q} \underset{\cong}{\xrightarrow{i^{\ast}_{\CRIS}}} D_{\phi}^{\mathrm{pcr}}(Y_0, M_{Y_0})^{\mathrm{nd}}_{\Q}.
\]    
\end{thm}

\begin{proof}
This is proved in \cite[\S~1.15, Thm.-Const.]{Beilinson-crystalline-period-map-arXiv}. Note that the big crystalline set-up is also applicable by \cite[Rem.~10.3]{du-moon-shimizu-cris-pushforward} and $g^{\ast}_{\CRIS}=Lg^{\ast}_{\CRIS}$.    
\end{proof}

\begin{rem}\label{rem:Beilinson-epsilon-functor}
If $V$ is a $(\phi,N)$-module over $K_0$, then $\epsilon(V)$ underlies a finite locally free $F$-crystal on $(Y_0, M_{Y_0})_\CRIS$ explicitly given by a module with log connection.
If $V\in D_{\phi, N}(K_0)^{\mathrm{eff}}$ is represented by a bounded complex $\cdots\rightarrow V_n\rightarrow V_{n+1}\rightarrow\cdots$ of effective $(\phi,N)$-modules over $K_0$, then $\epsilon_f(V)$ is represented by a bounded complex $\cdots\rightarrow f_\CRIS^\ast\epsilon(V_n)\rightarrow f_\CRIS^\ast\epsilon(V_{n+1})\rightarrow\cdots$ of $F$-isocrystals on $(Z,M_Z)_\CRIS$.
\end{rem}

We need to deduce several properties from the construction of $\epsilon$. 
Let $\calT^{\sharp} = (\calT, M_{\calT}, \calJ_{\calT}, \gamma_{\calT})$ be an affine $p$-adic log PD-formal scheme such that $M_{\calT}$ is integral and quasi-coherent (see \cite[\S~10]{du-moon-shimizu-cris-pushforward}).
Suppose that there exists a strict morphism $(U,M_U)\coloneqq (V(\calJ_{\calT}), M_{\calT}|_U)\rightarrow (Z, M_Z)$.
Then $\calT = \Spf A$ for a $p$-adically complete ring $A$, and $\calJ_{\calT}$ corresponds to an open PD-ideal $J$ such that $U = \Spec A/J$. For $\calK_\Q\in D_{\phi}^{\mathrm{pcr}}(Z, M_Z)^{\mathrm{nd}}_{\Q}$, we have $R\Gamma((U,\calT),\calK_\Q)\coloneqq R\Gamma((U,\calT),\calK)\otimes_\Z\Q\in D_\perf(A[p^{-1}])$ with $R\Gamma((U,\calT),\calK)=R\varprojlim_n R\Gamma((U,\Spec A/p^n),\calK)$ (cf.~\cite[Const.~10.6, Prop.~10.7]{du-moon-shimizu-cris-pushforward}).
We will write $\calK_\Q(A)$ for $R\Gamma((U,\calT),\calK_\Q)$ for simplicity.

When $(Z,M_Z)=(Y_1,M_{Y_1})$, the main two examples are $A=S_K$ with $J=\Ker(S_K\rightarrow\calO_K/p)$ and log structure associated to $\N\rightarrow S_K$, $1\mapsto u$, and $(\calO_K,(p),M_\can)$.  

To discuss the case $Z=\overline{Y}_1\coloneqq \Spec \calO_C/p$, let us first introduce two period rings.

\begin{defn}\label{def:AlogcrisK-and-Asthat}
Consider the log structures on $A=\calO_K, \calO_C,A_\inf,A_\cris$ associated to the pre-log structures $\N\rightarrow A$ sending $1$ to $\pi, \pi, [\pi^\flat], [\pi^\flat]$, respectively.
\begin{enumerate}
  \item Define $\widehat{A}_\st$ to be the $p$-completed log PD-envelope of the surjection $A_{\inf}\otimes_{W(k)}S_K \rightarrow \calO_C$, where $S_K\rightarrow \calO_C$ sends $u$ to $\pi$; if $A_\cris\{ v\}_\mathrm{PD}$ denotes the $p$-completed PD-polynomial algebra in $v$, a $p$-adically continuous $A_\cris$-algebra homomorphism $A_\cris\langle v\rangle_\mathrm{PD}\rightarrow \widehat{A}_\st$ sending $v^{[n]}$ to $(u/[\pi^\flat]-1)^{[n]}$ is an isomorphism. Write $\widehat{B}_{\st}^+ \coloneqq \widehat{A}_{\st}[p^{-1}]$.
  
  \item Let $\theta_K\colon A_{\inf, K}\coloneqq A_{\inf}\otimes_{W(k)} \calO_K \rightarrow \calO_C$ be the $\calO_K$-linear extension of $\theta\colon A_{\inf} \rightarrow \calO_C$. Define $A_{\logcris, K}$ to be the $p$-completed log PD-envelope of the surjection $\theta_K$; note that $S_K\rightarrow\calO_K$, $u\mapsto \pi$, induces an isomorphism $\widehat{A}_\st\otimes_{S_K}\calO_K\cong A_{\logcris,K}$. Write $M_{A_{\logcris,K}}$ for the resulting log structure.
  Let $\iota_1\colon S_K \rightarrow A_{\logcris, K}$ (resp. $\iota_2\colon S_K \rightarrow A_{\logcris, K}$) be the map given by $u \mapsto [\pi^{\flat}]$ (resp. $u \mapsto \pi$); $\iota_1$ factors through $A_{\cris}$, and $\iota_2$ factors through $\calO_K$.  
\end{enumerate}
\end{defn}

\begin{lem}\label{lem:crystalline-interpretation-Asthat}
Consider $h\colon (\overline{Y}_1,M_{\overline{Y}_1})\coloneqq(\Spec \calO_C/p,(\calO_{C}\smallsetminus\{0\})^a)\rightarrow (Y_1,M_{Y_1})$.
\begin{enumerate}
  \item $A=A_\cris$ with log structure associated to $\calO_{C^\flat}\smallsetminus\{0\}\xrightarrow{[-]} W(\calO_{C^\flat})\rightarrow A_\cris$ and PD-ideal $J=\Ker(A_\cris\rightarrow \calO_C/p)$ gives the initial ind-object of $(\overline{Y}_1,M_{\overline{Y}_1})_\CRIS$.
  \item $h_{\CRIS,\ast}\calO_{\overline{Y}_1/\Z_p}(S_K)=\widehat{A}_{\st}$ and $R^qh_{\CRIS,\ast}\calO_{\overline{Y}_1/\Z_p}(S_K)=0$ for $q\geq 1$.
  \item $h_{\CRIS,\ast}\calO_{\overline{Y}_1/\Z_p}(\calO_K)=A_{\logcris,K}$ and $R^qh_{\CRIS,\ast}\calO_{\overline{Y}_1/\Z_p}(\calO_K)=0$ for $q\geq 1$.
\end{enumerate}
\end{lem}

\begin{proof}
(1) Observe that the surjection $A_\cris\rightarrow \calO_C/p$ is exact with respect to the log structures we consider and that $(\calO_{C}\smallsetminus\{0\})^a/(\calO_C/p)^\times$ is uniquely $p$-divisible. Hence the assertion follows from \cite[\S~1.17, Lem.]{Beilinson-crystalline-period-map-arXiv}.

(2) We follow the notation in \cite[\S~10]{du-moon-shimizu-cris-pushforward}. Let $\calS^\sharp$ denote the affine $p$-adic log PD-formal scheme associated to $\calS=\Spf S_K$ (with the other structures as above). We deduce from \cite[Prop.~10.5, 10.7, Lem.~10.9]{du-moon-shimizu-cris-pushforward} that 
\[
Rh_{\CRIS,\ast}\calO_{\overline{Y}_1/\Z_p}(S_K)=R\Gamma(((\overline{Y}_1,M_{\overline{Y}_1})/\calS^\sharp)_\CRIS,\calO_{\overline{Y}_1/\calS}).
\]
Observe that $\widehat{A}_\st$ is identified with the $p$-completed log PD-envelope of the surjection $A_\cris\otimes_{W(k)}S_K\rightarrow \calO_C/p$ where the log structure on $A_\cris$ (resp.~$\calO_C/p$) is the one associated to $\calO_{C^\flat}\smallsetminus\{0\}$ (resp.~$\calO_{C}\smallsetminus\{0\}$). One can deduce from this and (1) that $\widehat{A}_\st$ (with the new PD-ideal and log structure) gives the initial ind-object of $((\overline{Y}_1,M_{\overline{Y}_1})/\calS^\sharp)_\CRIS^\op$.
Hence $H^q(((\overline{Y}_1,M_{\overline{Y}_1})/\calS^\sharp)_\CRIS,\calO_{\overline{Y}_1/\calS})=H^q((\overline{Y}_1,\calS),\calO_{\overline{Y}_1/\calS})$; the latter is $\widehat{A}_\st$ when $q=0$, and zero when $q\geq 1$ by an argument similar to Lemma~\ref{lem: affine vanishing on crystalline site}.

(3) This is proved exactly in the same way as (2).
\end{proof}

\begin{rem}\label{rem:Asthat-and-Tsuji'sring}
\hfill
\begin{enumerate}
 \item Recall the $A_\cris$-algebra isomorphism $A_\cris\langle v\rangle_\mathrm{PD}\cong \widehat{A}_\st$ sending $v$ to $u/[\pi^\flat]-1$.
It follows that the log connection $\nabla_{S_K}\colon S_K\rightarrow S_K\,d\operatorname{\log}u$ extends to a $p$-adically continuous $A_\cris$-linear log connection $\nabla_{\widehat{A}_\st}\colon \widehat{A}_\st\rightarrow \widehat{A}_\st\,d\operatorname{\log}u$. Define an $A_\cris$-linear derivation $N$ on $\widehat{A}_\st$ by the formula $\nabla_{\widehat{A}_\st}(x)=N(x)\,d\operatorname{\log}u$.
 \item It is not difficult to see that our $\widehat{A}_\st$ agrees with $\varprojlim_n P_n$ in \cite[\S~1.6]{Tsuji-Cst} as $A_\cris$-algebras with all the additional structures by sending $u/[\pi^\flat]-1$ to $(v_{[\pi^\flat]\bmod{p^n}})_n$; then (1) is compatible with \cite[Lem.~1.6.5]{Tsuji-Cst}. We remark that $\varprojlim_n P_n$ is initially introduced in \cite[\S~3]{Kato-exposeVI} but we follow the normalization of \cite{Tsuji-Cst}.
 Similarly, our $A_{\logcris,K}$ agrees with $A_{\mathrm{log\text{-}crys},S}$ in \cite[\S~4.6]{Tsuji-Cst}.
\end{enumerate}
\end{rem}

\begin{defn}[{cf.~\cite[\S~4.1]{Tsuji-Cst}}]\label{def:Bst}
Fix a uniformizer $\pi$ of $\calO_K$ and choose $\pi^\flat\in \calO_{C^\flat}$ given by a compatible system of $p$-power roots of $\pi$. We define $B_\st^+\coloneqq B_{\st,\pi}^+$ to be the subring $B_\cris[u_{\pi^\flat}]$ of $B_\dR^+$ where $u_{\pi^\flat}=\operatorname{log}([\pi^\flat]/\pi)$; it is independent of the choice of $\pi^\flat$ as a subring. Extend $\phi$ on $B_\cris^+$ to $B_\st^+$ by $\phi(u_{\pi^\flat})=pu_{\pi^\flat}$ and define a $B_\cris^+$-linear derivation $N=N_\pi$ on $B_\st^+$ by the formula $N(u_{\pi^\flat})=-1/e$ where $e$ is the ramification index of $K$. Set $B_\st\coloneqq B_\st^+[t^{-1}]$ and extend $\phi$ and $N$ in an obvious way.
\end{defn}

\begin{rem}\label{rem:semistable-period-ring-convention}
Let $B_{\st,\mathrm{Fon}}^+$ denote $\operatorname{Sym} (C^{\flat,\times})\otimes_{\operatorname{Sym} (\calO_{C^\flat}^\times),\lambda}B_\cris^+$, where $\lambda\colon \calO_{C^\flat}^\times\rightarrow B_\cris^+$ is an additive map defined by $\lambda(y)=\sum_{n\geq 1}(-1)^{n+1}[y]^n/n$ for $y\in 1+\fkm_{\calO_{C^\flat}}$.
This is the positive semistable period ring defined in \cite[\S~3]{fontaine-period-ring}, and a choice $\pi^\flat\in C^\flat$ identifies $B_{\st,\mathrm{Fon}}^+$  with $B_{\st,\pi}^+$ above. Our sign convention of the monodromy operator $N$ on $B_\st^+$ is compatible with that of \cite{Tsuji-Cst} and the opposite of \cite{fontaine-period-ring}. We use the canonical normalization in the sense of \cite[Rem.~3.2.4]{fontaine-period-ring}, whereas \cite{Tsuji-Cst} uses the normalization $N(u_{\pi^\flat})=-1$. We choose our normalization so that it is compatible with the explicit construction of Theorem~\ref{thm: Beilison-HK-equiv}; see Proposition~\ref{prop:Beilinson-epsilon-functor-properties} below.
\end{rem}

\begin{prop}\label{prop:Bst-vs-Asthat}
There exists a well-defined injective $B_\cris^+$-algebra map $\iota\colon B_{\st}^+\rightarrow\widehat{B}_{\st}^+$ sending $u_{\pi^\flat}$ to $\log ([\pi^\flat]/u)$. Moreover, $\iota$ is compatible with $\phi$, $N$, and $\Gal(\overline{K}/K)$-actions and identifies $B_{\st}^+$ with the subring $(\widehat{B}_{\st}^+)^{N\text{-nilp}}$ of $\widehat{B}_{\st}^+$ consisting of elements on which $N$ acts nilpotently. Finally, the diagram
\[
\xymatrix{
B_{\st}^+ \ar[r]^-\iota\ar[d]& \widehat{B}_{\st}^+\ar[d]\\
B_\dR^+ & A_{\logcris,K}[p^{-1}]\ar[l]
}
\]
is commutative and $\Gal(\overline{K}/K)$-equivariant. 
\end{prop}

\begin{proof}
This follows from \cite[Prop.~4.1.3]{Tsuji-Cst} and Remark~\ref{rem:Asthat-and-Tsuji'sring}.
\end{proof}

\begin{prop}\label{prop:Beilinson-epsilon-functor-properties}
Keep the notation as above.
Take any $V\in D_{\phi, N}(K_0)^{\mathrm{eff}}$ and set $\calE_V\coloneqq i_\CRIS^{0,\ast}\epsilon(V) \in D_{\phi}^{\mathrm{pcr}}(Y_1, M_{Y_1})^{\mathrm{nd}}_{\Q}$.
The following properties hold.
\begin{enumerate}
 \item 
 $V\cong (i_\CRIS^\ast \calE_V)(W(k))$ as $\phi$-modules in perfect $K_0$-complexes, where $W(k)$ comes with PD-ideal $(p)$ and log structure associated to $\N\rightarrow W(k)$, $1\mapsto 0$.
 \item $V\otimes_{K_0}^L S_K[p^{-1}]\cong \calE_V(S_K)$ as $\phi$-modules in perfect $S_K[p^{-1}]$-complexes;
 if V is a $(\phi,N)$-module so that $\calE_V$ is a quasi-coherent crystal and  $\calE_V(S_K)$ comes with a log connection $\nabla$, then via the above isomorphism, $\nabla$ corresponds to 
 \[
 \nabla(v\otimes x)=eN(v)\otimes x\,d\operatorname{log} u+v\otimes dx.
 \]
  \item Under (2), the map $S_K\rightarrow \calO_K$, $u\mapsto\pi$, yields
 \[
 \rho_\pi\colon V\otimes_{K_0}^LK\cong \calE_V(\calO_K).
 \]
as perfect $K$-complexes. 
For another uniformizer $\pi'=a\pi$ ($a\in \calO_K^\times$), we have 
\[
\rho_{\pi'}=\rho_{\pi}\circ \exp(\operatorname{log}(a) eN)\quad\text{with}\quad
\exp(\operatorname{log}(a) eN)\coloneqq \sum_{n\geq 0}\frac{(\operatorname{log}(a) e)^n}{n!}N^n.
\]
 \item $h_{\CRIS}^\ast \calE_V=\epsilon_{i^0\circ h}(V)$, and $(h_{\CRIS}^\ast \calE_V)(A_\cris)$ is a perfect complex in $B_\cris^+$-modules with commuting $\phi$ and $\Gal_K$-actions.

 \item  We have isomorphisms
 \[
 V\otimes_{K_0}^L\widehat{B}_\st^+\cong (h_{\CRIS}^\ast \calE_V)(\widehat{A}_\st)\cong (h_{\CRIS}^\ast \calE_V)(A_\cris)\otimes_{B_\cris^+}^L\widehat{B}_\st^+
 \]
 of perfect complexes in $\widehat{B}_\st^+$-modules compatible with $\phi$, $N$, and $\Gal(\overline{K}/K)$-actions, where $N$ on the LHS is $N\otimes 1+1\otimes N$ and $N$ on the RHS is $1\otimes N$.
 This induces an isomorphism of perfect complexes in $B_\st^+$-modules
 \[
\rho_{\st}\colon V\otimes_{K_0}^L B_\st^+\cong  (h_{\CRIS}^\ast \calE_V)(A_\cris)\otimes_{B_\cris^+}^L B_\st^+
 \]
 compatible with $\phi$, $N$, and $\Gal(\overline{K}/K)$-actions.
 \item We have isomorphisms
\[
\calE_V(\calO_K)\otimes_K^L A_{\logcris,K}[p^{-1}]\cong (h_{\CRIS}^\ast \calE_V)(A_{\logcris,K})\cong (h_{\CRIS}^\ast \calE_V)(A_\cris)\otimes_{B_\cris^+}^L A_{\logcris,K}[p^{-1}].
\]
  The map $S_K\rightarrow \calO_K$ sending $u$ to $\pi$ yields $B_\st^+\rightarrow \widehat{B}_\st^+\rightarrow A_{\logcris,K}[p^{-1}]$, which gives the commutative diagram
\[
\xymatrix{
V\otimes_{K_0}^L A_{\logcris,K}[p^{-1}]\ar[d]^-{\rho_\pi\otimes A_{\logcris,K}[p^{-1}]}\ar[r]^-{\overset{\rho_\st\otimes A_{\logcris,K}[p^{-1}]}{\phantom{a}}}& (h_{\CRIS}^\ast \calE_V)(A_\cris)\otimes_{B_\cris^+}^L A_{\logcris,K}[p^{-1}]\\
\calE_V(\calO_K)\otimes_K^L A_{\logcris,K}[p^{-1}].\ar[ur]&
}
\]
\end{enumerate}
\end{prop}

\begin{proof}
(1) This is \cite[\S~1.15, Rem.~(iv)]{Beilinson-crystalline-period-map-arXiv}.

(2) The first part follows from Dwork's trick applied to the $\phi$-equivariant surjection $S_K\rightarrow W$; see \cite[\S~1.13, Prop.]{Beilinson-crystalline-period-map-arXiv}.
The second part follows from \cite[\S~1.15, Thm.-Const.]{Beilinson-crystalline-period-map-arXiv} and its proof. To explain this, we freely use the notation therein. 
First we remark that $i_\CRIS^{0,\ast}\epsilon$ is also given by $(i_\CRIS^\ast)^{-1}\circ(i_\CRIS^\ast i_\CRIS^{0,\ast}\epsilon)$, where $(i_\CRIS^\ast)^{-1}$ is given by Dwork's trick for $F$-crystals (cf.~Proof (iii)(b) therein). We now unwind the latter description as it is compatible with the isomorphism in the first part.

Let $\overline{p}\in \Gamma(Y_1^0,M_{Y_1^0})$ be the image of $p\in M_\can$ on $W(k)$ and let $\overline{\pi}\in \Gamma(Y_0,M_{Y_0})$ be the image of $\pi\in M_\can$ on $\calO_K$. Write $(i^0\circ i)^{-1}(\overline{p})=a\overline{\pi}^e$ ($a\in k^\times$). Then $i_\CRIS^\ast i_\CRIS^{0,\ast}\epsilon$ is $\epsilon_{a\overline{\pi}^e}$ by construction, and $\epsilon_{a\overline{\pi}^e}(V,\phi,N)=\epsilon_{\overline{\pi}}(V,\phi,eN)$.

To describe $\epsilon_{\overline{\pi}}$ and $(i_\CRIS^\ast)^{-1}$, let $S_{K_0}=W(k)\{ s\}_\mathrm{PD}$ denote the $p$-adically completed PD-polynomial algebra over $W(k)$ with log structure given by $\N\rightarrow S_{K_0}$ sending $1$ to $s$. Then the $F$-isocrystal $\calE\coloneqq \epsilon_{\overline{\pi}}(V,\phi,eN)$ is given by an $S_{K_0}[p^{-1}]$-module $V\otimes_{K_0}S_{K_0}[p^{-1}]$ with Frobenius $\phi\otimes \phi_{S_{K_0}[p^{-1}]}$ and topologically nilpotent log connection $\nabla$ given by $\nabla(v\otimes x)=eN(v)\otimes x\,d\operatorname{log} s+v\otimes dx$ (cf.~Proof (iii)(a)).
Now take $m\gg0$ so that we have a map $\phi_m\colon S_{K_0}\rightarrow S_K$ sending $s$ to $u^{p^m}$. Then the log connection $\phi_m^{\ast}(V\otimes_{K_0}S_{K_0}[p^{-1}],\nabla)$ on $S_K[p^{-1}]$ represents the $F$-isocrystal $(F_{Y_1,\CRIS}^\ast)^m ((i_\CRIS^\ast)^{-1}\calE)$ whose underlying $S_K[p^{-1}]$-module given by \cite[\S~1.13, Prop.]{Beilinson-crystalline-period-map-arXiv} is $(\phi_{K_0}^\ast)^mV\otimes_{K_0}S_K[p^{-1}]$. Here $F_{Y_1,\CRIS}^\ast$ corresponds to the pullback along $\phi_{S_K[p^{-1}]}$ on log connections. We deduce the second assertion from  $\phi_m^{\ast}d\operatorname{\log}s=p^md\operatorname{\log}u$ and $(\phi_{S_K[p^{-1}]}^m)^{\ast}d\operatorname{\log}u=p^md\operatorname{\log}u$.

(3) The first part follows from the crystal property of $\calE_V$. For the second part, we may assume that $V$ is an effective $(\phi,N)$-module and $\calE_V$ is an $F$-isocrystal.
In this case, the formula follows from (2) and a standard computation involving the self-product $S_K^{(1)}$ of $S_K$; see \cite[Thm.~5.1]{hyodo-kato} or \cite[Rem.~4.4.18]{Tsuji-Cst} for a relevant computation.

(4) This is obvious.

(5) We may assume that $V$ is an effective $(\phi,N)$-module by Remark~\ref{rem:Beilinson-epsilon-functor}. The isomorphisms in the first part are given by the crystal property of $h_\CRIS^\ast\calE$ and (2). By construction, these isomorphisms are compatible with $\phi$ and $\Gal(\overline{K}/K)$-actions. So it remains to show that they are also compatible with monodromy operators.
From Lemma~\ref{lem:crystalline-interpretation-Asthat}(2) and a similar argument, we see that $\widehat{A}_\st$ is a weakly initial ind-object in $(\overline{Y}_1,M_{\overline{Y}_1})_\CRIS^\op$ and its self-coproduct is given by $(h_{\CRIS,\ast}\calO_{\overline{Y}_1/\Z_p})(S_K^{(1)})$.
We see from these and a standard argument that the category of quasi-coherent $\calO_{\overline{Y}_1/\Z_p}$-modules is equivalent to the category consisting of pairs $(E,\nabla)$ where $E$ is a $p$-adically complete $\widehat{A}_\st$-module and $\nabla\colon E\rightarrow E\,d\operatorname{log}u$ is a topologically nilpotent log connection compatible with $\nabla_{\widehat{A}_\st}$. Let $\nabla$ denote the log connection on the $\widehat{B}_\st^+$-module $(h_{\CRIS}^\ast \calE_V)(\widehat{A}_\st)$ given by the isocrystal $h_{\CRIS}^\ast \calE_V$ and define an endomorphism $\calN$ on $(h_{\CRIS}^\ast \calE_V)(\widehat{A}_\st)$ by $\nabla(v)=e\calN(v)\,d\operatorname{log}u$. Since $\nabla$ is the pullback of the log connection in (2), we see that the identification $ V\otimes_{K_0}\widehat{B}_\st^+\cong (h_{\CRIS}^\ast \calE_V)(\widehat{A}_\st)$ sends $N\otimes 1+1\otimes N$ to $\calN$. Since $A_\cris$ is the initial ind-object in $(\overline{Y}_1,M_{\overline{Y}_1})_\CRIS^\op$ (Lemma~\ref{lem:crystalline-interpretation-Asthat}(1)), we deduce that the identification $(h_{\CRIS}^\ast \calE_V)(\widehat{A}_\st)\cong (h_{\CRIS}^\ast \calE_V)(A_\cris)\otimes_{B_\cris^+}\widehat{B}_\st^+$ sends $\calN$ to $1\otimes N$. The second part follows from the first, Proposition~\ref{prop:Bst-vs-Asthat}, and the nilpotence of $N$ on $V$.

(6) The isomorphisms in the first part are given by the crystal property of $h_\CRIS^\ast\calE$ and (3). The second assertion follows from the construction in (2), (3), (5) using the crystal property and Dwork's trick, and the following obvious commutative diagram
\[
\xymatrix{
W(k)\ar@<.5ex>[r]& S_K \ar[r]\ar@<.5ex>[d]\ar[l] & \widehat{A}_\st\ar[d] & A_\cris\ar[l]\ar[dl]\\
& \calO_K\ar[r] & A_{\logcris,K}.&
}
\]
\end{proof}

\begin{rem}\label{rem:another-definition-of-N-in-Beilinson-Hyodo-Kato}
Part (2) of Proposition~\ref{prop:Beilinson-epsilon-functor-properties} is generalized as follows. Let $p_0,p_1\colon S_K\rightarrow S_K^{(1)}$ denote the projections. Then $p_1 \colon S_K\rightarrow S_K^{(1)}$ extends to an $S_K$-algebra isomorphism $S_K\{v\}_\mathrm{PD}\xrightarrow{\cong}S_K^{(1)}$ where $S_K\{v\}_\mathrm{PD}$ denotes the $p$-completed PD-polynomial algebra in $v$, and the map is defined by sending $v$ to $p_0(u)/p_1(u)-1$. Composing the inverse of this isomorphism with the $S_K$-linear projection $S_K\{v\}_\mathrm{PD}=\widehat{\oplus}_{n\geq 0}S_K v^{[n]}\rightarrow S_K v^{[1]}=S_K$ and inverting $p$, we obtain an $S_K[p^{-1}]$-linear map $f\colon S_K^{(1)}[p^{-1}]\rightarrow S_K[p^{-1}]$.

Let $V$ and $\calE_V$ be as in Proposition~\ref{prop:Beilinson-epsilon-functor-properties}. Consider the $K_0$-linear map
\[
e\calN\colon \calE_V(S_K)\rightarrow \calE_V(S_K)\otimes_{S_K,p_0}^LS_K^{(1)}\cong \calE_V(S_K^{(1)})\cong \calE_V(S_K)\otimes_{S_K,p_1}^LS_K^{(1)}\rightarrow \calE_V(S_K),
\]
where the first map (resp.~the last map) is induced from $p_0$ (resp.~$f$).
Under the isomorphism $V\cong \calE_V(S_K)\otimes_{S_K[p^{-1}]}^LK_0$ given by (2), the monodromy operator $N$ on $V$ corresponds to $\calN\otimes \id_{K_0}$; for this, we may assume that $V$ is an effective $(\phi,N)$-module, in which case the claim follows from the second part of (2) and the fact that $d\operatorname{log} u$ corresponds to the image of $p_0(u)/p_1(u)-1$ in $J/J^{[2]}$ where $J=\Ker(S_K^{(1)}\rightarrow S_K)$.
\end{rem}

\subsection{The Hyodo--Kato cohomology} \label{subsec: hyodo-kato-cohom}

In this subsection, we apply the Hyodo--Kato theory in \S~\ref{subsec: hyodo-kato-theory} to the cohomology of the crystalline realization of a completed prismatic $F$-crystal. For this, we first need to take a detour.

Let $(Y_1, M_{Y_1})\coloneqq (\Spf \calO_K, M_{\can})\times_{\Z} \Z/p$ as in \S~\ref{subsec: hyodo-kato-theory}. Write $\calD_{\perf}((Y_1, M_{Y_1})_{\CRIS})$ for the derived $\infty$-category of perfect complexes of $\calO_{Y_1}/\Z_p$-modules, i.e., the full subcategory of $\calD((Y_1, M_{Y_1})_{\CRIS}, \calO_{Y_1}/\Z_p)$ consisting of perfect complexes. Let $\mathcal{D}_{\perf}^{\phi}((Y_1, M_{Y_1})_{\CRIS})$ be the $\infty$-category of pairs $(\mathcal{K}, \phi_{\mathcal{K}})$ where $\mathcal{K} \in \mathcal{D}_{\perf}((Y_1, M_{Y_1})_{\CRIS})$ and $\phi_{\mathcal{K}}\colon LF_{Y_1, \CRIS}^*\mathcal{K} \stackrel{\cong}{\rightarrow} \mathcal{K}$ is an isomorphism in $\mathcal{D}_{\perf}((Y_1, M_{Y_1})_{\CRIS})_{\mathbf{Q}}\coloneqq \calD_{\perf}((Y_1, M_{Y_1})_{\CRIS})\otimes_{\calD(\Z)} \calD(\Q)$. 
Note that the homotopy category of $\mathcal{D}_{\perf}((Y_1, M_{Y_1})_{\CRIS})$ is identified with $D_{\phi}^{\mathrm{pcr}}(Z, M_Z)$ in \S~\ref{subsec: hyodo-kato-theory}.
By $p$-completely faithfully flat descent, we have a natural equivalence
\begin{equation} \label{eq: F-isocrys-perfect-cx-as-lim}
\cal{D}_{\perf}^{(\phi)}((Y_1, M_{Y_1})_{\CRIS})\cong \lim_{[n] \in \Delta} \calD_{\perf}^{(\phi)}(S_K^{[n]}),
\end{equation}
where $S_K^{[n]}$ is the $(n+1)$-st self-coproduct of the Breuil log prism (cf.~Theorem~\ref{thm: general-base-change-BK-cohom}).

Let $(X, M_X)$ be a \emph{proper} semistable $p$-adic formal scheme over $\Spf \calO_K$. 
 and set $(X_i, M_{X_i})\coloneqq (X, M_{X})\times_{(\calY, M_{\calY})} (Y_i, M_{Y_i})$ for $i=0,1$. Let $\calE_{\Prism} \in \Vect^{\an, \phi}((X,M_X)_\Prism)$ be an analytic prismatic $F$-crystal on $(X, M_X)_{\Prism}$ and write $(\calE_{\cris,\Q},\phi_{\calE_{\cris,\Q}})$ for the crystalline realization of $j_{\ast}\calE_{\Prism}$ as in Definition~\ref{def:crystalline-realization}.

\begin{construction} \label{const: F-isocrys-perfect-cx}
Keep the set-up as above. Recall that $R\Gamma_{\Br}(X, i(\calE_{\Prism}))$ satisfies the following properties:
\begin{itemize}
  \item $R\Gamma_{\Br}(X, i(\calE_{\Prism}))$ lies in $\calD_{\perf}(S_K)$ and satisfies the Frobenius isogeny property over $S_K[p^{-1}]$ (Theorem~\ref{thm: BK-cohom-perfect-cx}, Proposition~\ref{prop: BK-cohom}, Theorems~\ref{thm: general-base-change-BK-cohom}(1) and \ref{thm: pris-cris-comparison-over-Breuil S});
  \item for any projection map $S_K \rightarrow S_K^{[n]}$, the $\phi$-equivariant map
\[
R\Gamma_{\Br}(X, i(\calE_{\Prism}))\otimes^L_{S_K} S_K^{[n]} \rightarrow R\Gamma((X_{S_K^{[n]}/p}/(S_K^{[n]}, M_{\Spf S_K^{[n]}}))_{\Prism}, i(\calE_{\Prism}))
\]
is an isomorphism (Theorem~\ref{thm: general-base-change-BK-cohom}(1)).
\end{itemize}

Via the equivalence \eqref{eq: F-isocrys-perfect-cx-as-lim}, this gives an $F$-isocrystal in perfect complexes $(\mathcal{K}_\cris\coloneqq \mathcal{K}_\cris(i(\calE_{\Prism})),\phi_{\calK_\cris})\in \cal{D}_{\perf}^{\phi}((Y_1, M_{Y_1})_{\CRIS})$ satisfying 
\[
\mathcal{K}_\cris(i(\calE_{\Prism}))(S_K^{[n]}) \cong R\Gamma((X_{S_K^{[n]}/p}/(S_K^{[n]}, M_{\Spf S_K^{[n]}}))_{\Prism}, i(\calE_{\Prism})).
\]
By unwinding the construction, we see that there exists $m\geq 0$ such that $(\calK_\cris,p^m\phi_{\calK_\cris})$ lies in $D_{\phi}^{\mathrm{pcr}}(Z, M_Z)^{\mathrm{nd}}\subset D_{\phi}^{\mathrm{pcr}}(Z, M_Z)$.
So by applying Theorem~\ref{thm: Beilison-HK-equiv} to $(\calK_\cris,p^m\phi_{\calK_\cris})_\Q\in D_{\phi}^{\mathrm{pcr}}(Z, M_Z)^{\mathrm{nd}}_\Q$ and twisting back the Frobenius by $p^{-m}$, we obtain $(V,\phi,N)\in D_{\phi, N}(K_0)$.
\end{construction}

\begin{thm}\label{thm:Hyodo-Kato-cohomology}
The above $(V,\phi,N)\in D_{\phi, N}(K_0)$ satisfies the following properties.
\begin{enumerate}
 \item As $\phi$-modules in perfect $K_0$-complexes,
 we have
 \[
 V\cong R\Gamma(((X_0, M_{X_0})/(W(k), M_0))_{\CRIS}, \calE_{\cris, \Q}).
 \] 
 \item As $\phi$-modules in perfect $S_K[p^{-1}]$-complexes, we have
 \[
V\otimes_{K_0}^L S_K[p^{-1}] \cong R\Gamma(((X_1, M_{X_1})/(S_K, M_{S_K}))_{\CRIS}, \calE_{\cris, \Q}).
 \]
  \item Each uniformizer $\pi\in\calO_K$ yields an isomorphism of perfect $K$-complexes
 \[
 \rho_\pi\colon V\otimes_{K_0}^LK\cong R\Gamma(X, (E\otimes_{\calO_X}\omega^{\bullet}_{X/\calO_K}, \nabla_E))\otimes_{\calO_K}^L K,
 \]
 where $(E, \nabla_E)$ is the integrable log connection on $X$ associated to $\calE_{\cris}$ in Definition~\ref{def:crystalline-realization}.
For $\pi'=a\pi$ ($a\in \calO_K^\times$), we have $\rho_{\pi'}=\rho_{\pi}\circ \exp(\operatorname{log}(a) eN)$.
\item  We have an isomorphism of perfect complexes in $B_\st^+$-modules
 \[
\rho_{\st}\colon V\otimes_{K_0}^L B_\st^+\cong  R\Gamma(((X_{\calO_C/p}, M_{X_{\calO_C/p}})/(A_{\cris}, M_{A_\cris}))_{\CRIS}, \calE_{\cris, \Q})\otimes_{B_\cris^+}^L B_\st^+
 \]
 compatible with $\phi$, $N$, and $\Gal(\overline{K}/K)$-actions.
 \item We have isomorphisms
\begin{align*}
R\Gamma&(X, (E\otimes_{\calO_X}\omega^{\bullet}_{X/\calO_K}, \nabla_E))\otimes_{\calO_K}^L A_{\logcris,K}[p^{-1}]\\
&\cong
R\Gamma(((X_{\calO_C/p}, M_{X_{\calO_C/p}})/(A_{\logcris, K}, M_{A_{\logcris, K}}))_{\CRIS}, \calE_{\cris, \Q})\\
&\cong R\Gamma(((X_{\calO_C/p}, M_{X_{\calO_C/p}})/(A_{\cris}, M_{A_\cris}))_{\CRIS}, \calE_{\cris, \Q})\otimes_{B_\cris^+}^L A_{\logcris,K}[p^{-1}].    
\end{align*}
 The map $S_K\rightarrow \calO_K$ sending $u$ to $\pi$ yields $B_\st^+\rightarrow \widehat{B}_\st^+\rightarrow A_{\logcris,K}[p^{-1}]$, which identifies the above composite with $(\rho_\st\otimes A_{\logcris,K}[p^{-1}])\circ (\rho_\pi^{-1}\otimes A_{\logcris,K}[p^{-1}])$.
\end{enumerate}
We write $R\Gamma_{\HK}((X_0, M_{X_0}), \calE_{\cris, \Q}))$ for $(V,\phi,N)$ and call it the \emph{Hyodo--Kato complex} associated to $\calE_{\cris, \Q}$.
\end{thm}

We refer the reader to Theorem~\ref{thm: pris-cris-comparison-over-Breuil S} (and the paragraph that follows) and Proposition~\ref{prop: base-change-cris-cohom-Acris} below for the notation in the RHS in the above statements.

\begin{proof}
We use Proposition~\ref{prop:Beilinson-epsilon-functor-properties}.
Part (2) is given by the isomorphisms
 \begin{align*}
V\otimes_{K_0}^L S_K[p^{-1}]&\cong R\Gamma_{\Br}(X, i(\calE_{\Prism}))\otimes_{S_K}^LS_K[p^{-1}]\cong R\Gamma_{\Br}(X, j_\ast\calE_{\Prism})\otimes_{S_K}^LS_K[p^{-1}]\\
&\cong R\Gamma(((X_1, M_{X_1})/(S_K, M_{S_K}))_{\CRIS}, \calE_{\cris, \Q})
 \end{align*}
where the first isomorphism follows from the construction and Proposition~\ref{prop:Beilinson-epsilon-functor-properties}(2), the second from Theorem~\ref{thm: general-base-change-BK-cohom} ($A=S_K$), and the third from Theorem~\ref{thm: pris-cris-comparison-over-Breuil S}. 
Using (2), part (1) also follows from Proposition~\ref{prop:Beilinson-epsilon-functor-properties}(1), Theorems~\ref{thm: general-base-change-BK-cohom} ($A=W(k)$) and \ref{thm: pris-cris-comparison-over-Breuil S}(1).
Similarly, for $h\colon (\overline{Y}_1,M_{\overline{Y}_1})\coloneqq(\Spec \calO_C/p,(\calO_{C}\smallsetminus\{0\})^a)\rightarrow (Y_1,M_{Y_1})$, part (2), Theorems~\ref{thm: general-base-change-BK-cohom} ($A=A_\cris$) and \ref{thm: pris-cris-comparison-over-Breuil S}(3) give an isomorphism of perfect complex in $B_\cris^+$-modules $(h_\CRIS^\ast \calK_\cris)(A_\cris)\cong R\Gamma(((X_{\calO_C/p}, M_{X_{\calO_C/p}})/(A_{\cris}, M_{A_\cris}))_{\CRIS}, \calE_{\cris, \Q})$; here we have used the identification given by \cite[Prop.~4.4]{du-moon-shimizu-cris-pushforward}:
\[
(h_\CRIS^\ast \calK_\cris)(A_\cris,(\calO_{C^\flat}\smallsetminus\{0\}\rightarrow A_\cris)^a)=\calK_\cris(A_\cris,(\N\rightarrow A_\cris, 1\mapsto [\pi^\flat])^a).
\]
Hence (4) follows from Proposition~\ref{prop:Beilinson-epsilon-functor-properties}(5).
Part (3) follows from (2), Theorem~\ref{thm: log-dR-cris-comparison}, Propositions~\ref{prop:Beilinson-epsilon-functor-properties}(3) and \ref{prop: base-change-cris-cohom-Acris}(2) below. 
Part (5) follows from (2), Propositions~\ref{prop:Beilinson-epsilon-functor-properties}(6) and \ref{prop: base-change-cris-cohom-Acris}(1), (3) below. 
\end{proof}

\begin{prop} \label{prop: base-change-cris-cohom-Acris}
Keep the set-up as above. 
\begin{enumerate}
    \item With the notation as in Definition~\ref{def:AlogcrisK-and-Asthat}(2), the maps
    \begin{align*}
     R\Gamma(((X_1, M_{X_1})&/(S_K, M_{S_K}))_{\CRIS}, \calE_{\cris, \Q})\otimes_{S_K[p^{-1}]}^L A_{\logcris, K}[p^{-1}] \\
    &  \rightarrow R\Gamma(((X_{\calO_C/p}, M_{X_{\calO_C/p}})/(A_{\logcris, K}, M_{A_{\logcris, K}}))_{\CRIS}, \calE_{\cris, \Q})
    \end{align*}
    induced by $\iota_1, \iota_2\colon S_K \rightarrow A_{\logcris, K}$ are isomorphisms in $D(A_{\logcris, K}[p^{-1}])$.

    \item The map
    \begin{align*}
     R\Gamma(((X_1, M_{X_1})&/(S_K, M_{S_K}))_{\CRIS}, \calE_{\cris, \Q})\otimes_{S_K[p^{-1}]}^L K \\
    &  \rightarrow R\Gamma(((X_1, M_{X_1})/(\calO_K, M_\can))_{\CRIS}, \calE_{\cris, \Q})
    \end{align*}
    given by $S_{K} \rightarrow \calO_K$ ($u \mapsto \pi$) is an isomorphism in $D(K)$.

    \item The base change
    \begin{align*}
    & R\Gamma(((X_{\calO_C/p}, M_{X_{\calO_C/p}})/(A_{\cris}, M_{A_\cris}))_{\CRIS}, \calE_{\cris, \Q})\otimes_{B_{\cris}^+}^L A_{\logcris, K}[p^{-1}] \\
    & ~~ \rightarrow R\Gamma(((X_{\calO_C/p}, M_{X_{\calO_C/p}})/(A_{\logcris, K}, M_{A_{\logcris, K}}))_{\CRIS}, \calE_{\cris, \Q})
    \end{align*}
    is an isomorphism in $D(A_{\logcris, K}[p^{-1}])$.
\end{enumerate}
\end{prop}

\begin{proof}
The assertions are formulated in terms of the crystalline site in the sense of Definition~\ref{def: Koshikawa-crystalline-site}, but we use the one in Definition~\ref{def:DMS-crystalline-site} in the proof via Proposition~\ref{prop:comparison of crystals in two crystalline sites}, as in Proposition~\ref{prop:comparison-of-cohomology-of-crystalline-realization}.

We use the \v{C}ech nerve $(X_1, \Spf S_{\Lambda}^{\rel, [\bullet]})$ and the isomorphisms
\[
i(\calE_{\Prism})(S_{\Lambda}^{\rel, [\bullet]})\otimes_{S_{\Lambda}^{\rel, [\bullet]}}^L S_{\Lambda}^{\rel, [\bullet]}[p^{-1}] \cong j_{\ast}\calE_{\Prism}(S_{\Lambda}^{\rel, [\bullet]})\otimes_{S_{\Lambda}^{\rel, [\bullet]}}^L S_{\Lambda}^{\rel, [\bullet]}[p^{-1}] \cong  \calE_{\cris, \Q}(S_{\Lambda}^{\rel, [\bullet]})
\]
in Construction~\ref{const: Cech-nerve-cristalline-site-over-SK} and the proof of Proposition~\ref{prop:comparison-of-cohomology-of-crystalline-realization}.

(1) We have
\[
R\Gamma(((X_1, M_{X_1})/(\Spf S_K)^\sharp)_{\CRIS}, \calE_{\cris, \Q}) \cong \Tot(\calE_{\cris, \Q}(S_{\Lambda}^{\rel, [\bullet]})).
\]
Similarly, considering the \v{C}ech nerve on $((X_{\calO_C/p}, M_{X_{\calO_C/p}})/(\Spf A_{\logcris, K})^\sharp)_\CRIS$ given by $(X_{\calO_C/P}, \Spf S_{\Lambda}^{\rel, [\bullet]}\widehat{\otimes}_{S_K, \iota_1} A_{\logcris, K})$, we have an isomorphism
\[
R\Gamma(((X_{\calO_C/p}, M_{X_{\calO_C/p}})/(\Spf A_{\logcris, K})^\sharp)_{\CRIS}, \calE_{\cris, \Q}) \cong \Tot(\calE_{\cris, \Q}(S_{\Lambda}^{\rel, [\bullet]}\widehat{\otimes}_{S_K, \iota_1} A_{\logcris, K})).
\]
Since $\iota_1\colon S_K \rightarrow A_{\logcris, K}$ has finite $p$-complete Tor amplitude by Lemmas~\ref{lem: finite-p-compl-Tor-amp} and \ref{lem: Acris-to-AstK-finite-Tor-amp} below, the base change given by $\iota_1$
\begin{align*}
     R\Gamma(((X_1, M_{X_1})&/(\Spf S_K)^\sharp)_{\CRIS}, \calE_{\cris, \Q})\otimes_{S_K, \iota_1}^L A_{\logcris, K}[p^{-1}] \\
    &  \rightarrow R\Gamma(((X_{\calO_C/p}, M_{X_{\calO_C/p}})/(\Spf A_{\logcris, K})^\sharp)_{\CRIS}, \calE_{\cris, \Q})
\end{align*}
is an isomorphism; use Lemma~\ref{lem:algebraic-lemmas-for-base-change} as in the proof of Theorem~\ref{thm: general-base-change-BK-cohom}. This gives (1) for $\iota_1$.

For $\iota_2$, recall the $F$-isocrystal in perfect complexes $\calK_\cris \in \cal{D}_{\perf}^{\phi}((Y_1, M_{Y_1})_{\CRIS})$ given in Construction~\ref{const: F-isocrys-perfect-cx}; we have
\[
\calK_\cris(S_K)\otimes_{S_K, \iota_1}^L A_{\logcris, K} \cong \calK_\cris(A_{\logcris, K}) \cong \calK_\cris(S_K)\otimes_{S_K, \iota_2}^L A_{\logcris, K}.
\]
Since $\calK_\cris(S_K)\otimes_{S_K}^L S_K[p^{-1}] \cong R\Gamma(((X_1, M_{X_1})/(\Spf S_K)^\sharp)_{\CRIS}, \calE_{\cris, \Q})$ by Proposition~\ref{prop: BK-cohom} and Theorems~\ref{thm: general-base-change-BK-cohom} and \ref{thm: pris-cris-comparison-over-Breuil S}, this shows (1) for $\iota_2$.

(2) We have
\[
R\Gamma(((X_1, M_{X_1})/(\Spf \calO_K)^\sharp)_{\CRIS}, \calE_{\cris}) \cong \Tot(\calE_{\cris}(S_{\Lambda}^{\rel, [\bullet]}\widehat{\otimes}_{S_K}\calO_K)).
\]
Since $\calO_K \rightarrow A_{\logcris, K}$ is flat, the base change
\begin{align*}
    (R\Gamma(((X_1, M_{X_1})&/(\Spf \calO_K)^\sharp)_{\CRIS}, \calE_{\cris})\widehat{\otimes}_{\calO_K}^L A_{\logcris, K})\otimes_{A_{\logcris, K}}^L A_{\logcris, K}[p^{-1}] \\
    &  \rightarrow R\Gamma(((X_{\calO_C/p}, M_{X_{\calO_C/p}})/(\Spf A_{\logcris, K})^\sharp)_{\CRIS}, \calE_{\cris, \Q})
\end{align*}
is an isomorphism by Lemma~\ref{lem:algebraic-lemmas-for-base-change} similarly as above. Note that 
\[
R\Gamma(((X_1, M_{X_1})/(\Spf\calO_K)^\sharp)_{\CRIS}, \calE_{\cris})\otimes_{\calO_K}^L A_{\logcris, K}
\]
is derived $p$-complete as $R\Gamma(((X_1, M_{X_1})/(\Spf \calO_K)^\sharp)_{\CRIS}, \calE_{\cris}) \in D_{\perf}(\calO_K)$ (Theorem~\ref{thm: log-dR-cris-comparison}) and $A_{\logcris, K}$ is $p$-complete. Thus, $\widehat{\otimes}_{\calO_K}^L=\otimes_{\calO_K}^L$ in the above, and the map
\begin{align*} 
    R\Gamma(((X_1, M_{X_1})/&(\Spf\calO_K)^\sharp)_{\CRIS}, \calE_{\cris, \Q})\otimes_{\calO_K}^L A_{\logcris, K}[p^{-1}] \\
    &   \rightarrow R\Gamma(((X_{\calO_C/p}, M_{X_{\calO_C/p}})/(\Spf A_{\logcris, K})^\sharp)_{\CRIS}, \calE_{\cris, \Q})
\end{align*}
is an isomorphism.

Now, consider the base change
\begin{align*}
     R\Gamma(((X_1, M_{X_1})/(\Spf S_K)^\sharp)_{\CRIS}, &\calE_{\cris, \Q})\otimes_{S_K}^L K \\
    &  \rightarrow R\Gamma(((X_1, M_{X_1})/(\Spf\calO_K)^\sharp)_{\CRIS}, \calE_{\cris, \Q}).
\end{align*}
This becomes an isomorphism after $-\otimes_K^L A_{\logcris, K}[p^{-1}]$ by (1) for $\iota_2$ and the above. Since $\calO_K \rightarrow A_{\logcris, K}$ is classically faithful flat, we deduce (2).

(3) Note that
\[
R\Gamma(((X_{\calO_C/p}, M_{X_{\calO_C/p}})/(\Spf A_{\cris})^\sharp)_{\CRIS}, \calE_{\cris, \Q}) \cong \Tot(\calE_{\cris, \Q}(S_{\Lambda}^{\rel, [\bullet]}\widehat{\otimes}_{S_K} A_{\cris})).
\]
Since $A_{\cris} \rightarrow  A_{\logcris, K}$ has finite $p$-complete Tor amplitude by Lemma~\ref{lem: Acris-to-AstK-finite-Tor-amp} below, (3) follows from Lemma~\ref{lem:algebraic-lemmas-for-base-change} similarly as in (1).
\end{proof}

\begin{lem} \label{lem: Acris-to-AstK-finite-Tor-amp}
The map $A_{\cris} \rightarrow A_{\logcris, K}$ has finite $p$-complete Tor amplitude.    
\end{lem}

\begin{proof}
First consider $(S_K\otimes_{W(k)}\calO_K)[\frac{v}{\pi}]\coloneqq (S_K\otimes_{W(k)}\calO_K)[x]/(\pi x -v)$ where $v\coloneqq u-\pi$. Then the map $S_K \rightarrow (S_K\otimes_{W(k)}\calO_K)[\frac{v}{\pi}]$ has finite $p$-complete Tor amplitude. Let $(S_K\otimes_{W(k)}\calO_K)\{ \frac{v}{\pi} \}_{\PD}$ be the $p$-completed PD-envelope of the surjection $(S_K\otimes_{W(k)}\calO_K)[\frac{v}{\pi}]\rightarrow \calO_K$ sending $u$ to $\pi$ and $\frac{v}{\pi}$ to $0$. We claim that the sequence
\[
0\rightarrow
(S_K\otimes_{W(k)}\calO_K\bigl[\frac{v}{\pi}\bigr])\{ T\}_{\PD} \rightarrow (S_K\otimes_{W(k)}\calO_K\bigl[\frac{v}{\pi}\bigr])\{ T\}_{\PD} \rightarrow (S_K\otimes_{W(k)}\calO_K)\bigl\{ \frac{v}{\pi} \bigr\}_{\PD}    
\rightarrow 0
\]
is exact, 
where $(S_K\otimes_{W(k)}\calO_K[\frac{v}{\pi}])\{ T\}_{\PD}$ denotes the $p$-completed PD-polynomial algebra in $T$ over $(S_K\otimes_{W(k)}\calO_K)[\frac{v}{\pi}]$, the first injective map is the multiplication by $T-\frac{v}{\pi}$, and the second surjective map is given by $T^{[n]} \mapsto (\frac{v}{\pi})^{[n]}$; the exactness in the middle can be checked by realizing $(S_K\otimes_{W(k)}\calO_K[\frac{v}{\pi}])\{ T\}_{\PD}$ and $(S_K\otimes_{W(k)}\calO_K)\{ \frac{v}{\pi} \}_{\PD}$ as subrings of $K[\![v, T]\!]$ and $K[\![v]\!]$, respectively.
 Since $(S_K\otimes_{W(k)}\calO_K[\frac{v}{\pi}])\{ T\}_{\PD}$ is $p$-completely flat over $(S_K\otimes_{W(k)}\calO_K[\frac{v}{\pi}])$, we deduce that $S_K \rightarrow (S_K\otimes_{W(k)}\calO_K)\{ \frac{v}{\pi} \}_{\PD}$ has finite $p$-complete Tor amplitude.

Now, for each $n \geq 1$, the map $A_{\cris}/p^n \rightarrow A_{\logcris, K}/p^n$ is obtained from $S_K/p^n \rightarrow (S_K\otimes_{W(k)}\calO_K)\{ \frac{v}{\pi} \}_{\PD}/p^n$ by the base change $S_K/p^n \rightarrow A_{\cris}/p^n$ which is classically flat (Lemma~\ref{lem: finite-p-compl-Tor-amp}). Thus, $A_{\cris} \rightarrow A_{\logcris, K}$ has finite $p$-complete Tor amplitude. 
\end{proof}

\begin{rem}
Using Remark~\ref{rem:another-definition-of-N-in-Beilinson-Hyodo-Kato}, we can describe the monodromy operator $N$ on $R\Gamma_{\HK}((X_0, M_{X_0}), \calE_{\cris, \Q}))$ as follows.
As in Theorem~\ref{thm:Hyodo-Kato-cohomology}(2), 
we have an isomorphism
\[
\mathcal{K}_\cris(i(\calE_{\Prism}))(S_K^{(1)})\cong R\Gamma(((X_1, M_{X_1})/(S_K^{(1)}, M_{S_K^{(1)}}))_{\CRIS}, \calE_{\cris, \Q})
\]
by Theorems~\ref{thm: general-base-change-BK-cohom} ($A=S_K^{[1]}=S_K^{(1)}$) and \ref{thm: pris-cris-comparison-over-Breuil S}(2), and this isomorphism is compatible with $\mathcal{K}_\cris(i(\calE_{\Prism}))(S_K) \cong R\Gamma(((X_1, M_{X_1})/(S_K, M_{S_K}))_{\CRIS}, \calE_{\cris, \Q})$ under the base change along $p_0,p_1\colon S_K\rightarrow S_K^{(1)}$.
Define the $K_0$-linear map $e\calN$ as the composite
\begin{align*}
R\Gamma((X_1/S_K)_{\CRIS}, \calE_{\cris, \Q})
&\rightarrow R\Gamma((X_1/S_K)_{\CRIS}, \calE_{\cris, \Q})\otimes_{S_K[p^{-1}],p_0}^LS_K^{(1)}[p^{-1}]\\
\cong R\Gamma((X_1/S_K^{(1)})_{\CRIS}, \calE_{\cris, \Q})
&\cong R\Gamma((X_1/S_K)_{\CRIS}, \calE_{\cris, \Q})\otimes_{S_K[p^{-1}],p_1}^LS_K^{(1)}[p^{-1}]\\
&\rightarrow R\Gamma((X_1/S_K)_{\CRIS}, \calE_{\cris, \Q})
\end{align*}
as in Remark~\ref{rem:another-definition-of-N-in-Beilinson-Hyodo-Kato} (with the obvious simplification of notation).
Via the identification $R\Gamma_{\HK}((X_0, M_{X_0}), \calE_{\cris, \Q}))\cong R\Gamma((X_1/S_K)_{\CRIS}, \calE_{\cris, \Q})\otimes_{S_K[p^{-1}]}^LK_0$, $N$ corresponds to $\calN\otimes \id_{K_0}$.
When $j_\ast\calE_\Prism=\calO_\Prism$, our $N$ is consistent with those in \cite[(3.5)]{hyodo-kato} and \cite[\S~4.3]{Tsuji-Cst} up to the multiplication-by-$e$ (see the last paragraph of \cite[\S~4.3]{Tsuji-Cst}).
\end{rem}

\begin{rem}\label{rem:deRham-convergent}
Let $(E,\nabla_E)$ be a finitely generated $\calO_X$-module with log integrable connection on $X$ corresponding to the finitely generated $p$-adically completed crystal $\calE_\cris$ on $((X_1, M_{X_1})/(\calO_K, M_\can))_{\CRIS}$ by \cite[Prop.~6.18]{du-moon-shimizu-cris-pushforward} and Proposition~\ref{prop:comparison of crystals in two crystalline sites}, and write $(E_\Q,\nabla_{E_\Q})$ for the associated finite locally free isocoherent sheaf with integrable connection  on $X$ (see \cite[Def.~B.12]{du-liu-moon-shimizu-purity-F-crystal}).
Theorem~\ref{thm: log-dR-cris-comparison} gives the isomorphism 
\[
R\Gamma(((X_1, M_{X_1})/(\calO_K, M_\can))_{\CRIS}, \calE_{\cris,\Q}) \cong R\Gamma(X, (E_\Q\otimes_{\calO_X}\omega^{\bullet}_{X/\calO_K}, \nabla_{E_\Q})).
\]

Let $((X,M_X)/(\calO_K,M_\can))_\conv$ be the log convergent site as in \cite[Def.~B.27]{du-liu-moon-shimizu-purity-F-crystal}. 
By \cite[Prop.~B.32, B.30]{du-liu-moon-shimizu-purity-F-crystal}, we associate to the finite locally free $F$-isocrystal
$(\calE_{\cris, \Q},\phi_{\calE_{\cris, \Q}})$ on $(X_1, M_{X_1})_{\CRIS}$ an isocrystal $\calE_\conv$ on $((X,M_X)/(\calO_K,M_\can))_\conv$ and a vector bundle with integrable connection $(\mathbf{E}, \nabla_{\mathbf{E}})$ on the generic fiber $X_{K}$. We see from \cite[Rem.~B.33]{du-liu-moon-shimizu-purity-F-crystal} that $(\mathbf{E}, \nabla_{\mathbf{E}})$ is naturally identified with  $(E_\Q,\nabla_{E_\Q})$.
We also have isomorphisms
\begin{align*}
R\Gamma(X, (E_\Q\otimes_{\calO_X}\omega^{\bullet}_{X/\calO_K}, \nabla_{E_\Q}))
&\cong  R\Gamma_\dR(X_K, \mathbf{E})\coloneqq  R\Gamma(X_K, (\mathbf{E}\otimes \Omega^\bullet_{X_K/K}, \nabla_{\mathbf{E}}))\\
&\cong R\Gamma(((X,M_X)/(\calO_K,M_\can))_\conv, E_\conv),
\end{align*}
where the first one is easily deduced from the comparison of the cohomology of an isocoherent sheaf on $X$ with that of the associated coherent sheaf on $X_K$ by a spectral sequence argument, and the second one follows from \cite[Cor.~2.3.9, Def.~2.2.12]{Shiho-II} and \cite[Rem.~B.29(3)]{du-liu-moon-shimizu-purity-F-crystal}. According to Theorem~\ref{thm:Hyodo-Kato-cohomology}(3), each uniformizer $\pi\in\calO_K$ identifies $R\Gamma_{\HK}((X_0, M_{X_0}), \calE_{\cris, \Q}))\otimes_{K_0}^LK$ with any of the above mutually isomorphic complexes.
\end{rem}

\section{Application to the \texorpdfstring{$C_{\st}$}{Cst}-conjecture} \label{sec: Cst}

Throughout this section, let $(X, M_X)$ be a semistable $p$-adic formal scheme over $\Spf \calO_K$ such that $X \rightarrow \Spf \calO_K$ is proper. We study the $C_{\st}$-conjecture for $(X, M_X)$ with coefficients given by semistable local systems, by combining the prismatic-\'etale comparison theorem proved in \cite{Tian-prism-etale-comparison} and the results in previous sections.

\subsection{Semistable \texorpdfstring{$\Z_p$}{Zp}-local systems and analytic prismatic $F$-crystals} \label{subsec: semist-local-syst}

We first recall the definition of semistable $\Z_p$-local systems on $X_{K, \proet}$ via association studied in \cite{du-liu-moon-shimizu-purity-F-crystal} (which originates from the work of Faltings \cite[V.f)]{Faltings-CryscohandGalrep}). Write $\Perfd/X_{K,\proet}$ for the full subcategory of $X_{K,\proet}$ consisting of \emph{affinoid} perfectoids \emph{that contain all $p$-power roots of unity}, and equip $\Perfd/X_{K,\proet}$ with the induced topology. The inclusion induces an equivalence $\operatorname{Sh}(X_{K,\proet})\xrightarrow{\cong}\operatorname{Sh}(\Perfd/X_{K,\proet})$ of topoi. Let $\bA_\cris$, $\bB_\cris^+$, and $\bB_\cris$ denote the crystalline period sheaves on $X_{K,\proet}$ introduced in \cite[\S~2A]{Tan-Tong}.

\begin{construction}[{\cite[Const.~3.37, 3.38]{du-liu-moon-shimizu-purity-F-crystal}}]\label{const:DLMS-perfect-prism-attached-to-affinoid-perfectoid}
For $U\in \Perfd/X_{K,\proet}$, write $\widehat{U}=\Spa(A,A^+)$ for the associated affinoid perfectoid. Equip $\Spf A^+$ with the pullback log structure $M_{\Spf A^+}$ from $M_X$ via the map $\Spf A^+\rightarrow X$ of $p$-adic formal schemes. Since $A^+$ is (integral) perfectoid, the surjection $\theta\colon \A_{\mathrm{inf}}(A^+)\coloneqq W((A^+)^\flat)\rightarrow A^+$ gives a perfect prism over $X$. By \cite[Const.~3.37]{du-liu-moon-shimizu-purity-F-crystal}, this gives a log prism $(\Spf \Ainf(A^+),\Ker\theta, M_{\theta}) \in (X,M_X)_{\Prism}$.\footnote{In our work, $\Ainf(A^+)$ yields two log prisms depending whether we use $\Ker\theta$ or $\Ker(\theta\circ\phi^{-1})$. See Remark~\ref{rem:two-structure-map-for-Ainf} below for the reason.}

Let $\A_{\cris}(A^+)$ denote the $p$-completed (log) PD-envelope of the exact surjection $\theta$. Write $\B^+_{\cris}(A^+)\coloneqq \A_{\cris}(A^+)[p^{-1}]$ and $\B_{\cris}(A^+)\coloneqq \B_{\cris}^+(A^+)[t^{-1}]$ with $t \coloneqq \operatorname{log}[\epsilon]\in \A_{\cris}(A^+)$. This yields an ind-object $(\Spec (A^+/p), \Spf(\A_{\cris}(A^+)))$ in $(X_1,M_{X_1})_{\CRIS}^{\mathrm{aff}}$. We have a map $\A_{\cris}(A^+)\rightarrow \bA_{\cris}(U)$, which induces an isomorphism $\B_{\cris}(A^+) \xrightarrow{\cong} \bB_{\cris}(U)$ by \cite[Cor.~2.8(2)]{Tan-Tong}. When $A^+ = \calO_C$, we simply write $A_{\inf}$ for $\A_{\inf}(\calO_C)$, and similarly for $A_{\cris}$, $B_{\cris}$, etc..

For a finite locally free isocrystal $\calE_\Q$ on $(X_1,M_{X_1})_{\CRIS}$ (cf. \cite[Def.~B.15]{du-liu-moon-shimizu-purity-F-crystal}), define a sheaf $\calE_\Q(\bB_\cris^+)$ of $\bB^+_\cris$-modules on $X_{K,\proet}$ by
\[
\calE_\Q(\bB_\cris^+)(U)\coloneqq \calE_\Q(\A_{\cris}(A^+))\otimes_{\B_{\cris}^+(A^+)}\bB_\cris^+(U)
\]
where $\calE_\Q(\A_{\cris}(A^+))$ denotes the evaluation of $\calE_\Q$ on $(\Spec (A^+/p), \Spf(\A_{\cris}(A^+)))$ (as an ind-object; cf.~\cite[Const.~B.14]{du-liu-moon-shimizu-purity-F-crystal}), which is a finite projective $\B_{\cris}^+(A^+)$-module. Note that $\calE_\Q(\bB_\cris^+)$ satisfies the sheaf condition on $\Perfd/X_{K,\proet}$, by the isocrystal property and finite projectivity since $\bB_\cris^+$ is a sheaf. Finally, define
\[
\calE_\Q(\bB_\cris)\coloneqq \calE_\Q(\bB_\cris^+)\otimes_{\bB_\cris^+}\bB_\cris,
\]
which is a sheaf of $\bB_\cris$-modules on $X_{K,\proet}$. If $\calE_\Q$ is an $F$-isocrystal, then $\calE_\Q(\bB_{\cris})$ is equipped with the induced Frobenius $F$.
\end{construction}

\begin{defn}[Semistable local systems; cf.~{\cite[Def.~3.39]{du-liu-moon-shimizu-purity-F-crystal}}]
Let $\bL$ be a $\mathbf{Z}_p$-local system on $X_{K,\proet}$ and let $(\calE_\Q,\phi_{\calE_\Q})$ be an $F$-isocrystal on $(X_1,M_{X_1})_{\CRIS}$. We say that $\bL$ and $(\calE_\Q,\phi_{\calE_\Q})$ are \emph{associated} if there exists an isomorphism 
\[
\alpha_{\cris,\bL,\calE_\Q}\colon \calE_\Q(\bB_{\cris})\xrightarrow{\cong}\bL\otimes_{\Z_p}\bB_\cris
\]
as sheaves of $\bB_\cris$-modules on $X_{K, \proet}$ such that $\alpha_{\cris,\bL,\calE_\Q}$ is compatible with Frobenii.

Call $\bL$ \emph{semistable} if $\bL$ is associated to an $F$-isocrystal on $(X_1,M_{X_1})_{\CRIS}$. Let $\Loc_{\Z_p}^\st(X_K)$ denote the full subcategory of $\Loc_{\Z_p}(X_K)$ consisting of semistable ones.
\end{defn}

\begin{rem} \hfill
\begin{enumerate}
    \item Unlike the work of Faltings \cite{Faltings-CryscohandGalrep}, we do not require $\calE_\Q$ to be a \emph{filtered} $F$-isocrystal. See \cite[Prop.~3.45]{du-liu-moon-shimizu-purity-F-crystal} for canonical filtration structures. 
    \item When $X = \Spf\calO_K$, above definition of semistable local systems agrees with the classical notion of semistable representations by Fontaine using the period ring $B_{\st}$ (\cite[Thm.~4.1]{du-liu-moon-shimizu-purity-F-crystal}).
\end{enumerate}    
\end{rem}

Now, we recall the relation between semistable $\Z_p$-local systems and analytic prismatic $F$-crystals in \cite{du-liu-moon-shimizu-purity-F-crystal}. Keep the assumption that $(X, M_X)$ is a semistable $p$-adic formal scheme over $\calO_K$.

\begin{defn} \label{defn: Laurent-F-cryst}
A \emph{Laurent $F$-crystal} on $(X,M_X)_\Prism$ consists of a pair $(\mathscr{E},\phi_{\mathscr{E}})$ where $\mathscr{E}$ is a crystal in vector bundles on the ringed site $((X,M_X)_\Prism, \calO_\Prism[\calI_\Prism^{-1}]^\wedge_p)$ and $\phi_{\mathscr{E}}$ is an isomorphism $\phi_{\mathscr{E}}\colon \phi^\ast {\mathscr{E}}\xrightarrow{\cong}{\mathscr{E}}$. Write $\Vect((X,M_X)_\Prism,\calO_\Prism[\calI_\Prism^{-1}]^\wedge_p)^{\phi=1}$ for the category of Laurent $F$-crystals on $(X,M_X)_\Prism$.    
\end{defn}

Let $\Vect(A[I^{-1}]^\wedge_p)^{\phi=1}$ denote the category of \'etale $\phi$-modules over $A[I^{-1}]^\wedge_p$, i.e., pairs $(\mathcal{M},\phi_{\mathcal{M}})$ where $\mathcal{M}$ is a finite locally free $A[I^{-1}]^\wedge_p$-module, and $\phi_{\mathcal{M}}\colon \phi^\ast \mathcal{M} \xrightarrow{\cong} \mathcal{M}$ is an isomorphism of $A[I^{-1}]^\wedge_p$-modules. 

\begin{lem}[{\cite[Lem.~3.13]{du-liu-moon-shimizu-purity-F-crystal}}] \label{lem: Laurent-F-cryst-as-limit}
There is a natural equivalence of categories
\[
\Vect((X,M_X)_\Prism,\calO_\Prism[\calI_\Prism^{-1}]^\wedge_p)^{\phi=1}\cong
\lim_{(A,I,M_{\Spf A})\in (X,M_X)_\Prism^\op}\Vect(A[I^{-1}]^\wedge_p)^{\phi=1}.
\]
\end{lem}

Furthermore, we have the following result from \cite{du-liu-moon-shimizu-purity-F-crystal} based on \cite[Cor.~3.7, 3.8]{bhatt-scholze-prismaticFcrystal} relating Laurent $F$-crystals and $\Z_p$-local systems:

\begin{thm}[cf. {\cite[Cor.~3.7, 3.8]{bhatt-scholze-prismaticFcrystal}, \cite[Thm.~3.14]{du-liu-moon-shimizu-purity-F-crystal}}] \label{thm: Laurent-F-cryst-loc-syst}
There is a natural equivalence of categories
\[
\Vect((X,M_X)_\Prism, \calO_\Prism[\calI_\Prism^{-1}]^\wedge_p)^{\phi=1}\cong \Loc_{\Z_p}(X_K).
\]    
\end{thm}

Note that for $(A, I, M_{\Spf A})\in (X,M_X)_\Prism$, restricting vector bundles over $\Spec(A)\smallsetminus V(p,I)$ to $\Spec(A)\smallsetminus V(I)$ and then extending the scalar along $A[I^{-1}]\to A[I^{-1}]^\wedge_p$ yields a functor $\Vect^{\an,\phi}(A,I) \to \Vect(A[I^{-1}]^\wedge_p)^{\phi=1}$. By Lemma~\ref{lem: Laurent-F-cryst-as-limit}, we obtain the scalar extension functor
\[
\Vect^{\an,\phi}((X,M_X)_\Prism) \to \Vect((X,M_X)_\Prism,\calO_\Prism[\calI_\Prism^{-1}]^\wedge_p)^{\phi=1}.
\]

\begin{defn}[\'Etale realization of analytic prismatic $F$-crystals] \label{def: etale-realization}

The \emph{\'etale realization} functor is defined to be the composite
\[
T=T_X\colon \Vect^{\an,\phi}((X,M_X)_\Prism)\rightarrow \Vect((X,M_X)_\Prism,\calO_\Prism[\calI_\Prism^{-1}]^\wedge_p)^{\phi=1}\xrightarrow{\cong} \Loc_{\Z_p}(X_{K}),
\]
where the second equivalence is given in Theorem~\ref{thm: Laurent-F-cryst-loc-syst}.    
\end{defn}

\begin{thm}[{\cite[Cor.~3.46, 5.2]{du-liu-moon-shimizu-purity-F-crystal}}]\label{thm:DLMS-association-via-prismatic-F-crystal}
The \'etale realization functor $T$ induces an equivalence of categories between $\Vect^{\an,\phi}((X,M_X)_\Prism)$ and $\Loc_{\Z_p}^\st(X_K)$. 

More precisely, for $\calE_\Prism\in \Vect^{\an,\phi}((X,M_X)_\Prism)$, $\bL=T_X(\calE_X)$ is associated with the crystalline realization $(\calE_{\cris,\Q},\phi_{\calE_{\cris,\Q}})$ of $j_\ast\calE_\Prism$;
moreover, $\bL$ is de Rham, and the associated filtered vector bundle with integrable connection $(D_\dR(\bL),\nabla_{D_\dR(\bL)},\Fil^\bullet D_{\dR}(\mathbb{L}))$ on $X_K$ is identified with the vector bundle with integrable connection $(\mathbf{E}, \nabla_{\mathbf{E}})$ on $X_{K}$ induced from $(\calE_{\cris,\Q},\phi_{\calE_{\cris,\Q}})$, together with the unique Griffiths-transverse filtration that makes $\alpha_{\cris,\bL,\calE_{\cris,\Q}}\otimes \mathbb{B}_\dR$ a filtered isomorphism.
\end{thm}

See Definition~\ref{def:crystalline-realization} for the crystalline realization and Remark~\ref{rem:deRham-convergent} for the construction of $(\mathbf{E}, \nabla_{\mathbf{E}})$.

\subsection{The \texorpdfstring{$C_{\st}$}{Cst}-conjecture} \label{subsec: Cst-conj}

Keep the assumption that $(X, M_X)$ is a semistable $p$-adic formal scheme over $\Spf \calO_K$ such that $X \rightarrow \Spf \calO_K$ is proper. We now study a comparison between the \'etale cohomology and Hyodo--Kato cohomology, with coefficients given by semistable $\Z_p$-local systems and associated $F$-isocrystals respectively.

Let $\bL$ be a semistable $\Z_p$-local system on $X_{K}$, and let $\calE_{\cris, \Q}$ be an $F$-isocrystal on $(X_1, M_{X_1})_{\CRIS}$ associated to $\bL$. By Theorem~\ref{thm:Hyodo-Kato-cohomology}, $\calE_{\cris, \Q}$ gives rise to the Hyodo--Kato complex $R\Gamma_{\HK}((X_0, M_{X_0}), \calE_{\cris, \Q}))\in D_{\phi, N}(K_0)$, which in fact depends only on $\calE_{\cris, \Q}|_{(X_0, M_{X_0})_{\CRIS}}$.

We let $(\mathbf{E}, \nabla_{\mathbf{E}},\Fil^\bullet \mathbf{E})$ denote the filtered vector bundle with integrable connection $(D_\dR(\bL),\nabla_{D_\dR(\bL)},\Fil^\bullet D_{\dR}(\mathbb{L}))$ on $X_K$ (cf.~Remark~\ref{rem:deRham-convergent} and Theorem~\ref{thm:DLMS-association-via-prismatic-F-crystal}).
Recall from \cite[Thm.~8.4]{scholze-p-adic-hodge} the \'etale-de Rham comparison 
\begin{equation}\label{eq:etale-deRham-comparison}
R\Gamma(X_{C, \et}, \bL)\otimes_{\Z_p}^L B_{\dR} \cong R\Gamma_{\dR}(X_K, \mathbf{E})\otimes_K^L B_{\dR},
\end{equation}
which is an isomorphism in $D(B_{\dR})$ compatible with $\Gal(\overline{K}/K)$-actions.
A choice of uniformizer $\pi\in K$ gives the Hyodo--Kato isomorphism
\[
\rho_\pi\colon R\Gamma_{\HK}((X_0, M_{X_0}), \mathcal{E}_{\cris, \Q})\otimes_{K_0}^LK\cong R\Gamma_{\dR}(X_K, \mathbf{E}).
\]
Combining this with the embedding $B_\st \rightarrow B_\dR$ given by the choice of $\pi$ (see Definition~\ref{def:Bst}), we obtain
\begin{equation}\label{eq:HT-to-dR-map}
R\Gamma_{\HK}((X_0, M_{X_0}), \mathcal{E}_{\cris, \Q})\otimes_{K_0}^L B_{\st}\rightarrow R\Gamma_{\dR}(X_K, \mathbf{E})\otimes_K^L B_{\dR}.
\end{equation}

We now give the precise statement of the $C_{\st}$-conjecture stated in Theorem~\ref{thm: Cst-intro}.

\begin{thm} \label{thm: Cst-conj-main}
Keep notations as above. There exists a commutative diagram
\[
\xymatrix{
R\Gamma(X_{C, \et}, \mathbb{L})\otimes_{\Z_p}^L B_{\st} \ar[r]^-{\alpha_{\st}}_-\cong \ar[d] & R\Gamma_{\HK}((X_0, M_{X_0}), \mathcal{E}_{\cris, \Q})\otimes_{K_0}^L B_{\st}\ar[d]^-{\eqref{eq:HT-to-dR-map}} \\
R\Gamma(X_{C, \et}, \bL)\otimes_{\Z_p}^L B_{\dR} \ar[r]^-{\eqref{eq:etale-deRham-comparison}} & R\Gamma_{\dR}(X_K, \mathbf{E})\otimes_K^L B_{\dR},
}
\]
where the top horizontal map is an isomorphism in $D(B_\st)$ that is compatible with Galois actions, Frobenii, and monodromy actions, and the left vertical map is induced by the embedding $B_\st \rightarrow B_\dR$.
\end{thm}

Let us first construct $\alpha_{\st}$ (Definition~\ref{def:Cst-comparison-map} below).
By Theorem~\ref{thm:DLMS-association-via-prismatic-F-crystal}, there exists an analytic prismatic $F$-crystals $\calE_{\Prism} \in \Vect^{\an,\phi}((X,M_X)_\Prism)$ such that $T(\calE_{\Prism}) \cong \mathbb{L}$, and the $F$-isocrystal $\calE_{\cris, \Q}$ on $(X_1, M_{X_1})_{\CRIS}$ is given by the crystalline realization $\calE_{\cris}$ of the completed prismatic $F$-crystal $j_{\ast}\calE_{\Prism}$. 

Recall the notations as in \S~\ref{sec: semistable-case} (especially, Example~\ref{eg: examples-log-prisms}). Consider the map $\fkS_K \rightarrow A_{\inf}$ of log prisms in $(\Spf\calO_K, M_{\can})^{\op}$ given by $u\mapsto [\pi^{\flat}]^p$. Write $\mu\coloneqq [\epsilon]-1 \in A_{\inf}$. The following comparison between the prismatic cohomology and \'etale cohomology is proved in \cite{Tian-prism-etale-comparison}, which serves as a key input.

\begin{thm}[{\cite[Thm.~0.4(2), 5.6(2)]{Tian-prism-etale-comparison}}] \label{thm: pris-etale comparison}
Keep notations as above. The \'etale realization induces a Frobenius and $\Gal(\overline{K}/K)$-equivariant isomorphism
\[
 R\Gamma_{A_\inf}(X, j_{\ast}\calE_\Prism)\otimes_{A_{\inf}}^L A_{\inf}[\mu^{-1}]
 \cong R\Gamma(X_{C, \et}, T(\calE_{\Prism}))\otimes_{\Z_p}^L A_{\inf}[\mu^{-1}].
\]
\end{thm}

Recall $B_{\cris} = A_{\cris}[p^{-1}, t^{-1}]$. 

\begin{cor} \label{cor: semist-comparison-over-S}
Keep notations as above. We have an isomorphism
\[
R\Gamma(X_{C, \et}, \mathbb{L})\otimes_{\Z_p}^L B_{\cris} \cong R\Gamma(((X_{\calO_C/p}, M_{X_{\calO_C/p}})/(A_{\cris}, M_{A_{\cris}}))_{\CRIS}, \calE_{\cris, \Q})\otimes_{A_{\cris}}^L B_{\cris}
\]
compatible with the Frobenii and $\Gal(\overline{K}/K)$-actions. 
\end{cor}

\begin{proof}
By Theorem~\ref{thm: pris-etale comparison}, we have an isomorphism 
\[
R\Gamma_{A_\inf}(X, j_{\ast}\calE_\Prism)\otimes_{A_{\inf}}^L B_{\cris} \cong R\Gamma(X_{C, \et}, \mathbb{L})\otimes_{\Z_p}^L B_{\cris}
\]
compatible with the Frobenii and $\Gal(\overline{K}/K)$-actions. On the other hand, Theorem~\ref{thm: pris-cris-comparison-over-Breuil S}(3) gives a $\phi$-equivariant isomorphism
\[
R\Gamma_{A_{\cris}}(X, j_{\ast}\calE_{\Prism})\otimes_{A_{\cris}}^L B_{\cris} \cong R\Gamma(((X_{\calO_C/p}, M_{X_{\calO_C/p}})/(A_{\cris}, M_{A_{\cris}}))_{\CRIS}, \calE_{\cris, \Q})\otimes_{A_{\cris}}^L B_{\cris}.
\]
The base change along $A_{\inf} \rightarrow A_{\cris}$ yields the isomorphism
\[
R\Gamma_{A_\inf}(X, j_{\ast}\calE_\Prism)\otimes_{A_{\inf}}^L B_{\cris} \cong R\Gamma_{A_{\cris}}(X, j_{\ast}\calE_{\Prism})\otimes_{A_{\cris}}^L B_{\cris}
\]
by Theorem~\ref{thm: general-base-change-BK-cohom} (cf.~\eqref{eq: diag-log-prism-site-O_K}), hence the statement.
\end{proof}

\begin{defn}\label{def:Cst-comparison-map}
Define $\alpha_{\st}$ to be the composite
\begin{align*}
R\Gamma(X_{C, \et}, \mathbb{L})\otimes_{\Z_p}^L B_{\st}
&\overset{\text{Cor.~\ref{cor: semist-comparison-over-S}}}{\underset{\cong}{\longrightarrow}} R\Gamma(((X_{\calO_C/p}, M_{X_{\calO_C/p}})/(A_{\cris}, M_{A_{\cris}}))_{\CRIS}, \calE_{\cris, \Q})\otimes_{A_{\cris}}^L B_{\st}\\
&\overset{\text{Thm.~\ref{thm:Hyodo-Kato-cohomology}(4)}}{\underset{\cong}{\longleftarrow}} R\Gamma_{\HK}((X_0, M_{X_0}), \mathcal{E}_{\cris, \Q})\otimes_{K_0}^L B_{\st}.
\end{align*}
Note that $\alpha_{\st,\pi}$ is an isomorphism compatible with Galois actions, Frobenii, and monodromy actions.
\end{defn}

\begin{proof}[Proof of Theorem~\ref{thm: Cst-conj-main}]
It remains to show the commutativity of the diagram, and we will verify it in \S~\ref{subsec: etale-deRham comparison}.
\end{proof}

\begin{rem}\label{rem:two-structure-map-for-Ainf}
Formally, there are two conventions for the \'etale realization (\cite{du-liu-moon-shimizu-completed-prismatic-F-crystal-loc-system,du-liu-moon-shimizu-purity-F-crystal} and \cite{GuoReinecke-Ccris, Tian-prism-etale-comparison}), but they are compatible in the following sense.
Take $U\in \Perfd/X_{K,\proet}$ with the affinoid perfectoid  $\widehat{U}=\Spa(A,A^+)$ as in Construction~\ref{const:DLMS-perfect-prism-attached-to-affinoid-perfectoid}. We have a morphism of log prisms 
\[
\phi\coloneqq \phi_{\Ainf(A^+)}\colon (\Ainf(A^+),\Ker\theta, M_{\theta}) \rightarrow (\Ainf(A^+),\Ker\widetilde{\theta}, M_{\widetilde{\theta}})
\]
in $(X,M_X)_{\Prism}^\op$, where $\widetilde{\theta}\coloneqq \theta\circ \phi^{-1}\colon \Ainf(A^+)\rightarrow A^+$ and  $M_{\widetilde{\theta}}$ is the pullback log structure of $M_{\theta}$ along $\phi$. This induces a map $\can\colon \calE_\Prism(\A_\inf(A^+),\Ker\theta)\rightarrow \calE_\Prism(\A_\inf(A^+),\Ker\widetilde{\theta})$, where we omit the log structure from the evaluation for simplicity.
The \'etale realization comes with the following commutative diagram of isomorphisms
\[
\xymatrix{
T(\calE_\Prism)(U)\otimes \A_\inf(A^+)[1/\Ker\theta]^\wedge_p\ar[d]_-{\id\otimes \phi}^-\cong\ar[r]^-\cong & j_\ast\calE_\Prism(\A_\inf(A^+),\Ker\theta)\otimes\A_\inf(A^+)[1/\Ker\theta]^\wedge_p \ar[d]_-{\can\otimes \phi}^-\cong\\
T(\calE_\Prism)(U)\otimes\A_\inf(A^+)[1/\Ker\widetilde{\theta}]^\wedge_p \ar[r]^-\cong& j_\ast\calE_\Prism(\A_\inf(A^+),\Ker\widetilde{\theta})\otimes\A_\inf(A^+)[1/\Ker\widetilde{\theta}]^\wedge_p,
}
\]
which is functorial in $U$.
The top horizontal isomorphism is the isomorphism used in \cite[Prop.~3.24]{du-liu-moon-shimizu-purity-F-crystal} to describe the \'etale realization, whereas the bottom horizontal isomorphism is the one used in \cite{GuoReinecke-Ccris, Tian-prism-etale-comparison} (see \cite[Lem.~5.5(1), Pf.]{Tian-prism-etale-comparison}). 
The right vertical isomorphism descends to an isomorphism $\can\otimes \phi \colon \calE_\Prism(\A_\inf(A^+),\Ker\theta)\otimes \A_\inf(A^+)[(\phi^{-1}(\mu))^{-1}] \xrightarrow{\cong}\calE_\Prism(\A_\inf(A^+),\Ker\widetilde{\theta})\otimes \A_\inf(A^+)[\mu^{-1}]$, which yields the following commutative diagram of isomorphisms
\[
\xymatrix{
j_\ast\calE_\Prism(\A_\inf(A^+),\Ker\theta)\otimes_{\A_\inf(A^+),\phi} \B_{\cris}(A^+) \ar[d]_-{\can\otimes \phi\otimes \id_{\B_{\cris}(A^+)}}^-\cong\ar[rd]^-{\alpha\circ (\can\otimes\phi\otimes \id)}_-\cong& \\
j_\ast\calE_\Prism(\A_\inf(A^+),\Ker\widetilde{\theta})\otimes_{\A_\inf(A^+)} \B_{\cris}(A^+)\ar[r]^-{\alpha}_-\cong & \calE_{\cris,\Q}(\A_{\cris}(A^+))\otimes_{\B^+_{\cris}(A^+)}\B_{\cris}(A^+).
}
\]
The isomorphism $\alpha\circ (\can\otimes\phi\otimes \id)$ agrees with the one in \cite[Rem.~3.41(2)]{du-liu-moon-shimizu-purity-F-crystal} and thus is used in the proof of Theorem~\ref{thm:DLMS-association-via-prismatic-F-crystal}, whereas $\alpha$ is used in the proof of Theorem~\ref{thm: pris-cris-comparison-over-Breuil S}(3). Via the above diagram, they are compatible, and we use this implicitly in the following arguments.
\end{rem}

\begin{rem}\label{rem:independence-of-Cst-comparison-map}
Recall that $B_\st\coloneqq B_{\st,\pi}$ is introduced in Definition~\ref{def:Bst} as a subring of $B_\dR$, which depends on the choice of a uniformizer $\pi$.
However, the map $\alpha_\st=\alpha_{\st,\pi}\colon R\Gamma(X_{C, \et}, \mathbb{L})\otimes_{\Z_p}^L B_{\st,\pi}\rightarrow R\Gamma_{\HK}((X_0, M_{X_0}), \mathcal{E}_{\cris, \Q})\otimes_{K_0}^L B_{\st,\pi}$ in Definition~\ref{def:Cst-comparison-map} is independent of the choice $\pi$ in the following sense. Recall the canonical semistable period ring $B_{\st,\mathrm{Fon}}\coloneqq B_{\st,\mathrm{Fon}}^+[t^{-1}]$ with an embedding $\iota_\pi\colon B_{\st,\mathrm{Fon}}\xrightarrow{\cong}B_{\st,\pi}\subset B_\dR$ given by $\pi$ (see Remark~\ref{rem:semistable-period-ring-convention}). Define $\alpha_\pi$ and $i_\pi$ to be the maps making the following diagram commutative: 
\[
\xymatrix{
R\Gamma(X_{C, \et}, \mathbb{L})\otimes_{\Z_p}^L B_{\st,\mathrm{Fon}} \ar@{..>}[r]^-{\alpha_{\pi}}_-\cong \ar[d]^-\cong  \ar@<-7ex>@/_30pt/[dd]_-{\id\otimes \iota_\pi}
& R\Gamma_{\HK}((X_0, M_{X_0}), \mathcal{E}_{\cris, \Q})\otimes_{K_0}^L B_{\st,\mathrm{Fon}}\ar[d]^-\cong \ar@<7ex>@/^60pt/[dd]^-{i_\pi} \\
R\Gamma(X_{C, \et}, \mathbb{L})\otimes_{\Z_p}^L B_{\st,\pi} \ar[r]^-{\alpha_{\st,\pi}}_-\cong \ar[d] & R\Gamma_{\HK}((X_0, M_{X_0}), \mathcal{E}_{\cris, \Q})\otimes_{K_0}^L B_{\st,\pi}\ar[d]^-{\eqref{eq:HT-to-dR-map}} \\
R\Gamma(X_{C, \et}, \bL)\otimes_{\Z_p}^L B_{\dR} \ar[r]^-{\eqref{eq:etale-deRham-comparison}} & R\Gamma_{\dR}(X_K, \mathbf{E})\otimes_K^L B_{\dR},
}
\]
We claim that $\alpha_\pi$ is indeed independent of $\pi$. For this, take any $c\in R\Gamma(X_{C, \et}, \mathbb{L})$\footnote{More precisely, after taking cohomology here and in what follows.} and we show that $\alpha_\pi(c\otimes 1)$ is independent of $\pi$. Since $\iota_\pi|_{B_\cris}$ and \eqref{eq:etale-deRham-comparison} are independent of $\pi$, so is $i_{\pi}(\alpha_\pi(c\otimes 1))\in R\Gamma_{\dR}(X_K, \mathbf{E})\otimes_K^L B_{\dR}$. Take another uniformizer $\pi'=a\pi$ ($a\in\calO_K^\times)$, and let us compare $i_\pi$ and $i_{\pi'}$.
For $x\otimes y\in R\Gamma_{\HK}((X_0, M_{X_0}), \mathcal{E}_{\cris, \Q})\otimes_{K_0}^L B_{\st,\mathrm{Fon}}$, we have $i_\pi(x\otimes y)=\rho_\pi(x)\otimes \iota_\pi(y)$ by definition. By Theorem~\ref{thm:Hyodo-Kato-cohomology}(3), we get $\rho_{\pi'}(x)=\sum_{n\geq 0}\frac{(\operatorname{log}(a) e)^n}{n!}\rho_{\pi}(N^n(x))$. Let $N_\pi$ denote the monodromy operator on $B_{\st,\mathrm{Fon}}$ induced from the one on $B_{\st,\pi}$. Then we know from \cite[(4.1.2)]{Tsuji-Cst} that $\iota_{\pi'}(y)=\sum_{n\geq 0}\frac{(\operatorname{log}(a) e)^n}{n!}\iota_{\pi}(N_\pi^n(y))$ (recall our normalization of $N_\pi$). Hence we deduce $i_{\pi'}(x\otimes y)=\sum_{n\geq 0}\frac{(\operatorname{log}(a) e)^n}{n!}i_{\pi}((N\otimes 1+1\otimes N_\pi)^n(x\otimes y))$. In particular, $i_\pi$ and $i_{\pi'}$ coincide on $\Ker(N\otimes 1+1\otimes N_\pi)$. Since $\alpha_{\st,\pi}$ is compatible with monodromy operators, we see that $\alpha_\pi(c\otimes 1)$ lies in $\Ker(N\otimes 1+1\otimes N_\pi)$ and thus equals $\alpha_{\pi'}(c\otimes 1)$.
\end{rem}

\subsection{\'Etale-de Rham comparison} \label{subsec: etale-deRham comparison}

In this subsection, we show the commutativity of the diagram in Theorem~\ref{thm: Cst-conj-main}. We mainly follow \cite[\S~10]{GuoReinecke-Ccris} based on the infinitesimal cohomology over $B_{\dR}^+$, which is introduced in \cite[\S~13]{bhatt-morrow-scholze-integralpadic} for the smooth case and generalized in \cite[\S~6]{Cesnavicius-Koshikawa} to the semistable case. As many arguments in this subsection are straightforward generalizations of those in \cite[\S~10]{GuoReinecke-Ccris}, we will emphasize additional necessary inputs (e.g. from \cite{Cesnavicius-Koshikawa}) and refer to \textit{loc. cit.} for details.

\subsubsection{Infinitesimal cohomology over \texorpdfstring{$B_{\dR}^+$}{BdR}} \label{subsec: inf-cohom}

Write $A_{\inf, K}\coloneqq A_{\inf}\otimes_{W(k)}\calO_K$. Choose a map $\calO_K \rightarrow B_{\dR}^+$ compatible with $K \rightarrow C$ under the surjection $B_{\dR}^+ \rightarrow C$. Let $\widetilde{\theta}\coloneqq \theta\circ \phi^{-1}\colon A_{\inf} \rightarrow \calO_C$, which extends to $\widetilde{\theta}_K\colon A_{\inf, K} \rightarrow \calO_C$ via $\calO_{K}\rightarrow \calO_C$. Note that the map $A_{\inf} \rightarrow A_{\inf, K}$ induces an isomorphism $(A_{\inf, K}[p^{-1}])^{\wedge}_{\Ker(\widetilde{\theta}_K)} \cong B_{\dR}^+$. For each positive integer $c$, write $B_{\dR, c}^+\coloneqq B_{\dR}^+/(\Ker(\widetilde{\theta}_K))^c$; it becomes an f-adic ring of which the image of $A_{\inf,L}$ (with $p$-adic topology) is a ring of definition.

Let $X$ be a semistable $p$-adic formal scheme over $\calO_K$ with generic fiber $X_K$. Denote $X_{B_{\dR, c}^+}$ for the base change of $X_K$ along $K \rightarrow A_{\inf, K}[p^{-1}] \rightarrow B_{\dR, c}^+$. 

\begin{defn}[{\cite[Def.~10.1]{GuoReinecke-Ccris}}] 
Define the infinitesimal site $X_{C, \et}/B_{\dR, \inf}^+$ as follows:
\begin{itemize}
    \item An object is a pair $(U, T)$, called an \emph{infinitesimal thickening of $X_C$}, where $U \in X_{C, \et}$ and $T$ is an adic space topologically of finite type over $B_{\dR, c}^+$ for some $c\geq 1$ together with a Zariski closed immersion $U \rightarrow T$ given by a nilpotent ideal.
    \item A morphism between objects $(U_1, T_1)\rightarrow (U_2, T_2)$ is given by a morphism $U_1 \rightarrow U_2$ in $X_{C, \et}$ and a compatible map of adic spaces $T_1\rightarrow T_2$ over $B_{\dR, c}^+$ for some $c$.
    \item A covering of $(U, T)$ is a family of morphisms $\{(U_i, T_i) \rightarrow (U, T)\}$ such that $\{U_i \rightarrow U\}$ and $\{T_i\rightarrow T\}$ are \'etale coverings.
\end{itemize}
The structure sheaf $\calO_{\inf}$ is given by the assignment $(U, T)\mapsto \calO_T(T)$. 

A \emph{crystal in vector bundles on $X_{C, \et}/B_{\dR, \inf}^+$} is a sheaf of $\calO_{\inf}$-modules $\calF$ such that $\calF(U, T)$ is a vector bundle over $\calO_T$ for each $(U, T)$ and satisfies the base change isomorphism for each $(U_1, T_1) \rightarrow (U_2, T_2)$. Write $\Vect(X_{C, \et}/B_{\dR, \inf}^+)$ for the category of crystals in vector bundles on $X_{C, \et}/B_{\dR, \inf}^+$.
\end{defn}

\begin{lem}\label{lem: inf-site}
The fiber products and equalizers exist in $X_{C, \et}/B_{\dR, \inf}^+$. Furthermore, the non-empty finite products are ind-representable in $X_{C, \et}/B_{\dR, \inf}^+$.    
\end{lem}

\begin{proof}
This follows directly from the proof of \cite[Lem.~2.1.4]{Guo-cris-cohom-rigid-analy-sp}.    
\end{proof}

\begin{lem} \label{lem: inf-site-weakly-final-ind-syst}
Let $Z$ be a smooth adic space over $K$ with a closed immersion $X_C \rightarrow Z_C$.
\begin{itemize}
    \item Write $Z_{m, c}$ for the $m$-th infinitesimal neighborhood of $X_C$ in $Z_{B_{\dR, c}^+}$. Then the ind-system $D_Z(X)\coloneqq \{Z_{m, c}\}_{m, c}$ is a weakly final object in $X_{C, \et}/B_{\dR, \inf}^+$.
    \item The $(n+1)$-th self-product of the ind-system $\{Z_{m, c}\}_{m, c}$ is representable by the ind-system $\{Z_{m, c}^n\}_{m, c}$, where $Z^n$ denotes the $(n+1)$-th self-product of $Z$ over $K$ and $Z_{m, c}^n$ the $m$-th infinitesimal neighborhood of $X_C$ in $Z^n_{B_{\dR, c}^+}$.
\end{itemize}
\end{lem}

\begin{proof}
Using Lemma~\ref{lem: inf-site}, this is proved as in \cite[Lem.~10.3]{GuoReinecke-Ccris}.   
\end{proof}

For any topologically finite type $\calO_K$-algebra $A$, let $A_{B_{\dR, c}^+}$ denote the base change of $A$ along $\calO_K \rightarrow A_{\inf, K} \rightarrow B_{\dR, c}^+$ as f-adic rings, and set $A_{B_{\dR}^+}=\varprojlim_c A_{B_{\dR, c}^+}$.
The following examples of the above lemma will be used in computations.

\begin{eg}[cf. {\cite[Ex.~10.5]{GuoReinecke-Ccris}}] \label{eg: enlarged-framing-inf-site} 
Suppose $X = \Spf R$ admits an integral \emph{enlarged framing} $(\Psi, \Sigma)$ as in \cite[\S~6]{Cesnavicius-Koshikawa}, given by finite subsets $\Psi \subset R^{\times}$ and $\Sigma \subset R\cap R[p^{-1}]^{\times}$ such that the map $\calO_K\langle (T_{\psi}^{\pm 1})_{\psi \in \Psi}, (T_{\sigma})_{\sigma \in \Sigma} \rangle \rightarrow R$ (given by $T_{\psi} \mapsto \psi$, $T_{\sigma} \mapsto \sigma$) is surjective and that there exist subsets $\Psi' \subset \Psi$ and $\Sigma' \subset \Sigma$ with $\prod_{\sigma \in \Sigma'} T_{\sigma} = \pi$ in $R$ such that the induced map 
\[
X \rightarrow \Spf \calO_K\langle (T_{\psi}^{\pm 1})_{u \in \Psi'}, (T_{\sigma})_{\sigma \in \Sigma'}\rangle /(\prod_{\sigma \in \Sigma'} T_{\sigma} - \pi)
\]
is \'etale. Note that this holds \'etale locally on $X$. 
We equip $\Spf \calO_K\langle (T_{\psi}^{\pm 1})_{\psi \in \Psi}, (T_{\sigma})_{\sigma \in \Sigma} \rangle$ with the log structure $M_{\Psi, \Sigma}$ given by
\[
\N\times\N^{\Sigma'} \rightarrow \calO_K\langle (T_{\psi}^{\pm 1})_{\psi \in \Psi}, (T_{\sigma})_{\sigma \in \Sigma} \rangle, ~~ (1, 0)\mapsto \pi, ~~(0, e_{\sigma}) \mapsto T_{\sigma} ~~\text{for}~~ \sigma \in \Sigma'.
\]

Let $Z$ be the generic fiber of $\Spf \calO_K\langle (T_{\psi}^{\pm 1})_{\psi \in \Psi}, (T_{\sigma})_{\sigma \in \Sigma} \rangle$. Then the value of $\calO_{\inf}$ on the ind-system $\{Z_{m, c}\}_{m, c}$ is the formal completion ring $D_{\Psi, \Sigma}(R)$ of the surjection 
\[
\calO_K\langle (T_{\psi}^{\pm 1})_{\psi \in \Psi}, (T_{\sigma})_{\sigma \in \Sigma} \rangle_{B_{\dR}^+} \twoheadrightarrow R_C.
\]
Furthermore, for each $n \geq 1$, $D_{(\Psi, \Sigma)^n}(R)\coloneqq \calO_{\inf}(\{Z_{m, c}^n\}_{m, c})$ is equal to the formal completion of the surjection $(\calO_K\langle (T_{\psi}^{\pm 1})_{\psi \in \Psi}, (T_{\sigma})_{\sigma \in \Sigma} \rangle^{\widehat{\otimes}_{\calO_K}^{n+1}})_{B_{\dR}^+} \rightarrow R_C$.
\end{eg}

Let $\calF$ be a crystal in vector bundles over $X_{C, \et}/B_{\dR, \inf}^+$. We consider two ways to compute $R\Gamma(X_{C, \et}/B_{\dR, \inf}^+, \calF)$, one by \v{C}ech--Alexander complexes and another by de Rham complexes, as in \cite[Const.~10.6]{GuoReinecke-Ccris}. Let $Z$ be a smooth affinoid adic space over $K$, and assume there is a closed immersion $X_C \rightarrow Z_C$. The \v{C}ech--Alexander complex for the covering $D_Z(X)$ is the cosimplicial complex $\calF_{D_{Z^{\bullet}}(X)}$ with $\calF_{D_{Z^n}(X)} \coloneqq \lim_{m, c}\calF(Z_{m, c}^n)$. If $Z$ is given by an enlarged framing $(\Psi, \Sigma)$ as in Example~\ref{eg: enlarged-framing-inf-site}, then $\calF_{D_{Z^{\bullet}}(X)} = \calF_{D_{(\Psi, \Sigma)^{\bullet}}(R)}$. By Lemma~\ref{lem: inf-site-weakly-final-ind-syst} and the vanishing of higher cohomology of vector bundles over affinoid rigid spaces, we have a natural isomorphism
\[
R\Gamma(X_{C, \et}/B_{\dR, \inf}^+, \calF) \cong \Tot(\calF_{D_{Z^{\bullet}}(X)}).
\]

On the other hand, by \cite[Thm.~3.3.1]{Guo-cris-cohom-rigid-analy-sp}, the vector bundle $\calF_{D_Z(X)}$ over $\calO_{\inf}(D_Z(X))$ is equipped with a natural flat connection $\nabla\colon \calF_{D_Z(X)} \rightarrow \calF_{D_Z(X)}\otimes_{\calO_Z} \Omega_{Z/K}^1$. Write $\dR(\calF_{D_Z(X)})$ for the de Rham complex
\[
\dR(\calF_{D_Z(X)})\coloneqq \calF_{D_Z(X)} \rightarrow \calF_{D_Z(X)}\otimes_{\calO_Z} \Omega_{Z/K}^1 \rightarrow \cdots \rightarrow \calF_{D_Z(X)}\otimes_{\calO_Z} \Omega_{Z/K}^d 
\]
where $d$ is the dimension of $Z/K$ (since $Z$ is affinoid, we also regard this as a complex of modules). 
Note that when $Z$ is given by an enlarged framing $(\Psi, \Sigma)$ as in Example~\ref{eg: enlarged-framing-inf-site}, $\dR(\calF_{D_{\Psi, \Sigma}(R)})\coloneqq \dR(\calF_{D_Z(X)})$ is represented by the Koszul complex for the derivations $\frac{\partial}{\partial \log(T_{\lambda})}\coloneqq T_{\lambda}\frac{\partial}{\partial T_{\lambda}}$, $\lambda \in \Psi \cup \Sigma$ (cf. \cite[\S~6.3]{Cesnavicius-Koshikawa}). The de Rham complex also computes $R\Gamma(X_{C, \et}/B_{\dR, \inf}^+, \calF)$.

\begin{thm}[{\cite[Thm.~10.7]{GuoReinecke-Ccris}}] \label{thm: inf-cohom-double-cx}
Keep the notation as above. Let $\calF$ be a crystal in vector bundles over $X_{C, \et}/B_{\dR, \inf}^+$. Let $Z$ be a smooth affinoid space over $K$ with a closed immersion $X_K \rightarrow Z$. Then there is a natural double complex
\[
M^{a, b}\coloneqq \calF_{D_{Z^a}(X)}\otimes_{\calO_{Z^a}}\Omega_{Z^a/K}^b,
\]
which is functorial with respect to $\calF$ and the choice of $Z$ (and the choice of enlarged framing $(\Psi, \Sigma)$ in the case of Example~\ref{eg: enlarged-framing-inf-site}. It satisfies the following properties.
\begin{itemize}
    \item For each $b \geq 1$, the cosimplical complex $M^{\bullet, b}$ is acyclic.
    \item Any degeneracy map $[a] \rightarrow [0]$ induces an isomorphism of de Rham complexes $M^{0, \bullet} \rightarrow M^{a, \bullet}$.
\end{itemize}
In particular, the total complex of $M^{a, b}$ is isomorphic to both $M^{\bullet, 0} = \Tot(\calF_{D_{Z^{\bullet}}(X)})$ and $M^{0, \bullet} = \dR(\calF_{D_Z(X)})$. 
\end{thm}

We now study the relation between the log crystalline cohomology and infinitesimal cohomology. We refer to \cite[Def.~B.15]{du-liu-moon-shimizu-purity-F-crystal} for the notion of finite locally free isocrystals on $(X_1, M_{X_1})_{\CRIS}$.

\begin{prop}[cf. {\cite[Prop.~10.10]{GuoReinecke-Ccris}}] \label{prop: cris crystal-to-inf crystal}
Let $X$ be a separated semistable $p$-adic formal scheme over $\calO_K$. There is a natural functor from the category of finite locally free isocrystals on $(X_1, M_{X_1})_{\CRIS}$ to $\Vect(X_{C, \et}/B_{\dR, \inf}^+)$.
\end{prop}

\begin{proof}
We follow a similar construction as in the proof of \cite[Prop.~10.10]{GuoReinecke-Ccris}. Let $\calE_{\cris, \Q}$ be a finite locally free isocrystal on $(X_1, M_{X_1})_{\CRIS}$. To associate a crystal in vector bundles on $X_{C, \et}/B_{\dR, \inf}^+$ to $\calE_{\cris, \Q}$, we work \'etale locally on $X$ and consider integral enlarged framings as in Example~\ref{eg: enlarged-framing-inf-site}.

Assume $X = \Spf R$ with an integral enlarged framing $(\Psi, \Sigma)$. 
Consider the $A_{\inf,K}$-algebra $A_{\inf, K}\langle (T_\psi^{\pm 1})_{\psi \in \Psi}, (T_\sigma)_{\sigma\in \Sigma}\rangle$ equipped with the log structure given by $M_{A_{\inf}}$ in Theorem~\ref{thm: pris-cris-comparison-over-Breuil S} and $M_{\Psi, \Sigma}$ in Example~\ref{eg: enlarged-framing-inf-site}.
Let $J\coloneqq \Ker(A_{\inf, K}\langle (T_\psi^{\pm 1})_{\psi \in \Psi}, (T_\sigma)_{\sigma\in \Sigma}\rangle \twoheadrightarrow R_{\calO_C})$. Write $D_{\pd, \Psi, \Sigma}(R)$ for the $p$-completed log PD-envelope of $A_{\inf, K}\langle (T_\psi^{\pm 1})_{\psi \in \Psi}, (T_\sigma)_{\sigma\in \Sigma}\rangle$ with respect to $J$. The induced map $D_{\pd, \Psi, \Sigma}(R) \rightarrow A_{\inf, K}\langle (T_\psi^{\pm 1})_{\psi \in \Psi}, (T_\sigma)_{\sigma\in \Sigma}\rangle[p^{-1}]/J^m$ becomes surjective after inverting $p$. Taking the inverse limit over $m$, we obtain
\[
D_{\pd, \Psi, \Sigma}(R) \rightarrow D_{\Psi, \Sigma}(R) \cong (D_{\pd, \Psi, \Sigma}(R)[p^{-1}])^{\wedge}_J,
\]
which factors through $D_{\pd, \Psi, \Sigma}(R)\widehat{\otimes}_{A_{\inf, K}}B_{\dR}^+\coloneqq \varprojlim_m D_{\pd, \Psi, \Sigma}(R)[p^{-1}]/\Ker(\widetilde{\theta}_K)^m$.

Similarly, for each $n \geq 0$, let $D_{\pd, (\Psi, \Sigma)^n}(R)$ be the $p$-completed log PD-envelope of 
\[
A_{\inf, K}\langle (T_\psi^{\pm 1})_{\psi \in \Psi}, (T_\sigma)_{\sigma\in \Sigma}\rangle^{\widehat{\otimes}_{A_{\inf, K}}^{n+1}} \twoheadrightarrow R_{\calO_C}.
\]
Then we have a natural map $D_{\pd, (\Psi, \Sigma)^n}(R) \rightarrow D_{(\Psi, \Sigma)^n}(R)$.

For each $n$, consider the finite projective $D_{(\Psi, \Sigma)^n}(R)$-module
\[
\calE_{\cris, \Q}(D_{\pd, (\Psi, \Sigma)^n}(R))\otimes_{D_{\pd, (\Psi, \Sigma)^n}(R)} D_{(\Psi, \Sigma)^n}(R),
\]
where  the evaluation of $\calE_{\cris, \Q}$ at the ind-object $(\Spec R_{\calO_C/p}, \Spf D_{\pd, (\Psi, \Sigma)^n}(R))$ is denoted by $\calE_{\cris, \Q}(D_{\pd, (\Psi, \Sigma)^n}(R))$
as in Construction~\ref{const:DLMS-perfect-prism-attached-to-affinoid-perfectoid}.
By the crystal property of $\calE_{\cris, \Q}$ (see Lemma~\ref{lem:algebraic-lemmas-for-base-change}(2)), the corresponding cosimplicial object satisfies the natural base change isomorphism under any $[n] \rightarrow [m]$ in $\Delta$. By Lemma~\ref{lem: inf-site-weakly-final-ind-syst}, this yields a crystal in vector bundles $\calE_{\inf} \in \Vect(X_{C, \et}/B_{\dR, \inf}^+)$ such that $\calE_{\inf}(D_{(\Psi, \Sigma)^n}(R)) = \calE_{\cris}(D_{\pd, (\Psi, \Sigma)^n}(R))\otimes_{D_{\pd, (\Psi, \Sigma)^n}(R)} D_{(\Psi, \Sigma)^n}(R)$. 

Now, since the construction above is functorial with respect to $R$ and enlarged framing $(\Psi, \Sigma)$, it globalizes to a functor from the category of finite locally free isocrystals on $(X_1, M_{X_1})_{\CRIS}$ to $\Vect(X_{C, \et}/B_{\dR, \inf}^+)$. 
\end{proof}

\begin{prop}[cf. {\cite[Prop.~10.11]{GuoReinecke-Ccris}}] \label{prop: comparison-cris-cohom-inf-cohom}
Let $X$ be a proper semistable $p$-adic formal scheme over $\calO_K$. Let $\calE_{\cris, \Q}$ be a finite locally free isocrystal on $(X_1, M_{X_1})_{\CRIS}$ given by the crystalline realization $\calE_{\cris}$ of an analytic prismatic crystal on $(X, M_X)_{\Prism}$ (Definition~\ref{def:crystalline-realization}), and write $\calE_{\inf} \in \Vect(X_{C, \et}/B_{\dR, \inf}^+)$ for the associated crystal in vector bundles on $X_{C, \et}/B_{\dR, \inf}^+$ given by Proposition~\ref{prop: cris crystal-to-inf crystal}. Then we have a natural isomorphism in $D(B_{\dR}^+)$
\begin{align*}
R\Gamma(((X_{\calO_C/p}, M_{X_{\calO_C/p}})&/(A_{\logcris, K}, M_{A_{\logcris, K}}))_{\CRIS}, \calE_{\cris, \Q})\otimes_{A_{\logcris, K}[p^{-1}]}^L B_{\dR}^+\\ &\stackrel{\cong}{\rightarrow} R\Gamma(X_{C, \et}/B_{\dR, \inf}^+, \calE_{\inf}). 
\end{align*}
\end{prop}

\begin{proof}
By Proposition~\ref{prop: base-change-cris-cohom-Acris}(1)(2), it suffices to show a natural isomorphism
\[
(R\Gamma(((X_1, M_{X_1})/(\calO_K, M_\can))_{\CRIS}, \calE_{\cris, \Q})\otimes_K^L B_{\dR}^+ \stackrel{\cong}{\rightarrow} R\Gamma(X_{C, \et}/B_{\dR, \inf}^+, \calE_{\inf}).
\]
Recall the notation in Remark~\ref{rem:deRham-convergent} and  isomorphisms therein:
\begin{align*}
R\Gamma(((X_1, M_{X_1})/(\calO_K, M_\can))_{\CRIS}, \calE_{\cris, \Q}) &\cong R\Gamma(X, (E_{\Q}\otimes_{\calO_X}\omega^{\bullet}_{X/\calO_K}, \nabla_{E_{\Q}})) \\
    &\cong R\Gamma(X_K, (\mathbf{E}\otimes \Omega^\bullet_{X_K/K}, \nabla_{\mathbf{E}})).
\end{align*}

On the other hand, since $X_K$ is smooth over $K$, we see from applying Theorem~\ref{thm: inf-cohom-double-cx} for the identity map $X_K \rightarrow Z \coloneqq X_K$ that $R\Gamma(X_{C, \et}/B_{\dR, \inf}^+, \calE_{\inf})$ is computed by $\dR((\calE_{\inf})_{D_Z(X)})$. By the construction of $\calE_{\inf}$ in the proof of Proposition~\ref{prop: cris crystal-to-inf crystal}, we have $E\otimes_{\calO_K} B_{\dR}^+ \cong (\calE_{\inf})_{D_Z(X)}$, and have a natural isomorphism
\[
(\mathbf{E}\otimes \Omega^\bullet_{X_K/K}, \nabla_{\mathbf{E}})\otimes_{\calO_K} B_{\dR^+} \cong \dR((\calE_{\inf})_{D_Z(X)}). 
\]
\end{proof}

\subsubsection{Compatibility of comparison maps}

We proceed to proving the commutativity of the diagram in Theorem~\ref{thm: Cst-conj-main}. Similarly as in \cite[\S~10.2]{GuoReinecke-Ccris}, we show a commutative diagram \eqref{eq: summary-of-comparison-maps} of isomorphisms of cohomology complexes over $B_{\dR}$ in Theorem~\ref{thm: comparison-maps-compatibility} below. In the diagram, $\gamma_{\et}$ and $\gamma_{\cris}$ are given respectively by Theorem~\ref{thm: pris-etale comparison} and Theorem~\ref{thm: pris-cris-comparison-over-Breuil S}, and $\rho_{\st}$ and $\rho_{\pi}$ are given by Theorem~\ref{thm:Hyodo-Kato-cohomology}. The dotted isomorphisms in the diagram are given by \cite{scholze-p-adic-hodge}, so the commutativity in question follows from Theorem~\ref{thm: comparison-maps-compatibility}.  

Keep the notation as in \S~\ref{subsec: Cst-conj}: let $(X, M_X)$ be a proper semistable $p$-adic formal scheme over $\calO_K$. Let $\calE_{\Prism}$ be an analytic prismatic $F$-crystal on $(X, M_X)_{\Prism}$, and $\bL = T(\calE_{\Prism})$ be the corresponding semistable $\Z_p$-local system on $X_{K}$ associated to the finite locally free $F$-isocrystal $\calE_{\cris, \Q}$ on $(X_1, M_{X_1})_{\CRIS}$ given by the crystalline realization $\calE_{\cris}$ of $\calE_{\Prism}$. Write $(A, I, M_{\Spf A}) = (A_{\inf}, ([p]_{[\epsilon]}), M_{A_{\inf}}) \in (\Spf\calO_K, M_{\Spf\calO_K})_{\Prism}$, and make following abbreviations for the various cohomology theories involved:
\begin{itemize}
    \item $R\Gamma_{\Prism}\coloneqq R\Gamma_{A_{\inf}}(X, j_{\ast}\calE_{\Prism})$

    \item $R\Gamma_{\et}\coloneqq R\Gamma(X_{C, \et}, \bL)$

    \item $R\Gamma_{\cris}\coloneqq R\Gamma(((X_{\calO_C/p}, M_{X_{\calO_C/p}})/(A_{\cris}, M_{A_{\cris}}))_{\CRIS}, \calE_{\cris, \Q})$

    \item $R\Gamma_{\HK}\coloneqq R\Gamma_{\HK}((X_0, M_{X_0}), \calE_{\cris, \Q}))$.

    \item Let $\calE_{\inf}$ be the crystal in vector bundles on $X_{C, \et}/B_{\dR, \inf}^+$ associated to $\calE_{\cris, \Q}$ by Proposition~\ref{prop: cris crystal-to-inf crystal}. Write $R\Gamma_{\inf}\coloneqq R\Gamma(X_{C, \et}/B_{\dR, \inf}^+, \calE_{\inf})$

    \item We have the $\bB_{\dR}$-local system $\calE_{\cris, \Q}(\bB_{\dR})$ over the pro-\'etale site $X_{C, \proet}$ given by \cite[Const.~3.43]{du-liu-moon-shimizu-purity-F-crystal}. Write $R\Gamma_{\proet}\coloneqq R\Gamma(X_{C, \proet}, \calE_{\cris, \Q}(\bB_{\dR}))$.

    \item Let $(\mathbf{E}, \nabla_{\mathbf{E}})$ be the vector bundle with integral connection on $X_{K}/K$ associated to $\calE_{\cris, \Q}$, equipped with filtration $\Fil^{\bullet}\mathbf{E}$ (\cite[Rem.~B.33, Prop.~3.45]{du-liu-moon-shimizu-purity-F-crystal}). Write $R\Gamma_{\dR}\coloneqq R\Gamma_{\dR}(X_{K}/K, \mathbf{E})$.

    \item Let $\dR(\calE_{\cris, \Q}(\bB_{\dR})\otimes_{\bB_{\dR}} \calO\bB_{\dR}) \coloneqq \calE_{\cris, \Q}(\bB_{\dR})\otimes_{\bB_{\dR}} \dR(\calO\bB_{\dR}, d_{X_{K}/K})$ be the associated filtered de Rham complex equipped with the tensor product filtration. By \cite[Prop.~3.45]{du-liu-moon-shimizu-purity-F-crystal}, we have a natural filtered isomorphism
    \[
    \dR(\calE_{\cris, \Q}(\bB_{\dR})\otimes_{\bB_{\dR}} \calO\bB_{\dR}) \cong \dR(\mathbf{E}, \nabla_{X_{K}/K})\otimes_{\calO_{X_{K}}} \calO\bB_{\dR}. 
    \]
    Write $R\Gamma_{\proetdR}\coloneqq R\Gamma(X_{C, \proet}, \dR(\calE_{\cris, \Q}(\bB_{\dR})\otimes_{\bB_{\dR}} \calO\bB_{\dR}))$ for its filtered pro-\'etale cohomology.
\end{itemize}

\begin{thm}[cf. {\cite[Thm.~10.16]{GuoReinecke-Ccris}}] \label{thm: comparison-maps-compatibility}
We have a natural commutative diagram of isomorphisms of cohomology complexes over $B_{\dR}$ 
\begin{equation} \label{eq: summary-of-comparison-maps}
\xymatrix{
& R\Gamma_{\Prism}\otimes_{A_{\inf}}^L B_{\dR} \ar[ld]_{\gamma_{\et}} \ar[d] \ar[r]^{\gamma_{\cris}} & R\Gamma_{\cris}\otimes_{A_{\cris}}^L B_{\dR} \ar[d]\\
R\Gamma_{\et}\otimes_{\Z_p}^L B_{\dR} \ar@{.>}[r]
& R\Gamma_{\proet} \ar@{.>}[d] & R\Gamma_{\inf}\otimes_{B_{\dR}^+}^L B_{\dR} \ar@<-.5ex>[l] \ar@<.5ex>[l] & R\Gamma_{\HK}\otimes_{K_0}^L B_{\dR} \ar[ld]^{\rho_{\pi}\otimes B_{\dR}} \ar[lu]_{\rho_{\st}\otimes B_{\dR}}\\
& R\Gamma_{\proetdR} & R\Gamma_{\dR}\otimes_K^L B_{\dR} \ar@{.>}[l] \ar[u]
}    
\end{equation}
where the dotted arrows are filtered isomorphisms given in \cite{scholze-p-adic-hodge}.
\end{thm}

\begin{proof}
Write $R\Gamma_{\logcris, K}\coloneqq R\Gamma(((X_{\calO_C/p}, M_{X_{\calO_C/p}})/(A_{\logcris, K}, M_{A_{\logcris, K}}))_{\CRIS}, \calE_{\cris, \Q})$. Note first that we have a commutative diagram
\[
\xymatrix{
R\Gamma_{\HK}\otimes_{K_0}^L B_{\dR} \ar[r]^{\rho_{\st}\otimes B_{\dR}} \ar[d]_{\rho_{\pi}\otimes B_{\dR}} & R\Gamma_{\cris}\otimes_{A_{\cris}}^L B_{\dR} \ar[d]\\
R\Gamma_{\dR}\otimes_K^L B_{\dR} \ar[r] & R\Gamma_{\logcris, K}\otimes_{A_{\logcris, K}}^L B_{\dR} 
}
\]
given by Theorem~\ref{thm:Hyodo-Kato-cohomology}(3)(4)(5) and the base change along $A_{\logcris, K}[p^{-1}] \rightarrow B_{\dR}$, where all the maps are isomorphisms. The rightmost part of the diagram~\eqref{eq: summary-of-comparison-maps} is obtained by composing the above with the isomorphism in Proposition~\ref{prop: comparison-cris-cohom-inf-cohom}. 

For the other parts of the diagram~\eqref{eq: summary-of-comparison-maps}, the argument is very similar to the proof of \cite[Thm.~10.16]{GuoReinecke-Ccris}; for local computations we use enlarged framing as in Example~\ref{eg: enlarged-framing-inf-site} since $X$ is semistable. So we omit some details and refer to \textit{loc. cit.} for corresponding parts.

First consider the diagram
\begin{equation} \label{eq: diagram-gamma-et}
\xymatrix{
& R\Gamma_{\Prism}\otimes_{A_{\inf}}^L B_{\dR} \ar[ld]_{\gamma_{\et}} \ar[d]^{(1.a)}\\
R\Gamma_{\et}\otimes_{\Z_p}^L B_{\dR} \ar@{.>}[r]^{(1.b)}
& R\Gamma_{\proet}.
}    
\end{equation}

\noindent $\bullet$ \textbf{(1.a):} For any affinoid perfectoid space $\Spa(B, B^+)$ over $X_{C, \proet}$, consider the perfect prism $(\A_{\inf}(B^+), I)$. By \cite[Prop.~1.23]{Tian-prism-etale-comparison}, there is a unique $\delta_{\log}$-structure $M_{\Spf \A_{\inf}(B^+)}$ on $\Spf \A_{\inf}(B^+)$ such that $(\A_{\inf}(B^+), I, M_{\Spf \A_{\inf}(B^+)})$ becomes a log prism of $((X_C, M_{X_C})/(A, I, M_{\Spf A}))_{\Prism}^{\op}$. This defines a cocontinuous functor
\[
\Perfd / X_{C, \proet} \rightarrow ((X_{\calO_C}, M_{X_{\calO_C}})/(A, I, M_{\Spf A}))_{\Prism, \et}
\]
and thus gives a morphism of topoi
$\Sh(X_{C, \proet}) \rightarrow \Sh(((X_{\calO_C}, M_{X_{\calO_C}})/(A, I, M_{\Spf A}))_{\Prism, \et})$.
The right vertical map (1.a) is induced by this.\\ 

\noindent $\bullet$ \textbf{(1.b):} The map
$R\Gamma_{\et}\otimes_{\Z_p}^L B_{\dR} \rightarrow R\Gamma_{\proet}$ 
is defined in \cite[Thm.~8.8]{scholze-p-adic-hodge}, and is a filtered isomorphism. \\

\noindent $\bullet~\gamma_{\et}:$ The isomorphism 
$\gamma_{\et}\colon R\Gamma_{\Prism}\otimes_{A_{\inf}}^L B_{\dR} \rightarrow R\Gamma_{\et}\otimes_{\Z_p}^L B_{\dR}$
is given by the base change along $A_{\inf}[\mu^{-1}] \rightarrow B_{\dR}$ of the prismatic-\'etale comparison in \cite{Tian-prism-etale-comparison}: we have a natural isomorphism of $\A_{\inf}(B^+)$-algebras
\[
T(\calE_{\Prism})\otimes_{\Z_p}\A_{\inf}(B^+)[\mu^{-1}] \cong \calE_{\Prism}(\A_{\inf}(B^+), I, M_{\Spf \A_{\inf}(B^+)})[\mu^{-1}]
\]
functorial in $\Spa(B, B^+)$, by \cite[Lem.~5.5(1)]{Tian-prism-etale-comparison}. By the proof of \cite[Thm.~5.6]{Tian-prism-etale-comparison}, the comparison $\gamma_{\et}$ is given by the above functorial isomorphism so that the diagram \eqref{eq: diagram-gamma-et} is commutative.\\ 

Next, we consider the diagram
\[
\xymatrix{
R\Gamma_{\Prism}\otimes_{A_{\inf}}^L B_{\dR} \ar[d] \ar[r]^{\gamma_{\cris}} & R\Gamma_{\cris}\otimes_{A_{\cris}}^L B_{\dR} \ar[d]\\
R\Gamma_{\proet} & R\Gamma_{\inf}\otimes_{B_{\dR}^+}^L B_{\dR} \ar[l].
}    
\]
The right vertical map of the above diagram is given by Proposition~\ref{prop: comparison-cris-cohom-inf-cohom} together with the factorization
\[
R\Gamma_{\cris}\otimes_{A_{\cris}}^L B_{\dR} \rightarrow R\Gamma_{\logcris, K}\otimes_{A_{\logcris, K}}^L B_{\dR} \rightarrow R\Gamma_{\inf}\otimes_{B_{\dR}^+}^L B_{\dR},
\]
where the first map is an isomorphism by Proposition~\ref{prop: base-change-cris-cohom-Acris}(3). Consider the induced diagram
\begin{equation} \label{eq: diagram-gamma-cris-O_K}
\xymatrix{
R\Gamma_{\Prism}\otimes_{A_{\inf}}^L B_{\dR} \ar[d]^{(1.a)} \ar[r]^-{\gamma_{\cris, K}} & R\Gamma_{\logcris, K}\otimes_{A_{\logcris, K}}^L B_{\dR} \ar[d]^-{(2.a)}\\
R\Gamma_{\proet} & R\Gamma_{\inf}\otimes_{B_{\dR}^+}^L B_{\dR} \ar[l]^-{(2.b)}.
}    
\end{equation}

We have two \'etale local computations of $R\Gamma_{\logcris, K}$: let $\Spf R \in X_{\et}$ with an enlarged framing $(\Psi, \Sigma)$ as in Example~\ref{eg: enlarged-framing-inf-site}. In particular, we have an $\calO_K$-linear surjection $\calO_K\langle (T_\psi^{\pm 1})_{\psi \in \Psi}, (T_\sigma)_{\sigma \in \Sigma} \rangle \rightarrow R$. For each $\psi \in \Psi$ and $\sigma \in \Sigma$, choose compatible systems of $p$-power roots $T_\psi^{1/p^{\infty}}$ and $T_\sigma^{1/p^{\infty}}$, and let $R_{\infty}$ be the $p$-completed base change of $R$ along $\calO_K\langle (T_\psi^{\pm 1})_{\psi \in \Psi}, (T_\sigma)_{\sigma \in \Sigma} \rangle \rightarrow \calO_K\langle (T_\psi^{\pm 1/p^{\infty}})_{\psi \in \Psi}, (T_\sigma^{1/p^{\infty}})_{\sigma \in \Sigma} \rangle$. Then the map $R \rightarrow R_{\infty}$ is quasi-syntomic. Write $G = \prod_{\Psi} \Z_p(1)\times \prod_{\Sigma} \Z_p(1)$. Note that $\Spa(R_{\infty, C}, R_{\infty, \calO_C})$ is an affinoid perfectoid and the cover $\Spa(R_{\infty, C}, R_{\infty, \calO_C}) \rightarrow X_C$ is a $G$-torsor. For each $n \geq 0$, denote by $R^n$ (resp. $R_{\infty}^n$) the $p$-complete self-tensor product $R^{\widehat{\otimes}_{\calO_K}^{n+1}}$ (resp. $R_{\infty}^{\widehat{\otimes}_{\calO_K}^{n+1}}$). 

As in the proof of Proposition~\ref{prop: cris crystal-to-inf crystal}, let $D_{\pd, (\Psi, \Sigma)^n}(R)$ be the $p$-completed log PD-envelope of 
\[
A_{\inf, K}\langle (T_\psi^{\pm 1})_{\psi \in \Psi}, (T_\sigma)_{\sigma\in \Sigma}\rangle^{\widehat{\otimes}_{A_{\inf, K}}^{n+1}} \twoheadrightarrow R_{\calO_C}.
\]
Let $\A_{\logcris, K}(R_{\infty, \calO_C}^n)$ be the $p$-completed log PD-envelope of the map $\A_{\inf}(R_{\infty, \calO_C}^n)\otimes_{W(k)}\calO_K \twoheadrightarrow R_{\infty, \calO_C}$. By mapping $T_\psi$'s and $T_\sigma$'s respectively to $[(T_\psi^{1/p^m})_{m}]$ and $[(T_\sigma^{1/p^m})_{m}]$ in $A_{\inf}(R_{\infty, \calO_C})$, we get
\[
\lim_{[n] \in \Delta} \calE_{\cris}(D_{\pd, (\Psi, \Sigma)^n}(R)) \rightarrow \lim_{[n] \in \Delta} \calE_{\cris}(\A_{\logcris, K}(R_{\infty, \calO_C}^n)).   
\]
This map is an isomorphism since both compute the same crystalline cohomology $R\Gamma(((\Spf R_{\calO_C/p}, M_{\Spf R_{\calO_C/p}}) / (A_{\logcris, K}, M_{A_{\logcris, K}}))_{\CRIS}, \calE_{\cris})$ (note $R \rightarrow R_{\infty}$ is quasi-syntomic; see also \cite[Thm.~6.6]{du-moon-shimizu-cris-pushforward}).

Under this local situation, we have following descriptions of the maps in \eqref{eq: diagram-gamma-cris-O_K} in terms of \v{C}ech--Alexander complexes:\\

\noindent $\bullet$ \textbf{(2.a):} Recall from the proof of Proposition~\ref{prop: cris crystal-to-inf crystal} that 
\[
\calE_{\inf}(D_{(\Psi, \Sigma)^n}(R)) = \calE_{\cris}(D_{\pd, (\Psi, \Sigma)^n}(R))\otimes_{D_{\pd, (\Psi, \Sigma)^n}(R)} D_{(\Psi, \Sigma)^n}(R)
\]
and we have a natural map
\[
(\lim_{[n] \in \Delta} \calE_{\cris}(D_{\pd, (\Psi, \Sigma)^n}(R))\widehat{\otimes}_{A_{\logcris, K}}^L B_{\dR}^+ \rightarrow \lim_{[n] \in \Delta} \calE_{\inf}(D_{(\Psi, \Sigma)^n}(R)). 
\]
Since the second complex computes $R\Gamma_{\inf}$, the above map is an isomorphism by Proposition~\ref{prop: comparison-cris-cohom-inf-cohom}. The base change along $B_{\dR}^+ \rightarrow B_{\dR}$ yields the map (2.a).\\

\noindent $\bullet~\gamma_{\cris, K}:$ To describe the map $\gamma_{\cris, K}$ in \eqref{eq: diagram-gamma-cris-O_K}, note that for any choice of enlarged framing $(\Psi, \Sigma)$, we have subsets $\Psi' \subset \Psi$ and $\Sigma' \subset \Sigma$ with $\prod_{\sigma \in \Sigma'} T_\sigma = \pi$ in $R$ such that the induced map 
\[
\square\colon \calO_K\langle (T_\psi^{\pm 1})_{\psi \in \Psi'}, (T_\sigma)_{\sigma \in \Sigma'}\rangle /(\prod_{\sigma \in \Sigma'} T_\sigma - \pi) \rightarrow R
\]
is $p$-completely \'etale. So we have $\fkS_{\square} \rightarrow \A_{\inf}(R_{\infty, \calO_C})$ given by $T_\psi \mapsto [(T_\psi^{1/p^m})_{m}]$ for $\psi \in \Psi'$, $T_{\sigma} \mapsto [(T_\sigma^{1/p^m})_{m}]$ for $\sigma \in \Sigma'$, and $u\mapsto \prod_{\sigma \in \Sigma'} [(T_\sigma^{1/p^m})_{m}]$. This extends uniquely to $S_{\square} \rightarrow \A_{\logcris, K}(R_{\infty, \calO_C})$ and induces a map of cosimplicial rings 
\[
S_{\square}^{\rel, [\bullet]} \rightarrow \A_{\logcris, K}(R_{\infty, \calO_C}^{\bullet}),
\]
where $S_{\square}^{\rel, [\bullet]}$ is defined in Lemma~\ref{lem: Breuil-prism-self-products-relative} and the paragraph above. Recall that by Theorem~\ref{thm: general-base-change-BK-cohom} and its proof, we have
\[
R\Gamma_{\Prism}\otimes_{A_{\inf}}^L A_{\cris}[p^{-1}] \cong R\Gamma_{\Br}(X, j_{\ast}\calE_{\Prism})\widehat{\otimes}_{S_K}^L A_{\cris}[p^{-1}] \cong (\lim_{[n]\in \Delta} j_{\ast}\calE_{\Prism}(S_{\square}^{\rel, [\bullet]}))\widehat{\otimes}_{S_K}^L A_{\cris}[p^{-1}]. 
\]
Thus, by the proof of Theorem~\ref{thm: pris-cris-comparison-over-Breuil S}, the map $\gamma_{\cris, K}$ is given by
\begin{align*}
\gamma_{\cris, K}\colon R\Gamma_{\Prism}\otimes_{A_{\inf}}^L B_{\dR} & \cong (\lim_{[n]\in \Delta} j_{\ast}\calE_{\Prism}(S_{\square}^{\rel, [n]}))\widehat{\otimes}_{S_K}^L B_{\dR}\\
    & \rightarrow (\lim_{[n] \in \Delta} \calE_{\cris}(\A_{\logcris, K}(R_{\infty, \calO_C}^n))\widehat{\otimes}_{A_{\logcris, K}}^L B_{\dR}.    
\end{align*}

\noindent $\bullet$ \textbf{(1.a):} For the left vertical map of \eqref{eq: diagram-gamma-cris-O_K}, write $\Spa(R_{\infty, C}, R_{\infty, \calO_C})^{\bullet}$ for the \v{C}ech nerve of $\Spa(R_{\infty, C}, R_{\infty, \calO_C}) \rightarrow X_C$ in $X_{C, \proet}$. We have
\[
R\Gamma(X_{C, \proet}, \calE_{\cris, \Q}(\bB_{\dR})) \cong \lim_{[n] \in \Delta} \calE_{\cris, \Q}(\bB_{\dR})(\Spa(R_{\infty, C}, R_{\infty, \calO_C})^n).
\]
The map $\fkS_{R, \square} \rightarrow \A_{\inf}(R_{\infty, \calO_C})$ as above induces a map of cosimplicial rings  
\[
\fkS_{\square}^{\rel, [\bullet]} \rightarrow \bA_{\inf}(\Spa(R_{\infty, C}, R_{\infty, \calO_C})^{\bullet}).
\]
So the map (1.a) is given by 
\[
R\Gamma_{\BK}(X, j_{\ast}\calE_{\Prism})  \cong \lim_{[n] \in \Delta} j_{\ast}\calE_{\Prism}(\fkS_{\square}^{\rel, [n]})
  \rightarrow \lim_{[n] \in \Delta} j_{\ast}\calE_{\Prism}(\bA_{\inf}(\Spa(R_{\infty, C}, R_{\infty, \calO_C})^{n})),
\]
where the first isomorphism is given in the proof of Proposition~\ref{prop: BK-cohom}.

\noindent $\bullet$ \textbf{(2.b):} Similarly, $T_{\lambda}\mapsto [(T_{\lambda}^{1/p^m})_{m}]$ for $\lambda \in \Psi \cup \Sigma$ gives
\[
\lim_{[n] \in \Delta} \calE_{\inf}(D_{(\Psi, \Sigma)^n}(R)) \rightarrow \lim_{[n] \in \Delta} \calE_{\cris, \Q}(\bB_{\dR})(\Spa(R_{\infty, C}, R_{\infty, \calO_C})^n),
\]
which yields the bottom map in \eqref{eq: diagram-gamma-cris-O_K}.\\

From the above explicit descriptions in terms of \v{C}ech--Alexander complexes, it follows that the diagram~\eqref{eq: diagram-gamma-cris-O_K} commutes and its local constructions are functorial with respect to $\Spf R \in X_{\et}$ and integral enlarged framing $(\Psi, \Sigma)$. Furthermore, the map (2.b) is a quasi-isomorphism globally, i.e., after passing to global sections on $X$, since the other maps in the diagram are so.

The remaining parts of the diagram~\eqref{eq: summary-of-comparison-maps} are given by essentially the same argument as in the corresponding parts of the proof of \cite[Thm.~10.16]{GuoReinecke-Ccris} (using enlarged framing in Example~\ref{eg: enlarged-framing-inf-site} for our case), from which the commutativity of the entire diagram also follows.
\end{proof}

\appendix

\section{Complements on log formal schemes} \label{sec:log-formal-scheme}

Let $X$ be a $p$-adic formal scheme. For each $n\geq 1$, write $X_n$ for the scheme $(\lvert X\rvert, \calO_X/p^n)$ and let $i_n\colon X_{n,\et}\rightarrow X_\et$ denote the induced morphism of small \'etale sites. Continue to write $\calO_X$ for the structure sheaf on $X_\et$. 
Recall $\calO_X\xrightarrow{\cong}\varprojlim_ni_{n,\ast}\calO_{X_n}$ by \cite[Lem.~I.6.2.11]{fujiwara-kato}.

\begin{prop}\label{lem:log structure on formal scheme as inverse limit}
Let $M_X$ be a log structure on $X$ and write $M_{X_n}$ for the log structure on $X_n$ attached to the prelog structure $i_n^{-1}M_X\rightarrow \calO_{X_n}$.
The maps 
\[
M_X\rightarrow \varprojlim_n i_{n,\ast}M_{X_n}\quad\text{and} \quad\Gamma(X_\et,M_X) \rightarrow \varprojlim_n\Gamma(X_{n,\et},M_{X_n})
\]
are isomorphisms.
\end{prop}

\begin{proof}
Consider the short exact sequence of monoid sheaves
\[
1\rightarrow \calO_X^\times \rightarrow M_X \rightarrow \overline{M}_X\rightarrow 1
\quad\text{and}\quad
1\rightarrow \calO_{X_n}^\times \rightarrow M_{X_n} \rightarrow \overline{M}_{X_n}\rightarrow 1,
\]
where $\overline{M}_X\coloneqq M_X/\calO_X^\times$ and $\overline{M}_{X_n}\coloneqq M_{X_n}/\calO_{X_n}^\times$. By (a formal scheme version of) \cite[Rem.~III.1.1.6]{Ogus-log}, we have $\overline{M}_{X_n}=i_n^{-1}(\overline{M}_X)$. So we obtain a short exact sequence
\[
1\rightarrow i_{n,\ast}\calO_{X_n}^\times \rightarrow i_{n,\ast}M_{X_n} \rightarrow i_{n,\ast}i_n^{-1}(\overline{M}_X)\rightarrow 1
\]
with $i_{n,\ast}i_n^{-1}(\overline{M}_X)=\overline{M}_X$. Taking the inverse limit yields a commutative diagram with exact rows:
\[
\xymatrix{
1\ar[r] & \calO_X^\times \ar[r]\ar[d] &M_X \ar[r]\ar[d] & \overline{M}_X \ar[r]\ar@{=}[d] &1.\\
1\ar[r] & \varprojlim i_{n,\ast}\calO_{X_n}^\times \ar[r] & \varprojlim i_{n,\ast}M_{X_n} \ar[r] & \overline{M}_X  & 
}
\]
Since $\calO_X^\times \xrightarrow{\cong}\varprojlim i_{n,\ast}\calO_{X_n}^\times$, we conclude $M_X\xrightarrow{\cong} \varprojlim_n i_{n,\ast}M_{X_n}$.
The second assertion follows from the first by taking the global sections.
\end{proof}

\begin{lem}\label{lem:ff descent of log str}
Let $X'$ be a $p$-adic formal scheme and $f\colon X'\rightarrow X$ a quasi-compact adically faithfully flat morphism. Set $g\colon X''\coloneqq X'\times_XX' \rightarrow X$. Let $M_X$ be an integral log structure that admits a chart Zariski locally and let $M_{X'}$ (resp.~$M_{X''}$) denote the pullback log structure on $X'$ (resp.~$X''$). Then the diagram of sheaves of monoids
\[
M_X\rightarrow f_\ast M_{X'}\rightrightarrows g_\ast M_{X''}
\]
is exact.
\end{lem}

\begin{proof}
Set $X'_n\coloneqq X'\times_XX_n$ and $X''_n\coloneqq X''\times_XX_n$. Then they are schemes and $X''_n$ is identified with $X'_n\times_{X_n}X'_n$. By assumption, the induced map $f_n\colon X'_n\rightarrow X_n$ is quasi-compact faithfully flat.
By \cite[Prop.~II.1.1.8.3, Lem.~III.1.4.2]{Ogus-log}, $M_{X_n}\rightarrow f_{n,\ast} M_{X'_n}\rightrightarrows g_{n, \ast} M_{X''_n}$ is exact where $M_{Y_n}$ denotes the pullback log structure of $M_Y$ on $Y_n$ and $g_n$ denotes the induced map $X''_n\rightarrow X_n$. Now the lemma follows from Lemma~\ref{lem:log structure on formal scheme as inverse limit}.
\end{proof}

\section{Variants of crystalline sites} \label{sec: cryst-sites-variants}

\subsection{Crystalline sites}\label{subsec:crystalline-sites}

We introduce variants of crystalline sites used in \cite{koshikawa}, which we will use to establish the prismatic-crystalline comparison theorem in \S~\ref{sec: prismatic-crystalline comparison}. All the formal schemes in this subsection are $p$-adic.

Here is the set-up of this subsection: let $A$ be a \emph{$p$-torsion free} $p$-complete ring, and let $M_A\rightarrow A$ be a prelog structure on $A$ such that $M_A$ is integral. Let $I$ be a $p$-complete PD-ideal of $A$ containing $p$ with PD-structure $\gamma_A$. Let $(Y, M_Y)$ be a log $p$-adic formal scheme defined over $(A,M_A)$ such that $M_Y$ is integral and quasi-coherent.

\begin{defn}[{\cite[Rem.~6.6]{koshikawa}}]\label{def: Koshikawa-crystalline-site}
Define $((Y,M_Y)/(A,M_A))_\CRIS$ to be a site whose opposite category is given as follows:
an object is a quadruple $(B,M_{\Spf B},J,f)$ (or written as $(B,J)$ or $(B,B/J)$ for simplicity) where
\begin{itemize}
 \item  $(\Spf B,M_{\Spf B})$ is a log formal scheme associated with a \textit{prelog ring} $(B,M_B)$ \textit{over} $(A,M_A)$ such that $B$ is a \emph{$p$-torsion free} $p$-complete ring and $M_{\Spf B}$ is integral;
 $J$ is a $p$-complete PD-ideal of $B$ and that $B/J$ is $p$-complete;
 \item $f$ is a morphism $\Spf B/J\rightarrow Y$ of $p$-adic formal schemes;
 \item an exact closed immersion $(\Spf B/J,f^\ast M_Y)\hookrightarrow (\Spf B, M_{\Spf B})$.
\end{itemize}
A morphism from $(B,M_{\Spf B},J,f)$ to $(B',M_{\Spf B'},J',f')$ is a ring homomorphism $g\colon B\rightarrow B'$ that is compatible with the other structures. We use the same symbol $g$ to denote the induced morphism $(\Spf B',M_{\Spf B'})\rightarrow (\Spf B,M_{\Spf B})$.
We equip $((Y,M_Y)/(A,M_A))_\CRIS$ with the strict \'etale topology: $\{g_i\colon (B,M_{\Spf B},J,f)\rightarrow (B_i,M_{\Spf B_i},J_i,f_i) \}$ is a covering if $g_i\colon (\Spf B_i,M_{\Spf B_i})\rightarrow (\Spf B,M_{\Spf B})$ is $p$-completely \'etale and strict and jointly surjective, and if $J_i=(JB_i)^\wedge_p$ for each $i$.
To see that this defines a site, observe that the pushout exists for a diagram
\[
(B_1,M_{\Spf B_1},J_1,f_1)\xleftarrow{g_1}(B,M_{\Spf B},J,f)\xrightarrow{g_2}(B_2,M_{\Spf B_2},J_2,f_2)
\]
 in $((Y,M_Y)/(A,M_A))_\CRIS^\op$ if $g_1$ is $p$-completely \'etale and strict and if $J_1=(JB_1)^\wedge_p$, and in this case, the pushout $(B',M_{\Spf B'},J',f')$ is given by $B'=B_1\widehat{\otimes}_BB_2\coloneqq(B_1\otimes_BB_2)^\wedge_p$ and $J'=(J_2B')^\wedge_p$.

The association $(B,M_{\Spf B},J,f)\mapsto B$ defines 
a sheaf $\calO_\CRIS$ of rings by $p$-completely faithfully flat descent. 
\end{defn}

The notation $((Y,M_Y)/(A,M_A))_\CRIS$ is an abbreviation of $((Y,M_Y)/(A,(p), M_A))_\CRIS$ as the PD-structure on $J\subset B$ is compatible with the canonical PD-structure on $pB$ (see Remark~\ref{rem: PD-struct-compatible}).
When $(A,M_A)$ is $\Z_p$ with the trivial log structure, we will simply use $(Y,M_Y)_\CRIS$ to denote $((Y,M_Y)/(A,M_A))_\CRIS$ and call it the \textit{absolute crystalline site}.

\begin{rem}[{\cite[Rem.~6.5]{koshikawa}}] \label{rem: PD-struct-compatible}
For every $(B,J)\in ((Y,M_Y)/(A,M_A))_\CRIS^\op$, the PD-structures on $I$ and $J$ are compatible since $B$ is $p$-torsion free.
\end{rem}

\begin{rem}\label{rem: two absolute crystalline sites}
Strictly speaking, the above $(Y,M_Y)_\CRIS$ is different from the absolute crystalline site $(Y,M_Y)_\CRIS\coloneqq ((Y,M_Y)/\Z_p^\sharp)_\CRIS$ in \cite[Def.~10.1]{du-moon-shimizu-cris-pushforward}. However, we have comparison results as in Proposition~\ref{prop:comparison of crystals in two crystalline sites} below. So we often ignore the difference.
\end{rem}

For a morphism $h\colon (Y',M_{Y'})\rightarrow (Y,M_Y)$ of integral and quasi-coherent log $p$-adic formal schemes over $(A,M_A)$, one can define a morphism of topoi 
\[
h_\CRIS\colon ((Y',M_{Y'})/(A,M_A))_\CRIS\rightarrow ((Y,M_Y)/(A,M_A))_\CRIS
\]
in a way similar to \cite[\S~4]{du-moon-shimizu-cris-pushforward}; we leave the detail to the reader.

\begin{prop}\label{prop: proj to etale for cris topos}
There exists a morphism of topoi (called the \emph{projection from crystalline topos to \'etale topos})
\[
u_{Y} \coloneqq u_{Y/(A,M_A)}\colon \Sh(((Y,M_Y)/(A,M_A))_{\CRIS})\rightarrow \Sh(Y_\et)
\]
given by
\[
(u_{Y/(A,M_A),\ast}\calF)(V)=\Gamma(((V,M_V)/(A,M_A))_\CRIS,\calF|_{((V,M_V)/(A,M_A))_\CRIS})
\]
(where $M_V$ is the pullback log structure from $M_Y$) and
\[
(u_{Y/(A,M_A)}^\ast\calG)(B,M_{\Spf B},J,f)=(f_\et^\ast\calG)(\Spf B/J).
\]
\end{prop}

Note that there is an obvious restriction functor $\Sh(((Y,M_Y)/(A,M_A))_\CRIS)\rightarrow\Sh(((V,M_V)/(A,M_A))_\CRIS)$ for $V\in Y_\et$.
\begin{proof}
The proof of \cite[Prop.~5.1]{du-moon-shimizu-cris-pushforward} also works in this situation.
\end{proof}

\begin{defn}\label{defn:p-complete crystals over OCRIS}
An $\calO_\CRIS$-module $\calE$ is called a \emph{$p$-adically completed crystal} if 
\begin{enumerate}
  \item for each $(B,J)\in ((Y,M_Y)/(A,M_A))_\CRIS^\op$, $\calE(B,J)$ is a $p$-adically complete $B$-module, and if
  \item for each map $(B,J)\rightarrow (B',J')$ in $((Y,M_Y)/(A,M_A))_\CRIS^\op$, the induced map
\[
\calE(B,J)\widehat{\otimes}_BB'\coloneqq \varprojlim_n (\calE(B,J)\otimes_BB')/p^n(\calE(B,J)\otimes_BB') \rightarrow \calE(B',J')
\]
is an isomorphism.
\end{enumerate}
\end{defn}

\begin{lem}\label{lem: affine vanishing on crystalline site}
Let $\calE$ be a $p$-adically completed crystal on $((Y,M_Y)/(A,M_A))_\CRIS$.
\begin{enumerate}
  \item For each $n\geq 1$, the presheaf $\calE_n$ given by
$\calE_n(B,J)\coloneqq \calE(B,J)/p^n\calE(B,J)$ is a sheaf and satisfies $H^q((B,J),\calE_n)=0$ for every $(B,J)$ and $q>0$.
  \item The map $\calE\rightarrow \varprojlim_n \calE_n$ is an isomorphism, and $R^q\varprojlim_n \calE_n=0$ for every $q>0$. Moreover, $H^q((B,J),\calE)=0$ for every $(B,J)$ and $q>0$.
\end{enumerate}
\end{lem}

\begin{proof}
Fix $(B,J)\in ((Y,M_Y)/(A,M_A))_\CRIS^\op$ and set $B_n=B/p^n$. 
Note that every covering of $(B,J)$ is refined to a covering of the form $g\colon (B,J)\rightarrow (B',J')$ with $g$ being $p$-completely \'etale and faithfully flat and $J'=(JB')^\wedge_p$, which we also fix. Since $\calE$ is a $p$-adically completed crystal, we have $\calE_n(B'',J'')\cong \calE_n(B,J)\otimes_{B_n}B_n''$ for every $(B,J)\rightarrow (B'',J'')$. In particular, the \v{C}ech cohomology $\check{H}^q(\{(B,J)\rightarrow (B',J')\},\calE_n)$ is $\calE_n(B,J)$ if $q=0$ and zero if $q>0$ since $B_n\rightarrow B_n'$ is faithfully flat.
This implies (1) by \cite[03F9]{stacks-project} and (2) by \cite[Lem.~3.18]{scholze-p-adic-hodge}.
\end{proof}

Let us also recall the big crystalline site studied in \cite[\S~10]{du-moon-shimizu-cris-pushforward}.

\begin{defn}[{See \cite[Def.~10.1]{du-moon-shimizu-cris-pushforward} for the detail}]\label{def:DMS-crystalline-site}
Keep the set-up as before and assume further that $(Y,M_Y)$ is a log (formal) scheme over $A/I$.
Let $\calS^\sharp=(\calS,M_\calS,\calJ_\calS,\gamma_\calS)$ be the $p$-adic log PD-formal scheme associated to $(A,M_A,I,\gamma_A)$ and write $S^\sharp_n$ for the log PD-scheme $\calS^\sharp \times_{\Spf \Z_p} \Spec\Z_p/p^n$. 

Let $((Y,M_Y)/\calS^\sharp)_{\CRIS}$ denote the \textit{big relative logarithmic crystalline site}, which is the colimit of $((Y,M_Y)/S_n^\sharp)_{\CRIS}$: concretely, an object is a tuple $(U,T,M_T, f, i,\gamma)$ where $f\colon U\rightarrow Y$ is a morphism of schemes over $S_1$, $(T,M_T)$ is a log $S_n^\sharp$-scheme with $M_T$ integral and quasi-coherent for some $n\geq 1$, $i$ is an exact closed immersion $(U,M_U\coloneqq f^\ast M_Y)\rightarrow (T,M_T)$ over $(S_n,M_{S_n})$, and $\gamma$ is a PD-structure on the ideal of $\calO_T$ defining $U$ compatible with the PD-structure on $\calJ_{S_n}\calO_T$. We often abbreviate it as $(U,T)$. The morphisms are the obvious ones and we consider the strict \'etale topology. The association sending $(U,T)$ to $\calO_T(T)$ defines the structure sheaf $\calO_{Y/\calS}$, and we have the projection morphism of topoi $u_{Y,\calS}\colon \Sh(((Y,M_Y)/\calS^\sharp)_{\CRIS})\rightarrow \Sh(Y_\et)$.
When $\calS=\Spf \Z_p$ with trivial log structure and canonical PD-ideal $(p)$, we write $(Y,M_Y)_\CRIS$ for $((Y,M_Y)/\calS^\sharp)_{\CRIS}$.

We define the additive functor
\[
\alpha\colon \Ab(((Y,M_Y)/\calS^\sharp)_{\CRIS})\rightarrow \Ab(((Y,M_Y)/(A,M_A))_{\CRIS})
\]
as in \cite[Construction~10.6]{du-moon-shimizu-cris-pushforward}. 
For $(B,M_{\Spf B},J,f)\in ((Y,M_Y)/(A,M_A))_{\CRIS}^\op$, let 
$\overline{B}\coloneqq B/J$ and $B_n\coloneqq B/p^nB$. Then $(\Spec \overline{B}, \Spec B_n, M_{\Spec B_n}, f, \Spec \overline{B} \hookrightarrow \Spec B_n,\gamma)$ gives an object of $((Y,M_Y)/S_n^\sharp)_{\CRIS}$, for which we write $(\Spec \overline{B},\Spec B_n)$ for simplicity. We let $(\Spec \overline{B},\Spf B)\in \Sh(((Y,M_Y)/\calS^\sharp)_{\CRIS})$ denote the colimit over $n$ of the sheaves represented by $(\Spec \overline{B},\Spec B_n)$. For $\calF\in \Ab(((Y,M_Y)/\calS^\sharp)_{\CRIS})$, define $\alpha(\calF)$ by
\[
\alpha(\calF)\colon (B,M_{\Spf B},J,f)\mapsto \calF(\Spec \overline{B},\Spf B)=\varprojlim_n \calF(\Spec \overline{B},\Spec B_n).
\]
Since $\alpha(\calO_{Y/\calS})=\calO_\CRIS$, it yields $\alpha\colon \Mod(\calO_{Y/\calS})\rightarrow \Mod(\calO_\CRIS)$. Obviously, $\alpha$ is compatible with \'etale localization on $Y$.
\end{defn}

\begin{prop}\label{prop:comparison of crystals in two crystalline sites}
Assume that $M_A$ is fine and $\pi\colon (Y,M_Y)\rightarrow (\Spec A/I, M_A^a)$ is a smooth and integral morphism between fine log schemes. 
Then $\alpha$ induces an equivalence between the category of quasi-coherent $\calO_{Y/\calS}$-modules on $((Y,M_Y)/\calS^\sharp)_\CRIS$ and that of $p$-adically completed crystals on $((Y,M_Y)/(A,M_A))_\CRIS$.
Moreover, for a quasi-coherent $\calO_{Y/\calS}$-module $\calE$, there are functorial isomorphisms $Ru_{Y/\calS,\ast}\calE\cong Ru_{Y,\ast}\alpha$ in $D(Y_\et,A)$ and $R\Gamma(((Y,M_Y)/\calS^\sharp)_\CRIS,\calE)\cong R\Gamma(((Y,M_Y)/(A,M_A))_\CRIS,\calE)$ in $D(A)$
\end{prop}

\begin{proof}
Let $((Y,M_Y)/\calS^\sharp)_{\CRIS}^\aff\subset ((Y,M_Y)/\calS^\sharp)_{\CRIS}$ denote the full subcategory consisting of $(U,T)$ with $T$ being affine. Then the inclusion is a special cocontinuous functor (see \cite[03CG]{stacks-project}) and thus induces an equivalence on the associated topoi. We may and do regard $\alpha$ as a functor from $\Ab(((Y,M_Y)/\calS^\sharp)_{\CRIS}^\aff)$ in the rest of the proof.

For the first assertion, we may assume that $Y=\Spec R$ is affine and $M_Y$ admits a chart over $M_A$ since $\alpha$ is compatible with \'etale localization.
Then one can lift $(Y,M_Y)\rightarrow (\Spec A/I, M_A^a)$ to a morphism $(\calY,M_{\calY})=(\Spf \widetilde{R}, M_{\Spf \widetilde{R}})\rightarrow (\calS,M_{\calS})$ of affine $p$-adic log formal schemes such that $(Y_n,M_{Y_n})\coloneqq(\Spec \widetilde{R}/p^n, M_{\Spec \widetilde{R}/p^n})\rightarrow (S_n,M_{S_n})$ is smooth and integral for every $n\geq 1$ (see \cite[Prop.~3.14, Prop.~4.1]{Kato-log}). 
Write $(\calY^{[m]},M_{\calY^{[m]}})$ for the $(m+1)$-st self-fiber product of $(\calY,M_{\calY})$ over $(\calS,M_{\calS})$ and let $(Y,M_Y)\hookrightarrow(\calD^{[m]},M_{\calD^{[m]}})$ be the $p$-adically completed PD-envelope of the closed immersion $(Y,M_Y)\hookrightarrow(\calY^{[m]},M_{\calY^{[m]}})$ relative to $\calS^\sharp$. Then $\calD^{[m]}=\Spf B^{[m]}$ is affine (note $B^{[0]}=\widehat{R}$) and $B^{[m]}/p^n$ is flat over $A/p^n$ for every $n\geq 1$ by \cite[Cor.~4.5, Prop.~6.5]{Kato-log}. In particular, $B^{[m]}$ is a $p$-torsion free $p$-complete ring, and if we set $J^{[m]}\coloneqq \Ker(B^{[m]}\rightarrow R)$, then $(B^{[m]}, M_{\calD^{[m]}}, J^{[m]},\pi)$ defines an object of $((Y,M_Y)/\calS^\sharp)_{\CRIS}^\op$. We know from \cite[Prop.~IV.3.1.4.2]{Ogus-log} that the sheaf represented by $(B^{[0]}, M_{\calD^{[0]}}, J^{[0]},\pi)$ surjects onto the final object of the associated topos as morphisms of presheaves and $(B^{[m]}, M_{\calD^{[m]}}, J^{[m]},\pi)$ represents its $(m+1)$-st product as sheaves. Similarly, if we set $D^{[m]}_n\coloneqq \Spec B^{[m]}/p^n$, analogous properties hold for $(Y,D^{[0]}_n), (Y, D^{[m]}_n)\in ((Y,M_Y)/S_n^\sharp)_\CRIS^\aff$.

For a quasi-coherent $\calO_{Y/\calS}$-module $\calE$, set $E^{[m]}\coloneqq \alpha(\calE)(B^{[m]}, M_{\calD^{[m]}}, J^{[m]},\pi)=\varprojlim_n \calE(Y, D^{[m]}_n)$ for every $m\geq 1$. Note that each $E^{[m]}$ is a $p$-adically complete $B^{[m]}$-module. It follows easily from the preceding paragraph and the standard discussion on the HPD-stratification that both $\calE$ and $\alpha(\calE)$ determine and are determined by the $2$-truncated cosimplicial module $(E^{[\bullet]})_{0\leq \bullet\leq 2}$ over the $2$-truncated cosimplicial ring $(B^{[\bullet]})_{0\leq \bullet\leq 2}$ such that the induced map $E^{[\bullet]}\otimes_{B^{[\bullet]}}B^{[\bullet']}\rightarrow E^{[\bullet']}$ is an isomorphism for every map $[\bullet]\rightarrow [\bullet']$ (cf.~\cite[Prop.~6.8]{du-moon-shimizu-cris-pushforward}). One can similarly see that the category of $p$-adically completed crystals on $((Y,M_Y)/(A,M_A))_\CRIS$ is equivalent to the same category. The first assertion follows from these observations. 

For the second assertion, let us first construct a comparison map.
Let $n\geq 1$. The inclusion $((Y,M_Y)/\calS_n^\sharp)_\CRIS^\aff \hookrightarrow ((Y,M_Y)/\calS^\sharp)_\CRIS^\aff$ and the functor $((Y,M_Y)/(A,M_A))_{\CRIS} \rightarrow ((Y,M_Y)/\calS_n^\sharp)_\CRIS^\aff$ sending $(B,M_{\Spf B},J,f)$ to $(\Spec \overline{B},\Spec B_n)$ are both continuous and cocontinuous. They give rise to morphisms of topoi, which appear in the following commutative diagram of topoi
\[
\xymatrix{
& \Sh((Y,M_Y)/\calS_n^\sharp)_\CRIS^\aff)\ar[ld]_-{i_n}\ar[d]^-{u_{Y/\calS_n}} & \\
\Sh(((Y,M_Y)/\calS^\sharp)_{\CRIS}^\aff)\ar[r]^-{u_{Y/\calS}}& \Sh(Y_\et) & \Sh(((Y,M_Y)/(A,M_A))_{\CRIS})\ar[l]_-{u_{Y}}\ar[lu]_-{\beta_n}.
}
\]
Here one sees the commutative by verifying $i_n^\ast\circ u_{Y/\calS}=u_{Y\calS_n}$ and $\beta_n^\ast\circ u_{Y/\calS_n}^\ast=u_{Y}^\ast$.
Let $\calE$ be a quasi-coherent $\calO_{Y/\calS}$-module. By \cite[Prop.~10.5]{du-moon-shimizu-cris-pushforward}, we have $Ru_{Y/\calS,\ast}\calE\xrightarrow{\cong}R\varprojlim_n Ru_{Y/\calS_n,\ast}i_n^\ast\calE$ in $D(Y_\et,A)$\footnote{In \textit{loc.~cit.}, we prove an isomorphism in $D(Y_\et,Z)$ but the same proof works for the $A$-modules.}. On the other hand, we compute $\beta_n^\ast i_n^\ast \calE=\alpha(\calE)_n$ and thus $\alpha(\calE)\xrightarrow{\cong}R\varprojlim_n \beta_n^\ast i_n^\ast \calE$ by Lemma~\ref{lem: affine vanishing on crystalline site}. Hence $Ru_{Y,\ast}\alpha(\calE)\xrightarrow{\cong}R\varprojlim_n Ru_{Y,\ast} \beta_n^\ast i_n^\ast \calE$ by \cite[0A07]{stacks-project}. It is easy to see that the adjunction maps $i_n^\ast \calE \rightarrow R\beta_{n,\ast}\beta_n^\ast i_n^\ast \calE$ induce a comparison map
\begin{align*}
Ru_{Y/\calS,\ast}\calE\xrightarrow{\cong}R\varprojlim_n Ru_{Y/\calS_n,\ast}i_n^\ast\calE  \rightarrow & R\varprojlim_n Ru_{Y/\calS_n,\ast}R\beta_{n,\ast}\beta_n^\ast i_n^\ast \calE \\
&\xrightarrow{\cong}R\varprojlim_n Ru_{Y,\ast} \beta_n^\ast i_n^\ast \calE\xleftarrow{\cong}Ru_{Y,\ast}\alpha(\calE).    
\end{align*}
By further applying $R\Gamma(Y_\et,-)$, we also obtain 
\[
R\Gamma(((Y,M_Y)/\calS^\sharp)_\CRIS^\aff,\calE)\rightarrow R\Gamma(((Y,M_Y)/(A,M_A))_\CRIS,\calE)\quad \text{in} \quad D(A).
\]
To show that these are isomorphisms, it is enough to show that the map  
\[
Ru_{Y/\calS_n,\ast}i_n^\ast\calE  \rightarrow Ru_{Y/\calS_n,\ast}R\beta_{n,\ast}\beta_n^\ast i_n^\ast \calE=Ru_{Y,\ast} \beta_n^\ast i_n^\ast\calE
\]
is an isomorphism for each $n$, which can be checked \'etale locally on $Y$. More precisely, it is enough to show that for each \'etale morphism $Y'=\Spec R\rightarrow Y$ from an affine scheme, the induced map 
\begin{align*}
R\Gamma(((Y',M_{Y'})/\calS_n^\sharp &)_\CRIS^\aff,i_n^\ast\calE)=R\Gamma(Y',Ru_{Y/\calS_n,\ast}i_n^\ast\calE)\\ 
&\rightarrow R\Gamma(Y',Ru_{Y,\ast} \beta_n^\ast i_n^\ast\calE)= R\Gamma(((Y',M_{Y'})/(A,M_A))_\CRIS,\beta_n^\ast i_n^\ast\calE)
\end{align*}
is an isomorphism. For this, we may assume that $Y=Y'=\Spec R$ and consider the objects introduced in the second paragraph of the current proof. It follows from the \v{C}ech--Alexandar method (Lemma~\ref{lem: affine vanishing on crystalline site}, its analogue for $((Y,M_Y)/\calS_n)_\CRIS^\aff$, and \cite[p.~1184, fn.~10]{bhatt-scholze-prismaticcohom}) that they are naturally represented by the cosimplicial module $\calE(Y,D_n^{[\bullet]})$, which completes the proof.
\end{proof}

\begin{rem}
It is straightforward to see the following claims from the proof of Proposition~\ref{prop:comparison of crystals in two crystalline sites}:
\begin{enumerate}
 \item We may relax the assumptions in Proposition~\ref{prop:comparison of crystals in two crystalline sites}: for example, it continues to hold if $(Y,M_Y)\rightarrow (\Spec A/I,N_A^a)\rightarrow (\Spf A,M_A^a)$ is the base change along $(A_0,M_{A_0})\rightarrow (A,M_A)$ of morphisms $(Y_0,M_{Y_0})\rightarrow (\Spec A_0/I_0,N_{A_0}^a)\rightarrow (\Spf A_0,M_{A_0}^a)$ satisfying the assumptions.
 \item 
 One could formulate and show that the equivalence in Proposition~\ref{prop:comparison of crystals in two crystalline sites} (in a derived sense) is compatible with $Rh_{\CRIS,\ast}$; for example, this is the case if $h\colon (Y,M_Y)\rightarrow (\Spec A/I,M_A^a)$ is the reduced or mod $p$ fiber of a proper semistable formal scheme $X\rightarrow \Spf \calO_K$ (use a \v{C}ech--Alexander method similar to the one in \S~\ref{sec: semistable-case}).
\end{enumerate}
\end{rem}

\subsection{The \texorpdfstring{$\delta_\log$}{deltalog}-crystalline site} \label{sec: delta-log-cris-site}

In this subsection, we will review the $\delta_\log$-crystalline site considered in \cite[\S~6.1]{koshikawa}. Unless otherwise specified, we let $(A, (p), M_A)$ denote a bounded prelog prism with integral $M_A$, such that it is of rank $1$, or $(A, M_A)$ is a log ring. Let $I\subset A$ be a PD-ideal containing $p$. Let $(Y, M_Y)$ be a smooth log scheme over $(A/I, M_A)$ of Cartier type (cf.~\cite[Def.~4.8]{Kato-log}).

\begin{defn}[{\cite[Def.~6.4]{koshikawa}}]
Write $((Y, M_Y)/(A, M_A))_{\delta\text{-}\CRIS}$ to be the $\delta_{\log}$-\textit{crystalline site}, which is defined as the opposite category whose objects consist of
\begin{itemize}
\item a log prism $(B, (p), M_{\text{Spf}(B)}) = (B, (p), M_B)^a$ associated to a prelog prism $(B, (p), M_B)$ over $(A, (p), M_A)$ with integral log structure, and a $p$-completed PD ideal $J \subset B$ such that $B/J$ is classically $p$-complete (such object $(B, J, M_B)^a$ is called a $\delta_{\log}$-PD triple over $(A, M_A)$),
\item a map of $p$-adic formal schemes $f: \text{Spf}(B/J) \to Y$ over $A$, and
\item an exact closed immersion of log $p$-adic formal schemes
\[
(\text{Spf}(B/J), f^* M_Y) \hookrightarrow (\text{Spf}(B), M_{\text{Spf}(B)})
\]
over $(A, M_A)$.
\end{itemize}

Morphisms are those that preserve all structures. We endow it with the \'etale topology. The functor $\mathcal{O}_{\delta\text{-}\CRIS}$ sending $(B, J, M_B)^a$ to $B$ is indeed a sheaf and called the structure sheaf.  
\end{defn}

When $(A,I)=(\Z_p,(p))$, where $(p)$ is with the canonical PD-structure, and $A$ is equipped it with the trivial log structure, we will simply write $((Y,M_Y)/(A,M_A))_{\dCRIS}$ as $(Y,M_Y)_{\dCRIS}$ and call it the \textit{absolute} $\delta_\log$-crystalline site. The $\delta$-structure defines a natural endomorphism $\phi=\phi_{\delta\text{-}\CRIS}$ on the structure sheaf $\mathcal{O}_{\delta\text{-}\CRIS}$.

There are two natural functions associated with $((Y, M_{Y})/(A,M_A))_{\delta\text{-CRIS}}$ that play a crucial role in the comparison theorems discussed in the body of this paper.

\begin{construction}[{\cite[Rem.~6.6]{koshikawa}}]\label{const: delta-cris to affine cris}
Forgetting the $\delta$- and $\delta_\log$-structures defines a functor $((Y, M_{Y})/(A,M_A))_{\delta\text{-}\CRIS} \to ((Y, M_{Y})/(A,M_A))_{\CRIS}$. It is cocontinuous and induces a morphism of topoi
\[
\nu_\CRIS\coloneqq \nu_{Y/(A,M_A),\CRIS}\colon \Sh(((Y, M_{Y})/(A,M_A))_{\delta\text{-}\CRIS}) \to \Sh(((Y, M_{Y})/(A,M_A))_{\CRIS}), 
\]
Note $\nu_{\CRIS}^{-1}\calO_\CRIS=\calO_{\dCRIS}$. 
Define the \textit{projection to \'etale topos} functor
\[
u^\delta_Y\coloneqq u^\delta_{Y/(A,M_A)}\colon\Sh(((Y,M_{Y})/(A,M_A))_{\delta\text{-}\CRIS})\rightarrow \Sh(Y_\et)
\]
to be the composite $u_{Y/(A,M_A)}\circ \nu_{Y/(A,M_A),\CRIS}$ where $u_{Y/(A,M_A)}$ is defined in Proposition~\ref{prop: proj to etale for cris topos}. 

As
$\nu_\CRIS^\ast\colon \Sh(((Y, M_{Y})/(A,M_A))_{\CRIS},\calO_\CRIS)\rightarrow \Sh(((Y,M_{Y})/(A,M_A))_{\delta\text{-}\CRIS},\calO_{\delta\text{-}\CRIS})$ is exact, we have an adjunction morphism
\[
Ru_{Y/(A,M_A),\ast}\calF_\CRIS\rightarrow Ru_{Y/(A,M_A),\ast}^\delta \nu_\CRIS^\ast\calF_\CRIS
\]
for every $\calO_\CRIS$-module $\calF_\CRIS$.
In Proposition~\ref{prop:comparison-delta-log-cris-log-cris}, we show that this is a quasi-isomorphism if $Y$ is smooth of Cartier type over $(A/I,M_A)$ and $\calF_\CRIS$ is a $p$-adically completed crystal.
\end{construction}

\begin{construction}\label{const: f^J_Prism}
There is a cocontinuous functor
\[
f_{\Prism/(A,M_A)}\colon((Y,M_{Y})/(A,M_A))_{\delta\text{-CRIS}} \to ((Y^{(1)},M_Y^{(1)})/(\phi_\ast A, \phi_\ast M_A))_{\Prism,\et}
\]
defined at the beginning of \cite[\S~6.2]{koshikawa}. Here $(\phi_\ast A, \phi_\ast M_A)$ denotes $(A, M_A)$ regarded as a prelog ring over itself by $\phi$. The Frobenius endomorphism $\phi$ on the prelog ring $(A/(p), M_A)$ factors as
$(A/p, M_A)\rightarrow (A/I, M_A)\xrightarrow{\psi} (\phi_\ast A/p,\phi_\ast M_A)$, 
and we set $(Y^{(1)}, M_{Y^{(1)}})\coloneqq (Y, M_Y)\times_{(\Spec A/I,M_A^a),\psi}(\Spec \phi_\ast A/p,(\phi_\ast M_A)^a)$. 

Let us recall the construction in the case where $(Y,M_Y)$ is associated to a prelog ring $(R,P)$ over $(A,M_A)$ with $P$ integral. Take $(B,J,M_B)^a\in ((Y,M_{Y})/(A,M_A))_{\delta\text{-CRIS}}^\op$ such that $(B,M_B)$ is a log ring.
Use similar notations for $(B,J,M_B)$ and consider the pushout $M_B \to M_B^{(1)}$ of $M_A \xrightarrow{\phi_{M_A}} \phi_\ast M_A$ along $M_A \to M_B$ in the category of monoids. It defines  morphisms $(B,J,M_B) \xrightarrow{\psi_B} (\phi_\ast B,(p),M_B^{(1)})\rightarrow (\phi_\ast B,\phi_\ast M_B)$ of prelog rings. 

Write $f_B\colon (\Spec(B/J), M_B^a) \to (Y,M_Y)$ for the structure map and observe that the map $(\Spec(\phi_\ast B/p),(\phi_{\ast} M_B)^a) \xrightarrow{\psi_B} (\Spec(B/J), M_B^a) \xrightarrow{f_B} (Y,M_Y)$ factors through $Y^{(1)}$ over $\phi_\ast A/p$ and defines an object $(\phi_\ast B, (p), M_B^{(1)})^{a} \in (Y^{(1)},M_Y^{(1)})_\Prism$. This construction is functorial and defines the desired $f_{\Prism/(A,M_A)}$. It induces a morphism of ringed topoi
\[
\nu_\Prism \coloneqq (\nu_\Prism^{-1},\nu_{\Prism,\ast})\colon \Sh(((Y,M_Y)/(A,M_A))_{\dCRIS}) \rightarrow \Sh(((Y^{(1)},M_Y^{(1)})/(\phi_\ast A, \phi_\ast M_A)^a)_{\Prism,\et}).
\]
By construction, we have $\nu_\Prism^{-1}\calO_{\Prism} = \phi_{\ast} \calO_{\dCRIS}$.
\end{construction}

\bibliographystyle{amsalpha}
\bibliography{library}

\end{document}